\documentclass[12pt,reqno,twoside]{amsart}%

\usepackage{amsmath,amsthm,amscd,amsthm,upref,indentfirst}
\usepackage{amsfonts,mathrsfs}
\usepackage{amssymb,amsbsy,bm}
\usepackage{graphicx}
\usepackage{times}
\usepackage[foot]{amsaddr}
\usepackage{soul}
\usepackage{slashbox}

\usepackage[all]{xy}
\xyoption{tips}
\SelectTips {xy} {12}
\usepackage[UglyObsolete,small,nohug,heads=LaTeX]{diagrams}
\diagramstyle[labelstyle=\scriptstyle]
\usepackage[usenames]{color}

\theoremstyle{plain}
\newtheorem{theorem}{Theorem}
\newtheorem{assertion}[theorem]{Assertion}
\newtheorem{lemma}[theorem]{Lemma}
\newtheorem{proposition}[theorem]{Proposition}

\theoremstyle{definition}
\newtheorem{definition}[theorem]{Definition}
\newtheorem{condition}[theorem]{Construction}
\newtheorem{conjecture}[theorem]{Conjecture}
\newtheorem{corollary}[theorem]{Corollary}

\theoremstyle{remark}
\newtheorem{remark}[theorem]{Remark}

\newtheorem{example}[theorem]{Example}
\numberwithin{equation}{section}
\numberwithin{theorem}{section}

\usepackage[scaled=0.86]{helvet}
\renewcommand{\mathfrak}[1]{{\textbf{\upshape #1}}}
\renewcommand{\mathbf}{\bm}
\renewcommand{\mathrm}[1]{\scalebox{1.15}{\textsf{\upshape #1}}}
\renewcommand{\emph}[1]{\textrm{{\upshape #1}}}
\usepackage[scr=euler,scrscaled=1.2,bb=txof,bbscaled=1.16,cal=boondox,calscaled=1.1]{mathalfa}
\renewcommand{\mathit}[1]{\mathscr #1}
\renewcommand{\mathtt}[1]{\scalebox{.9}{\bf \texttt{ \upshape #1}}}
\usepackage{exscale}

\numberwithin{equation}{section}
\numberwithin{theorem}{section}

\usepackage[normalem]{ulem}
\usepackage[square,comma]{natbib}

\usepackage{mparhack}

\usepackage{fancyhdr}
\author{\scshape S\lowercase{teven} D\lowercase{uplij}}

\address{\noindent \itshape Address\textrm{:} \tiny \scshape
Universit\"at M\"unster,
Einsteinstrasse 62,
D-48149 M\"unster,
Deutschland}
\email{\small \sf douplii@uni-muenster.de; http://ivv5hpp.uni-muenster.de/u/douplii}
\title[Polyadic Hopf algebras and quantum groups
]{\textbf{Polyadic Hopf algebras and quantum groups}}

\date{November 6, 2018}

\renewcommand{\refname}{\textsc{References}}

\let\origsection\section
\renewcommand{\section}[1]{\sectionmark{#1}\origsection{#1}}
\let\origsubsection\subsection
\renewcommand{\subsection}[1]{\subsectionmark{#1}\origsubsection{#1}}

\makeatletter
\renewenvironment{thebibliography}[1]{%
  \@xp\origsection\@xp*\@xp{\refname}%
  \normalfont\footnotesize\labelsep .9em\relax
  \renewcommand\theenumiv{\arabic{enumiv}}\let\p@enumiv\@empty
  \vspace*{-5pt}
  \list{\@biblabel{\theenumiv}}{\settowidth\labelwidth{\@biblabel{#1}}%
    \leftmargin\labelwidth \advance\leftmargin\labelsep
    \usecounter{enumiv}}%
  \sloppy \clubpenalty\@M \widowpenalty\clubpenalty
  \sfcode`\.=\@m
}{%
  \def\@noitemerr{\@latex@warning{Empty `thebibliography' environment}}%
  \endlist
}
\makeatother

\begin{document}

\mbox{}

\begin{abstract}
\noindent
This article continues the study of concrete algebra-like structures in our polyadic approach, where the arities of all operations are initially taken as arbitrary, but the relations between them, the arity shapes, are to be found from some natural conditions  (``arity freedom principle''). In this way, generalized associative algebras, coassociative coalgebras, bialgebras and Hopf algebras are defined and investigated. They have many unusual features in comparison with the binary case. For instance, both the algebra and its underlying field can be zeroless and nonunital, the existence of the unit and counit is not obligatory, and the dimension of the algebra is not arbitrary, but ``quantized''. The polyadic convolution product and bialgebra can be defined, and when the algebra and coalgebra have unequal arities, the polyadic version of the antipode, the querantipode, has different properties. As a possible application to quantum group theory, we introduce the polyadic version of braidings, almost co-commutativity, quasitriangularity and the equations for the $R$-matrix (which can be treated as a polyadic analog of the Yang-Baxter equation). Finally, we propose another concept of deformation which is governed not by the twist map, but by the medial map, where    only the latter is unique in the polyadic case. We present the corresponding braidings, almost co-mediality and $M$-matrix, for which the compatibility equations are found.

\end{abstract}

\maketitle
\thispagestyle{empty}

\begin{small}
\tableofcontents
\end{small}%

\pagestyle{fancy}
\addtolength{\footskip}{15pt}

\renewcommand{\sectionmark}[1]{%
\markboth{\textmd{
\  \thesection.} { \scshape #1}}{}}
\renewcommand{\subsectionmark}[1]{%
\markright{
\mbox{\;}\\[5pt]
\textmd{#1}}{}}
\fancyhead{}
\fancyhead[EL,OR]{\leftmark}
\fancyhead[ER,OL]{\rightmark}
\fancyfoot[C]{\scshape - \thepage \ -}
\renewcommand\headrulewidth{0.5pt}
\fancypagestyle {plain1}{ %
\fancyhf{}
\renewcommand {\headrulewidth }{0pt}
\renewcommand {\footrulewidth }{0pt}
}

\fancypagestyle{plain}{ %
\fancyhf{}
\fancyhead[C]{\scshape S\lowercase{teven} D\lowercase{uplij} \hskip 0.7cm \MakeUppercase{Polyadic Hopf algebras and quantum groups}}
\fancyfoot[C]{\scshape - \thepage \ -}
\renewcommand {\headrulewidth }{0pt}
\renewcommand {\footrulewidth }{0pt}
}

\fancypagestyle{fancyref}{ %
\fancyhf{} 
\fancyhead[C]{\scshape R\lowercase{eferences} }
\fancyfoot[C]{\scshape - \thepage \ -}
\renewcommand {\headrulewidth }{0.5pt}
\renewcommand {\footrulewidth }{0pt}
}

\fancypagestyle{emptyf}{
\fancyhead{}
\fancyfoot[C]{\thepage}
\renewcommand{\headrulewidth}{0pt}
}

\mbox{}

\bigskip

\thispagestyle{plain}

\mbox{}

\section{\textsc{Introduction}}

Since Hopf algebras were introduced in connection with algebraic topology
\cite{swe,abe}, their role has increased significantly (see, e.g., \cite{radford}),
with numerous applications in diverse areas, especially in relation to quantum groups
\cite{dri0,shn/ste,cha/pre,kassel,majid}. There have been many generalizations of
Hopf algebras (for a brief review, see, e.g., \cite{kar2008}).

From another perspective, the concepts of polyadic vector space, polyadic
algebras and polyadic tensor product over general polyadic fields were
introduced in \cite{dup2017}. They differ from the standard definitions
of $n$-ary algebras \cite{azc/izq,mil/vin,goz/rau} in considering an
arbitrary arity shape for all operations, and not the algebra multiplication
alone. This means that the arities of addition in the algebra, the multiplication
and addition in the underlying field can all be different from binary and the
number of places in the multiaction (polyadic module) can be more than one
\cite{dup2018a}. The connection between arities is determined by their arity
shapes \cite{dup2017} (``arity freedom principle'').
Note that our approach is somewhat different from the
operad approach (see, e.g., \cite{mar/shn/sta,lod/val}).

Here we propose a similar and consequent polyadic generalization of Hopf algebras.
First, we define polyadic coalgebras and study their homomorphisms and tensor
products. In the construction of the polyadic convolution product and
bialgebras we propose considering different arities for the algebra and coalgebra,
which is a crucial difference from the binary case. Instead of the antipode,
we introduce its polyadic version, the querantipode, by analogy with the
querelement in $n$-ary groups \cite{dor3}. We then consider polyadic analogs
of braidings, almost co-commutativity and the $R$-matrix, together with the
quasitriangularity equations. This description is not unique, as with the
polyadic analog of the twist map, while the medial map is unique for all
arities. Therefore, a new (unique) concept of deformation is proposed: almost
co-mediality with the corresponding $M$-matrix. The medial analogs of braidings
and quasitriangularity are introduced, and the equations for $M$-matrix are obtained.

\section{\label{sec-polf}\textsc{Polyadic fields and vector spaces}}

Let $\Bbbk=\Bbbk^{\left(  m_{k},n_{k}\right)  }=\left\langle K\mid\nu
_{k}^{\left(  m_{k}\right)  },\mu_{k}^{\left(  n_{k}\right)  }\right\rangle $
be a polyadic or $\left(  m_{k},n_{k}\right)  $-ary field with $n_{k}$-ary
multiplication $\mu_{k}^{\left(  n_{k}\right)  }:K^{n_{k}}\rightarrow K$ and
$m_{k}$-ary addition $\nu_{k}^{\left(  m_{k}\right)  }:K^{m_{k}}\rightarrow K$
which are (polyadically) associative and distributive, such that $\left\langle
K\mid\mu^{\left(  n_{k}\right)  }\right\rangle $ and $\left\langle K\mid
\nu^{\left(  m_{k}\right)  }\right\rangle $ are both commutative polyadic
groups \cite{cro2,lee/but}. This means that $\mu_{k}^{\left(  n_{k}\right)
}=\mu_{k}^{\left(  n_{k}\right)  }\circ\tau_{n_{k}}$ and $\nu_{k}^{\left(
m_{k}\right)  }=\nu_{k}^{\left(  m_{k}\right)  }\circ\tau_{m_{k}}$, where
$\tau_{n_{k}}\in\mathrm{S}_{n_{k}}$, $\tau_{m_{k}}\in\mathrm{S}_{m_{k}}$, and
$\mathrm{S}_{n_{k}},\mathrm{S}_{m_{k}}$ are the symmetry permutation groups. A
polyadic field $\Bbbk^{\left(  m_{k},n_{k}\right)  }$ is \textit{derived}, if
$\mu_{k}^{\left(  n_{k}\right)  }$ and $\nu_{k}^{\left(  m_{k}\right)  }$ are
iterations of the corresponding binary operations: ordinary multiplication and
addition. The polyadic fields considered in \cite{lee/but} were derived. The
simplest example of a nonderived $\left(  2,3\right)  $-ary field is
$\Bbbk^{\left(  2,3\right)  }=\left\{  i\mathbb{R}\right\}  $, and of a
nonderived $\left(  3,3\right)  $-ary field is $\Bbbk^{\left(  3,3\right)
}=\left\{  ip/q\right\}  $, where $p,q\in\mathbb{Z}^{odd}$ ($i^{2}=-1$, and
the operations are in $\mathbb{C}$). Polyadic analogs of prime Galois fields
including nonderived ones were presented in \cite{dup2017a}.

Recall that a \textit{polyadic zero} $z$ in any $\left\langle X\mid
\nu^{\left(  m\right)  }\right\rangle $ (with $\nu^{\left(  m\right)  }$ being
an \textsf{addition-like} operation) is defined (if it exists) by %
\begin{equation}
\nu^{\left(  m\right)  }\left[  \mathbf{\hat{x}},z\right]  =z,\ \ \ \forall
\mathbf{\hat{x}}\in X^{m-1}, \label{z}%
\end{equation}
where $z$ can be on any place, and $\mathbf{x}$ is any \textit{polyad} of
length $m-1$ (as a sequence of elements) in $X$. A \textit{polyadic unit} in any
$\left\langle X\mid\mu^{\left(  n\right)  }\right\rangle $ (with $\mu$ being a
\textsf{multiplication-like} operation) is an $e\in X$ (if it exists) such that %
\begin{equation}
\mu^{\left(  n\right)  }\left[  e^{n-1},x\right]  =x,\ \ \ \forall x\in X,
\label{e}%
\end{equation}
where $x$ can be on any place, and the repeated entries in a polyad are denoted by a
power $\overset{n}{\overbrace{x,\ldots,x}}\equiv x^{n}$. It follows from
(\ref{e}), that for $n\geq3$ the polyad $\mathbf{e}$ can play the role of a
unit, and is called a neutral sequence \cite{usa}%
\begin{equation}
\mu^{\left(  n\right)  }\left[  \mathbf{\hat{e}},x\right]  =x,\ \ \ \forall
x\in X,\ \ \mathbf{\hat{e}}\in X^{n-1}. \label{me}%
\end{equation}
This is a crucial difference from the binary case, as the neutral
sequence $\mathbf{\hat{e}}$ can (possibly) be nonunique.

The nonderived polyadic fields obey unusual properties: they can have several
(polyadic) units or no units at all (\textit{nonunital}, as in $\Bbbk^{\left(
2,3\right)  }$ and $\Bbbk^{\left(  3,3\right)  }$ above), no (polyadic) zeros
(\textit{zeroless}, as $\Bbbk^{\left(  3,3\right)  }$ above), or they can consist
of units only (for some examples, see \cite{dup/wer,dup2017a}). This may lead, in general, to new features of the algebraic
structures using the polyadic fields as the underlying fields (e.g. scalars
for vector spaces, etc.) \cite{dup2017}.

Moreover, polyadic invertibility is not connected with units, but is
governed by the special element, analogous to an inverse, the so called
\textit{querelement} $\bar{x}$, which for any $\left\langle X\mid\mu^{\left(
n\right)  }\right\rangle $ is defined by \cite{dor3}%
\begin{equation}
\mu^{\left(  n\right)  }\left[  x^{n-1},\bar{x}\right]  =x,\ \ \ \forall x\in
X, \label{q}%
\end{equation}
where $\bar{x}$ can be on any place (instead of the binary inverse
\textquotedblleft$xx^{-1}=e$\textquotedblright). An element $x\in X$ for
which (\ref{q}) has a solution under $\bar{x}$ is called \textit{querable} or
\textquotedblleft polyadically invertible\textquotedblright. If all elements
in $X$ are querable, and the operation $\mu^{\left(  n\right)  }$ is
polyadically associative, then $\left\langle X\mid\mu^{\left(  n\right)
}\right\rangle $ is a $n$-ary group. \textit{Polyadic associativity} in
$\left\langle X\mid\mu^{\left(  n\right)  }\right\rangle $ can be defined as a
kind of invariance relationship \cite{dup2018a}%
\begin{equation}
\mu^{\left(  n\right)  }\left[  \mathbf{\hat{x}},\mu^{\left(  n\right)
}\left[  \mathbf{\hat{y}}\right]  ,\mathbf{\hat{z}}\right]  =invariant,
\label{a}%
\end{equation}
where $\mathbf{\hat{x}},\mathbf{\hat{y}},\mathbf{\hat{z}}$ are polyads of the
needed size in $X$, and $\mu^{\left(  n\right)  }\left[  \mathbf{\hat{y}%
}\right]  $ can be on any place, and we therefore will not use additional
brackets. Using polyadic associativity (\ref{a}) we introduce $\ell
$-\textit{iterated multiplication} by%
\begin{equation}
\left(  \mu^{\left(  n\right)  }\right)  ^{\circ\ell}\left[  \mathbf{\hat{x}%
}\right]  =\overset{\ell}{\overbrace{\mu^{\left(  n\right)  }[\mu^{\left(
n\right)  }[\ldots\mu^{\left(  n\right)  }[}}\mathbf{\hat{x}]]]}%
,\ \ \ \ \mathbf{\hat{x}}\in X^{\ell\left(  n-1\right)  +1}, \label{mx}%
\end{equation}
where $\ell$ is \textquotedblleft number of multiplications\textquotedblright.
Therefore, the \textit{admissible} length of any $n$-ary word is \textsf{not
arbitrary}, as in the binary $n=2$ case, but fixed (\textquotedblleft
quantized\textquotedblright) to $\ell\left(  n-1\right)  +1$.

\begin{example}
\label{ex-tern}Consider the nonunital zeroless polyadic field $\Bbbk^{\left(
3,3\right)  }=\left\{  ip/q\right\}  $, $i^{2}=-1$, $p,q\in\mathbb{Z}^{odd}$
(from the example above). Both the ternary addition $\nu^{\left(  3\right)
}\left[  x,y,t\right]  =x+y+t$ and the ternary multiplication $\mu^{\left(
3\right)  }\left[  x,y,t\right]  =xyt$ are nonderived, ternary associative and
distributive. For each $x=ip/q$ ($p,q\in\mathbb{Z}^{odd}$) the
\textit{additive} \textit{querelement} (denoted by a wave, a ternary analog of an
inverse element with respect to addition) is $\tilde{x}=-ip/p^{\prime}$, and
the \textit{multiplicative querelement} is $\bar{x}=-iq/p$ (see (\ref{q})).
Therefore, both $\left\langle \left\{  ip/q\right\}  \mid\mu^{\left(
3\right)  }\right\rangle $ and $\left\langle \left\{  ip/q\right\}  \mid
\nu^{\left(  3\right)  }\right\rangle $ are ternary groups (as it should be
for a $\left(  3,3\right)  $-field), but they contain \textsf{no neutral
elements} (unit or zero).
\end{example}

The polyadic analogs of vector spaces and tensor products were introduced in
\cite{dup2017}. Briefly, consider a set $V$ of \textquotedblleft polyadic
vectors\textquotedblright\ with the addition-like $m_{v}$-ary operation
$\nu_{V}^{\left(  m_{v}\right)  }$, such that $\left\langle V\mid\nu
_{v}^{\left(  m_{v}\right)  }\right\rangle $ is a \textsf{commutative} $m_{v}%
$-ary group. The key differences from the binary case are: 1) The zero vector
$z_{v}$ does not necessarily exist (see the above example for $\Bbbk^{\left(
3,3\right)  }$ field); 2) The role of a negative vector is played by the
additive querelement $\tilde{v}$ in $\left\langle V\mid\nu_{V}^{\left(
m_{v}\right)  }\right\rangle $ (which \textsf{does not} imply the existence of
$z_{v}$). A polyadic analog of the binary multiplication by a scalar ($\lambda
v$) is the \textit{multiaction} $\rho^{\left(  r_{\rho}\right)  }$ introduced
in \cite{dup2018a}%
\begin{equation}
\rho_{V}^{\left(  r_{v}\right)  }:K^{r_{v}}\times V\rightarrow V. \label{r}%
\end{equation}
If the unit $e_{k}$ exists in $\Bbbk^{\left(  m_{k},n_{k}\right)  }$, then
the multiaction can be normalized (analog of \textquotedblleft$1v=v$%
\textquotedblright) by%
\begin{equation}
\rho_{V}^{\left(  r_{v}\right)  }\left(  e_{k}^{\times r_{v}}\mid v\right)
=v,\ \ \ v\in V. \label{re}%
\end{equation}

Under the composition $\circ_{n_{\rho}}$ (given by the \textit{arity changing
formula} \cite{dup2018a}), the set of multiactions form a $n_{\rho}$-ary
semigroup $\mathrm{S}_{\rho}^{\left(  n_{\rho}\right)  }=\left\langle \left\{
\rho_{V}^{\left(  r_{v}\right)  }\right\}  \mid\circ_{n_{\rho}}\right\rangle
$. Its arity is less or equal than $n_{k}$ and depends on one integer
parameter (the number of intact elements in the composition), which is less than
$\left(  r_{v}-1\right)  $ (for details see \cite{dup2017}).

A \textit{polyadic} \textit{vector space} over the polyadic field
$\Bbbk^{\left(  m_{k},n_{k}\right)  }$ is%
\begin{equation}
\mathrm{V}=\mathrm{V}^{\left(  m_{v};m_{k},n_{k};r_{\rho}\right)
}=\left\langle V,K\mid\nu_{V}^{\left(  m_{v}\right)  };\nu_{k}^{\left(
m_{k}\right)  },\mu_{k}^{\left(  n_{k}\right)  };\rho_{V}^{\left(
r_{v}\right)  }\right\rangle , \label{v}%
\end{equation}
where $\left\langle V\mid\nu_{V}^{\left(  m_{v}\right)  }\right\rangle $ is a
commutative $m_{v}$-ary group, $\left\langle K\mid\mu_{k}^{\left(
n_{k}\right)  },\nu_{k}^{\left(  m_{k}\right)  }\right\rangle $ is a polyadic
field, $\left\langle \left\{  \rho_{V}^{\left(  r_{v}\right)  }\right\}
\mid\circ_{n_{\rho}}\right\rangle $ is a $n_{\rho}$-ary semigroup, the
multiaction $\rho^{\left(  r_{\rho}\right)  }$ is distributive with respect to
the polyadic additions $\nu_{V}^{\left(  m_{v}\right)  }$, $\nu_{k}^{\left(
m_{k}\right)  }$ and compatible with $\mu_{k}^{\left(  n_{k}\right)  }$ (see
(2.15), (2.16), and (2.9) in \cite{dup2017}). If instead of the underlying
field, we consider a ring, then (\ref{v}) define a \textit{polyadic module}
together with (\ref{r}). The \textit{dimension} $d_{v}$ of a polyadic vector
space is the number of elements in its polyadic basis, and we denote it
$\mathrm{V}_{d_{v}}=\mathrm{V}_{d_{v}}^{\left(  m_{v};m_{k},n_{k}%
;r_{v}\right)  }$. The polyadic direct sum and polyadic tensor product of
polyadic vector spaces were constructed in \cite{dup2017} (see (3.25) and
(3.39) there). They have an unusual peculiarity (which is not possible in
the binary case): the polyadic vector spaces of \textsf{different arities} can
be added and multiplied. The polyadic tensor product is \textquotedblleft%
$\Bbbk$-linear\textquotedblright\ in the usual sense, only instead of
\textquotedblleft multiplication by scalar\textquotedblright\ one uses the
multiaction $\rho_{V}^{\left(  r_{v}\right)  }$ (see \cite{dup2017} for
details). Because of associativity, we will use the \textsf{binary-like}
notation for polyadic tensor products (implying $\otimes=\otimes_{\Bbbk}$) and
powers of them (for instance, $\overset{n}{\overbrace{x\otimes x\otimes
\ldots\otimes x}}=x^{\otimes n}$) to be clearer in computations and as customary
in diagrams.

\section{\textsc{Polyadic associative algebras}}

Here we introduce operations on elements of a polyadic vector space, which
leads to the notion of a polyadic algebra.

\subsection{\textquotedblleft Elementwise\textquotedblright%
\ description\label{sec-el}}

Here we formulate the polyadic algebras in terms of sets and operations
written in a manifest form. The arities will be initially taken as
arbitrary, but then relations between them will follow from compatibility
conditions (as in \cite{dup2017}).

\begin{definition}
\label{def-alg-as}A \textit{polyadic (associative) algebra} (or $\Bbbk
$-algebra) is a tuple consisting of 2 sets and 5 operations%
\begin{equation}
\mathrm{A}=\mathrm{A}^{\left(  m_{a},n_{a};m_{k},n_{k};r_{a}\right)
}=\left\langle A,K\mid\nu_{A}^{\left(  m_{a}\right)  },\mu_{A}^{\left(
n_{a}\right)  };\nu_{k}^{\left(  m_{k}\right)  },\mu_{k}^{\left(
n_{k}\right)  };\rho_{A}^{\left(  r_{a}\right)  }\right\rangle , \label{al}%
\end{equation}
where:

\begin{enumerate}
\item $\Bbbk^{\left(  m_{k},n_{k}\right)  }=\left\langle K\mid\nu_{k}^{\left(
m_{k}\right)  },\mu_{k}^{\left(  n_{k}\right)  }\right\rangle $ is a polyadic
field with the $m_{k}$-ary \textit{field (scalar) addition} $\nu_{k}^{\left(
m_{k}\right)  }:K^{m_{k}}\rightarrow K$ and $n_{k}$-ary \textit{field (scalar)
multiplication} $\mu_{k}^{\left(  n_{k}\right)  }:K^{m_{k}}\rightarrow K$;

\item
\begin{equation}
\mathrm{A}_{vect}=\mathrm{A}^{\left(  m_{a};m_{k},n_{k};r_{a}\right)
}=\left\langle A,K\mid\nu_{A}^{\left(  m_{a}\right)  };\nu_{k}^{\left(
m_{k}\right)  },\mu_{k}^{\left(  n_{k}\right)  };\rho_{A}^{\left(
r_{a}\right)  }\right\rangle \label{av}%
\end{equation}
is a polyadic vector space with the $m_{a}$-ary \textit{vector addition}
$\nu_{A}^{\left(  m_{a}\right)  }:A^{n_{a}}\rightarrow A$ and the $r_{a}%
$-\textit{place} \textit{multiaction} $\rho_{A}^{\left(  r_{a}\right)
}:K^{r_{a}}\times A\rightarrow A$;

\item The map $\mu_{A}^{\left(  n_{a}\right)  }:A^{n_{a}}\rightarrow A$ is a
$\Bbbk$-linear map (\textquotedblleft\textit{vector multiplication}%
\textquotedblright) satisfying total associativity%
\begin{equation}
\mu_{A}^{\left(  n_{a}\right)  }\left[  \mathbf{\hat{a}},\mu_{A}^{\left(
n_{a}\right)  }\left[  \mathbf{\hat{b}}\right]  ,\mathbf{\hat{c}}\right]
=invariant, \label{mn}%
\end{equation}
where the second product $\mu_{A}^{\left(  n_{a}\right)  }$ can be on any
place in brackets and $\mathbf{\hat{a}},\mathbf{\hat{b}},\mathbf{\hat{c}}$ are polyads;

\item The multiacton $\rho_{A}^{\left(  r_{a}\right)  }$ is compatible with
vector and field operations $\left(  \nu_{A}^{\left(  m_{a}\right)  },\mu
_{A}^{\left(  n_{a}\right)  };\nu_{k}^{\left(  m_{k}\right)  },\mu
_{k}^{\left(  n_{k}\right)  }\right)  $.
\end{enumerate}
\end{definition}

\begin{definition}
We call the tuple $\left(  m_{a},n_{a};m_{k},n_{k};r_{a}\right)  $ an
\textit{arity shape} of the polyadic algebra $\mathrm{A}$.
\end{definition}

The compatibility of the multiaction $\rho_{A}^{\left(  r_{a}\right)  }$
(\textquotedblleft linearity\textquotedblright) consists of
\cite{dup2018a,dup2017}:

\textbf{1}) Distributivity with respect to the $m_{a}$-ary vector addition
$\nu_{A}^{\left(  m_{a}\right)  }$ (\textquotedblleft$\lambda\left(
a+b\right)  =\lambda a+\lambda b$\textquotedblright)%
\begin{align}
&  \rho_{A}^{\left(  r_{a}\right)  }\left\{  \lambda_{1},\ldots\ldots
,\lambda_{r_{a}}\mid\nu_{A}^{\left(  m_{a}\right)  }\left[  a_{1}%
,\ldots,a_{m_{a}}\right]  \right\} \nonumber\\
&  =\nu_{A}^{\left(  m_{a}\right)  }\left[  \rho_{A}^{\left(  r_{a}\right)
}\left\{  \lambda_{1},\ldots\ldots,\lambda_{r_{a}}\mid a_{1}\right\}
,\ldots,\rho_{A}^{\left(  r_{a}\right)  }\left\{  \lambda_{1},\ldots
\ldots,\lambda_{r_{a}}\mid a_{m_{a}}\right\}  \right]  . \label{c0}%
\end{align}

\textbf{2}) Compatibility with $n_{a}$-ary \textquotedblleft vector
multiplication\textquotedblright\ $\mu_{A}^{\left(  n_{a}\right)  }$
(\textquotedblleft$\left(  \lambda a\right)  \cdot\left(  \mu b\right)
=\left(  \lambda\mu\right)  \left(  a\cdot b\right)  $\textquotedblright)%
\begin{align}
&  \mu_{A}^{\left(  n_{a}\right)  }\left[  \rho_{A}^{\left(  r_{a}\right)
}\left\{  \lambda_{1},\ldots\ldots,\lambda_{r_{a}}\mid a_{1}\right\}
,\ldots,\rho_{A}^{\left(  r_{a}\right)  }\left\{  \lambda_{r_{a}\left(
n_{a}-1\right)  },\ldots\ldots,\lambda_{r_{a}n_{a}}\mid a_{n_{a}}\right\}
\right] \nonumber\\
&  =\rho_{A}^{\left(  r_{a}\right)  }\Bigg\{\overset{\ell}{\overbrace{\mu
_{k}^{\left(  n_{k}\right)  }\left[  \lambda_{1},\ldots\ldots,\lambda_{m_{k}%
}\right]  ,\ldots,\mu_{k}^{\left(  n_{k}\right)  }\left[  \lambda
_{m_{k}\left(  \ell-1\right)  },\ldots\ldots,\lambda_{m_{k}\ell}\right]  }%
},\Bigg.\nonumber\\
&  \Bigg.\lambda_{m_{k}\ell+1},\ldots\ldots,\lambda_{r_{a}n_{a}}\mid\mu
_{A}^{\left(  n_{a}\right)  }\left[  a_{1},\ldots,a_{n_{a}}\right]
\Bigg\},\label{c1}\\
&  \ell\left(  n_{k}-1\right)  =r_{a}\left(  n_{a}-1\right)  , \label{c1a}%
\end{align}
where $\ell$ is an integer, and $\ell\leq r_{a}\leq\ell\left(  n_{k}-1\right)
$, $2\leq n_{a}\leq n_{k}$.

\textbf{3}) Distributivity with respect to the $m_{k}$-ary field addition
$\nu_{k}^{\left(  m_{k}\right)  }$ (\textquotedblleft$\left(  \lambda
+\mu\right)  a=\lambda a+\mu a$\textquotedblright)%
\begin{align}
&  \rho^{\left(  r_{a}\right)  }\left\{  \overset{\ell^{\prime}}%
{\overbrace{\nu_{k}^{\left(  m_{k}\right)  }\left[  \lambda_{1},\ldots
\ldots,\lambda_{m_{k}}\right]  ,\ldots,\nu_{k}^{\left(  m_{k}\right)  }\left[
\lambda_{m_{k}\left(  \ell^{\prime}-1\right)  },\ldots\ldots,\lambda
_{m_{k}\ell^{\prime}}\right]  }},\lambda_{m_{k}\ell^{\prime}+1},\ldots
\ldots,\lambda_{r_{a}m_{a}}\mid a\right\} \nonumber\\
&  =\nu_{A}^{\left(  m_{a}\right)  }\left[  \rho^{\left(  r_{a}\right)
}\left\{  \lambda_{1},\ldots\ldots,\lambda_{r_{a}}\mid a\right\}  ,\ldots
,\rho^{\left(  r_{a}\right)  }\left\{  \lambda_{r_{a}\left(  m_{a}-1\right)
},\ldots\ldots,\lambda_{r_{a}m_{a}}\mid a\right\}  \right]  ,\label{c2}\\
&  \ell^{\prime}\left(  m_{k}-1\right)  =r_{a}\left(  m_{a}-1\right)  ,
\end{align}
where $\ell^{\prime}$ is an integer, and $\ell^{\prime}\leq r_{a}\leq
\ell^{\prime}\left(  m_{k}-1\right)  $, $2\leq m_{a}\leq m_{k}$.

\textbf{4}) Compatibility $n_{k}$-ary field multiplication $\mu_{k}^{\left(
n_{k}\right)  }$ (\textquotedblleft$\lambda\left(  \mu a\right)  =\left(
\lambda\mu\right)  a$\textquotedblright)%
\begin{align}
&  \rho_{A}^{\left(  r_{a}\right)  }\overset{n_{\rho}}{\overbrace{\left\{
\lambda_{1},\ldots\ldots,\lambda_{r_{a}}\mid\ldots\rho_{A}^{\left(
r_{a}\right)  }\left\{  \lambda_{r_{a}\left(  n_{\rho}-1\right)  }%
,\ldots\ldots,\lambda_{r_{a}n_{\rho}}\mid a\right\}  \ldots\right\}  }%
}\nonumber\\
&  =\rho_{A}^{\left(  r_{a}\right)  }\left\{  \overset{\ell^{\prime\prime}%
}{\overbrace{\mu_{k}^{\left(  n_{k}\right)  }\left[  \lambda_{1},\ldots
\ldots,\lambda_{n_{k}}\right]  ,\ldots,\mu_{k}^{\left(  n_{k}\right)  }\left[
\lambda_{n_{k}\left(  \ell^{\prime\prime}-1\right)  },\ldots\ldots
,\lambda_{n_{k}\ell^{\prime\prime}}\right]  }},\lambda_{n_{k}\ell
^{\prime\prime}+1},\ldots\ldots,\lambda_{r_{a}n_{\rho}}\mid a\right\}
,\label{c3}\\
&  \ell^{\prime\prime}\left(  n_{k}-1\right)  =r_{a}\left(  n_{\rho}-1\right)
, \label{c3a}%
\end{align}
where $\ell^{\prime\prime}$ is an integer, and $\ell^{\prime\prime}\leq
r_{a}\leq\ell^{\prime\prime}\left(  n_{k}-1\right)  $, $2\leq n_{\rho}\leq
n_{k}$.

\begin{remark}
In the binary case, we have $m_{a}=n_{a}=m_{k}=n_{k}=n_{\rho}=2$, $r_{a}%
=\ell=\ell^{\prime}=\ell^{\prime\prime}=1$. The $n$-ary algebras
\cite{azc/izq,mil/vin} have only one distinct arity $n_{a}=n$.
\end{remark}

\begin{definition}
\label{def-ar}We call the triple $\left(  \ell,\ell^{\prime},\ell
^{\prime\prime}\right)  $ a $\ell$-\textit{arity shape} of the polyadic
algebra $\mathrm{A}$.
\end{definition}

\begin{proposition}
In the limiting $\ell$-arity shapes the arity shape of {\upshape$\mathrm{A}$}
is determined by \textsf{three} integers$\left(  m,n,r\right)  $, such that:

\begin{enumerate}
\item For the maximal $\ell=\ell^{\prime}=\ell^{\prime\prime}=r_{a}$, the
arity shape of the algebra and underlying field coincide%
\begin{align}
m_{a}  &  =m_{k}=m,\\
n_{a}  &  =n_{k}=n_{\rho}=n,\\
r_{a}  &  =r.
\end{align}

\item For the minimal $\ell$-arities $\ell=\ell^{\prime}=\ell^{\prime\prime
}=1$ it should be $r_{a}|\left(  m_{k}-1\right)  $ and $r_{a}|\left(
n_{k}-1\right)  $, and%
\begin{align}
m_{a}  &  =1+\frac{m-1}{r},\\
n_{a}  &  =n_{\rho}=1+\frac{n-1}{r},\\
m_{k}  &  =m,\\
n_{k}  &  =n,\\
r_{a}  &  =r.
\end{align}

\end{enumerate}
\end{proposition}

\begin{proof}
This follows directly from the compatibility conditions (\ref{c1})--(\ref{c3}).
\end{proof}

\begin{proposition}
If the multiaction $\rho_{A}^{\left(  r_{a}\right)  }$ is an ordinary action
$K\times A\rightarrow A$, then all $\ell$-arities are minimal $\ell
=\ell^{\prime}=\ell^{\prime\prime}=1$, and the arity shape of
{\upshape$\mathrm{A}$} is determined by \textsf{two} integers $\left(
m,n\right)  $, such that the arities of the algebra and underlying field are
equal, and the arity $n_{\rho}$ of the action semigroup $\mathrm{S}_{\rho}$ is
equal to the arity of multiplication in the underlying field%
\begin{align}
m_{a}  &  =m_{k}=m,\\
n_{a}  &  =n_{k}=n_{\rho}=n.
\end{align}

\end{proposition}

As it was shown in \cite{dup2017a}, there exist zeroless and nonunital
polyadic fields and rings. Therefore, the main difference with the binary
algebras is the possible \textsf{absence} of a zero and/or unit in the polyadic
field $\Bbbk^{\left(  m_{k},n_{k}\right)  }$ and/or in the polyadic ring%
\begin{equation}
\mathrm{A}_{ring}=\mathrm{A}^{\left(  m_{a},n_{a}\right)  }=\left\langle
A\mid\nu_{A}^{\left(  m_{a}\right)  },\mu_{A}^{\left(  n_{a}\right)
}\right\rangle ,
\end{equation}
and so the additional axioms are needed iff such elements exist. This was the
reason we have started from \textbf{Definition \ref{def-alg-as}}, where
no existence of zeroes and units in $\Bbbk^{\left(  m_{k},n_{k}\right)  }$ and
$\mathrm{A}_{ring}$ is implied.

If they exist, denote possible units and zeroes by $e_{k}\in\Bbbk^{\left(
m_{k},n_{k}\right)  }$, $z_{k}\in\Bbbk^{\left(  m_{k},n_{k}\right)  }$ and
$e_{A}\in\mathrm{A}^{\left(  m_{a},n_{a}\right)  }$, $z_{A}\in\mathrm{A}%
^{\left(  m_{a},n_{a}\right)  }$. In this way we have 4 choices for each
$\Bbbk^{\left(  m_{k},n_{k}\right)  }$ and $\mathrm{A}^{\left(  m_{a}%
,n_{a}\right)  }$, and these 16 possible kinds of polyadic algebras are presented in
 \textsc{Table \ref{tab1}}. The most exotic case is at the bottom right,
where both $\Bbbk^{\left(  m_{k},n_{k}\right)  }$ and $\mathrm{A}^{\left(
m_{a},n_{a}\right)  }$ are \textsf{zeroless nonunital}, which cannot exist in either
binary algebras or $n$-ary algebras \cite{azc/izq}.

\begin{table}[th]
\caption{Kinds of polyadic algebras depending on zeroes and units.}%
\label{tab1}
\begin{center}
{\tiny \resizebox{\textwidth}{!}{
\begin{tabular}
[c]{||l||l|l|l|l||}\hline\hline
\backslashbox{$\Bbbk^{\left(  m_{k},n_{k}\right)  }$}{$\mathrm{A}^{\left(  m_{a},n_{a}\right)  }$ }
& $
\begin{array}
[c]{c}z_{A}\\
e_{A}\end{array}
$ & $\begin{array}
[c]{c}z_{A}\\
\text{no }e_{A}\end{array}
$ & $\begin{array}
[c]{c}\text{no }z_{A}\\
e_{A}\end{array}
$ & $\begin{array}
[c]{c}\text{no }z_{A}\\
\text{no }e_{A}\end{array}
$\\\hline\hline
$\begin{array}
[c]{c}z_{k}\\
e_{k}\end{array}
$ & $\begin{array}
[c]{c}\text{unital $\mathrm{A}$}\\
\text{unital }\Bbbk
\end{array}
$ & $\begin{array}
[c]{c}\text{nonunital $\mathrm{A}$}\\
\text{unital }\Bbbk
\end{array}
$ & $\begin{array}
[c]{c}\text{unital zeroless $\mathrm{A}$}\\
\text{unital }\Bbbk
\end{array}
$ & $\begin{array}
[c]{c}\text{nonunital zeroless $\mathrm{A}$}\\
\text{unital }\Bbbk
\end{array}
$\\\hline
$\begin{array}
[c]{c}z_{k}\\
\text{no }e_{k}\end{array}
$ & $\begin{array}
[c]{c}\text{unital $\mathrm{A}$}\\
\text{nonunital }\Bbbk
\end{array}
$ & $\begin{array}
[c]{c}\text{nonunital $\mathrm{A}$}\\
\text{nonunital }\Bbbk
\end{array}
$ & $\begin{array}
[c]{c}\text{unital zeroless $\mathrm{A}$}\\
\text{nonunital }\Bbbk
\end{array}
$ & $\begin{array}
[c]{c}\text{nonunital zeroless $\mathrm{A}$}\\
\text{nonunital }\Bbbk
\end{array}
$\\\hline
$\begin{array}
[c]{c}\text{no }z_{k}\\
e_{k}\end{array}
$ & $\begin{array}
[c]{c}\text{unital $\mathrm{A}$}\\
\text{unital zeroless }\Bbbk
\end{array}
$ & $\begin{array}
[c]{c}\text{nonunital $\mathrm{A}$}\\
\text{unital zeroless }\Bbbk
\end{array}
$ & $\begin{array}
[c]{c}\text{unital zeroless $\mathrm{A}$}\\
\text{unital zeroless }\Bbbk
\end{array}
$ & $\begin{array}
[c]{c}\text{nonunital zeroless $\mathrm{A}$}\\
\text{unital zeroless }\Bbbk
\end{array}
$\\\hline
$\begin{array}
[c]{c}\text{no }z_{k}\\
\text{no }e_{k}\end{array}
$ & $\begin{array}
[c]{c}\text{unital $\mathrm{A}$}\\
\text{nonunital zeroless }\Bbbk
\end{array}
$ & $\begin{array}
[c]{c}\text{nonunital $\mathrm{A}$}\\
\text{nonunital zeroless }\Bbbk
\end{array}
$ & $\begin{array}
[c]{c}\text{unital zeroless $\mathrm{A}$}\\
\text{nonunital zeroless }\Bbbk
\end{array}
$ & $\begin{array}
[c]{c}\text{nonunital zeroless $\mathrm{A}$}\\
\text{nonunital zeroless }\Bbbk
\end{array}
$\\\hline\hline
\end{tabular}} }
\end{center}
\end{table}

The standard case is that in the upper left corner, when both $\Bbbk^{\left(
m_{k},n_{k}\right)  }$ and $\mathrm{A}^{\left(  m_{a},n_{a}\right)  }$ have
a zero and unit.

\begin{example}
\label{ex-a33}Consider the (\textquotedblleft$\Bbbk$-linear\textquotedblright)
associative polyadic algebra $\mathrm{A}^{\left(  3,3;3,3;2\right)  }$ over
the zeroless nonunital $\left(  3,3\right)  $-field $\Bbbk^{\left(
3,3\right)  }$ (from \textit{Example} \ref{ex-tern}). The elements of $A$ are
pairs $a=\left(  \lambda,\lambda^{\prime}\right)  \in\Bbbk^{\left(
3,3\right)  }\times\Bbbk^{\left(  3,3\right)  }$, and for them the ternary
addition and ternary multiplication are defined by%
\begin{align}
\mu_{A}^{\left(  3\right)  }\left[  \left(  \lambda_{1},\lambda_{1}^{\prime
}\right)  \left(  \lambda_{2},\lambda_{2}^{\prime}\right)  \left(  \lambda
_{3},\lambda_{3}^{\prime}\right)  \right]   &  =\left(  \lambda_{1}\lambda
_{2}^{\prime}\lambda_{3},\lambda_{1}^{\prime}\lambda_{2}\lambda_{3}^{\prime
}\right)  ,\label{ma1}\\
\nu_{A}^{\left(  3\right)  }\left[  \left(  \lambda_{1},\lambda_{1}^{\prime
}\right)  \left(  \lambda_{2},\lambda_{2}^{\prime}\right)  \left(  \lambda
_{3},\lambda_{3}^{\prime}\right)  \right]   &  =\left(  \lambda_{1}%
+\lambda_{2}+\lambda_{3},\lambda_{1}^{\prime}+\lambda_{2}^{\prime}+\lambda
_{3}^{\prime}\right)  ,\ \ \ \lambda_{i},\lambda_{i}^{\prime}\in\Bbbk^{\left(
3,3\right)  } \label{ma2}%
\end{align}
where operations on the r.h.s. are in $\mathbb{C}$. If we introduce an element
$0\notin\Bbbk^{\left(  3,3\right)  }$ with the property $0\cdot\lambda
=\lambda\cdot0=0$, then (\ref{ma1})--(\ref{ma2}) can be presented as the
ordinary multiplication and addition of three anti-diagonal $2\times2$ formal
matrices $\left(
\begin{array}
[c]{cc}%
0 & \lambda\\
\lambda^{\prime} & 0
\end{array}
\right)  $. There is no unit or zero in the ternary ring $\left\langle
A\mid\nu_{A}^{\left(  3\right)  },\mu_{A}^{\left(  3\right)  }\right\rangle $,
but both $\left\langle A\mid\mu_{A}^{\left(  3\right)  }\right\rangle $ and
$\left\langle A\mid\nu_{A}^{\left(  3\right)  }\right\rangle $ are ternary
groups, because each $a=\left(  ip/q,ip^{\prime}/q^{\prime}\right)  \in A$ has
the \textsf{unique} additive querelement $\tilde{a}=\left(  -ip/q,-ip^{\prime
}/q^{\prime}\right)  $ and the unique multiplicative querelement $\bar
{a}=\left(  -iq^{\prime}/p^{\prime},-iq/p\right)  $. The $2$-place action
(\textquotedblleft$2$-scalar product\textquotedblright) is defined by
$\rho^{\left(  2\right)  }\left(  \lambda_{1},\lambda_{2}\mid\left(
\lambda,\lambda^{\prime}\right)  \right)  =\left(  \lambda_{1}\lambda
_{2}\lambda,\lambda_{1}\lambda_{2}\lambda^{\prime}\right)  $. The arity shape
(see \textit{Definition} \ref{def-ar}) of this \textsf{zeroless nonunital}
polyadic algebra $\mathrm{A}^{\left(  3,3;3,3;2\right)  }$ is $\left(
2,2,2\right)  $, and the compatibilities (\ref{c0})--(\ref{c3a}) hold.
\end{example}

\subsection{Polyadic analog of the functions on group}

In the search for a polyadic version of the algebra of $\Bbbk$-valued functions
(which is isomorphic and dual to the corresponding group algebra) we can not
only have more complicated arity shapes than in the binary case, but also the
exotic possibility that the arities of the field and group are different as can
be possible for multiplace functions.

Let us consider a $n_{g}$-ary group $\mathrm{G}=\mathrm{G}^{\left(
n_{g}\right)  }=\left\langle G\mid\mu_{g}^{\left(  n_{g}\right)
}\right\rangle $, which does not necessarily contain the identity $e_{g}$, and where
each element is querable (see (\ref{q})). Now we introduce the set $A_{f}$ of
multiplace ($s$-place) functions $f_{i}\left(  g_{1},\ldots,g_{s}\right)  $
(of finite support) which take value in the polyadic field $\Bbbk^{\left(
m_{k},n_{k}\right)  }$ such that $f_{i}:G^{s}\rightarrow K$. To endow $A_{f}$ with
the structure of a polyadic associative algebra (\ref{al}), we should
consistently define the $m_{k}$-ary addition $\nu_{f}^{\left(  m_{k}\right)
}:A_{f}^{\left(  m_{k}\right)  }\rightarrow A_{f}$, $n_{k}$-ary multiplication
(\textquotedblleft convolution\textquotedblright) $\mu_{f}^{\left(
n_{k}\right)  }:A_{f}^{\left(  n_{k}\right)  }\rightarrow A_{f}$ and the
multiaction $\rho_{f}^{\left(  r_{f}\right)  }:K^{r_{f}}\times A_{f}%
\rightarrow A_{f}$ (\textquotedblleft scalar multiplication\textquotedblright%
). Thus we write for the algebra of $\Bbbk$-valued functions%
\begin{equation}
F_{\Bbbk}\left(  \mathrm{G}\right)  =\left\langle A_{f}\mid\nu_{f}^{\left(
m_{k}\right)  },\mu_{f}^{\left(  n_{k}\right)  };\nu_{k}^{\left(
m_{k}\right)  },\mu_{k}^{\left(  n_{k}\right)  };\rho_{f}^{\left(
r_{f}\right)  }\right\rangle . \label{fg}%
\end{equation}

The simplest operation here is the addition of the $\Bbbk$-valued functions which,
obviously, coincides with the field addition $\nu_{f}^{\left(  m_{k}\right)
}=\nu_{k}^{\left(  m_{k}\right)  }$.

\begin{condition}
Because all arguments of the multiacton $\rho_{f}^{\left(  r_{f}\right)  }$
are in the field, the only possibility for the r.h.s. is its multiplication
(similar to the regular representation)%
\begin{equation}
\rho_{f}^{\left(  r_{f}\right)  }\left(  \lambda_{1},\ldots,\lambda_{r_{f}%
}\mid f\right)  =\mu_{k}^{\left(  n_{k}\right)  }\left[  \lambda_{1}%
,\ldots,\lambda_{r_{f}},f\right]  ,\ \ \lambda_{i}\in K,\ \ f\in A_{f},
\label{rr}%
\end{equation}
and in addition we have the arity shape relation%
\begin{equation}
n_{k}=r_{f}+1, \label{nr}%
\end{equation}
which is satisfied \textquotedblleft automatically\textquotedblright\ in the
binary case.
\end{condition}

The polyadic analog of $\Bbbk$-valued function convolution (\textquotedblleft%
$\left(  f_{1}\ast f_{2}\right)  \left(  g\right)  =\Sigma_{h_{1}h_{2}=g}%
f_{1}\left(  h_{1}\right)  f_{2}\left(  h_{2}\right)  $\textquotedblright),
which is denoted by $\mu_{f}^{\left(  n_{k}\right)  }$ here, while the sum in
the field is $\nu_{k}^{\ell_{\nu}\left(  m_{k}-1\right)  +1}$, where
$\ell_{\nu}$ is the \textquotedblleft number of additions\textquotedblright,
can be constructed according to the arity rules from \cite{dup2018a,dup2017}.

\begin{definition}
The \textit{polyadic convolution} of $s$-place $\Bbbk$-valued functions is
defined as the admissible polyadic sum of $\ell_{\nu}\left(  m_{k}-1\right)
+1$ products%
\begin{align}
&  \mu_{f}^{\left(  n_{k}\right)  }\left[  f_{1}\left(  g_{1},\ldots
,g_{s}\right)  ,\ldots,f_{n_{k}}\left(  g_{1},\ldots,g_{s}\right)  \right]
=\nonumber\\
&  \underset{\left\{  \substack{ \mu_{g}^{\left(  n_{g}\right)  }\left[
h_{1},\ldots,h_{n_{g}}\right]  =g_{1},\\ \mu_{g}^{\left(  n_{g}\right)
}\left[  h_{n_{g}+1},\ldots,h_{2n_{g}}\right]  =g_{2},\\ \vdots\\ \mu
_{g}^{\left(  n_{g}\right)  }\left[  h_{\left(  s-\ell_{\operatorname*{id}%
}-1\right)  n_{g}},\ldots,h_{\left(  s-\ell_{\operatorname*{id}}\right)
n_{g}}\right]  =g_{s-\ell_{\operatorname*{id}}},\\ h_{\left(  s-\ell
_{\operatorname*{id}}+1\right)  n_{g}}=g_{s-\ell_{\operatorname*{id}}+1},\\
\vdots\\ h_{sn_{g}}=g_{s}}\right.  }{\left(  \nu_{k}^{\left(  m_{k}\right)
}\right)  ^{\circ\ell_{\nu}}}\left[  \mu_{k}^{\left(  n_{k}\right)  }\left[
f_{1}\left(  h_{1},\ldots,h_{s}\right)  ,\ldots,f_{n_{k}}\left(  h_{s\left(
n_{k}-1\right)  },\ldots,h_{sn_{k}}\right)  \right]  \right]  , \label{mf}%
\end{align}
where $\ell_{\operatorname*{id}}$ is the number of intact elements in the
determining equations (\textquotedblleft$h_{1}h_{2}=g$\textquotedblright)
under the field sum $\nu_{k}$. The arity shape is determined by%
\begin{equation}
sn_{k}=\left(  s-\ell_{\operatorname*{id}}\right)  n_{g}+\ell
_{\operatorname*{id}}, \label{sn}%
\end{equation}
which gives the connection between the field and the group arities.
\end{definition}

\begin{example}
If $n_{g}=3$, $n_{k}=2$, $m_{k}=3$, $s=2$, $\ell_{\operatorname*{id}}=1$, then
we obtain the arity changing polyadic convolution%
\begin{align}
&  \mu_{f}^{\left(  2\right)  }\left[  f_{1}\left(  g_{1},g_{2}\right)
,f_{2}\left(  g_{1},g_{2}\right)  \right]  =\nonumber\\
&  \underset{\left\{  \substack{ \mu_{g}^{\left(  3\right)  }\left[
h_{1},h_{2},h_{3}\right]  =g_{1},\\ \mu_{g}^{\left(  4\right)  } =g_{2}%
}\right.  }{\left(  \nu_{k}^{\left(  3\right)  }\right)  ^{\circ\ell_{\nu}}%
}\left[  \mu_{k}^{\left(  2\right)  }\left[  f_{1}\left(  h_{1},h_{2}\right)
,f_{2}\left(  h_{3},h_{4}\right)  \right]  \right]  ,
\end{align}
where the $\ell_{\nu}$ ternary additions are taken on the support. Now the
multiaction (\ref{rr}) is one-place%
\begin{equation}
\rho_{f}^{\left(  1\right)  }\left(  \lambda\mid f\right)  =\mu_{k}^{\left(
2\right)  }\left[  \lambda,f\right]  ,\ \ \lambda\in K,\ \ f\in A_{f},
\end{equation}
as it follows from (\ref{nr}).
\end{example}

\begin{remark}
The general polyadic convolution (\ref{mf}) is inspired by the main
heteromorphism equation (5.14) and the arity changing formula (5.15) of
\cite{dup2018a}. The graphical dependence of the field arity $n_{k}$ on the
number of places $s$ is similar to that on \textsc{Figure 1}, and the
\textquotedblleft quantization\textquotedblright\ rules (following from the
solutions of (\ref{sn}) in integers) are in \textsc{Table 1} there.
\end{remark}

\begin{proposition}
The multiplication {\upshape(\ref{mf})} is associative.
\end{proposition}

\begin{proof}
This follows from the associativity quiver technique of \cite{dup2018a} applied
to the polyadic convolution.
\end{proof}

\begin{corollary}
The $\Bbbk$-valued multiplace functions $\left\{  f_{i}\right\}  $ form a
polyadic associative algebra $F_{\Bbbk}\left(  \mathrm{G}\right)  $.
\end{corollary}

\subsection{\textquotedblleft Diagrammatic\textquotedblright\ description}

Here we formulate the polyadic algebra axioms in the more customary
\textquotedblleft diagrammatic\textquotedblright\ form using the polyadic
tensor products and mappings between them (denoted by bold corresponding
letters). Informally, the $\Bbbk$-linearity is already \textquotedblleft
automatically encoded\textquotedblright\ by the polyadic tensor algebra over
$\Bbbk$, and therefore the axioms already contain the algebra multiplication (but not
the scalar multiplication).

Let us denote the $\Bbbk$-\textit{linear algebra multiplication map} by
$\mathbf{\mu}^{\left(  n\right)  }$ ($\mu^{\left(  n\right)  }\equiv\mu
_{A}^{\left(  n_{a}\right)  }$ from (\ref{al})) defined as%
\begin{equation}
\mathbf{\mu}^{\left(  n\right)  }\circ\left(  a_{1}\otimes\ldots\otimes
a_{n}\right)  =\mu^{\left(  n\right)  }\left[  a_{1},\ldots,a_{n}\right]
,\ \ \ a_{1},\ldots,a_{n}\in A. \label{ma}%
\end{equation}

\begin{definition}
[\textsl{Algebra associativity axiom}]\label{def-alg-as1}A \textit{polyadic
(associative }$n$\textit{-ary) algebra} (or $\Bbbk$-algebra) is a vector space
$\mathrm{A}_{vect}$ over the polyadic field $\Bbbk$ (\ref{av}) with the
$\Bbbk$-linear algebra multiplication map%
\begin{equation}
\mathrm{A}^{\left(  n\right)  }=\left\langle \mathrm{A}_{vect}\mid\mathbf{\mu
}^{\left(  n\right)  }\right\rangle ,\ \ \ \ \ \ \mathbf{\mu}^{\left(
n\right)  }:A^{\otimes n}\rightarrow A, \label{aav}%
\end{equation}
which is totally associative%
\begin{align}
\mathbf{\mu}^{\left(  n\right)  }\circ\left(  \operatorname*{id}%
\nolimits_{A}^{\otimes\left(  n-1-i\right)  }\otimes\mathbf{\mu}^{\left(
n\right)  }\otimes\operatorname*{id}\nolimits_{A}^{\otimes i}\right)   &
=\mathbf{\mu}^{\left(  n\right)  }\circ\left(  \operatorname*{id}%
\nolimits_{A}^{\otimes\left(  n-1-j\right)  }\otimes\mathbf{\mu}^{\left(
n\right)  }\otimes\operatorname*{id}\nolimits_{A}^{\otimes j}\right)
,\nonumber\\
\forall i,j  &  =0,\ldots n-1,\ i\neq j,\ \ \operatorname*{id}\nolimits_{A}%
:A\rightarrow A, \label{as}%
\end{align}
such that the diagram%
\begin{equation}
\begin{diagram} A^{\otimes \left(2n-1\right)} & \rTo^{ \operatorname*{id}\nolimits_{A}^{\otimes\left( n-1-i\right) }\otimes \;\mathbf{\mu}^{\left( n \right) }\otimes\;\operatorname*{id}\nolimits_{A}^{\otimes i} } &A^{\otimes n} \\ \dTo^{\operatorname*{id}\nolimits_{A}^{\otimes\left( n-1-j\right) }\otimes\;\mathbf{\mu}^{\left( n\right) }\otimes\;\operatorname*{id}\nolimits_{A}^{\otimes j}} & & \dTo_{\mathbf{\mu}^{\left( n\right) }} \\ A^{\otimes n} & \rTo^{ \mathbf{\mu}^{\left( n\right) }} &A \\ \end{diagram} \label{dia3}%
\end{equation}
commutes.
\end{definition}

\begin{definition}
A polyadic algebra $\mathrm{A}^{\left(  n\right)  }$ is called \textit{totally
commutative}, if%
\begin{equation}
\mathbf{\mu}^{\left(  n\right)  }=\mathbf{\mu}^{\left(  n\right)  }%
\circ\mathbf{\tau}_{n}, \label{mt}%
\end{equation}
where $\mathbf{\tau}_{n}\in\mathrm{S}_{n}$, and $\mathrm{S}_{n}$ is the
symmetry permutation group on $n$ elements.
\end{definition}

\begin{remark}
\label{rem-zu}Initially, there are no other axioms in the definition of a
polyadic algebra, because polyadic fields and vector spaces do not necessarily
contain zeroes and units (see \textsc{Table \ref{tab1}}).
\end{remark}

A special kind of polyadic algebra can appear, when the multiplication is
\textquotedblleft iterated\textquotedblright\ from lower arity ones, which is
one of 3 kinds of arity changing for polyadic systems \cite{dup2018a}.

\begin{definition}
A polyadic multiplication is called \textit{derived}, if the map $\mathbf{\mu
}_{der}^{\left(  n\right)  }$ is $\ell_{\mu}$-iterated from the maps
$\mathbf{\mu}_{0}^{\left(  n_{0}\right)  }$ of lower arity $n_{0}<n$%
\begin{equation}
\mathbf{\mu}_{der}^{\left(  n\right)  }=\overset{\ell_{\mu}}{\overbrace
{\mathbf{\mu}_{0}^{\left(  n_{0}\right)  }\circ\left(  \mathbf{\mu}%
_{0}^{\left(  n_{0}\right)  }\circ\ldots\left(  \mathbf{\mu}_{0}^{\left(
n_{0}\right)  }\otimes\operatorname*{id}\nolimits^{\otimes\left(
n_{0}-1\right)  }\right)  \otimes\ldots\otimes\operatorname*{id}%
\nolimits^{\otimes\left(  n_{0}-1\right)  }\right)  }}, \label{mder}%
\end{equation}
where%
\begin{equation}
n=\ell_{\mu}\left(  n_{0}-1\right)  +1,\ \ \ \ \ell_{\mu}\geq2,
\end{equation}
and $\ell_{\mu}$ is the \textquotedblleft number of
iterations\textquotedblright.
\end{definition}

\begin{example}
In the ternary case $n=3$ and $n_{0}=2$, we have $\mathbf{\mu}_{der}^{\left(
3\right)  }=\mathbf{\mu}_{0}^{\left(  2\right)  }\circ\left(  \mathbf{\mu}%
_{0}^{\left(  2\right)  }\otimes\operatorname*{id}\right)  $, which in the
\textquotedblleft elementwise\textquotedblright\ description is $\left[
a_{1},a_{2},a_{3}\right]  _{der}=a_{1}\cdot\left(  a_{2}\cdot a_{3}\right)  $,
where $\mu_{der}^{\left(  3\right)  }=\left[  \ ,\ ,\ \right]  _{der}$ and
$\mu_{0}^{\left(  2\right)  }=\left(  \cdot\right)  $.
\end{example}

Introduce a $\Bbbk$-linear \textit{multiaction map} $\mathbf{\rho}^{\left(
r\right)  }$ corresponding to the multiaction $\rho^{\left(  r\right)  }%
\equiv\rho_{A}^{\left(  r_{a}\right)  }$ (\ref{r}) (by analogy with
(\ref{ma})) as%
\begin{equation}
\mathbf{\rho}^{\left(  r\right)  }\circ\left(  \lambda_{1}\otimes\ldots
\otimes\lambda_{r}\otimes a\right)  =\rho^{\left(  r\right)  }\left(
\lambda_{1},\ldots,\lambda_{r}\mid a\right)  ,\ \ \ \lambda_{1},\ldots
,\lambda_{r}\in K,\ a\in A. \label{r1}%
\end{equation}

Let $\Bbbk$ and $\mathrm{A}^{\left(  n\right)  }$ both be unital, then we can
construct a $\Bbbk$-linear \textit{polyadic unit map} $\mathbf{\eta}$ by
\textquotedblleft polyadizing\textquotedblright\ \textquotedblleft$\mu
\circ\left(  \eta\otimes\operatorname*{id}\right)  =\operatorname*{id}%
$\textquotedblright\ and the scalar product \textquotedblleft$\lambda
a=\rho\left(  \lambda\mid a\right)  =\eta\left(  \lambda\right)
a$\textquotedblright\ with \textquotedblleft$\eta\left(  e_{k}\right)  =e_{a}%
$\textquotedblright, using the normalization (\ref{re}), and taking into
account the standard identification $\Bbbk^{\otimes r}\otimes\mathrm{A}%
\cong\mathrm{A}$ \cite{yokonuma}.

\begin{definition}
[\textsl{Algebra unit axiom}]The \textit{unital polyadic algebra}
$\mathrm{A}^{\left(  n\right)  }$ (\ref{aav}) contains in addition a $\Bbbk
$-linear \textit{polyadic (right) unit} map $\mathbf{\eta}^{\left(
r,n\right)  }:K^{\otimes r}\rightarrow A^{\otimes\left(  n-1\right)  }$
satisfying%
\begin{equation}
\mathbf{\mu}^{\left(  n\right)  }\circ\left(  \mathbf{\eta}^{\left(
r,n\right)  }\otimes\operatorname*{id}\nolimits_{A}\right)  =\mathbf{\rho
}^{\left(  r\right)  }, \label{mu}%
\end{equation}
such that the diagram%
\begin{equation}
\begin{diagram} K^{\otimes r}\otimes A &\rTo^{\; \mathbf{\eta}^{\left( r,n\right) }\otimes\operatorname*{id}\nolimits_{A}} & A^{\otimes n \;\;\; }\\ \dTo^{\mathbf{\rho}^{\left( r\right) }}&\;\;\ldTo(3,2)_{\mathbf{\mu}^{\left( n\right) }}&\\ A& & \end{diagram} \label{dia-n}%
\end{equation}
commutes.
\end{definition}

The normalization of the multiaction (\ref{re}) gives the corresponding
normalization of the map $\mathbf{\eta}^{\left(  r,n\right)  }$ (instead of
\textquotedblleft$\eta\left(  e_{k}\right)  =e_{a}$\textquotedblright)%
\begin{equation}
\mathbf{\eta}^{\left(  r,n\right)  }\circ\left(  \overset{r}{\overbrace
{e_{k}\otimes\ldots\otimes e_{k}}}\right)  =\overset{n-1}{\overbrace
{e_{a}\otimes\ldots\otimes e_{a}}}. \label{ne}%
\end{equation}

\begin{assertion}
\label{as-unit}In the \textquotedblleft elementwise\textquotedblright%
\ description (see \textbf{Subsection}\emph{ }\ref{sec-el}) the polyadic unit
$\eta^{\left(  r,n\right)  }$ of {\upshape$\mathrm{A}^{\left(  n\right)  }$}
is a $\left(  n-1\right)  $-valued function of $r$ arguments.
\end{assertion}

\begin{proposition}
The polyadic unit map $\mathbf{\eta}^{\left(  r,n\right)  }$ is (polyadically)
multiplicative in the following sense%
\begin{equation}
\overset{r}{\overbrace{\mathbf{\mu}^{\left(  n\right)  }\circ\ldots
\circ\mathbf{\mu}^{\left(  n\right)  }}}\circ\left(  \overset{r\left(
n-1\right)  +1}{\overbrace{\mathbf{\eta}^{\left(  r,n\right)  }\otimes
\ldots\otimes\mathbf{\eta}^{\left(  r,n\right)  }}}\right)  =\mathbf{\eta
}^{\left(  r,n\right)  }\circ\left(  \mu_{k}^{\left(  n_{k}\right)  }\right)
^{\circ\ell}.
\end{equation}

\end{proposition}

\begin{proof}
This follows from the compatibility of the multiaction with the
\textquotedblleft vector multiplication\textquotedblright\ (\ref{c1}) and the
relation between corresponding arities (\ref{c1a}), such that the number of
arguments (\textquotedblleft scalars\textquotedblright\ $\lambda_{i}$) in
r.h.s. becomes $\ell\left(  n_{k}-1\right)  +1=r\left(  n-1\right)  +1$, where
$\ell$ is an integer.
\end{proof}

Introduce a \textquotedblleft derived\textquotedblright\ version of the
polyadic unit by analogy with the neutral sequence (\ref{me}).

\begin{definition}
\label{def-nder}The $\Bbbk$-linear \textit{derived polyadic unit
}(\textit{neutral unit sequence})\textit{ }of $n$-ary algebra $\mathrm{A}%
^{\left(  n\right)  }$ is the set $\mathbf{\hat{\eta}}^{\left(  r\right)
}=\left\{  \mathbf{\eta}_{i}^{\left(  r\right)  }\right\}  $ of $n-1$ maps
$\mathbf{\eta}_{i}^{\left(  r\right)  }:K^{\otimes r}\rightarrow A$,
$i=1,\ldots,n-1$, satisfying%
\begin{equation}
\mathbf{\mu}^{\left(  n\right)  }\circ\left(  \mathbf{\eta}_{1}^{\left(
r\right)  }\otimes\ldots\otimes\mathbf{\eta}_{n-1}^{\left(  r\right)  }%
\otimes\operatorname*{id}\nolimits_{A}\right)  =\mathbf{\rho}^{\left(
r\right)  }, \label{mn1}%
\end{equation}
where $\operatorname*{id}\nolimits_{A}$ can be on any place. If $\mathbf{\eta
}_{1}^{\left(  r\right)  }=\ldots=\mathbf{\eta}_{n-1}^{\left(  r\right)
}=\mathbf{\eta}_{0}^{\left(  r\right)  }$, we call $\mathbf{\eta}_{0}^{\left(
r\right)  }$ the \textit{strong derived polyadic unit}. Formally (comparing
(\ref{mu}) and (\ref{mn1})), we can write%
\begin{equation}
\mathbf{\eta}_{der}^{\left(  r,n\right)  }=\mathbf{\eta}_{1}^{\left(
r\right)  }\otimes\ldots\otimes\mathbf{\eta}_{n-1}^{\left(  r\right)  }.
\label{nder}%
\end{equation}

\end{definition}

The normalization of the maps $\mathbf{\eta}_{i}^{\left(  r\right)  }$ is given by%
\begin{equation}
\mathbf{\eta}_{i}^{\left(  r\right)  }\circ\left(  \overset{r}{\overbrace
{e_{k}\otimes\ldots\otimes e_{k}}}\right)  =e_{a},\ \ \ i=1,\ldots
,n-1,\ \ e_{a}\in A,\ \ e_{k}\in K,
\end{equation}
and in the \textquotedblleft elementwise\textquotedblright\ description
$\eta_{i}^{\left(  r\right)  }$ is a function of $r$ arguments, satisfying%
\begin{equation}
\eta_{i}^{\left(  r\right)  }\left(  \lambda_{1},\ldots,\lambda_{r}\right)
=\rho^{\left(  r\right)  }\left\{  \lambda_{1},\ldots,\lambda_{r}\mid
e_{a}\right\}  ,\ \ \ \lambda_{i}\in K,
\end{equation}
where $\rho^{\left(  r\right)  }$ is the multiaction (\ref{r}).

\begin{definition}
A polyadic associative algebra $\mathrm{A}_{der}^{\left(  n\right)
}=\left\langle \mathrm{A}_{vect}\mid\mathbf{\mu}_{der}^{\left(  n\right)
},\mathbf{\eta}_{der}^{\left(  r,n\right)  }\right\rangle $ is called
\textit{derived from }$\mathrm{A}_{0}^{\left(  n_{0}\right)  }=\left\langle
\mathrm{A}_{vect}\mid\mathbf{\mu}_{0}^{\left(  n_{0}\right)  },\mathbf{\eta
}_{0}^{\left(  r,n_{0}\right)  }\right\rangle $, if (\ref{mder}) holds and further%
\begin{equation}
\mathbf{\eta}_{der}^{\left(  r,n\right)  }=\overset{\ell_{\mu}}{\overbrace
{\mathbf{\eta}_{0}^{\left(  r,n_{0}\right)  }\otimes\ldots\otimes\mathbf{\eta
}_{0}^{\left(  r,n_{0}\right)  }}}%
\end{equation}
is true, where $\mathbf{\eta}_{0}^{\left(  r,n_{0}\right)  }%
=\overset{n_{0}-1}{\overbrace{\mathbf{\eta}_{0}^{\left(  r\right)  }%
\otimes\ldots\otimes\mathbf{\eta}_{0}^{\left(  r\right)  }}}$ (formally,
because $\operatorname*{id}\nolimits_{A}$ in (\ref{mn1}) can be on any place).
\end{definition}

The particular case $n=3$ and $r=1$ was considered in \cite{dup26,dup2018b}
(with examples).

Invertibility in a polyadic algebra is not connected with the unit or zero
(as in $n$-ary groups \cite{dor3}), but is determined by the querelement
(\ref{q}). Introduce the corresponding mappings for the subsets of the
\textit{additively querable elements} $A_{quer}^{\left(  add\right)
}\subseteq A$ and the \textit{multiplicatively querable elements}
$A_{quer}^{\left(  mult\right)  }\subseteq A$.

\begin{definition}
In the polyadic algebra $\mathrm{A}^{\left(  m,n\right)  }$, the
\textit{additive quermap} $\mathbf{q}_{add}:A_{quer}^{\left(  add\right)
}\rightarrow A_{quer}^{\left(  add\right)  }$ is defined by%
\begin{equation}
\mathbf{\nu}^{\left(  m\right)  }\circ\left(  \operatorname*{id}%
\nolimits_{A}^{\otimes\left(  m-1\right)  }\otimes\mathbf{q}_{add}\right)
\circ\mathbf{D}_{a}^{\left(  m\right)  }=\operatorname*{id}\nolimits_{A},
\label{mq}%
\end{equation}
and the \textit{multiplicative quermap} $\mathbf{q}_{mult}:A_{quer}^{\left(
mult\right)  }\rightarrow A_{quer}^{\left(  mult\right)  }$ is defined by%
\begin{equation}
\mathbf{\mu}^{\left(  n\right)  }\circ\left(  \operatorname*{id}%
\nolimits_{A}^{\otimes\left(  n-1\right)  }\otimes\mathbf{q}_{mult}\right)
\circ\mathbf{D}_{a}^{\left(  n\right)  }=\operatorname*{id}\nolimits_{A},
\label{mq1}%
\end{equation}
where $\mathbf{D}_{a}^{\left(  n\right)  }:A\rightarrow A^{\otimes n}$ is the
diagonal map given by $a\rightarrow\overset{n}{\overbrace{a\otimes
\ldots\otimes a}}$, while $\mathbf{q}_{add}$ and $\mathbf{q}_{mult}$ can be on
any place. They send an element to the additive querelement $a\overset
{\mathbf{q}_{add}}{\mapsto}\tilde{a}$, $a\in A_{quer}^{\left(  add\right)
}\subseteq A$ and multiplicative querelement $a\overset{\mathbf{q}_{mult}%
}{\mapsto}\bar{a}$, $a\in A_{quer}^{\left(  mult\right)  }\subseteq A$ (see
(\ref{q})), such that the diagrams%
\begin{equation}
\begin{diagram} A &\; \rTo^{\;\;\;\mathbf{D}_{a}^{\left( m\right)} \;\;\;}\;& A^{\otimes m }\\ \uTo^{\mathbf{\nu}^{\left( m\right) }}&\ldTo(3,2)_{\operatorname*{id}\nolimits_{A}^{\otimes\left( m-1\right) }\otimes\mathbf{q}_{add}}&\\ A^{\otimes m }\;\;& & \end{diagram}\ \ \ \ \ \ \ \begin{diagram} A &\; \rTo^{\;\;\;\mathbf{D}_{a}^{\left( n\right)} \;\;\;}\;& A^{\otimes n }\\ \uTo^{\mathbf{\mu}^{\left( n\right) }}&\ldTo(3,2)_{\operatorname*{id}\nolimits_{A}^{\otimes\left( n-1\right) }\otimes\mathbf{q}_{mult}}&\\ A^{\otimes n }\;\;& & \end{diagram} \label{dia5}%
\end{equation}
commute.
\end{definition}

\begin{example}
For the polyadic algebra $\mathrm{A}^{\left(  3,3;3,3;2\right)  }$ from
\textit{Example} \ref{ex-a33} all elements are additively and multiplicatively
querable, and so the sets of querable elements coincide $A_{quer}^{\left(
add\right)  }=A_{quer}^{\left(  mult\right)  }=A$. The additive quermap
$\mathbf{q}_{add}$ and multiplicative quermap $\mathbf{q}_{mult}$ act as
follows (the operations are in $\mathbb{C}$)%
\begin{align}
&  \left(  i\frac{p}{q},i\frac{p^{\prime}}{q^{\prime}}\right)  \overset
{\mathbf{q}_{add}}{\mapsto}\left(  -i\frac{p}{q},-i\frac{p^{\prime}}%
{q^{\prime}}\right)  ,\\
&  \left(  i\frac{p}{q},i\frac{p^{\prime}}{q^{\prime}}\right)  \overset
{\mathbf{q}_{mult}}{\mapsto}\left(  -i\frac{q^{\prime}}{p^{\prime}},-i\frac
{q}{p}\right)  ,\ \ \ \ i^{2}=-1,\ \ \ p,q\in\mathbb{Z}^{odd}.
\end{align}

\end{example}

\begin{example}
The polyadic field $\Bbbk^{\left(  m_{k},n_{k}\right)  }$ is a polyadic
algebra over itself. We identify $A=K$, $\mu_{A}^{\left(  n\right)  }=\mu
_{k}^{\left(  n_{k}\right)  }$, and the multiplication is defined by the
multiaction as%
\begin{equation}
\mathbf{\mu}^{\left(  n\right)  }\circ\left(  \lambda_{1}\otimes\ldots
\otimes\lambda_{r}\otimes\lambda\right)  =\rho^{\left(  r\right)  }\left(
\lambda_{1},\ldots,\lambda_{r}\mid\lambda\right)  .
\end{equation}
Therefore, we have the additional arity conditions%
\begin{equation}
n=r+1=n_{k},
\end{equation}
which are trivially satisfied in the binary case. Now the polyadic unit map
$\mathbf{\eta}^{\left(  r,n\right)  }$ (\ref{ne}) is the identity in each tensor component.
\end{example}

\subsection{\label{subsec-med}Medial map and polyadic permutations}

Recall that the \textit{binary} \textit{medial map} for the tensor product of
algebras (as vector spaces)
\begin{equation}
\mathbf{\tau}_{medial}:\left(  \mathrm{A}_{1}\otimes\mathrm{A}_{2}\right)
\otimes\left(  \mathrm{A}_{1}\otimes\mathrm{A}_{2}\right)  \rightarrow\left(
\mathrm{A}_{1}\otimes\mathrm{A}_{1}\right)  \otimes\left(  \mathrm{A}%
_{2}\otimes\mathrm{A}_{2}\right)  \label{tm}%
\end{equation}
is defined by (evaluation)%
\begin{equation}
\left(  a_{1}^{\left(  1\right)  }\otimes a_{1}^{\left(  2\right)  }\right)
\otimes\left(  a_{2}^{\left(  1\right)  }\otimes a_{2}^{\left(  2\right)
}\right)  \overset{\tau_{medial}}{\mapsto}\left(  a_{1}^{\left(  1\right)
}\otimes a_{2}^{\left(  1\right)  }\right)  \otimes\left(  a_{1}^{\left(
2\right)  }\otimes a_{2}^{\left(  2\right)  }\right)  . \label{sa}%
\end{equation}
It is obvious that%
\begin{equation}
\mathbf{\tau}_{medial}=\operatorname*{id}\nolimits_{A}\otimes\mathbf{\tau
}_{op}\otimes\operatorname*{id}\nolimits_{A}, \label{ti}%
\end{equation}
where $\mathbf{\tau}_{op}:\mathrm{A}_{1}\otimes\mathrm{A}_{2}\rightarrow
\mathrm{A}_{1}\otimes\mathrm{A}_{2}$ is the permutation of 2 elements
(twist/flip) of the tensor product, such that $a^{\left(  1\right)  }\otimes
a^{\left(  2\right)  }\overset{\tau_{op}}{\mapsto}a^{\left(  2\right)
}\otimes a^{\left(  1\right)  }$, $a^{\left(  1\right)  }\in\mathrm{A}_{1}$,
$a^{\left(  2\right)  }\in\mathrm{A}_{2}$, $\tau_{op}\in\mathrm{S}_{2}$.
This may be presented (\ref{sa}) in the matrix form%
\begin{equation}
\bigotimes\left(  \mathbf{a}\right)  _{2\times2}\overset{\tau_{medial}%
}{\mapsto}\bigotimes\left(  \mathbf{a}^{T}\right)  _{2\times2}%
,\ \ \ \ \bigotimes\left(  \mathbf{a}\right)  _{2\times2}=\bigotimes\left(
\begin{array}
[c]{cc}%
a_{1}^{\left(  1\right)  } & a_{1}^{\left(  2\right)  }\\
a_{2}^{\left(  1\right)  } & a_{2}^{\left(  2\right)  }%
\end{array}
\right)  , \label{an2}%
\end{equation}
where $T$ is the ordinary matrix transposition.

Let us apply (\ref{tm}) to arbitrary tensor products. By analogy, if we have a
tensor product of $mn$ elements (of \textsf{any} nature) grouped by $n$
elements (e.g. $m$ elements from $n$ different vector spaces), as in
(\ref{sa}), (\ref{an2}), we can write the tensor product in the $\left(
m\times n\right)  $-matrix form (cf. (3.18)--(3.19) in \cite{dup2018a})%
\begin{equation}
\bigotimes\left(  \mathbf{a}\right)  _{m\times n}=\bigotimes\left(
\begin{array}
[c]{cccc}%
a_{1}^{\left(  1\right)  } & a_{1}^{\left(  2\right)  } & \ldots &
a_{1}^{\left(  n\right)  }\\
a_{2}^{\left(  1\right)  } & a_{2}^{\left(  2\right)  } & \ldots &
a_{2}^{\left(  n\right)  }\\
\vdots & \vdots & \vdots & \vdots\\
a_{m}^{\left(  1\right)  } & a_{m}^{\left(  2\right)  } & \ldots &
a_{m}^{\left(  n\right)  }%
\end{array}
\right)  \label{an1}%
\end{equation}

\begin{definition}
The \textit{polyadic medial map} $\mathbf{\tau}_{medial}^{\left(  n,m\right)
}:\left(  A^{\otimes n}\right)  ^{\otimes m}\rightarrow\left(  A^{\otimes
m}\right)  ^{\otimes n}$ is defined as the transposition of the tensor product
matrix (\ref{an1}) by the evaluation (cf. the binary case (\ref{sa}))%
\begin{equation}
\bigotimes\left(  \mathbf{a}\right)  _{m\times n}\overset{\tau_{medial}%
^{\left(  n,m\right)  }}{\mapsto}\bigotimes\left(  \mathbf{a}^{T}\right)
_{n\times m}. \label{an}%
\end{equation}

\end{definition}

We can extend the mediality concept \cite{eva7,belousov} to polyadic algebras
using the medial map. If we have an algebra with $n$-ary multiplication
(\ref{ma}), then the mediality relation follows from (\ref{an1}) with $m=n$
and contains $\left(  n+1\right)  $ multiplications acting on $n^{2}$ elements.

\begin{definition}
A $\Bbbk$-linear polyadic algebra $\mathrm{A}^{\left(  n\right)  }$
(\ref{aav}) is called \textit{medial}, if its $n$-ary multiplication map
satisfies the relation%
\begin{equation}
\mathbf{\mu}^{\left(  n\right)  }\circ\left(  \left(  \mathbf{\mu}^{\left(
n\right)  }\right)  ^{\otimes n}\right)  =\mathbf{\mu}^{\left(  n\right)
}\circ\left(  \left(  \mathbf{\mu}^{\left(  n\right)  }\right)  ^{\otimes
n}\right)  \circ\mathbf{\tau}_{medial}^{\left(  n,n\right)  }, \label{mmm}%
\end{equation}
where $\mathbf{\tau}_{medial}^{\left(  n,n\right)  }$ is given by (\ref{an}),
or in the manifest elementwise form (evaluation)%
\begin{align}
&  \mu^{\left(  n\right)  }\left[  \mu^{\left(  n\right)  }\left[
a_{1}^{\left(  1\right)  },a_{1}^{\left(  2\right)  },\ldots,a_{1}^{\left(
n\right)  }\right]  ,\mu^{\left(  n\right)  }\left[  a_{2}^{\left(  1\right)
},a_{2}^{\left(  2\right)  },\ldots,a_{2}^{\left(  n\right)  }\right]
,\ldots,\mu^{\left(  n\right)  }\left[  a_{n}^{\left(  1\right)  }%
,a_{n}^{\left(  2\right)  },\ldots,a_{n}^{\left(  n\right)  }\right]  \right]
\nonumber\\
&  =\mu^{\left(  n\right)  }\left[  \mu^{\left(  n\right)  }\left[
a_{1}^{\left(  1\right)  },a_{2}^{\left(  1\right)  },\ldots,a_{n}^{\left(
1\right)  }\right]  ,\mu^{\left(  n\right)  }\left[  a_{1}^{\left(  2\right)
},a_{2}^{\left(  2\right)  },\ldots,a_{n}^{\left(  2\right)  }\right]
,\ldots,\mu^{\left(  n\right)  }\left[  a_{1}^{\left(  n\right)  }%
,a_{2}^{\left(  n\right)  },\ldots,a_{n}^{\left(  n\right)  }\right]  \right]
. \label{mma}%
\end{align}

\end{definition}

Let us \textquotedblleft polyadize\textquotedblright\ the binary twist map
$\mathbf{\tau}_{op}$ from (\ref{ti}), which can be suitable for operations
with polyadic tensor products. Informally, we can interpret (\ref{ti}), as
\textquotedblleft omitting the fixed points\textquotedblright\ of the binary
medial map $\mathbf{\tau}_{medial}$, and denote this procedure by
\textquotedblleft$\mathbf{\tau}_{op}=\mathbf{\tau}_{medial}\setminus
\operatorname*{id}$\textquotedblright.

\begin{definition}
A (\textit{medially allowed}) $\ell_{\tau}$-\textit{place polyadic twist map}
$\mathbf{\tau}_{op}^{\left(  \ell_{\tau}\right)  }$ is defined by%
\begin{equation}
\text{\textquotedblleft}\mathbf{\tau}_{op}^{\left(  \ell_{\tau}\right)
}=\mathbf{\tau}_{medial}^{\left(  n,m\right)  }\setminus\operatorname*{id}%
\text{\textquotedblright}, \label{top}%
\end{equation}
where $\ell_{\tau}=mn-k_{fixed}$, and $k_{fixed}$ is the number of fixed
points of the medial map $\mathbf{\tau}_{medial}^{\left(  n,m\right)  }$.
\end{definition}

\begin{assertion}
If $m\neq n$, then $\ell_{\tau}=mn-2$. If $m=n$, then the polyadic twist map
$\mathbf{\tau}_{op}^{\left(  \ell_{\tau}\right)  }$ is the reflection%
\begin{equation}
\mathbf{\tau}_{op}^{\left(  \ell_{\tau}\right)  }\circ\mathbf{\tau}%
_{op}^{\left(  \ell_{\tau}\right)  }=\operatorname*{id}\nolimits_{A}%
\end{equation}
and $\ell_{\tau}=n\left(  n-1\right)  $.
\end{assertion}

\begin{proof}
This follows from the matrix form (\ref{an1}) and (\ref{an}).
\end{proof}

Therefore the number of places $\ell_{\tau}$ is \textquotedblleft
quantized\textquotedblright\ and for lowest $m,n$ is presented in
\textsc{Table} \ref{tab-twist}.

\begin{table}[th]
\caption[Number of places in the polyadic twist map.]{Number of places
$\ell_{\tau}$ in the polyadic twist map $\mathbf{\tau}_{op}^{\left(
\ell_{\tau}\right)  }$.}%
\label{tab-twist}
\begin{center}%
\begin{tabular}
[c]{||c||c|c|c|c|c|c||}\hline\hline
\backslashbox{$m$}{$n$} & 2 & 3 & 4 & 5 & 6 & 7\\\hline\hline
2 & \textbf{2} & \textbf{4} & \textbf{6} & \textbf{8} & \textbf{10} &
\textbf{12}\\\hline
3 & \textbf{4} & \textbf{6} & \textbf{10} & \textbf{13} & \textbf{16} &
\textbf{19}\\\hline
4 & \textbf{6} & \textbf{10} & \textbf{12} & \textbf{18} & \textbf{22} &
\textbf{26}\\\hline
5 & \textbf{8} & \textbf{13} & \textbf{18} & \textbf{20} & \textbf{28} &
\textbf{33}\\\hline
6 & \textbf{10} & \textbf{16} & \textbf{22} & \textbf{28} & \textbf{30} &
\textbf{40}\\\hline
7 & \textbf{12} & \textbf{19} & \textbf{26} & \textbf{33} & \textbf{40} &
\textbf{42}\\\hline\hline
\end{tabular}
\end{center}
\end{table}

This generalizes the binary twist in a more unique way, which gives
polyadic commutativity.

\begin{remark}
The polyadic twist map $\mathbf{\tau}_{op}^{\left(  \ell_{\tau}\right)  }$ is
one element of the symmetry permutation group $\mathrm{S}_{\ell_{\tau}}$ which
is fixed by the medial map $\mathbf{\tau}_{medial}^{\left(  n,m\right)  }$ and
the special condition (\ref{top}), and it therefore respects polyadic tensor
product operations.
\end{remark}

\begin{example}
In the matrix representation we have%
\begin{equation}
\tau_{op}^{\left(  4\right)  }|_{n=3,m=2}=\left(
\begin{array}
[c]{cccc}%
0 & 0 & 1 & 0\\
1 & 0 & 0 & 0\\
0 & 0 & 0 & 1\\
0 & 1 & 0 & 0
\end{array}
\right)  ,\tau_{op}^{\left(  6\right)  }|_{n=3,m=3}=\left(
\begin{array}
[c]{cccccc}%
0 & 0 & 1 & 0 & 0 & 0\\
0 & 0 & 0 & 0 & 1 & 0\\
1 & 0 & 0 & 0 & 0 & 0\\
0 & 0 & 0 & 0 & 0 & 1\\
0 & 1 & 0 & 0 & 0 & 0\\
0 & 0 & 0 & 1 & 0 & 0
\end{array}
\right)  ,
\end{equation}%
\begin{equation}
\tau_{op}^{\left(  6\right)  }|_{n=4,m=2}=\left(
\begin{array}
[c]{cccccc}%
0 & 0 & 0 & 1 & 0 & 0\\
1 & 0 & 0 & 0 & 0 & 0\\
0 & 0 & 0 & 0 & 1 & 0\\
0 & 1 & 0 & 0 & 0 & 0\\
0 & 0 & 0 & 0 & 0 & 1\\
0 & 0 & 1 & 0 & 0 & 0
\end{array}
\right)  . \label{t6}%
\end{equation}

\end{example}

The introduction of the polyadic twist gives us the possibility to generalize
(in a way consistent with the medial map) the notion of the opposite algebra.

\begin{definition}
For a polyadic algebra $\mathrm{A}^{\left(  n\right)  }=\left\langle
A\mid\mathbf{\mu}^{\left(  n\right)  }\right\rangle $, an \textit{opposite
polyadic algebra}%
\begin{equation}
\mathrm{A}_{op}^{\left(  n\right)  }=\left\langle A\mid\mathbf{\mu}^{\left(
n\right)  }\circ\mathbf{\tau}_{op}^{\left(  n\right)  }\right\rangle
\label{aop}%
\end{equation}
exists if the number of places for the polyadic twist map (which coincides in
(\ref{aop}) with the arity of algebra multiplication $\ell_{\tau}=n$) is
\textsf{allowed} (see \textsc{Table} \ref{tab-twist}).
\end{definition}

\begin{definition}
A polyadic algebra $\mathrm{A}^{\left(  n\right)  }$ is called
\textit{medially commutative}, if%
\begin{equation}
\mathbf{\mu}_{op}^{\left(  n\right)  }\equiv\mathbf{\mu}^{\left(  n\right)
}\circ\mathbf{\tau}_{op}^{\left(  n\right)  }=\mathbf{\mu}^{\left(  n\right)
}, \label{mmt}%
\end{equation}
where $\mathbf{\tau}_{op}^{\left(  n\right)  }$ is the medially allowed
polyadic twist map.
\end{definition}

\subsection{Tensor product of polyadic algebras}

Let us consider a polyadic tensor product $\bigotimes_{i=1}^{n}\mathrm{A}%
_{i}^{\left(  n\right)  }$ of $n$ polyadic associative $n$-ary algebras
$\mathrm{A}_{i}^{\left(  n\right)  }=\left\langle A_{i}\mid\mathbf{\mu}%
_{i}^{\left(  n\right)  }\right\rangle $, $i=1,\ldots,n$, such that (see
(\ref{ma}))%
\begin{equation}
\mathbf{\mu}_{i}^{\left(  n\right)  }\circ\left(  a_{1}^{\left(  i\right)
}\otimes\ldots\otimes a_{n}^{\left(  i\right)  }\right)  =\mu_{A_{i}}^{\left(
n\right)  }\left[  a_{1}^{\left(  i\right)  },\ldots,a_{n}^{\left(  i\right)
}\right]  ,\ \ \ a_{1}^{\left(  i\right)  },\ldots,a_{n}^{\left(  i\right)
}\in A_{i},\ \ \mu_{A_{i}}^{\left(  n\right)  }:A_{i}^{\left(  n\right)
}\rightarrow A_{i}.
\end{equation}
To endow $\bigotimes_{i=1}^{n}\mathrm{A}_{i}^{\left(  n\right)  }$ with the
structure of an algebra, we will use the medial map $\mathbf{\tau}_{medial}%
^{\left(  n,m\right)  }$ (\ref{an}).

\begin{proposition}
\label{prop-tp}The tensor product of $n$ associative $n$-ary algebras
{\upshape$\mathrm{A}_{i}^{\left(  n\right)  }$} has the structure of the
polyadic algebra {\upshape$\mathrm{A}_{\otimes}^{\left(  n\right)
}=\left\langle \bigotimes_{i=1}^{n}\mathrm{A}_{i}^{\left(  n\right)  }%
\mid\mathbf{\mu}_{\otimes}\right\rangle $}, which is associative {\upshape(cf.
(\ref{as}))}%
\begin{align}
\mathbf{\mu}_{\otimes}\circ\left(  \operatorname*{id}\nolimits_{A_{\otimes}%
}^{\otimes\left(  n-1-i\right)  }\otimes\mathbf{\mu}_{\otimes}\otimes
\operatorname*{id}\nolimits_{A_{\otimes}}^{\otimes i}\right)   &
=\mathbf{\mu}_{\otimes}\circ\left(  \operatorname*{id}\nolimits_{A_{\otimes}%
}^{\otimes\left(  n-1-j\right)  }\otimes\mathbf{\mu}_{\otimes}\otimes
\operatorname*{id}\nolimits_{A_{\otimes}}^{\otimes j}\right)  ,\nonumber\\
\forall i,j  &  =0,\ldots n-1,\ i\neq j,\ \ \ \operatorname*{id}%
\nolimits_{A_{\otimes}}:A_{1}^{\otimes n}\otimes\ldots\otimes A_{n}^{\otimes
n}\rightarrow A, \label{as1}%
\end{align}
if%
\begin{equation}
\mathbf{\mu}_{\otimes}=\left(  \mathbf{\mu}_{1}^{\left(  n\right)  }%
\otimes\ldots\otimes\mathbf{\mu}_{n}^{\left(  n\right)  }\right)
\circ\mathbf{\tau}_{medial}^{\left(  n,n\right)  }. \label{mm}%
\end{equation}

\end{proposition}

\begin{proof}
We act by the multiplication map $\mathbf{\mu}_{\otimes}$ on the element's
tensor product matrix (\ref{an1}) and obtain%
\begin{align}
&  \mathbf{\mu}_{\otimes}\circ\left(  \left(  a_{1}^{\left(  1\right)
}\otimes a_{1}^{\left(  2\right)  }\otimes\ldots\otimes a_{1}^{\left(
n\right)  }\right)  \otimes\ldots\otimes\left(  a_{n}^{\left(  1\right)
}\otimes a_{n}^{\left(  2\right)  }\otimes\ldots\otimes a_{n}^{\left(
n\right)  }\right)  \right) \nonumber\\
&  =\left(  \mathbf{\mu}_{1}^{\left(  n\right)  }\otimes\ldots\otimes
\mathbf{\mu}_{n}^{\left(  n\right)  }\right)  \circ\mathbf{\tau}%
_{medial}^{\left(  n,n\right)  }\circ\left(  \left(  a_{1}^{\left(  1\right)
}\otimes a_{1}^{\left(  2\right)  }\otimes\ldots\otimes a_{1}^{\left(
n\right)  }\right)  \otimes\ldots\otimes\left(  a_{n}^{\left(  1\right)
}\otimes a_{n}^{\left(  2\right)  }\otimes\ldots\otimes a_{n}^{\left(
n\right)  }\right)  \right) \nonumber\\
&  =\left(  \mathbf{\mu}_{1}^{\left(  n\right)  }\otimes\ldots\otimes
\mathbf{\mu}_{n}^{\left(  n\right)  }\right)  \circ\left(  \left(
a_{1}^{\left(  1\right)  }\otimes a_{2}^{\left(  1\right)  }\otimes
\ldots\otimes a_{n}^{\left(  1\right)  }\right)  \otimes\ldots\otimes\left(
a_{1}^{\left(  n\right)  }\otimes a_{2}^{\left(  n\right)  }\otimes
\ldots\otimes a_{n}^{\left(  n\right)  }\right)  \right) \nonumber\\
&  =\mu_{1}^{\left(  n\right)  }\left[  a_{1}^{\left(  1\right)  }%
,a_{2}^{\left(  1\right)  },\ldots,a_{n}^{\left(  1\right)  }\right]
\otimes\ldots\otimes\mu_{n}^{\left(  n\right)  }\left[  a_{1}^{\left(
n\right)  },a_{2}^{\left(  n\right)  },\ldots,a_{n}^{\left(  n\right)
}\right]  , \label{mp}%
\end{align}
which proves that $\mathbf{\mu}_{\otimes}$ is indeed a polyadic algebra
multiplication. To prove the associativity (\ref{as1}) we repeat the same
derivation (\ref{mp}) twice and show that the result is independent of $i,j$.
\end{proof}

If \textsf{all} $\mathrm{A}_{i}^{\left(  n\right)  }$ have their polyadic unit
map $\mathbf{\eta}_{i}^{\left(  r,n\right)  }$ defined by (\ref{mu}) and
acting as (\ref{ne}), then we have

\begin{proposition}
The polyadic unit map of {\upshape$\mathrm{A}_{\otimes}^{\left(  n\right)  }$}
is $\mathbf{\eta}_{\otimes}:K^{\otimes nr}\rightarrow A_{1}^{\otimes\left(
n-1\right)  }\otimes\ldots\otimes A_{n}^{\otimes\left(  n-1\right)  }$ acting
as%
\begin{equation}
\mathbf{\eta}_{\otimes}\circ\left(  \overset{nr}{\overbrace{e_{k}\otimes
\ldots\otimes e_{k}}}\right)  =\left(  \overset{n-1}{\overbrace{e_{a_{1}%
}\otimes\ldots\otimes e_{a_{1}}}}\right)  \otimes\ldots\otimes\left(
\overset{n-1}{\overbrace{e_{a_{n}}\otimes\ldots\otimes e_{a_{n}}}}\right)  .
\end{equation}

\end{proposition}

\begin{assertion}
The polyadic unit of {\upshape$\mathrm{A}_{\otimes}^{\left(  n\right)  }$} is
a $\left(  n^{2}-n\right)  $-valued function of $nr$ arguments.
\end{assertion}

Note that concepts of tensor product and derived polyadic algebras are different.

\subsection{Heteromorphisms of polyadic associative algebras}

The standard homomorphism between binary associative algebras is defined as a
linear map $\varphi$ which \textquotedblleft commutes\textquotedblright\ with
the algebra multiplications (\textquotedblleft$\varphi\circ\mathbf{\mu}%
_{1}=\mathbf{\mu}_{2}\circ\left(  \varphi\otimes\varphi\right)  $%
\textquotedblright). In the polyadic case there exists the possibility to
change arity of the algebras, and for that, one needs to use the heteromorphisms
(or multiplace maps) introduced in \cite{dup2018a}. Let us consider two
polyadic $\Bbbk$-algebras $\mathrm{A}_{1}^{\left(  n_{1}\right)
}=\left\langle A_{1}\mid\mathbf{\mu}_{1}^{\left(  n_{1}\right)  }\right\rangle
$ and $\mathrm{A}_{2}^{\left(  n_{2}\right)  }=\left\langle A_{2}%
\mid\mathbf{\mu}_{2}^{\left(  n_{2}\right)  }\right\rangle $ (over the same
polyadic field $\Bbbk$).

\begin{definition}
A \textit{heteromorphism} between two polyadic $\Bbbk$-algebras $\mathrm{A}%
_{1}^{\left(  n_{1}\right)  }$ and $\mathrm{A}_{2}^{\left(  n_{2}\right)  }$
(of different arities $n_{1}$ and $n_{2}$) is a $s$-place $\Bbbk$-linear map
$\mathbf{\Phi}_{s}^{\left(  n_{1},n_{2}\right)  }:A_{1}^{\otimes s}\rightarrow
A_{2}$, such that%
\begin{equation}
\mathbf{\Phi}_{s}^{\left(  n_{1},n_{2}\right)  }\circ\left(  \overset
{s-\ell_{\operatorname*{id}}}{\overbrace{\mathbf{\mu}_{1}^{\left(
n_{1}\right)  }\otimes\ldots\otimes\mathbf{\mu}_{1}^{\left(  n_{1}\right)  }}%
}\otimes\overset{\ell_{\operatorname*{id}}}{\overbrace{\operatorname*{id}%
\nolimits_{A_{1}}\otimes\ldots\otimes\operatorname*{id}\nolimits_{A_{1}}}%
}\right)  =\mathbf{\mu}_{2}^{\left(  n_{2}\right)  }\circ\left(
\overset{n_{2}}{\overbrace{\mathbf{\Phi}_{s}^{\left(  n_{1},n_{2}\right)
}\otimes\ldots\otimes\mathbf{\Phi}_{s}^{\left(  n_{1},n_{2}\right)  }}%
}\right)  , \label{f}%
\end{equation}
and the diagram%
\begin{equation}
\begin{diagram} {A}_{1}^{\otimes s n_2} & \rTo^{\left(\mathbf{\Phi}_{s}^{\left( n_{1},n_{2}\right) }\right)^{\otimes n_2}\;\;\;\;\;\;} & {A}_{2}^{\otimes n_2} \\ \dTo^{\left(\mathbf{\mu}_{1}^{\left( n_{1}\right) }\right)^{\otimes \left( s-\ell_{\operatorname*{id}}\right)}\otimes \left(\operatorname*{id}\nolimits_{A_{1}} \right)^{\otimes \ell_{\operatorname*{id}}}} & & \dTo_{\mathbf{\mu}_{2}^{\left( n_{2}\right) }} \\ {A}_{1}^{\otimes s } & \rTo^{\mathbf{\Phi}_{s}^{\left( n_{1},n_{2}\right) } \;\;\;\;\;} & A_2 \\ \end{diagram} \label{dia-f}%
\end{equation}
commutes. The arities satisfy%
\begin{equation}
sn_{2}=n_{1}\left(  s-\ell_{\operatorname*{id}}\right)  +\ell
_{\operatorname*{id}}, \label{sn2}%
\end{equation}
where $0\leq\ell_{\operatorname*{id}}\leq s-1$ is an integer (the number of
\textquotedblleft intact elements\textquotedblright\ of $\mathrm{A}_{1}$), and
therefore $2\leq n_{2}\leq n_{1}$.
\end{definition}

\begin{assertion}
If $\ell_{\operatorname*{id}}=0$ (there are no \textquotedblleft intact
elements\textquotedblright), then the ($s$-place) heteromorphism does not
change the arity of the polyadic algebra.
\end{assertion}

\begin{definition}
A \textit{homomorphism} between two polyadic $\Bbbk$-algebras $\mathrm{A}%
_{1}^{\left(  n\right)  }$ and $\mathrm{A}_{2}^{\left(  n\right)  }$ (of the
same arity or \textit{equiary}) is a $1$-place $\Bbbk$-linear map
$\mathbf{\Phi}^{\left(  n\right)  }=\mathbf{\Phi}_{s=1}^{\left(  n,n\right)
}:A_{1}\rightarrow A_{2}$, such that%
\begin{equation}
\mathbf{\Phi}^{\left(  n\right)  }\circ\mathbf{\mu}_{1}^{\left(  n\right)
}=\mathbf{\mu}_{2}^{\left(  n\right)  }\circ\left(  \overset{n}{\overbrace
{\mathbf{\Phi}^{\left(  n\right)  }\otimes\ldots\otimes\mathbf{\Phi}^{\left(
n\right)  }}}\right)  , \label{ha}%
\end{equation}
and the diagram%
\begin{equation}
\begin{diagram} {A}_{1}^{\otimes n} & \rTo^{\left(\mathbf{\Phi}^{\left( n\right) }\right)^{\otimes n}} & {A}_{2}^{\otimes n} \\ \dTo^{\mathbf{\mu}_{1}^{\left( n\right) }} & & \dTo_{\mathbf{\mu}_{2}^{\left( n\right) }} \\ {A}_{1} & \rTo^{\mathbf{\Phi}^{\left( n\right) } } & A_2 \\ \end{diagram} \label{dia-f1}%
\end{equation}
commutes.
\end{definition}

The above definitions do not include the behavior of the polyadic unit under
heteromorphism, because a polyadic associative algebra need not contain a
unit. However, if both units exist, this will lead to strong arity restrictions.

\begin{proposition}
If in $\Bbbk$-algebras {\upshape$\mathrm{A}_{1}^{\left(  n_{1}\right)  }$} and
{\upshape$\mathrm{A}_{2}^{\left(  n_{2}\right)  }$} (of arities $n_{1}$ and
$n_{2}$) there exist both polyadic units {\upshape (\ref{mu})} $\mathbf{\eta}%
_{1}^{\left(  r,n_{1}\right)  }:K^{\otimes r}\rightarrow A_{1}^{\otimes\left(
n_{1}-1\right)  }$ and $\mathbf{\eta}_{2}^{\left(  r,n_{2}\right)
}:K^{\otimes r}\rightarrow A_{2}^{\otimes\left(  n_{2}-1\right)  }$, then

\begin{enumerate}
\item The heteromorphism {\upshape (\ref{f})} connects them as%
\begin{equation}
\overset{s}{\overbrace{\mathbf{\eta}_{2}^{\left(  r,n_{2}\right)  }%
\otimes\ldots\otimes\mathbf{\eta}_{2}^{\left(  r,n_{2}\right)  }}}=\left(
\overset{n_{1}-1}{\overbrace{\mathbf{\Phi}_{s}^{\left(  n_{1},n_{2}\right)
}\otimes\ldots\otimes\mathbf{\Phi}_{s}^{\left(  n_{1},n_{2}\right)  }}%
}\right)  \circ\left(  \overset{s}{\overbrace{\mathbf{\eta}_{1}^{\left(
r,n_{1}\right)  }\otimes\ldots\otimes\mathbf{\eta}_{1}^{\left(  r,n_{1}%
\right)  }}}\right)  ,
\end{equation}
and the diagram%
\begin{equation}
\begin{diagram} K^{r s} &\; \rTo^{\left( \mathbf{\eta}_{1}^{\left( r,n_{1}\right) } \right) ^{\otimes s}}& A_{1}^{\otimes s \left( n_{1}-1\right) }\\ \dTo^{\left( \mathbf{\eta}_{2}^{\left( r,n_{2}\right) } \right) ^{\otimes s}}&\ldTo(3,2)_{\left( \mathbf{\Phi}_{s}^{\left( n_{1},n_{2}\right) } \right) ^{\otimes \left( n_{1}-1\right)}}&\\ A_{2}^{\otimes s \left( n_{2}-1\right) \;\;\;}\; & &\end{diagram} \label{dia-n12}%
\end{equation}
commutes.

\item The number of \textquotedblleft intact elements\textquotedblright\ is
fixed by its maximum value%
\begin{equation}
\ell_{\operatorname*{id}}=s-1, \label{ls}%
\end{equation}
such that in the l.h.s. of {\upshape (\ref{f})} there is only one multiplication
$\mathbf{\mu}_{1}^{\left(  n_{1}\right)  }$.

\item The number of places $s$ in the heteromorphism $\mathbf{\Phi}%
_{s}^{\left(  n_{1},n_{2}\right)  }$ is fixed by the arities of the polyadic algebras%
\begin{equation}
s\left(  n_{2}-1\right)  =n_{1}-1. \label{s2}%
\end{equation}

\end{enumerate}
\end{proposition}

\begin{proof}
Using (\ref{sn2}) we obtain $s\left(  n_{2}-1\right)  =\left(  s-\ell
_{\operatorname*{id}}\right)  \left(  n_{1}-1\right)  $, then the $\left(
n_{1}-1\right)  $ power of the heteromorphism $\mathbf{\Phi}_{s}^{\left(
n_{1},n_{2}\right)  }$ maps $A_{1}^{\otimes s\left(  n_{1}-1\right)
}\rightarrow A_{2}^{\otimes\left(  n_{1}-1\right)  }$, and we have
$s-\ell_{\operatorname*{id}}=1$, which, together with (\ref{sn2}), gives
(\ref{ls}), and (\ref{s2}).
\end{proof}

\subsection{Structure constants}

Let $\mathrm{A}^{\left(  n\right)  }$ be a finite-dimensional polyadic algebra
(\ref{al}) having the basis $\mathrm{e}_{i}\in A,i=1,\ldots,N$, where $N$ is
its dimension as a polyadic vector space%
\begin{equation}
\mathrm{A}_{vect}=\left\langle A,K\mid\nu^{\left(  m\right)  };\nu
_{k}^{\left(  m_{k}\right)  },\mu_{k}^{\left(  n_{k}\right)  };\rho^{\left(
r\right)  }\right\rangle
\end{equation}
where we denote $\nu^{\left(  m\right)  }=\nu_{A}^{\left(  m_{a}\right)  }$ (see
(\ref{v}) and (\ref{av}), here $N=d_{v}$). In the binary case
\textquotedblleft$a=\sum_{i}\lambda^{\left(  i\right)  }\mathrm{e}_{i}%
$\textquotedblright, any element $a\in A$ is determined by the number
$N_{\lambda}$ of scalars $\lambda\in K$, which coincides with the algebra
dimension $N_{\lambda}=N$, because $r=1$. In the polyadic case, it can be that
$r>1$, and moreover with $m\geq2$ the admissible number of \textquotedblleft
words\textquotedblright\ (in the expansion of $a$ by $\mathrm{e}_{i}$) is
\textquotedblleft quantized\textquotedblright, such that $1,m,2m-1,3m-2,\ldots
\ell_{N}\left(  m-1\right)  +1$, where $\ell_{N}\in\mathbb{N}_{0}$ is the
\textquotedblleft number of additions\textquotedblright. So we have

\begin{definition}
In $N$-dimensional $n$-ary algebra $\mathrm{A}^{\left(  n\right)  }$ (with
$m$-ary addition and $r$-place \textquotedblleft scalar\textquotedblright%
\ multiplication) the expansion of any element $a\in A$ by the basis $\left\{
\mathrm{e}_{i}\mid i=1,\ldots,N\right\}  $ is%
\begin{equation}
a=\left(  \nu^{\left(  m\right)  }\right)  ^{\circ\ell_{N}}\left[
\rho^{\left(  r\right)  }\left\{  \lambda_{1}^{\left(  1\right)  }%
,\ldots,\lambda_{r}^{\left(  1\right)  }\mid\mathrm{e}_{1}\right\}
,\ldots,\rho^{\left(  r\right)  }\left\{  \lambda_{1}^{\left(  N\right)
},\ldots,\lambda_{r}^{\left(  N\right)  }\mid\mathrm{e}_{N}\right\}  \right]
, \label{ae}%
\end{equation}
and is determined by $N_{\lambda}\in\mathbb{N}$ \textquotedblleft
scalars\textquotedblright, where%
\begin{align}
N_{\lambda}  &  =rN,\label{nl}\\
N  &  =\ell_{N}\left(  m-1\right)  +1,\ \ \ \ell_{N}\in\mathbb{N}_{0}%
,\ \ N\in\mathbb{N},\ \ m\geq2. \label{nl1}%
\end{align}

\end{definition}

In the binary case $m=2$, the dimension $N$ of an algebra is not restricted
and is a natural number, because, $N=\ell_{N}+1$.

\begin{assertion}
\label{as-quanN}The dimension of $n$-ary algebra {\upshape$\mathrm{A}^{\left(
n\right)  }$} having $m$-ary addition is not arbitrary, but \textquotedblleft
quantized\textquotedblright\ and can only have the following values for
$m\geq3$%
\begin{align}
m  &  =3,\ \ N=1,3,5,\ldots,2\ell_{N}+1,\\
m  &  =4,\ \ N=1,4,7,\ldots,3\ell_{N}+1,\\
m  &  =5,\ \ N=1,5,9,\ldots,4\ell_{N}+1,\\
&  \ldots\nonumber
\end{align}

\end{assertion}

\begin{proof}
It follows from (\ref{nl1}) and demanding that the \textquotedblleft number of
additions\textquotedblright\ $\ell_{N}$ is natural or zero.
\end{proof}

In a similar way, by considering a product of the basis elements, which can
also be expanded in the basis \textquotedblleft$\mathrm{e}_{i}\mathrm{e}%
_{j}=\sum_{k}\chi_{\left(  i,j\right)  }^{\left(  k\right)  }\mathrm{e}_{k}%
$\textquotedblright, we can define a polyadic analog of the structure
constants $\chi_{\left(  i,j\right)  }^{\left(  k\right)  }\in K$.

\begin{definition}
The \textit{polyadic structure constants} $\chi_{r,\left(  i_{1},\ldots
i_{n}\right)  }^{\left(  j\right)  }\in K$, $i_{1},\ldots i_{n},j=1,\ldots,n$
of the $N$-dimensional $n$-ary algebra $\mathrm{A}^{\left(  n\right)  }$ (with
$m$-ary addition $\nu^{\left(  m\right)  }$ and $r$-place multiaction
$\rho^{\left(  r\right)  }$) are defined by the expansion of the $n$-ary
product of the basis elements $\left\{  \mathrm{e}_{i}\mid i=1,\ldots
,N\right\}  $ as%
\begin{align}
&  \mu^{\left(  n\right)  }\left[  \mathrm{e}_{i_{1}},\ldots,\mathrm{e}%
_{i_{n}}\right] \nonumber\\
&  =\left(  \nu^{\left(  m\right)  }\right)  ^{\circ\ell_{N}}\left[
\rho^{\left(  r\right)  }\left\{  \chi_{1,\left(  i_{1},\ldots i_{n}\right)
}^{\left(  1\right)  },\ldots,\chi_{r,\left(  i_{1},\ldots i_{n}\right)
}^{\left(  1\right)  }\mid\mathrm{e}_{1}\right\}  ,\ldots,\rho^{\left(
r\right)  }\left\{  \chi_{1,\left(  i_{1},\ldots i_{n}\right)  }^{\left(
N\right)  },\ldots,\chi_{r,\left(  i_{1},\ldots i_{n}\right)  }^{\left(
N\right)  }\mid\mathrm{e}_{N}\right\}  \right]  ,
\end{align}

where%
\begin{align}
N_{\chi}  &  =rN^{n+1},\ \ N,N_{\chi}\in\mathbb{N}\label{nx}\\
N  &  =\ell_{N}\left(  m-1\right)  +1,\ \ \ell_{N}\in\mathbb{N}_{0}%
,\ \ m\geq2.
\end{align}

\end{definition}

As in the binary case, we have

\begin{corollary}
The algebra multiplication $\mu^{\left(  n\right)  }$ of $\mathrm{A}^{\left(
n\right)  }$ is \textsf{fully} determined by the $rN^{n+1}$ polyadic structure
constants $\chi_{r,\left(  i_{1},\ldots i_{n}\right)  }^{\left(  j\right)
}\in K$.
\end{corollary}

Contrary to the binary case $m=2$, when $N_{\chi}$ can be any natural number,
we now have

\begin{assertion}
The number of the polyadic structure constants $N_{\chi}$ of the
finite-dimensional $n$-ary algebra {\upshape$\mathrm{A}^{\left(  n\right)  }$}
with $m$-ary addition and $r$-place multiaction is not arbitrary, but
\textquotedblleft quantized\textquotedblright\ according to%
\begin{equation}
N_{\chi}=r\left(  \ell_{N}\left(  m-1\right)  +1\right)  ^{n+1},\ \ r\in
\mathbb{N},\ \ \ell_{N}\in\mathbb{N}_{0},\ \ m,n\geq2.
\end{equation}

\end{assertion}

\begin{proof}
This follows from (\ref{nx}) and \textquotedblleft
quantization\textquotedblright\ of the algebra dimension $N$, see
\textbf{Assertion \ref{as-quanN}}.
\end{proof}

\section{\textsc{Polyadic coalgebras}}

\subsection{Motivation}

The standard motivation for introducing the comultiplication is from
representation theory \cite{cur/rei,kirillov}. The first examples come from
so-called addition formulas for special functions (\textquotedblleft
anciently\textquotedblright\ started from sin/cos), which actually arise from
representations of groups \cite{shn/ste,haz/gub/kir}.

In brief (and informally), let $\mathbf{\pi}$ be a finite-dimensional
representation of a group $\mathrm{G}$ in a vector space $\mathrm{V}$ over a
field $\mathrm{k}$, such that%
\begin{equation}
\mathbf{\pi}\left(  gh\right)  =\mathbf{\pi}\left(  g\right)  \mathbf{\pi
}\left(  h\right)  ,\ \ \ \ \mathbf{\pi}:\mathrm{G}\rightarrow
\operatorname*{End}\mathrm{V},\ \ \ \ g,h\in\mathrm{G}. \label{p}%
\end{equation}
In some basis of $\mathrm{V}$ the matrix elements $\pi_{ij}\left(  g\right)  $
satisfy $\pi_{ij}\left(  gh\right)  =\sum_{k}\pi_{ik}\left(  g\right)
\pi_{kj}\left(  h\right)  $ (from (\ref{p})) and span a finite dimensional
vector space $C_{\pi}$ of functions with a basis $e_{\pi_{m}}$ as $f_{\pi
}=\sum_{m}\alpha_{m}e_{\pi_{m}}$, $f_{\pi}\in C_{\pi}$. Now (\ref{p}) gives
$f_{\pi}\left(  gh\right)  =\sum_{m,n}\beta_{mn}e_{\pi_{m}}\left(  g\right)
e_{\pi_{n}}\left(  h\right)  $, $f_{\pi}\in C_{\pi}$. If we omit the
evaluation, it can be written in the vector space $C_{\pi}$ using an
additional linear map $\Delta_{\pi}:C_{\pi}\rightarrow C_{\pi}\otimes C_{\pi}%
$, in the following way%
\begin{equation}
\Delta_{\pi}\left(  f_{\pi}\right)  =\sum_{m,n}\beta_{mn}e_{\pi_{m}}\otimes
e_{\pi_{n}}\in C_{\pi}\otimes C_{\pi}. \label{dp}%
\end{equation}
Thus, to any finite-dimensional representation $\mathbf{\pi}$ one can define
the map $\Delta_{\pi}$ of vector spaces $C_{\pi}$ to functions on a group,
called a \textit{comultiplication}.

It is important that all the above operations are binary, and the defining
formula for comultiplication (\ref{dp}) is fully determined by the definition
of a representation (\ref{p}).

The polyadic analog of a representation was introduced and studied in
\cite{dup2018a}. In the case of multiplace representations, arities of the
initial group and its representation can be \textsf{different}. Indeed, let
$\mathrm{G}^{\left(  n\right)  }=\left\langle G\mid\mu_{G}^{\left(  n\right)
}\right\rangle $, $\mu_{G}^{\left(  n\right)  }:G^{\times n}\rightarrow G$, be
a $n$-ary group and $\mathrm{G}_{\operatorname*{End}\mathrm{V}}^{\left(
n^{\prime}\right)  }=\left\langle \left\{  \operatorname*{End}\mathrm{V}%
\right\}  \mid\mu_{E}^{\left(  n^{\prime}\right)  }\right\rangle $, $\mu
_{E}^{\left(  n^{\prime}\right)  }:\left(  \operatorname*{End}\mathrm{V}%
\right)  ^{\times n^{\prime}}\rightarrow\operatorname*{End}\mathrm{V}$, is a
$n^{\prime}$-ary group of endomorphisms of a polyadic vector space
$\mathrm{V}$ (\ref{v}). In \cite{dup2018a} $\mathrm{G}_{\operatorname*{End}%
\mathrm{V}}^{\left(  n^{\prime}\right)  }$ was considered as a derived one,
while here we \textsf{do not restrict} it in this way.

\begin{definition}
A \textit{polyadic (multiplace) representation} of $\mathrm{G}^{\left(
n\right)  }$ in $\mathrm{V}$ is a $s$-place mapping%
\begin{equation}
\mathbf{\Pi}_{s}^{\left(  n,n^{\prime}\right)  }:G^{\times s}\rightarrow
\operatorname*{End}\mathrm{V}, \label{pg}%
\end{equation}
satisfying the associativity preserving heteromorphism equation
\cite{dup2018a}%
\begin{align}
&  \mathbf{\Pi}_{s}^{\left(  n,n^{\prime}\right)  }\left(  \overset
{s-\ell_{\operatorname*{id}}^{\prime}}{\overbrace{\mu_{G}^{\left(  n\right)
}\left[  g_{1},\ldots g_{n}\right]  ,\ldots,\mu_{G}^{\left(  n\right)
}\left[  g_{n\left(  s-\ell_{\operatorname*{id}}^{\prime}-1\right)  }%
,\ldots,g_{n\left(  s-\ell_{\operatorname*{id}}^{\prime}\right)  }\right]  }%
},\ldots,\overset{\ell_{\operatorname*{id}}^{\prime}}{\overbrace{g_{n\left(
s-\ell_{\operatorname*{id}}^{\prime}\right)  +1},\ldots,g_{n\left(
s-\ell_{\operatorname*{id}}^{\prime}\right)  +\ell_{\operatorname*{id}%
}^{\prime}}}}\right) \nonumber\\
&  =\mu_{E}^{\left(  n^{\prime}\right)  }\left[  \mathbf{\Pi}_{s}^{\left(
n,n^{\prime}\right)  }\left(  g_{1},\ldots g_{s}\right)  ,\ldots,\mathbf{\Pi
}_{s}^{\left(  n,n^{\prime}\right)  }\left(  g_{s\left(  n^{\prime}-1\right)
},\ldots g_{sn^{\prime}}\right)  \right]  , \label{pp}%
\end{align}
such that the diagram
\begin{equation}
\begin{diagram} G^{\times s n^{\prime}} & \rTo^{\left( \mathbf{\Pi}_{s}^{\left( n,n^{\prime}\right) }\right) ^{\times n^{\prime}}} &\left( \operatorname*{End}V\right)^{\times n^{\prime}} \\ \dTo^{ \left(\mu_{G}^{\left( n\right)} \right)^{\times \left(s-\ell _{\operatorname*{id}}^{\prime} \right)} \times\left( \operatorname*{id}\nolimits_{G}\right)^{\times \ell _{\operatorname*{id}}^{\prime}}} & & \dTo_{\mu_{E}^{\left( n^{\prime}\right)} }\\ G^{\times n^{\prime}} & \rTo^{ \mathbf{\Pi}_{s}^{\left( n,n^{\prime}\right) }} &\operatorname*{End}V \end{diagram} \label{dia-p}%
\end{equation}
commutes, and the arity changing formula%
\begin{equation}
sn^{\prime}=n\left(  s-\ell_{\operatorname*{id}}^{\prime}\right)
+\ell_{\operatorname*{id}}^{\prime}, \label{sn1}%
\end{equation}
where $\ell_{\operatorname*{id}}^{\prime}$ is the number of \textquotedblleft
intact elements\textquotedblright\ in l.h.s. of (\ref{pp}), $0\leq
\ell_{\operatorname*{id}}^{\prime}\leq s-1$, $2\leq n^{\prime}\leq n$.
\end{definition}

\begin{remark}
Particular examples of $2$-place representations of ternary groups ($s=2$,
which change arity from $n=3$ to $n^{\prime}=2$), together with their matrix
representations, were presented in \cite{dup2018a,bor/dud/dup3}.
\end{remark}

\subsection{Polyadic comultiplication}

Our motivations say that in constructing a polyadic analog of the
comultiplication, one should not only \textquotedblleft reverse
arrows\textquotedblright, but also pay thorough attention to arities.

\begin{assertion}
The arity of polyadic comultiplication coincides with the arity of the
representation and can \textsf{differ} from the arity of the polyadic
algebra.\bigskip
\end{assertion}

\begin{proof}
It follows from (\ref{p}), (\ref{dp}) and (\ref{pg}).
\end{proof}

Let us consider a polyadic vector space over the polyadic field $\Bbbk
^{\left(  m_{k},n_{k}\right)  }$ as (see (\ref{av}))%
\begin{equation}
\mathrm{C}_{vect}=\left\langle C,K\mid\nu_{C}^{\left(  m_{c}\right)  };\nu
_{k}^{\left(  m_{k}\right)  },\mu_{k}^{\left(  n_{k}\right)  };\rho
_{C}^{\left(  r_{c}\right)  }\right\rangle , \label{cvect}%
\end{equation}
where $\nu_{C}^{\left(  m_{c}\right)  }:C^{\times m_{c}}\rightarrow C$ is
$m_{c}$-ary addition and $\rho_{C}^{\left(  r_{c}\right)  }:K^{\times r_{c}%
}\times C\rightarrow C$ is $r_{c}$-place action (see (\ref{r})).

\begin{definition}
\label{def-comult}A \textit{polyadic (}$n^{\prime}$\textit{-ary)}
\textit{comultiplication} is a $\Bbbk$-linear map $\mathbf{\Delta}^{\left(
n^{\prime}\right)  }:C\rightarrow C^{\otimes n^{\prime}}$.
\end{definition}

\begin{definition}
A \textit{polyadic (coassociative) coalgebra (or }$\Bbbk$-coalgebra) is the
polyadic vector space $\mathrm{C}_{vect}$ equipped with the polyadic
comultiplication%
\begin{equation}
\mathrm{C}=\mathrm{C}^{\left(  n^{\prime}\right)  }=\left\langle
\mathrm{C}_{vect}\mid\mathbf{\Delta}^{\left(  n^{\prime}\right)
}\right\rangle , \label{cnn}%
\end{equation}
which is (totally) \textit{coassociative}%
\begin{align}
\left(  \operatorname*{id}\nolimits_{C}^{\otimes\left(  n^{\prime}-1-i\right)
}\otimes\mathbf{\Delta}^{\left(  n^{\prime}\right)  }\otimes\operatorname*{id}%
\nolimits_{C}^{\otimes i}\right)  \circ\mathbf{\Delta}^{\left(  n^{\prime
}\right)  }  &  =\left(  \operatorname*{id}\nolimits_{C}^{\otimes\left(
n-1-j\right)  }\otimes\mathbf{\Delta}^{\left(  n^{\prime}\right)  }%
\otimes\operatorname*{id}\nolimits_{C}^{\otimes j}\right)  \circ
\mathbf{\Delta}^{\left(  n^{\prime}\right)  },\nonumber\\
\forall i,j  &  =0,\ldots n-1,\ i\neq j,\ \ \operatorname*{id}\nolimits_{C}%
:C\rightarrow C, \label{coa}%
\end{align}
and such that the diagram%
\begin{equation}
\begin{diagram} C^{\otimes \left(2n^{\prime}-1\right)} & \lTo^{ \operatorname*{id}\nolimits_{C}^{\otimes\left( n^{\prime}-1-i\right) }\otimes \;\mathbf{\Delta}^{\left( n^{\prime}\right) }\otimes\;\operatorname*{id}\nolimits_{C}^{\otimes i} } &C^{\otimes n^{\prime}} \\ \uTo^{\operatorname*{id}\nolimits_{C}^{\otimes\left( n^{\prime}-1-j\right) }\otimes\;\mathbf{\Delta}^{\left( n^{\prime}\right) }\otimes\;\operatorname*{id}\nolimits_{C}^{\otimes j}} & & \uTo_{\mathbf{\Delta}^{\left( n^{\prime}\right) }} \\ C^{\otimes n^{\prime}} & \lTo^{ \mathbf{\Delta}^{\left( n^{\prime}\right) }} &C \\ \end{diagram} \label{dia-D}%
\end{equation}
commutes (cf.(\ref{dia3})).
\end{definition}

\begin{definition}
A polyadic coalgebra $\mathrm{C}^{\left(  n^{\prime}\right)  }$ is called
\textit{totally co-commutative}, if%
\begin{equation}
\mathbf{\Delta}^{\left(  n^{\prime}\right)  }=\mathbf{\tau}_{n^{\prime}}%
\circ\mathbf{\Delta}^{\left(  n^{\prime}\right)  }, \label{dd}%
\end{equation}
where $\mathbf{\tau}_{n^{\prime}}\in\mathrm{S}_{n^{\prime}}$, and
$\mathrm{S}_{n^{\prime}}$ is the permutation symmetry group on $n^{\prime}$ elements.
\end{definition}

\begin{definition}
A polyadic coalgebra $\mathrm{C}^{\left(  n^{\prime}\right)  }$ is called
\textit{medially co-commutative}, if%
\begin{equation}
\mathbf{\Delta}_{cop}^{\left(  n^{\prime}\right)  }\equiv\mathbf{\tau}%
_{op}^{\left(  n^{\prime}\right)  }\circ\mathbf{\Delta}^{\left(  n^{\prime
}\right)  }=\mathbf{\Delta}^{\left(  n^{\prime}\right)  }, \label{cop}%
\end{equation}
where $\mathbf{\tau}_{op}^{\left(  n^{\prime}\right)  }$ is the medially
allowed polyadic twist map (\ref{top}).
\end{definition}

There are no other axioms in the definition of a polyadic coalgebra, following
the same reasoning as for a polyadic algebra: the possible absence of zeroes
and units (see \textit{Remark \ref{rem-zu}} and \textsc{Table \ref{tab1}}).
Obviously, in a polyadic coalgebra $\mathrm{C}^{\left(  n^{\prime}\right)  }$,
there is no \textquotedblleft unit element\textquotedblright, because there is
no multiplication, and a polyadic analog of counit can be \textsf{only
defined}, when the underlying field $\Bbbk^{\left(  m_{k},n_{k}\right)  }$ is
\textsf{unital} (which is not always the case \cite{dup2017a}).

By analogy with (\ref{mx}), introduce the $\ell^{\prime}$-\textit{coiterated
}$n^{\prime}$-\textit{ary comultiplication} by%
\begin{equation}
\left(  \mathbf{\Delta}^{\left(  n^{\prime}\right)  }\right)  ^{\circ
\ell^{\prime}}=\overset{\ell^{\prime}}{\overbrace{\left(  \operatorname*{id}%
\nolimits_{C}^{\otimes\left(  n^{\prime}-1\right)  }\otimes\ldots\left(
\operatorname*{id}\nolimits_{C}^{\otimes\left(  n^{\prime}-1\right)
}\mathbf{\Delta}^{\left(  n^{\prime}\right)  }\right)  \ldots\circ
\mathbf{\Delta}^{\left(  n^{\prime}\right)  }\right)  \circ\mathbf{\Delta
}^{\left(  n^{\prime}\right)  }}},\ \ \ \ \ell^{\prime}\in\mathbb{N}.
\label{dl}%
\end{equation}
Therefore, the \textit{admissible} length of any co-word is fixed
(\textquotedblleft quantized\textquotedblright) as $\ell^{\prime}\left(
n^{\prime}-1\right)  +1$, but not arbitrary, as in the binary case.

Let us introduce a co-analog of the derived $n$-ary multiplication
(\ref{mder}) by

\begin{definition}
A polyadic comultiplication $\mathbf{\Delta}_{der}^{\left(  n^{\prime}\right)
}$ is called \textit{derived}, if it is $\ell_{d}$-coiterated\ from the
comultiplication $\mathbf{\Delta}_{0}^{\left(  n_{0}^{\prime}\right)  }$ of
lower arity $n_{0}^{\prime}<n^{\prime}$%
\begin{equation}
\mathbf{\Delta}_{der}^{\left(  n^{\prime}\right)  }=\overset{\ell_{d}%
}{\overbrace{\left(  \operatorname*{id}\nolimits_{C}^{\otimes\left(
n_{0}^{\prime}-1\right)  }\otimes\ldots\left(  \operatorname*{id}%
\nolimits_{C}^{\otimes\left(  n_{0}^{\prime}-1\right)  }\mathbf{\Delta}%
_{0}^{\left(  n_{0}^{\prime}\right)  }\right)  \ldots\circ\mathbf{\Delta}%
_{0}^{\left(  n_{0}^{\prime}\right)  }\right)  \circ\mathbf{\Delta}%
_{0}^{\left(  n_{0}^{\prime}\right)  }}}, \label{ndder}%
\end{equation}
or%
\begin{equation}
\mathbf{\Delta}_{der}^{\left(  n^{\prime}\right)  }=\left(  \mathbf{\Delta
}_{0}^{\left(  n_{0}^{\prime}\right)  }\right)  ^{\circ\ell^{\prime}},
\end{equation}
where%
\begin{equation}
n^{\prime}=\ell_{d}\left(  n_{0}^{\prime}-1\right)  +1,
\end{equation}
and $\ell_{d}\geq2$ is the \textquotedblleft number of
coiterations\textquotedblright.
\end{definition}

The standard coiterations of $\mathbf{\Delta}$ are binary and restricted by
$n_{0}^{\prime}=2$ (\cite{swe}).

\begin{example}
The matrix coalgebra generated by the basis $e_{ij}$, $i,j=1,\ldots,N$ of
$\operatorname*{Mat}\nolimits_{N}\left(  \mathbb{C}\right)  $ with the binary
coproduct $\Delta_{0}^{\left(  2\right)  }\left(  e_{ij}\right)  =\sum
_{k}e_{ik}\otimes e_{kj}$ (see, e.g., \cite{abe}) can be extended to the
\textsf{derived} ternary coalgebra by $\Delta_{der}^{\left(  3\right)
}\left(  e_{ij}\right)  =\sum_{k,l}e_{ik}\otimes e_{kl}\otimes e_{lj}$, such
that (\ref{ndder}) becomes $\mathbf{\Delta}_{der}^{\left(  3\right)  }=\left(
\operatorname*{id}\nolimits_{C}\otimes\mathbf{\Delta}_{0}^{\left(  2\right)
}\right)  \mathbf{\Delta}_{0}^{\left(  2\right)  }=\left(  \mathbf{\Delta}%
_{0}^{\left(  2\right)  }\otimes\operatorname*{id}\nolimits_{C}\right)
\mathbf{\Delta}_{0}^{\left(  2\right)  }$.
\end{example}

\begin{example}
Let us consider the ternary coalgebra $\left\langle C\mid\mathbf{\Delta
}^{\left(  3\right)  }\right\rangle $ generated by two elements $\left\{
a,b\right\}  \in C$ with the von Neumann regular looking comultiplication%
\begin{equation}
\Delta^{\left(  3\right)  }\left(  a\right)  =a\otimes b\otimes
a,\ \ \ \ \Delta^{\left(  3\right)  }\left(  b\right)  =b\otimes a\otimes b.
\label{d3}%
\end{equation}
It is easy to check that $\mathbf{\Delta}^{\left(  3\right)  }$ is
coassociative and \textsf{nonderived}.
\end{example}

\begin{definition}
A polyadic coalgebra $\mathrm{C}^{\left(  n^{\prime}\right)  }$ (\ref{cnn}) is
called \textit{co-medial}, if its $n^{\prime}$-ary multiplication map
satisfies the relation%
\begin{equation}
\left(  \left(  \mathbf{\Delta}^{\left(  n^{\prime}\right)  }\right)
^{\otimes n^{\prime}}\right)  \circ\mathbf{\Delta}^{\left(  n^{\prime}\right)
}=\mathbf{\tau}_{medial}^{\left(  n^{\prime},n^{\prime}\right)  }\circ\left(
\left(  \mathbf{\Delta}^{\left(  n^{\prime}\right)  }\right)  ^{\otimes
n^{\prime}}\right)  \circ\mathbf{\Delta}^{\left(  n^{\prime}\right)  },
\label{com}%
\end{equation}
where $\mathbf{\tau}_{medial}^{\left(  n^{\prime},n^{\prime}\right)  }$ is the
polyadic medial map given by (\ref{an1})--(\ref{an}).
\end{definition}

Introduce a $\Bbbk$-linear $r^{\prime}$-\textit{place} \textit{action map}
$\mathbf{\bar{\rho}}^{\left(  r^{\prime}\right)  }:K^{\otimes r^{\prime}%
}\otimes C\rightarrow C$ corresponding to $\rho_{C}^{\left(  r_{c}\right)  }$
by (see (\ref{r1}))%
\begin{equation}
\mathbf{\bar{\rho}}^{\left(  r^{\prime}\right)  }\circ\left(  \lambda
_{1}\otimes\ldots\otimes\lambda_{r^{\prime}}\otimes c\right)  =\rho
_{C}^{\left(  r^{\prime}\right)  }\left(  \lambda_{1},\ldots,\lambda
_{r^{\prime}}\mid c\right)  ,\ \ \ \lambda_{1},\ldots,\lambda_{r^{\prime}}\in
K,\ c\in C. \label{rc}%
\end{equation}

Let $\Bbbk^{\left(  m_{k},n_{k}\right)  }$ be unital with unit $e_{k}$.

\begin{definition}
A $\Bbbk$-linear $r^{\prime}$-\textit{place} \textit{coaction map}
$\mathbf{\sigma}^{\left(  r^{\prime}\right)  }:C\rightarrow K^{\otimes
r^{\prime}}\otimes C$ is defined by%
\begin{equation}
c\mapsto\overset{r^{\prime}}{\overbrace{e_{k}\otimes\ldots\otimes e_{k}}%
}\otimes c. \label{c}%
\end{equation}

\end{definition}

\begin{assertion}
The coaction map $\mathbf{\sigma}^{\left(  r^{\prime}\right)  }$ is a
\textquotedblleft right inverse\textquotedblright\ for the multiaction map
$\mathbf{\bar{\rho}}^{\left(  r^{\prime}\right)  }$%
\begin{equation}
\mathbf{\bar{\rho}}^{\left(  r^{\prime}\right)  }\circ\mathbf{\sigma}^{\left(
r^{\prime}\right)  }=\operatorname*{id}\nolimits_{C}.
\end{equation}

\end{assertion}

\begin{proof}
This follows from the normalization (\ref{re}), (\ref{c}).
\end{proof}

\begin{remark}
The maps (\ref{rc}) and (\ref{c}) establish the isomorphism $\overset
{r^{\prime}}{\overbrace{\Bbbk\otimes\ldots\otimes\ \Bbbk}}\otimes
\mathrm{C}\cong\mathrm{C}$, which is well-known in the binary case (see, e.g.
\cite{yokonuma}).
\end{remark}

We can provide the definition of counit only in the case where the underlying
field $\Bbbk$ has a unit.

\begin{definition}
[\textsl{Counit axiom}]The polyadic coalgebra $\mathrm{C}^{\left(  n\right)
}$ (\ref{aav}) over the \textsf{unital}\textit{ }polyadic field\textit{
}$\Bbbk^{\left(  m_{k},n_{k}\right)  }$ contains a $\Bbbk$-linear
\textit{polyadic (right) counit map} $\mathbf{\varepsilon}^{\left(  n^{\prime
},r^{\prime}\right)  }:C^{\otimes\left(  n^{\prime}-1\right)  }\rightarrow
K^{\otimes r^{\prime}}$ satisfying%
\begin{equation}
\left(  \mathbf{\varepsilon}^{\left(  n^{\prime},r^{\prime}\right)  }%
\otimes\operatorname*{id}\nolimits_{C}\right)  \circ\mathbf{\Delta}^{\left(
n^{\prime}\right)  }=\mathbf{\sigma}^{\left(  r^{\prime}\right)  }, \label{ep}%
\end{equation}
such that the diagram%
\begin{equation}
\begin{diagram} K^{\otimes r^{\prime}}\otimes C &\lTo^{\; \mathbf{\varepsilon}^{\left( n^{\prime},r^{\prime}\right) }\otimes\operatorname*{id}\nolimits_{C}} & C^{\otimes n^{\prime} \; }\\ \uTo^{\mathbf{\sigma}^{\left( r^{\prime}\right) }}&\;\;\ruTo(3,2)_{\mathbf{\Delta}^{\left( n^{\prime}\right) }}&\\C& & \end{diagram} \label{dia-cou}%
\end{equation}
commutes (cf.(\ref{dia-n})).
\end{definition}

\begin{remark}
We cannot write the \textquotedblleft elementwise\textquotedblright%
\ normalization action for the counit analogous to (\ref{ne}) (and state the
\textbf{Assertion} \ref{as-unit}), because a unit element in a (polyadic)
coalgebra is not defined.
\end{remark}

By analogy with the derived polyadic unit (see (\ref{mn1}) and \textbf{Definition
\ref{def-nder}})\textbf{, }consider a \textquotedblleft
derived\textquotedblright\ version of the polyadic counit.

\begin{definition}
The $\Bbbk$-linear \textit{derived polyadic counit }(\textit{neutral counit
sequence})\textit{ }of the polyadic coalgebra $\mathrm{C}^{\left(  n^{\prime
}\right)  }$ is the set $\mathbf{\hat{\varepsilon}}^{\left(  r\right)
}=\left\{  \mathbf{\varepsilon}_{i}^{\left(  r^{\prime}\right)  }\right\}  $
of $n^{\prime}-1$ maps $\mathbf{\varepsilon}_{i}^{\left(  r^{\prime}\right)
}:C\rightarrow K^{\otimes r^{\prime}}$, $i=1,\ldots,n^{\prime}-1$, satisfying%
\begin{equation}
\left(  \mathbf{\varepsilon}_{1}^{\left(  r^{\prime}\right)  }\otimes
\ldots\otimes\mathbf{\varepsilon}_{n^{\prime}-1}^{\left(  r^{\prime}\right)
}\otimes\operatorname*{id}\nolimits_{C}\right)  \circ\mathbf{\Delta}^{\left(
n^{\prime}\right)  }=\mathbf{\sigma}^{\left(  r^{\prime}\right)  }, \label{ed}%
\end{equation}
where $\operatorname*{id}\nolimits_{C}$ can be on any place. If
$\mathbf{\varepsilon}_{1}^{\left(  r^{\prime}\right)  }=\ldots
=\mathbf{\varepsilon}_{n^{\prime}-1}^{\left(  r^{\prime}\right)
}=\mathbf{\varepsilon}_{0}^{\left(  r^{\prime}\right)  }$, we call it the
\textit{strong derived polyadic counit}. In general, we can define formally,
cf. (\ref{nder}),%
\begin{equation}
\mathbf{\varepsilon}_{der}^{\left(  n^{\prime},r^{\prime}\right)
}=\mathbf{\varepsilon}_{1}^{\left(  r^{\prime}\right)  }\otimes\ldots
\otimes\mathbf{\varepsilon}_{n^{\prime}-1}^{\left(  r^{\prime}\right)  }.
\end{equation}

\end{definition}

\begin{definition}
\label{def-derco}A polyadic coassociative coalgebra $\mathrm{C}_{der}^{\left(
n^{\prime}\right)  }=\left\langle \mathrm{C}_{vect}\mid\mathbf{\Delta}%
_{der}^{\left(  n^{\prime}\right)  },\mathbf{\varepsilon}_{der}^{\left(
n^{\prime},r^{\prime}\right)  }\right\rangle $ is called \textit{derived} from
$\mathrm{C}_{0}^{\left(  n^{\prime}\right)  }=\left\langle \mathrm{C}%
_{vect}\mid\mathbf{\Delta}_{0}^{\left(  n_{0}^{\prime}\right)  }%
,\mathbf{\varepsilon}_{0}^{\left(  n_{0}^{\prime},r^{\prime}\right)
}\right\rangle $, if (\ref{ndder}) and%
\begin{equation}
\mathbf{\varepsilon}_{der}^{\left(  n^{\prime},r^{\prime}\right)  }%
=\overset{\ell_{d}}{\overbrace{\mathbf{\varepsilon}_{0}^{\left(  n_{0}%
^{\prime},r^{\prime}\right)  }\otimes\ldots\otimes\mathbf{\varepsilon}%
_{0}^{\left(  n_{0}^{\prime},r^{\prime}\right)  }}}%
\end{equation}
hold, where $\mathbf{\varepsilon}_{0}^{\left(  n_{0}^{\prime},r^{\prime
}\right)  }=\overset{n_{0}^{\prime}-1}{\overbrace{\mathbf{\varepsilon}%
_{0}^{\left(  r^{\prime}\right)  }\otimes\ldots\otimes\mathbf{\varepsilon}%
_{0}^{\left(  r^{\prime}\right)  }}}$ (formally, because $\operatorname*{id}%
\nolimits_{C}$ in (\ref{ed}) can be on any place).
\end{definition}

In \cite{dup26,dup2018b} the particular case for $n^{\prime}=3$ and
$r^{\prime}=1$ was considered.

\subsection{Homomorphisms of polyadic coalgebras}

In the binary case, a morphism of coalgebras is a linear map $\psi
:C_{1}\rightarrow C_{2}$ which \textquotedblleft commutes\textquotedblright%
\ with comultiplications (\textquotedblleft$\left(  \psi\otimes\psi\right)
\circ\mathbf{\Delta}_{1}=\mathbf{\Delta}_{2}\circ\varphi$\textquotedblright).
It seems that for the polyadic coalgebras, one could formally change the
direction of all arrows in (\ref{dia-f}). However, we observed that arity
changing is possible for multivalued morphisms only. Therefore, here we
confine ourselves to homomorphisms (1-place heteromorphisms \cite{dup2018a}).

Let us consider two polyadic (equiary) $\Bbbk$-coalgebras $\mathrm{C}%
_{1}^{\left(  n^{\prime}\right)  }=\left\langle C_{1}\mid\mathbf{\Delta}%
_{1}^{\left(  n^{\prime}\right)  }\right\rangle $ and $\mathrm{C}_{2}^{\left(
n^{\prime}\right)  }=\left\langle C_{2}\mid\mathbf{\Delta}_{2}^{\left(
n^{\prime}\right)  }\right\rangle $ over the same polyadic field
$\Bbbk^{\left(  m_{k},n_{k}\right)  }$.

\begin{definition}
A (\textit{coalgebra}) \textit{homomorphism} between polyadic (equiary)
coalgebras $\mathrm{C}_{1}^{\left(  n^{\prime}\right)  }$ and $\mathrm{C}%
_{2}^{\left(  n^{\prime}\right)  }$ is a $\Bbbk$-linear map $\mathbf{\Psi
}^{\left(  n^{\prime}\right)  }:C_{1}\rightarrow C_{2}$, such that%
\begin{equation}
\left(  \overset{n^{\prime}}{\overbrace{\mathbf{\Psi}^{\left(  n^{\prime
}\right)  }\otimes\ldots\otimes\mathbf{\Psi}^{\left(  n^{\prime}\right)  }}%
}\right)  \circ\mathbf{\Delta}_{1}^{\left(  n^{\prime}\right)  }%
=\mathbf{\Delta}_{2}^{\left(  n^{\prime}\right)  }\circ\mathbf{\Psi}^{\left(
n^{\prime}\right)  }, \label{hc}%
\end{equation}
and the diagram%
\begin{equation}
\begin{diagram} {C}_{2}^{\otimes n^{\prime}} & \lTo^{\left(\mathbf{\Psi}^{\left( n^{\prime}\right) }\right)^{\otimes n^{\prime}}} & {C}_{1}^{\otimes n^{\prime}} \\ \uTo^{\mathbf{\Delta}_{2}^{\left( n^{\prime}\right) }} & & \uTo_{\mathbf{\Delta}_{1}^{\left( n^{\prime}\right) }} \\ {C}_{2} & \lTo^{\mathbf{\Psi}^{\left( n^{\prime}\right) } } & C_1 \\ \end{diagram} \label{dia-fc}%
\end{equation}
commutes (cf. (\ref{dia-f})).
\end{definition}

Only when the underlying field $\Bbbk$ is unital, we can also define a
morphism for counits.

\begin{definition}
The \textit{counit homomorphism} for $\mathbf{\varepsilon}_{1,2}^{\left(
n^{\prime},r^{\prime}\right)  }:C_{1,2}^{\otimes\left(  n^{\prime}-1\right)
}\rightarrow K^{\otimes r^{\prime}}$ is given by%
\begin{equation}
\mathbf{\varepsilon}_{2}^{\left(  n^{\prime},r^{\prime}\right)  }%
=\mathbf{\varepsilon}_{1}^{\left(  n^{\prime},r^{\prime}\right)  }\circ\left(
\overset{n^{\prime}-1}{\overbrace{\mathbf{\Psi}^{\left(  n^{\prime}\right)
}\otimes\ldots\otimes\mathbf{\Psi}^{\left(  n^{\prime}\right)  }}}\right)  ,
\end{equation}
and the diagram%
\begin{equation}
\begin{diagram} K^{r^{\prime}} &\; \lTo^{\mathbf{\varepsilon}_{2}^{\left( n^{\prime},r^{\prime}\right) } }& C_{2}^{\otimes \left( n^{\prime}-1\right) }\\ \uTo^{ \mathbf{\varepsilon}_{1}^{\left( n^{\prime},r^{\prime}\right) } }&\ruTo(3,2)_{\left( \mathbf{\Psi}^{\left( n^{\prime}\right) } \right) ^{\otimes \left( n^{\prime}-1\right)}}&\\ C_{1}^{\otimes \left( n^{\prime}-1\right) }\;\;\;\;\; & &\end{diagram} \label{dia-co}%
\end{equation}
commutes (cf. (\ref{dia-n12})).
\end{definition}

\subsection{Tensor product of polyadic coalgebras}

Let us consider $n^{\prime}$ polyadic equiary coalgebras $\mathrm{C}%
_{i}^{\left(  n^{\prime}\right)  }=\left\langle C_{i}\mid\mathbf{\Delta}%
_{i}^{\left(  n^{\prime}\right)  }\right\rangle $, $i=1,\ldots,n^{\prime}$.

\begin{proposition}
The tensor product of the coalgebras has a structure of the polyadic
coassociative coalgebra {\upshape$\mathrm{C}_{\otimes}^{\left(  n^{\prime
}\right)  }=\left\langle C_{\otimes}\mid\mathbf{\Delta}_{\otimes}^{\left(
n^{\prime}\right)  }\right\rangle $}, $C_{\otimes}=\bigotimes_{i=1}%
^{n^{\prime}}C_{i}$, if%
\begin{equation}
\mathbf{\Delta}_{\otimes}^{\left(  n^{\prime}\right)  }=\mathbf{\tau}%
_{medial}^{\left(  n^{\prime},n^{\prime}\right)  }\circ\left(  \mathbf{\Delta
}_{i}^{\left(  n^{\prime}\right)  }\otimes\ldots\otimes\mathbf{\Delta}%
_{i}^{\left(  n^{\prime}\right)  }\right)  ,
\end{equation}
where $\mathbf{\tau}_{medial}^{\left(  n^{\prime},n^{\prime}\right)  }$ is
defined in {\upshape(\ref{an})} and $\mathbf{\Delta}_{\otimes}^{\left(
n^{\prime}\right)  }:C_{\otimes}\rightarrow\overset{n^{\prime}}{\overbrace
{C_{\otimes}\otimes\ldots\otimes C_{\otimes}}}$.
\end{proposition}

The proof is in full analogy with that of \textbf{Proposition }\ref{prop-tp}.
If all of the coalgebras $\mathrm{C}_{i}^{\left(  n^{\prime}\right)  }$ have
counits, we denote them $\mathbf{\varepsilon}_{i}^{\left(  n^{\prime
},r^{\prime}\right)  }:C_{i}^{\otimes\left(  n^{\prime}-1\right)  }\rightarrow
K^{\otimes r^{\prime}}$, $i=1,\ldots,n^{\prime}$, and the counit map of
$\mathrm{C}_{\otimes}^{\left(  n^{\prime}\right)  }$ will be denoted by
$\mathbf{\varepsilon}_{\otimes}^{\left(  n^{\prime},r^{\prime}\right)
}:C_{\otimes}^{\otimes\left(  n^{\prime}-1\right)  }\rightarrow K^{\otimes
r^{\prime}}$. We have (in analogy to \textquotedblleft$\varepsilon_{C_{1}\otimes
C_{2}}\left(  c_{1}\otimes c_{2}\right)  =\varepsilon_{C_{1}}\left(
c_{1}\right)  \varepsilon_{C_{2}}\left(  c_{2}\right)  $\textquotedblright)

\begin{proposition}
The tensor product coalgebra {\upshape$\mathrm{C}_{\otimes}^{\left(
n^{\prime}\right)  }$} has a counit which is defined by%
\begin{align}
&  \mathbf{\varepsilon}_{\otimes}^{\left(  n^{\prime},r^{\prime}\right)
}\circ\left(  c_{1}\otimes\ldots\otimes c_{n^{\prime}\left(  n^{\prime
}-1\right)  }\right) \nonumber\\
&  =\mathbf{\mu}_{k}^{n_{k}}\circ\left(  \mathbf{\varepsilon}_{1}^{\left(
n^{\prime},r^{\prime}\right)  }\circ\left(  c_{1}\otimes\ldots\otimes
c_{\left(  n^{\prime}-1\right)  }\right)  \otimes\ldots\otimes
\mathbf{\varepsilon}_{n^{\prime}}^{\left(  n^{\prime},r^{\prime}\right)
}\circ\left(  c_{\left(  n^{\prime}-1\right)  \left(  n^{\prime}-1\right)
}\otimes\ldots\otimes c_{n^{\prime}\left(  n^{\prime}-1\right)  }\right)
\right)  ,\\
&  c_{i} \in C_{i},\ \ i=1,\ldots,n^{\prime}\left(  n^{\prime}-1\right)
,\nonumber
\end{align}
and the arity of the comultiplication coincides with the arity of the
underlying field%
\begin{equation}
n^{\prime}=n_{k}.
\end{equation}

\end{proposition}

\subsection{Polyadic coalgebras in the Sweedler notation}

The $\Bbbk$-linear coalgebra comultiplication map $\mathbf{\Delta}^{\left(
n^{\prime}\right)  }$ defined in \textbf{Definition \ref{def-comult}} is
useful for a \textquotedblleft diagrammatic\textquotedblright\ description of
polyadic coalgebras, and it corresponds to the algebra multiplication map
$\mathbf{\mu}^{\left(  n\right)  }$, which both manipulate with sets. However,
for concrete computations (with elements) we need an analog of the polyadic
algebra multiplication $\mu^{\left(  n\right)  }\equiv\mu_{A}^{\left(
n_{a}\right)  }$ from (\ref{mn}). The connection of $\mathbf{\mu}^{\left(
n\right)  }$ and $\mu^{\left(  n\right)  }$ is given by (\ref{ma}), which can
be treated as a \textquotedblleft bridge\textquotedblright\ between the
\textquotedblleft diagrammatic\textquotedblright\ and \textquotedblleft
elementwise\textquotedblright\ descriptions. The co-analog of (\ref{ma}) was
not considered, because the comultiplication has only one argument. To be
consistent, we introduce the \textquotedblleft elementwise\textquotedblright%
\ comultiplication $\Delta^{\left(  n^{\prime}\right)  }$ as the coanalog of
$\mu^{\left(  n\right)  }$ by the evaluation%
\begin{equation}
\mathbf{\Delta}^{\left(  n^{\prime}\right)  }\circ\left(  c\right)
=\Delta^{\left(  n^{\prime}\right)  }\left(  c\right)  ,\ \ \ c\in C.
\label{dc}%
\end{equation}
In general, one does not distinguish $\mathbf{\Delta}^{\left(  n^{\prime
}\right)  }$ and $\Delta^{\left(  n^{\prime}\right)  }$ and may use one symbol in
both descriptions.

In real \textquotedblleft elementwise\textquotedblright\ coalgebra
computations with many variables and comultiplications acting on them, the
indices and various letters reproduce themselves in such a way that it is
impossible to observe the structure of the expressions. Therefore, instead of
different letters in the binary decomposition (\textquotedblleft$\Delta\left(
c\right)  =\sum_{i}a_{i}\otimes b_{i}$\textquotedblright\ and (\ref{dp})) it
was proposed \cite{swe68} to use the \textsf{same letter} (\textquotedblleft%
$\Delta\left(  c\right)  =\sum_{i}c_{\left[  1\right]  ,i}\otimes c_{\left[
2\right]  ,i}$\textquotedblright), and then go from the \textsf{real} sum
$\sum_{i}$ to the \textsf{formal} sum $\sum_{\left[  c\right]  }$ as
(\textquotedblleft$\Delta\left(  c\right)  =\sum_{\left[  c\right]
}c_{\left[  1\right]  }\otimes c_{\left[  2\right]  }$\textquotedblright%
\ remembering the place of the components $c_{\left[  1\right]  },c_{\left[
2\right]  }$ only), because the real indices pullulate in complicated formulas
enormously. In simple cases, the sum sign was also omitted (\textquotedblleft%
$\Delta\left(  c\right)  =c_{\left[  1\right]  }\otimes c_{\left[  2\right]
}$\textquotedblright), which recalls the Einstein index summation rule in
physics. This trick abbreviated tedious coalgebra computations and was called
the (\textit{sumless}) \textit{Sweedler }(\textit{sigma})\textit{ notation}
(sometimes it is called the \textit{Heyneman-Sweedler notation} \cite{hey/swe}).

Now we can write $\Delta^{\left(  n^{\prime}\right)  }$ as a $n^{\prime}$-ary
decomposition in the manifest \textquotedblleft elementwise\textquotedblright%
\ form%
\begin{equation}
\Delta^{\left(  n^{\prime}\right)  }\left(  c\right)  =\left(  \nu^{\left(
m\right)  }\right)  ^{\circ\ell_{\Delta}}\left[  c_{\left[  1\right]
,1}\otimes c_{\left[  2\right]  ,1}\otimes\ldots\otimes c_{\left[  n^{\prime
}\right]  ,1},\ldots,c_{\left[  1\right]  ,N_{\Delta}}\otimes c_{\left[
2\right]  ,N_{\Delta}}\otimes\ldots\otimes c_{\left[  n^{\prime}\right]
,N_{\Delta}}\right]  ,\ \ c_{\left[  j\right]  ,i}\in C, \label{dn}%
\end{equation}
where $\ell_{\Delta}\in\mathbb{N}_{0}$ is a \textquotedblleft number of
additions\textquotedblright, and $N_{\Delta}\in\mathbb{N}$ is the
\textquotedblleft number of summands\textquotedblright. In the binary case,
the number of summands in the decomposition is not \textquotedblleft
algebraically\textquotedblright\ restricted, because $N_{\Delta}=\ell_{\Delta
}+1$. In the polyadic case, we have

\begin{assertion}
The admissible \textquotedblleft number of summands\textquotedblright%
\ $N_{\Delta}$\ in the polyadic comultiplication is%
\begin{equation}
N_{\Delta}=\ell_{\Delta}\left(  m-1\right)  +1,\ \ \ \ell_{\Delta}%
\in\mathbb{N}_{0},\ \ \ m\geq2. \label{nd}%
\end{equation}

\end{assertion}

Therefore, the \textquotedblleft quantization\textquotedblright\ of
$N_{\Delta}$ coincides with that of the $N$-dimensional polyadic algebra (see
\textbf{Assertion} \ref{as-quanN}).

Introduce the \textit{polyadic Sweedler notation} by exchanging in (\ref{dn})
the \textsf{real} $m$-ary addition $\nu^{\left(  m\right)  }$ by the
\textsf{formal} addition $\nu_{\left[  c\right]  }$ and writing%
\begin{equation}
\Delta^{\left(  n^{\prime}\right)  }\left(  c\right)  =\nu_{\left[  c\right]
}\left[  c_{\left[  1\right]  }\otimes c_{\left[  2\right]  }\otimes
\ldots\otimes c_{\left[  n^{\prime}\right]  }\right]  \Rightarrow c_{\left[
1\right]  }\otimes c_{\left[  2\right]  }\otimes\ldots\otimes c_{\left[
n^{\prime}\right]  }. \label{dv}%
\end{equation}
Remember here that we can formally add only $N_{\Delta}$ summands, because of
the \textquotedblleft quantization\textquotedblright\ (\ref{nd}) rule.

The polyadic Sweedler notation power can be seen in the following

\begin{example}
We apply (\ref{dv}) to the coassociativity (\ref{coa}) with $n^{\prime}=3$, to
obtain%
\begin{align}
\left(  \operatorname*{id}\otimes\operatorname*{id}\otimes\Delta^{\left(
3\right)  }\right)  \circ\Delta^{\left(  3\right)  }\left(  c\right)   &
=\left(  \operatorname*{id}\otimes\Delta^{\left(  3\right)  }\otimes
\operatorname*{id}\right)  \circ\Delta^{\left(  3\right)  }\left(  c\right)
=\left(  \Delta^{\left(  3\right)  }\otimes\operatorname*{id}\otimes
\operatorname*{id}\right)  \circ\Delta^{\left(  3\right)  }\left(  c\right)
\Rightarrow\\
&  =\nu_{\left[  c\right]  }\left[  c_{\left[  1\right]  }\otimes c_{\left[
2\right]  }\otimes\nu_{\left[  c_{2}\right]  }\left[  \left(  c_{\left[
3\right]  }\right)  _{\left[  1\right]  }\otimes\left(  c_{\left[  3\right]
}\right)  _{\left[  2\right]  }\otimes\left(  c_{\left[  3\right]  }\right)
_{\left[  3\right]  }\right]  \right] \nonumber\\
&  =\nu_{\left[  c\right]  }\left[  c_{\left[  1\right]  }\otimes\nu_{\left[
c_{2}\right]  }\left[  \left(  c_{\left[  2\right]  }\right)  _{\left[
1\right]  }\otimes\left(  c_{\left[  2\right]  }\right)  _{\left[  2\right]
}\otimes\left(  c_{\left[  2\right]  }\right)  _{\left[  3\right]  }\right]
\otimes c_{\left[  3\right]  }\right] \nonumber\\
&  =\nu_{\left[  c\right]  }\left[  \nu_{\left[  c_{2}\right]  }\left[
\left(  c_{\left[  1\right]  }\right)  _{\left[  1\right]  }\otimes\left(
c_{\left[  1\right]  }\right)  _{\left[  2\right]  }\otimes\left(  c_{\left[
1\right]  }\right)  _{\left[  3\right]  }\right]  \otimes c_{\left[  2\right]
}\otimes c_{\left[  3\right]  }\right]  .
\end{align}
After dropping the brackets and applying the Sweedler trick for the second
time, we get the same formal expression in all three cases%
\begin{equation}
\left(  \nu_{\left[  c\right]  }\right)  ^{\circ2}\left[  c_{\left[  1\right]
}\otimes c_{\left[  2\right]  }\otimes c_{\left[  3\right]  }\otimes
c_{\left[  4\right]  }\otimes c_{\left[  5\right]  }\right]  .
\end{equation}

\end{example}

Unfortunately, in the polyadic case the Sweedler notation looses too much information to be useful.

\begin{assertion}
The polyadic Sweedler notation can be applied to only the derived polyadic
coalgebras (see \textbf{Definition }\ref{def-derco}).
\end{assertion}

Nevertheless, if in an expression there are no coiterations, one can formally
use it (e.g., in the polyadic analog (\ref{ep}) of the counting axiom
\textquotedblleft$\sum\varepsilon\left(  c_{\left[  1\right]  }\right)
c_{\left[  2\right]  }=c$\textquotedblright).

\subsection{Polyadic group-like and primitive elements}

Let us consider some special kinds of elements in a polyadic coalgebra
$\mathrm{C}^{\left(  n^{\prime}\right)  }$. We should take into account that
in the polyadic case, as in (\ref{dn}), there can only be the admissible
\textquotedblleft number of summands\textquotedblright\ $N_{\Delta}$ (\ref{nd}).

\begin{definition}
An element $g\ $of $\mathrm{C}^{\left(  n^{\prime}\right)  }$ is called
\textit{polyadic semigroup-like}, if%
\begin{equation}
\Delta^{\left(  n^{\prime}\right)  }\left(  g\right)  =\overset{n^{\prime}%
}{\overbrace{g\otimes\ldots\otimes g}},,\ \ g\in C. \label{dg}%
\end{equation}
When $\mathrm{C}^{\left(  n^{\prime}\right)  }$ has the counit
$\mathbf{\varepsilon}^{\left(  n^{\prime},r^{\prime}\right)  }$ (\ref{ep}),
then $g$ is called \textit{polyadic group-like}, if (\textquotedblleft%
$\varepsilon\left(  g\right)  =1$\textquotedblright)%
\begin{equation}
\mathbf{\varepsilon}^{\left(  n^{\prime},r^{\prime}\right)  }\circ\left(
\overset{n^{\prime}-1}{\overbrace{g\otimes\ldots\otimes g}}\right)
=\overset{r^{\prime}}{\overbrace{e_{k}\otimes\ldots\otimes e_{k}}},
\end{equation}
where $e_{k}$ is the unit of the underlying polyadic field $\Bbbk$.
\end{definition}

\begin{definition}
An element $x\ $of $\mathrm{C}^{\left(  n^{\prime}\right)  }$ is called
\textit{polyadic skew }$k_{p}$-primitive, if (\textquotedblleft$\Delta\left(
x\right)  =g_{1}\otimes x+x\otimes g_{2}$\textquotedblright)%
\begin{align}
\Delta^{\left(  n^{\prime}\right)  }\left(  x\right)   &  =\left(
\nu^{\left(  m\right)  }\right)  ^{\circ\ell_{\Delta}}\left[  \left(
\overset{k_{p}}{\overbrace{g_{1}\otimes\ldots\otimes g_{k_{p}}}}%
\otimes\overset{n^{\prime}-k_{p}}{\overbrace{x\otimes\ldots\otimes x}}\right)
,\ldots,\right. \nonumber\\
&  \left.  \left(  \overset{n^{\prime}-k_{p}}{\overbrace{x\otimes\ldots\otimes
x}}\otimes\overset{k_{p}}{\overbrace{g_{\left(  N_{\Delta}-1\right)  k_{p}%
+1}\otimes\ldots\otimes g_{N_{\Delta}k_{p}}}}\right)  \right]  , \label{dx}%
\end{align}
where $1\leq k_{p}\leq n^{\prime}-1$, $N_{\Delta}=\ell_{\Delta}\left(
m-1\right)  +1\ $is the total \textquotedblleft number of
summands\textquotedblright, here $\ell_{\Delta}\in\mathbb{N}$ is the
\textquotedblleft number of $m$-ary additions\textquotedblright, and $g_{i}\in
C$, $i=1,\ldots,N_{\Delta}k_{p}$ are polyadic (semi-)group-like (\ref{dg}). In
(\ref{dx}) the $n^{\prime}-k_{p}$ elements $x$ move from the right to the left
by one.
\end{definition}

\begin{assertion}
If $k_{p}=n^{\prime}-1$, then $\Delta^{\left(  n^{\prime}\right)  }\left(
x\right)  $ is \textquotedblleft linear\textquotedblright\ in $x$, and
$n^{\prime}=\ell_{\Delta}\left(  m-1\right)  +1$.
\end{assertion}

In this case, we call $x$ a \textit{polyadic primitive element}.

\begin{example}
Let $n^{\prime}=3$ and $k_{p}=2$, then $m=3$, and we have only one ternary
addition $\ell_{\Delta}=1$%
\begin{align}
\Delta^{\left(  3\right)  }\left(  x\right)   &  =\nu^{\left(  3\right)
}\left[  g_{1}\otimes g_{2}\otimes x,\ g_{3}\otimes x\otimes g_{4},\ x\otimes
g_{5}\otimes g_{6}\right]  ,\\
\Delta^{\left(  3\right)  }\left(  g_{i}\right)   &  =g_{i}\otimes
g_{i}\otimes g_{i},\ \ \ i=1,\ldots,6.
\end{align}

The ternary coassociativity gives $g_{1}=g_{2}=g_{3}$ and $g_{4}=g_{5}=g_{6}$.
Therefore, the general form of the ternary primitive element is%
\begin{equation}
\Delta^{\left(  3\right)  }\left(  x\right)  =\nu^{\left(  3\right)  }\left[
g_{1}\otimes g_{1}\otimes x,\ g_{1}\otimes x\otimes g_{2},\ x\otimes
g_{2}\otimes g_{2}\right]  .
\end{equation}

Note that coassociativity leads to the \textsf{derived} comultiplication
(\ref{ndder}), because%
\begin{align}
\Delta^{\left(  3\right)  }\left(  x\right)   &  =\left(  \operatorname*{id}%
\otimes\Delta^{\left(  2\right)  }\right)  \Delta^{\left(  2\right)  }\left(
x\right)  =\left(  \Delta^{\left(  2\right)  }\otimes\operatorname*{id}%
\right)  \Delta^{\left(  2\right)  }\left(  x\right)  ,\\
\Delta^{\left(  2\right)  }\left(  x\right)   &  =g_{1}\otimes x+x\otimes
g_{2}.
\end{align}

The same situation occurs with the \textquotedblleft linear\textquotedblright%
\ comultiplication of any arity $n^{\prime}$, i.e. when $k_{p}=n^{\prime}-1$.
\end{example}

The most important difference with the binary case is the \textquotedblleft
intermediate\textquotedblright\ possibility $k_{p}<n^{\prime}-1$, when the
r.h.s. is \textquotedblleft nonlinear\textquotedblright\ in $x$.

\begin{example}
In the case where $n^{\prime}=3$ and $k_{p}=1$, we have $m=3$, and
$\ell_{\Delta}=1$%
\begin{align}
\Delta^{\left(  3\right)  }\left(  x\right)   &  =\nu^{\left(  3\right)
}\left[  g_{1}\otimes x\otimes x,\ x\otimes g_{2}\otimes x,\ x\otimes x\otimes
g_{3}\right]  ,\\
\Delta^{\left(  3\right)  }\left(  g_{i}\right)   &  =g_{i}\otimes
g_{i}\otimes g_{i},\ \ \ i=1,\ldots,3.
\end{align}

Now ternary coassociativity cannot be achieved with any values of $g_{i}$.
This is true for any arity $n^{\prime}$ and any \textquotedblleft
nonlinear\textquotedblright\ comultiplication.
\end{example}

Therefore, we arrive at the general structure

\begin{assertion}
In a polyadic coassociative coalgebra {\upshape$\mathrm{C}^{\left(  n^{\prime
}\right)  }$} polyadic primitive elements exist, if and only if the
$n^{\prime}$-ary comultiplication $\Delta^{\left(  n^{\prime}\right)  }$ is
derived {\upshape(\ref{ndder})} from the binary comultiplication
$\Delta^{\left(  2\right)  }$.
\end{assertion}

\subsection{Polyadic analog of duality}

The connection between binary associative algebras and coassociative
coalgebras (formally named as \textquotedblleft reversing
arrows\textquotedblright) is given in terms of the dual vector space (dual
module) concept. Informally, for a binary coalgebra $\mathrm{C}^{\left(
2\right)  }=\left\langle C\mid\Delta,\varepsilon\right\rangle $ considered as
a vector space over a binary field $\mathrm{k}$ (a $\mathrm{k}$-vector space),
its \textit{dual} is $C_{2}^{\ast}=\operatorname*{Hom}%
\nolimits_{\,{\scriptsize \mathrm{k}}\!}\left(  C,\mathrm{k}\right)  $ with
the natural \textit{pairing} $C^{\ast}\times C\rightarrow\mathrm{k}$ given by
$f\left(  c\right)  $, $f\in C^{\ast}$, $c\in C$. The \textit{canonical
injection} $\theta:C^{\ast}\otimes C^{\ast}\rightarrow\left(  C\otimes
C\right)  ^{\ast}$ is defined by%
\begin{equation}
\theta\left(  f_{1}\otimes f_{2}\right)  \circ\left(  c_{1}\otimes
c_{2}\right)  =f_{1}\left(  c_{1}\right)  f_{2}\left(  c_{2}\right)
,\ \ \ \ c_{1,2}\in C,\ \ \ f_{1,2}\in C^{\ast}, \label{ff}%
\end{equation}
which is an isomorphism in the finite-dimensional case. The transpose of
$\Delta:C\rightarrow C\otimes C$ is a $\mathrm{k}$-linear map $\Delta_{\ast
}:\left(  C\otimes C\right)  ^{\ast}\rightarrow C^{\ast}$ acting as
$\Delta_{\ast}\left(  \xi\right)  \left(  c\right)  =\xi\circ\left(
\Delta\left(  c\right)  \right)  $, where $\xi\in\left(  C\otimes C\right)
^{\ast}$, $c\in C$. The multiplication $\mu_{\ast}$ on the set $C^{\ast}$ is
the map $C^{\ast}\otimes C^{\ast}\rightarrow C^{\ast}$, and therefore we have
to use the canonical injection $\theta$ as follows%
\begin{align}
\mu_{\ast}  &  :C^{\ast}\otimes C^{\ast}\overset{\theta}{\rightarrow}\left(
C\otimes C\right)  ^{\ast}\overset{\Delta_{\ast}}{\rightarrow}C^{\ast},\\
\mu_{\ast}  &  =\Delta_{\ast}\circ\theta.
\end{align}

The associativity of $\mu_{\ast}$ follows from the coassociativity of $\Delta$.
Since $\mathrm{k}^{\ast}\simeq\mathrm{k}$, the dual of the counit is the unit
$\eta_{\ast}:\mathrm{k}\overset{\varepsilon^{\ast}}{\rightarrow}C^{\ast}$.
Therefore, $\mathrm{C}^{\left(  2\right)  \ast}=\left\langle C^{\ast}\mid
\mu_{\ast},\eta_{\ast}\right\rangle $ is a binary associative algebra which is
called the \textit{dual algebra }of the binary coalgebra\textit{ }%
$\mathrm{C}^{\left(  2\right)  }=\left\langle C\mid\Delta,\varepsilon
\right\rangle $ (see, e.g. \cite{radford}).

In the polyadic case, arities of the comultiplication, its dual multiplication
and the underlying field can be different, but connected by (\ref{ff}). Let us
consider a polyadic coassociative coalgebra $\mathrm{C}^{\left(  n^{\prime
}\right)  }$ with $n^{\prime}$-ary comultiplication $\Delta^{\left(
n^{\prime}\right)  }$ (\ref{dc}) over $\Bbbk^{\left(  m_{k},n_{k}\right)  }$.
In search of the most general polyadic analog of the injection (\ref{ff}), we
arrive at the possibility of multiplace morphisms.

\begin{definition}
For the polyadic coalgebra $\mathrm{C}^{\left(  n^{\prime}\right)  }$
considered as a polyadic vector space over $\Bbbk^{\left(  m_{k},n_{k}\right)
}$, a \textit{polyadic dual }is $C^{\ast}=\operatorname*{Hom}\nolimits_{\Bbbk
}\left(  C^{\otimes s},K\right)  $ with $s$-\textit{place pairing} $C^{\ast
}\times\overset{s}{\overbrace{C\times\ldots\times C}}\rightarrow K$ giving by
$f^{\left(  s\right)  }\left(  c_{1},\ldots,c_{s}\right)  $, $f\in C^{\ast}$,
$c_{i}\in C$, $s\in\mathbb{N}$.
\end{definition}

While constructing a polyadic analog of (\ref{ff}), recall that for any
$n^{\prime}$-ary operation the admissible length of a co-word is $\ell
^{\prime}\left(  n^{\prime}-1\right)  +1$, where $\ell^{\prime}$ is the number
of the iterated operation (\ref{dl}).

\begin{definition}
A \textit{polyadic canonical injection map} $\mathbf{\theta}^{\left(  n^{\ast
},n^{\prime},s\right)  }$ of $\mathrm{C}^{\left(  n^{\prime}\right)  }$ is
defined by%
\begin{align}
&  \mathbf{\theta}^{\left(  n^{\ast},n^{\prime},s\right)  }\circ\left(
f_{1}^{\left(  s\right)  }\otimes\ldots\otimes f_{n^{\ast}}^{\left(  s\right)
}\right)  \circ\left(  c_{1}\otimes\ldots\otimes c_{\ell^{\prime}\left(
n^{\prime}-1\right)  +1}\right)  =\nonumber\\
&  \left(  \mu_{k}^{\left(  n_{k}\right)  }\right)  ^{\circ\ell_{k}}\left[
f_{1}^{\left(  s\right)  }\left(  c_{1},\ldots,c_{s}\right)  ,\ldots
,f_{\ell_{k}\left(  n_{k}-1\right)  +1}^{\left(  s\right)  }\left(  c_{\left(
n^{\ast}-1\right)  s+1},\ldots,c_{n^{\ast}s}\right)  \right]  , \label{tn}%
\end{align}
where%
\begin{align}
n^{\ast}s  &  =\ell^{\prime}\left(  n^{\prime}-1\right)  +1,\ \ \ \ell
^{\prime}\in\mathbb{N},\ \ n^{\prime}\geq2,\label{ns1}\\
n^{\ast}  &  =\ell_{k}\left(  n_{k}-1\right)  +1,\ \ \ \ell_{k}\in
\mathbb{N},\ \ n_{k}\geq2. \label{ns2}%
\end{align}

\end{definition}

It is obvious that $\theta^{\left(  2,2,1\right)  }=\theta$ from (\ref{ff}).
Then, the \textit{polyadic transpose map} of the $n^{\prime}$-ary
comultiplication $\Delta^{\left(  n^{\prime}\right)  }:C\rightarrow
\overset{n^{\prime}}{\overbrace{C\otimes\ldots\otimes C}}$ \thinspace is a
$\Bbbk$-linear map $\Delta_{\ast}^{\left(  n^{\prime\prime}\right)  }:\left(
\overset{n^{\prime\prime}}{\overbrace{C\otimes\ldots\otimes C}}\right)
^{\ast}\rightarrow C^{\ast}$ such that%
\begin{align}
&  \Delta_{\ast}^{\left(  n^{\prime\prime}\right)  }\circ\left(  \xi^{\left(
n^{\prime\prime}\right)  }\right)  \left(  c\right)  =\xi^{\left(
n^{\prime\prime}\right)  }\circ\left(  \left(  \Delta^{\left(  n^{\prime
}\right)  }\right)  ^{\circ\ell^{\prime}}\left(  c\right)  \right)
,\nonumber\\
&  \xi^{\left(  n^{\prime\prime}\right)  }\in\left(  \overset{n^{\prime\prime
}}{\overbrace{C\otimes\ldots\otimes C}}\right)  ^{\ast},\ n^{\prime\prime
}=\ell^{\prime}\left(  n^{\prime}-1\right)  +1,\ c\in C \label{dns}%
\end{align}
where $\ell^{\prime}$ is the \textquotedblleft number of
comultiplications\textquotedblright\ (see (\ref{mx}) for multiplications and
(\ref{dl})).

\begin{definition}
\label{def-dual}A $n^{\ast}$-ary multiplication map $\mathbf{\mu}_{\ast
}^{\left(  n^{\ast}\right)  }$ which is (\textit{one way}) \textit{dual} to
the $n^{\prime}$-ary comultiplication map $\mathbf{\Delta}^{\left(  n^{\prime
}\right)  }$ is given by the composition of the polyadic canonical injection
$\mathbf{\theta}^{\left(  n^{\ast},n^{\prime},s\right)  }$ (\ref{tn}) and the
polyadic transpose $\mathbf{\Delta}_{\ast}^{\left(  n^{\prime\prime}\right)
}$ (\ref{dns}) by%
\begin{equation}
\mathbf{\mu}_{\ast}^{\left(  n^{\ast}\right)  }=\mathbf{\Delta}_{\ast
}^{\left(  n^{\prime\prime}\right)  }\circ\mathbf{\theta}^{\left(  n^{\ast
},n^{\prime},s\right)  }.
\end{equation}

\end{definition}

Indeed, using (\ref{tn}) and (\ref{dns}) we obtain (in Sweedler notation)%
\begin{align}
&  \mathbf{\mu}_{\ast}^{\left(  n^{\ast}\right)  }\circ\left(  f_{1}^{\left(
s\right)  }\otimes\ldots\otimes f_{n^{\ast}}^{\left(  s\right)  }\right)
\circ\left(  c\right)  =\mathbf{\Delta}_{\ast}^{\left(  n^{\prime\prime
}\right)  }\circ\mathbf{\theta}^{\left(  n^{\ast},n^{\prime},s\right)  }%
\circ\left(  f_{1}^{\left(  s\right)  }\otimes\ldots\otimes f_{n^{\ast}%
}^{\left(  s\right)  }\right)  \circ\left(  c\right) \nonumber\\
&  =\mathbf{\theta}^{\left(  n^{\ast},n^{\prime},s\right)  }\circ\left(
f_{1}^{\left(  s\right)  }\otimes\ldots\otimes f_{n^{\ast}}^{\left(  s\right)
}\right)  \circ\left(  \left(  \Delta^{\left(  n^{\prime}\right)  }\right)
^{\circ\ell^{\prime}}\left(  c\right)  \right) \nonumber\\
&  =\left(  \mu_{k}^{\left(  n_{k}\right)  }\right)  ^{\circ\ell_{k}}\left[
f_{1}^{\left(  s\right)  }\left(  c_{\left[  1\right]  },\ldots,c_{\left[
s\right]  }\right)  ,\ldots,f_{\ell_{k}\left(  n_{k}-1\right)  +1}^{\left(
s\right)  }\left(  c_{\left[  \left(  n^{\ast}-1\right)  s+1\right]  }%
,\ldots,c_{\left[  n^{\ast}s\right]  }\right)  \right]  , \label{mfc}%
\end{align}
and (\ref{ns1})--(\ref{ns2}) are valid, from which we arrive at

\begin{assertion}
\label{as-dual}In the polyadic case the arity $n^{\ast}$ of the multiplication
{\upshape$\mathbf{\mu}_{\ast}^{\left(  n^{\ast}\right)  }$} can be different
from the arity $n^{\prime}$ of the initial coalgebra {\upshape$\mathrm{C}%
^{\left(  n^{\prime}\right)  }$}.
\end{assertion}

\begin{remark}
If $n^{\ast}\neq n^{\prime}$ and $s\geq2$, the word \textquotedblleft
duality\textquotedblright\ can only be used conditionally.
\end{remark}

\subsection{Polyadic convolution product}

If $\mathrm{A}^{\left(  2\right)  }=\left\langle A\mid\mu,\eta\right\rangle $
is a binary algebra and $\mathrm{C}^{\left(  2\right)  }=\left\langle
C\mid\Delta,\varepsilon\right\rangle $ is a binary coalgebra over a binary
field $\mathrm{k}$, then a more general set of $\mathrm{k}$-linear maps
$\operatorname*{Hom}\nolimits_{\,{\scriptsize \mathrm{k}}\!}\left(
C,A\right)  $ can be considered, while its particular case where $\mathrm{A}%
^{\left(  2\right)  }=\mathrm{k}$ corresponds to the above duality. The
multiplication on $\operatorname*{Hom}\nolimits_{\,{\scriptsize \mathrm{k}}%
\!}\left(  C,A\right)  $ is the \textit{convolution product} $\left(
\star\right)  $ which can be \textsf{uniquely} constructed in the natural way:
by applying first comultiplication $\Delta$ and then multiplication $\mu
\equiv\left(  \cdot\right)  $ to an element of $C$, as $C\overset{\Delta
}{\longrightarrow}C\otimes C\overset{f\otimes g}{\longrightarrow}A\otimes
A\overset{\mu}{\longrightarrow}A$ or $f\star g=\mu\circ\left(  f\otimes
g\right)  \circ\Delta$, where $f,g\in\operatorname*{Hom}%
\nolimits_{\,{\scriptsize \mathrm{k}}\!}\left(  C,A\right)  $. The
associativity of the convolution product follows from the associativity of
$\mu$ and coassociativity of $\Delta$, and the role of the identity (neutral
element) in $\operatorname*{Hom}\nolimits_{\,{\scriptsize \mathrm{k}}%
\!}\left(  C,A\right)  $ is played by the composition of the unit map
$\eta:\mathrm{k}\rightarrow A$ and the counit map $\varepsilon:C\rightarrow
\mathrm{k}$, such that $e_{\star}=\eta\circ\varepsilon\in\operatorname*{Hom}%
\nolimits_{\,{\scriptsize \mathrm{k}}\!}\left(  C,A\right)  $, because
$e_{\star}\star f=f\star e_{\star}=f$. Indeed, from the obvious relation
$\operatorname*{id}\nolimits_{A}\circ f\circ\operatorname*{id}\nolimits_{C}=f$
and the unit and counit axioms it follows that%
\begin{equation}
\mu\circ\left(  \eta\otimes\operatorname*{id}\nolimits_{A}\right)
\circ\left(  \operatorname*{id}\nolimits_{K}\otimes f\right)  \circ\left(
\varepsilon\otimes\operatorname*{id}\nolimits_{C}\right)  \circ\Delta=\mu
\circ\left(  \eta\circ\operatorname*{id}\nolimits_{K}\circ\varepsilon\right)
\otimes\left(  \operatorname*{id}\nolimits_{A}\circ f\circ\operatorname*{id}%
\nolimits_{C}\right)  \circ\Delta=e_{\star}\star f=f, \label{mi}%
\end{equation}
or in Sweedler notation $\varepsilon\left(  c_{\left[  1\right]  }\right)
\cdot f\left(  c_{\left[  2\right]  }\right)  =f\left(  c_{\left[  1\right]
}\right)  \cdot\varepsilon\left(  c_{\left[  2\right]  }\right)  =f\left(
c\right)  $.

The polyadic analog of duality and (\ref{mfc}) offer an idea of how to
generalize the binary convolution product to the most exotic case, when the
algebra and coalgebra have \textsf{different arities} $n\neq n^{\prime}$.

Let $\mathrm{A}^{\left(  n\right)  }$ and $\mathrm{C}^{\left(  n^{\prime
}\right)  }$ be, respectively, a polyadic associative algebra and a
coassociative coalgebra over the same polyadic field $\Bbbk^{\left(
m_{k},n_{k}\right)  }$. If they are both unital and counital respectively,
then we can consider a polyadic analog of the composition $\eta\circ
\varepsilon$. The crucial difference from the binary case is that now
$\mathbf{\eta}^{\left(  r,n\right)  }$ and $\mathbf{\varepsilon}^{\left(
n^{\prime},r^{\prime}\right)  }$ are multiplace multivalued maps (\ref{mu})
and (\ref{ep}). Their composition is%
\begin{equation}
\mathbf{e}_{\star}^{\left(  n^{\prime},n\right)  }=\mathbf{\eta}^{\left(
r,n\right)  }\circ\mathbf{\gamma}^{\left(  r^{\prime},r\right)  }%
\circ\mathbf{\varepsilon}^{\left(  n^{\prime},r^{\prime}\right)  }%
\in\operatorname*{Hom}\nolimits_{\Bbbk}\left(  C^{\otimes\left(  n^{\prime
}-1\right)  },A^{\otimes\left(  n-1\right)  }\right)  , \label{es}%
\end{equation}
where the multiplace multivalued map $\mathbf{\gamma}^{\left(  r^{\prime
},r\right)  }\in\operatorname*{Hom}\nolimits_{\Bbbk}\left(  K^{\otimes
r^{\prime}},K^{\otimes r}\right)  $ is, obviously, $\left(  \simeq\right)  $,
and the diagram%
\begin{equation}
\begin{diagram} C^{\otimes\left( n^{\prime }-1\right) }& \rTo^{\mathbf{\varepsilon}^{\left( n^{\prime},r^{\prime}\right) }} & K^{\otimes r^{\prime},} \\ \dTo^{\mathbf{e}_{\star}^{\left( n,n^{\prime}\right) }} & & \dTo_{\mathbf{\gamma}^{\left( r^{\prime},r\right) } \left(\simeq\right)} \\ A^{\otimes\left( n-1\right) } & \lTo^{ \mathbf{\eta}^{\left( r,n\right) }} & K^{\otimes r} \\ \end{diagram}
\end{equation}
commutes.

The formula (\ref{es}) leads us to propose

\begin{conjecture}
A polyadic analog of the convolution should be considered for
\textsf{multiplace multivalued} $\Bbbk$-linear maps in $\operatorname*{Hom}%
\nolimits_{\Bbbk}\left(  C^{\otimes\left(  n^{\prime}-1\right)  }%
,A^{\otimes\left(  n-1\right)  }\right)  $.
\end{conjecture}

In this way, we arrive at the following

\begin{condition}
Introduce the $\Bbbk$-linear maps $\mathfrak{f}^{\left(  i\right)
}:C^{\otimes\left(  n^{\prime}-1\right)  }\rightarrow A^{\otimes\left(
n-1\right)  }$, $i=1,\ldots,n_{\star}$, where $n_{\star}\geq2$. To create a
closed $n_{\star}$-ary operation for them, we use the $\ell$-iterated
multiplication map $\left(  \mathbf{\mu}^{\left(  n\right)  }\right)
^{\circ\ell}:A^{\otimes\ell\left(  n-1\right)  +1}\rightarrow A$ and
$\ell^{\prime}$-iterated comultiplication map $\left(  \mathbf{\Delta
}^{\left(  n^{\prime}\right)  }\right)  ^{\circ\ell^{\prime}}:C\rightarrow
C^{\otimes\ell^{\prime}\left(  n^{\prime}-1\right)  +1}$. Then we compose the
above $\Bbbk$-linear maps in the same way as is done above for the binary
case%
\begin{align}
&
\begin{diagram} C^{\otimes\left( n^{\prime}-1\right) }&\rTo{\left(\left( \mathbf{\Delta }^{\left( n^{\prime}\right) }\right) ^{\circ\ell^{\prime}}\right)^{\otimes\left( n^{\prime}-1\right)}}& C^{\otimes\left( n^{\prime}-1\right) \left( \ell^{\prime}\left( n^{\prime}-1\right) +1\right) }& \rTo{\mathbf{\tau}_{medial}^{\left( n_{\star },n^{\prime}-1\right) } }C^{\otimes\left( n^{\prime}-1\right) \left( \ell^{\prime}\left( n^{\prime}-1\right) +1\right) }\end{diagram}\nonumber\\
&
\begin{diagram}&\rTo{\mathfrak{f}^{\left( 1\right) }\otimes\ldots\otimes\mathfrak{f}^{\left( n_{\star}\right) }}&A^{\otimes\left( n-1\right) \left( \ell^{\prime}\left( n^{\prime}-1\right) +1\right) } & \rTo{\mathbf{\tau}_{medial}^{\left( n-1,n_{\star}\right) }}&A^{\otimes\left( n-1\right) \left( \ell\left( n-1\right) +1\right) }&\rTo{\left(\left( \mathbf{\mu}^{\left( n\right) }\right) ^{\circ\ell}\right)^{\otimes\left( n-1\right)}}&A^{\otimes\left( n-1\right) }.\end{diagram} \label{cn}%
\end{align}
where $\mathbf{\tau}_{medial}^{\left(  n_{\star},n^{\prime}-1\right)  }$ and
$\mathbf{\tau}_{medial}^{\left(  n-1,n_{\star}\right)  }$ are the medial maps
(\ref{an}) acting on the Sweedler components of $c$ and $\mathfrak{f}^{\left(
i\right)  }$, respectively. To make the sequence of maps (\ref{cn})
consistent, the arity $n_{\star}$ is connected with the iteration numbers
$\ell,\ell^{\prime}$ by $n_{\star}=\ell\left(  n-1\right)  +1=\ell^{\prime
}\left(  n^{\prime}-1\right)  +1$, $\ell,\ell^{\prime}\in\mathbb{N}$.
\end{condition}

\begin{definition}
Let $\mathrm{A}^{\left(  n\right)  }$ and $\mathrm{C}^{\left(  n^{\prime
}\right)  }$ be a $n$-ary associative algebra and $n^{\prime}$-ary coassociative
coalgebra over a polyadic field $\Bbbk$ (the existence of the unit and counit
here is mandatory), then the set $\operatorname*{Hom}\nolimits_{\Bbbk}\left(
C^{\otimes\left(  n^{\prime}-1\right)  },A^{\otimes\left(  n-1\right)
}\right)  $ is closed under the $n_{\star}$-\textit{ary convolution product}
map $\mathbf{\mu}_{\star}^{\left(  n_{\star}\right)  }$ defined by%
\begin{align}
&  \mathbf{\mu}_{\star}^{\left(  n_{\star}\right)  }\circ\left(
\mathfrak{f}^{\left(  1\right)  }\otimes\ldots\otimes\mathfrak{f}^{\left(
n_{\star}\right)  }\right)  =\label{mff}\\
&  \left(  \left(  \mathbf{\mu}^{\left(  n\right)  }\right)  ^{\circ\ell
}\right)  ^{\otimes\left(  n-1\right)  }\circ\mathbf{\tau}_{medial}^{\left(
n-1,n_{\star}\right)  }\circ\left(  \mathfrak{f}^{\left(  1\right)  }%
\otimes\ldots\otimes\mathfrak{f}^{\left(  n_{\star}\right)  }\right)
\circ\mathbf{\tau}_{medial}^{\left(  n_{\star},n^{\prime}-1\right)  }%
\circ\left(  \left(  \mathbf{\Delta}^{\left(  n^{\prime}\right)  }\right)
^{\circ\ell^{\prime}}\right)  ^{\otimes\left(  n^{\prime}-1\right)
},\nonumber
\end{align}
and its arity is given by the following $n_{\star}$-\textit{consistency
condition}%
\begin{equation}
n_{\star}-1=\ell\left(  n-1\right)  =\ell^{\prime}\left(  n^{\prime}-1\right)
. \label{ns}%
\end{equation}

\end{definition}

\begin{definition}
The set of $\Bbbk$-linear maps $\mathfrak{f}^{\left(  i\right)  }%
\in\operatorname*{Hom}\nolimits_{\Bbbk}\left(  C^{\otimes\left(  n^{\prime
}-1\right)  },A^{\otimes\left(  n-1\right)  }\right)  $ endowed with the
convolution product (\ref{mff}) is called a \textit{polyadic convolution
algebra}%
\begin{equation}
\mathrm{C}_{\star}^{\left(  n^{\prime},n\right)  }=\left\langle
\operatorname*{Hom}\nolimits_{\Bbbk}\left(  C^{\otimes\left(  n^{\prime
}-1\right)  },A^{\otimes\left(  n-1\right)  }\right)  \mid\mathbf{\mu}_{\star
}^{\left(  n_{\star}\right)  }\right\rangle . \label{cs}%
\end{equation}

\end{definition}

\begin{example}
An important case is given by the binary algebra $\mathrm{A}^{\left(
2\right)  }$ and coalgebra $\mathrm{C}^{\left(  2\right)  }$ ($n=n^{\prime}%
=2$), when the number of iterations are equal $\ell=\ell^{\prime}$, and the
arity $n_{\star}$ becomes%
\begin{equation}
n_{\star}=\ell+1=\ell^{\prime}+1,\ \ \ \ell,\ell^{\prime}\in\mathbb{N},
\label{nll}%
\end{equation}
while the $n_{\star}$-ary convolution product in $\operatorname*{Hom}%
\nolimits_{\Bbbk}\left(  C,A\right)  $ takes the form%
\begin{equation}
\mathbf{\mu}_{\star}^{\left(  n_{\star}\right)  }\circ\left(  \mathfrak{f}%
^{\left(  1\right)  }\otimes\ldots\otimes\mathfrak{f}^{\left(  n_{\star
}\right)  }\right)  =\mathbf{\mu}^{\circ\left(  n_{\star}-1\right)  }%
\circ\left(  \mathfrak{f}^{\left(  1\right)  }\otimes\ldots\otimes
\mathfrak{f}^{\left(  n_{\star}\right)  }\right)  \circ\mathbf{\Delta}%
^{\circ\left(  n_{\star}-1\right)  },\ \ \mathfrak{f}^{\left(  i\right)  }%
\in\operatorname*{Hom}\nolimits_{\Bbbk}\left(  C,A\right)  , \label{mnf}%
\end{equation}
where $\mathbf{\mu=\mu}^{\left(  2\right)  }$ and $\mathbf{\Delta=\Delta
}^{\left(  2\right)  }$ are the binary multiplication and comultiplication
maps respectively.
\end{example}

\begin{definition}
\label{def-der}The polyadic convolution algebra $\mathrm{C}_{\star}^{\left(
2,2\right)  }$ determined by the binary algebra and binary coalgebra
(\ref{mnf}) is called \textit{derived}.
\end{definition}

\begin{corollary}
The arity $n_{\star}$ of the derived polyadic convolution algebra is
\textsf{unrestricted} and can take any integer value $n_{\star}\geq2$.
\end{corollary}

\begin{remark}
If the polyadic tensor product and the underlying polyadic field $\Bbbk$ are
derived (see discussion in \textsc{Section} \ref{sec-polf} and \cite{dup2017}%
), while all maps coincide $\mathfrak{f}^{\left(  i\right)  }=\mathfrak{f}$,
the convolution product (\ref{mnf}) is called the \textit{Sweedler power} of
$\mathfrak{f}$ \cite{kas/mon/ng} or the \textit{Adams operator} \cite{agu/lau}%
. In the binary case they denoted it by $\left(  \mathfrak{f}\right)
^{n_{\star}}$, but for the $n_{\star}$-ary product this is the \textsf{first}
polyadic power of $\mathfrak{f}$ (see (\ref{mx})).
\end{remark}

Obviously, some interesting algebraic objects are nonderived, and here they are
determined by $n+n^{\prime}\geq5$, and also the arities of the algebra and coalgebra
can be different $n\neq n^{\prime}$, which is a more exotic and exciting possibility.
Generally, the arity $n_{\star}$ of the convolution product (\ref{mff}) is
not arbitrary and is \textquotedblleft quantized\textquotedblright\ by solving
(\ref{ns}) in integers. The values $n_{\star}$ for minimal arities
$n,n^{\prime}$ are presented in \textsc{Table} \ref{tab-nstar}.

\begin{table}[th]
\caption[Arity values $n_{\star}$ of the polyadic convolution product.]{ Arity
values $n_{\star}$ of the polyadic convolution product (\ref{mff}), allowed by
(\ref{ns}). The framed box corresponds to the binary convolution product.}%
\label{tab-nstar}
\begin{center}
{ \resizebox{\textwidth}{!} {
\begin{tabular}
[c]{||c|c|c|c|c|c|c|c|c|c|c|c|c|c|}\hline\hline
\backslashbox{$n^{\prime}$}{$n$} & $\mathbf{\mu}$ & \multicolumn{3}{|c|}{$n=2$} & \multicolumn{3}{|c|}{$n=3$} & \multicolumn{3}{|c|}{$n=4$} &
\multicolumn{3}{|c|}{$n=5$}\\\hline
$\mathbf{\Delta}$ & \backslashbox{$^{\ell^\prime}$}{$\ell$} & $\ell=1$ &
$\ell=2$ & $\ell=3$ & $\ell=1$ & $\ell=2$ & $\ell=3$ & $\ell=1$ & $\ell=2$ &
$\ell=3$ & $\ell=1$ & $\ell=2$ & $\ell=3$\\\hline\hline
& $\ell^{\prime}=1$ & \multicolumn{1}{||c|}{\framebox{\textbf{2}}} &  &  &  &  &  &  &  &
&  &  & \\\cline{2-14}$n^{\prime}=2$ & $\ell^{\prime}=2$ & \multicolumn{1}{||c|}{} & \textbf{3} &  &
\textbf{3} &  &  &  &  &  &  &  & \\\cline{2-14}
& $\ell^{\prime}=3$ & \multicolumn{1}{||c|}{} &  & \textbf{4} &  &  &  &
\textbf{4} &  &  &  &  & \\\hline\hline
& $\ell^{\prime}=1$ & \multicolumn{1}{||c|}{} & \textbf{3} &  & \textbf{3} &
&  &  &  &  &  &  & \\\cline{2-14}$n^{\prime}=3$ & $\ell^{\prime}=2$ & \multicolumn{1}{||c|}{} &  &  &  &
\textbf{5} &  &  &  &  & \textbf{5} &  & \\\cline{2-14}
& $\ell^{\prime}=3$ & \multicolumn{1}{||c|}{} &  &  &  &  & \textbf{7} &  &
\textbf{7} &  &  &  & \\\hline\hline
& $\ell^{\prime}=1$ & \multicolumn{1}{||c|}{} &  & \textbf{4} &  &  &  &
\textbf{4} &  &  &  &  & \\\cline{2-14}$n^{\prime}=4$ & $\ell^{\prime}=2$ & \multicolumn{1}{||c|}{} &  &  &  &  &
\textbf{7} &  & \textbf{7} &  &  &  & \\\cline{2-14}
& $\ell^{\prime}=3$ & \multicolumn{1}{||c|}{} &  &  &  &  &  &  &  &
\textbf{10} &  &  & \\\hline\hline
& $\ell^{\prime}=1$ & \multicolumn{1}{||c|}{} &  &  &  & \textbf{5} &  &  &  &
& \textbf{5} &  & \\\cline{2-14}$n^{\prime}=5$ & $\ell^{\prime}=2$ & \multicolumn{1}{||c|}{} &  &  &  &  &  &
&  &  &  & \textbf{9} & \\\cline{2-14}
& $\ell^{\prime}=3$ & \multicolumn{1}{||c|}{} &  &  &  &  &  &  &  &  &  &  &
\textbf{13}\\\hline
\end{tabular}
}}
\end{center}
\end{table}

The most unusual possibility is the existence of nondiagonal entries, which
correspond to unequal arities of multiplication and comultiplication $n\neq
n^{\prime}$. The table is symmetric, which means that the arity $n_{\star}$ is
invariant under the exchange $\left(  n,\ell\right)  \longleftrightarrow
\left(  n^{\prime},\ell^{\prime}\right)  $ following from (\ref{ns}).

\begin{example}
[$\operatorname*{Hom}\nolimits_{\Bbbk}\left(  C,A\right)  $]In the simplest
derived case (\ref{mnf}), when both algebra $\mathrm{A}^{\left(
2\right)  }=\left\langle A\mid\mathbf{\mu}\right\rangle $ and coalgebra
$\mathrm{C}^{\left(  2\right)  }=\left\langle C\mid\mathbf{\Delta
}\right\rangle $ are binary with $n=2,\ell=2,n^{\prime}=2,\ell^{\prime}=2$, it
is possible to obtain the ternary convolution product $\mathbf{\mu}_{\star
}^{\left(  3\right)  }$ of the maps $\mathfrak{f}^{\left(  i\right)
}:C\rightarrow A$, $i=1,2,3$, using Sweedler notation for $\mathbf{\Delta
}^{\circ2}\equiv\left(  \operatorname*{id}\nolimits_{C}\otimes\mathbf{\Delta
}\right)  \circ\mathbf{\Delta}$ as $\Delta^{\circ2}\left(  c\right)
=c_{\left[  1\right]  }\otimes c_{\left[  2\right]  }\otimes c_{\left[
3\right]  }$, $\mathbf{\mu}^{\circ2}\equiv\mathbf{\mu\circ}\left(
\operatorname*{id}\nolimits_{A}\otimes\mathbf{\mu}\right)  :A^{\otimes
3}\rightarrow A$, and the elementwise description using the evaluation%
\begin{equation}
\mathbf{\mu}_{\star,der}^{\left(  3\right)  }\circ\left(  \mathfrak{f}%
^{\left(  1\right)  }\otimes\mathfrak{f}^{\left(  2\right)  }\otimes
\mathfrak{f}^{\left(  3\right)  }\right)  \circ\left(  c\right)  =\mu^{\circ
2}\left[  \mathrm{f}^{\left(  1\right)  }\left(  c_{\left[  1\right]
}\right)  ,\mathrm{f}^{\left(  2\right)  }\left(  c_{\left[  2\right]
}\right)  ,\mathrm{f}^{\left(  3\right)  }\left(  c_{\left[  3\right]
}\right)  \right]  . \label{msder}%
\end{equation}

\end{example}

\begin{example}
[$\operatorname*{Hom}\nolimits_{\Bbbk}\left(  C,A^{\otimes2}\right)  $,
$\operatorname*{Hom}\nolimits_{\Bbbk}\left(  C^{\otimes2},A\right)  $%
]Nonbinary, nonderived and nonsymmetric cases:

\begin{enumerate}
\item The ternary algebra $\mathrm{A}^{\left(  3\right)  }=\left\langle
A\mid\mathbf{\mu}^{\left(  3\right)  }\right\rangle $ and the binary coalgebra
$\mathrm{C}^{\left(  2\right)  }=\left\langle C\mid\mathbf{\Delta
}\right\rangle $, such that $n=3,\ell=1,n^{\prime}=2,\ell^{\prime}=2$ giving a
ternary convolution product of the maps $\mathfrak{f}^{\left(  i\right)
}:C\rightarrow A^{\otimes2}$, $i=1,2,3$. In the elementwise description
$\mathfrak{f}^{\left(  i\right)  }\circ\left(  c\right)  =\mathrm{f}_{\left[
1\right]  }^{\left(  i\right)  }\left(  c\right)  \otimes\mathrm{f}_{\left[
2\right]  }^{\left(  i\right)  }\left(  c\right)  $, $c\in C$. Using
(\ref{mff}), we obtain the manifest form of the nonderived ternary convolution
product by evaluation%
\begin{align}
&  \mathbf{\mu}_{\star}^{\left(  3\right)  }\circ\left(  \mathfrak{f}^{\left(
1\right)  }\otimes\mathfrak{f}^{\left(  2\right)  }\otimes\mathfrak{f}%
^{\left(  3\right)  }\right)  \circ\left(  c\right) \nonumber\\
&  =\mu^{\left(  3\right)  }\left[  \mathrm{f}_{\left[  1\right]  }^{\left(
1\right)  }\left(  c_{\left[  1\right]  }\right)  ,\mathrm{f}_{\left[
1\right]  }^{\left(  2\right)  }\left(  c_{\left[  2\right]  }\right)
,\mathrm{f}_{\left[  1\right]  }^{\left(  3\right)  }\left(  c_{\left[
3\right]  }\right)  \right]  \otimes\mu^{\left(  3\right)  }\left[
\mathrm{f}_{\left[  2\right]  }^{\left(  1\right)  }\left(  c_{\left[
1\right]  }\right)  ,\mathrm{f}_{\left[  2\right]  }^{\left(  2\right)
}\left(  c_{\left[  2\right]  }\right)  ,\mathrm{f}_{\left[  2\right]
}^{\left(  3\right)  }\left(  c_{\left[  3\right]  }\right)  \right]  ,
\end{align}
where $\Delta^{\circ2}\left(  c\right)  =c_{\left[  1\right]  }\otimes
c_{\left[  2\right]  }\otimes c_{\left[  3\right]  }.$

\item The algebra is binary $\mathrm{A}^{\left(  2\right)  }=\left\langle
A\mid\mathbf{\mu}\right\rangle $, and the coalgebra is ternary $\mathrm{C}%
^{\left(  3\right)  }=\left\langle C\mid\mathbf{\Delta}^{\left(  3\right)
}\right\rangle $, which corresponds to $n=2,\ell=2,n^{\prime}=3,\ell^{\prime
}=1$, the maps $\mathfrak{f}^{\left(  i\right)  }:C^{\otimes2}\rightarrow A$,
$i=1,2,3$ in the elementwise description are two place, $\mathfrak{f}^{\left(
i\right)  }\circ\left(  c_{1}\otimes c_{2}\right)  =\mathrm{f}^{\left(
i\right)  }\left(  c_{1},c_{2}\right)  $, $c_{1,2}\in C$, and $\left(
\Delta^{\left(  3\right)  }\right)  ^{\otimes2}\left(  c^{\left(  1\right)
}\otimes c^{\left(  2\right)  }\right)  =\left(  c_{\left[  1\right]
}^{\left(  1\right)  }\otimes c_{\left[  2\right]  }^{\left(  1\right)
}\otimes c_{\left[  3\right]  }^{\left(  1\right)  }\right)  \otimes\left(
c_{\left[  1\right]  }^{\left(  2\right)  }\otimes c_{\left[  2\right]
}^{\left(  2\right)  }\otimes c_{\left[  3\right]  }^{\left(  2\right)
}\right)  $. The ternary convolution product is%
\begin{equation}
\mathbf{\mu}_{\star}^{\left(  3\right)  }\circ\left(  \mathfrak{f}^{\left(
1\right)  }\otimes\mathfrak{f}^{\left(  2\right)  }\otimes\mathfrak{f}%
^{\left(  3\right)  }\right)  \circ\left(  c^{\left(  1\right)  }\otimes
c^{\left(  2\right)  }\right)  =\mu^{\circ2}\left[  \mathrm{f}^{\left(
1\right)  }\left(  c_{\left[  1\right]  }^{\left(  1\right)  },c_{\left[
1\right]  }^{\left(  2\right)  }\right)  ,\mathrm{f}^{\left(  2\right)
}\left(  c_{\left[  2\right]  }^{\left(  1\right)  },c_{\left[  2\right]
}^{\left(  2\right)  }\right)  ,\mathrm{f}^{\left(  3\right)  }\left(
c_{\left[  3\right]  }^{\left(  1\right)  },c_{\left[  3\right]  }^{\left(
2\right)  }\right)  \right]  .
\end{equation}

\end{enumerate}
\end{example}

\begin{example}
[$\operatorname*{Hom}\nolimits_{\Bbbk}\left(  C^{\otimes2},A^{\otimes
2}\right)  $]The last (fourth) possibility for the ternary convolution product
(see \textsc{Table} \ref{tab-nstar}) is nonderived and symmetric
$n=3,\ell=1,n^{\prime}=3,\ell^{\prime}=1$, with both a ternary algebra
$\mathrm{A}^{\left(  3\right)  }=\left\langle A\mid\mathbf{\mu}^{\left(
3\right)  }\right\rangle $ and coalgebra $\mathrm{C}^{\left(  3\right)
}=\left\langle C\mid\mathbf{\Delta}^{\left(  3\right)  }\right\rangle $. In
the elementwise description the maps $\mathfrak{f}^{\left(  i\right)
}:C^{\otimes2}\rightarrow A^{\otimes2}$, $i=1,2,3$ are $\mathfrak{f}^{\left(
i\right)  }\circ\left(  c_{1}\otimes c_{2}\right)  =\mathrm{f}_{\left[
1\right]  }^{\left(  i\right)  }\left(  c_{1},c_{2}\right)  \otimes
\mathrm{f}_{\left[  2\right]  }^{\left(  i\right)  }\left(  c_{1}%
,c_{2}\right)  $, $c_{1,2}\in C$. Then%
\begin{align}
\mathbf{\mu}_{\star}^{\left(  3\right)  }\circ\left(  \mathfrak{f}^{\left(
1\right)  }\otimes\mathfrak{f}^{\left(  2\right)  }\otimes\mathfrak{f}%
^{\left(  3\right)  }\right)  \circ\left(  c^{\left(  1\right)  }\otimes
c^{\left(  2\right)  }\right)   &  =\mu^{\left(  3\right)  }\left[
\mathrm{f}_{\left[  1\right]  }^{\left(  1\right)  }\left(  c_{\left[
1\right]  }^{\left(  1\right)  },c_{\left[  1\right]  }^{\left(  2\right)
}\right)  ,\mathrm{f}_{\left[  1\right]  }^{\left(  2\right)  }\left(
c_{\left[  2\right]  }^{\left(  1\right)  },c_{\left[  2\right]  }^{\left(
2\right)  }\right)  ,\mathrm{f}_{\left[  1\right]  }^{\left(  3\right)
}\left(  c_{\left[  3\right]  }^{\left(  1\right)  },c_{\left[  3\right]
}^{\left(  2\right)  }\right)  \right] \nonumber\\
&  \otimes\mu^{\left(  3\right)  }\left[  \mathrm{f}_{\left[  2\right]
}^{\left(  1\right)  }\left(  c_{\left[  1\right]  }^{\left(  1\right)
},c_{\left[  1\right]  }^{\left(  2\right)  }\right)  ,\mathrm{f}_{\left[
2\right]  }^{\left(  2\right)  }\left(  c_{\left[  2\right]  }^{\left(
1\right)  },c_{\left[  2\right]  }^{\left(  2\right)  }\right)  ,\mathrm{f}%
_{\left[  2\right]  }^{\left(  3\right)  }\left(  c_{\left[  3\right]
}^{\left(  1\right)  },c_{\left[  3\right]  }^{\left(  2\right)  }\right)
\right]  .
\end{align}

\end{example}

The above examples present clearly the possible forms of the $n_{\star}$-ary
convolution product, which can be convenient for lowest arity computations.

The \textit{general polyadic convolution product} (\ref{mff}) in Sweedler
notation can be presented as%
\begin{align}
&  \mathbf{\mu}_{\star}^{\left(  n_{\star}\right)  }\circ\left(
\mathfrak{f}^{\left(  1\right)  }\otimes\mathfrak{f}^{\left(  2\right)
}\otimes\mathfrak{\ldots}\otimes\mathfrak{f}^{\left(  n_{\star}\right)
}\right)  =\mathfrak{g},\ \ \ \mathfrak{f}^{\left(  i\right)  },\mathfrak{g}%
\in\operatorname*{Hom}\nolimits_{\Bbbk}\left(  C^{\otimes\left(  n^{\prime
}-1\right)  },A^{\otimes\left(  n-1\right)  }\right)  ,\nonumber\\
&  \mathfrak{g}_{\left[  j\right]  }\circ\left(  c^{\left(  1\right)  }%
\otimes\ldots\otimes c^{\left(  n^{\prime}-1\right)  }\right) \nonumber\\
&  =\left(  \mu^{\left(  n\right)  }\right)  ^{\circ\ell}\left[
\mathrm{f}_{\left[  j\right]  }^{\left(  1\right)  }\left(  \overset
{n^{\prime}-1}{\overbrace{c_{\left[  1\right]  }^{\left(  1\right)  }%
,\ldots,c_{\left[  1\right]  }^{\left(  n^{\prime}-1\right)  }}}\right)
,\mathrm{f}_{\left[  j\right]  }^{\left(  2\right)  }\left(  \overset
{n^{\prime}-1}{\overbrace{c_{\left[  2\right]  }^{\left(  1\right)  }%
,\ldots,c_{\left[  2\right]  }^{\left(  n^{\prime}-1\right)  }}}\right)
,\ldots,\mathrm{f}_{\left[  j\right]  }^{\left(  n_{\star}\right)  }\left(
\overset{n^{\prime}-1}{\overbrace{c_{\left[  n_{\star}\right]  }^{\left(
1\right)  },\ldots,c_{\left[  n_{\star}\right]  }^{\left(  n^{\prime
}-1\right)  }}}\right)  \right]  ,\nonumber\\
&  \mathrm{f}_{\left[  j\right]  }^{\left(  i\right)  }\in\operatorname*{Hom}%
\nolimits_{\Bbbk}\left(  C^{\otimes\left(  n^{\prime}-1\right)  },A\right)
,\ \ i\in1,\ldots,n_{\star},\ \ j\in1,\ldots,n-1,\ \ \ c\in C, \label{ms}%
\end{align}
where $\mathfrak{g}_{\left[  j\right]  }$ are the Sweedler components of
$\mathfrak{g}$.

Recall that the associativity of the binary convolution product $\left(
\star\right)  $ is transparent in the Sweedler notation. Indeed, if $\left(
f\star g\right)  \circ\left(  c\right)  =f\left(  c_{\left[  1\right]
}\right)  \cdot g\left(  c_{\left[  2\right]  }\right)  $, $f,g,h\in
\operatorname*{Hom}\nolimits_{\,{\scriptsize \mathrm{k}}\!}\left(  C,A\right)
$, $c\in C$, $\left(  \cdot\right)  \equiv\mu_{A}^{\left(  2\right)  }$, then
$\left(  \left(  f\star g\right)  \star h\right)  \circ\left(  c\right)
=\left(  f\left(  c_{\left[  1\right]  }\right)  \cdot g\left(  c_{\left[
2\right]  }\right)  \right)  \cdot h\left(  c_{\left[  3\right]  }\right)
=f\left(  c_{\left[  1\right]  }\right)  \cdot\left(  g\left(  c_{\left[
2\right]  }\right)  \cdot h\left(  c_{\left[  3\right]  }\right)  \right)
=\left(  f\star\left(  g\star h\right)  \right)  \circ\left(  c\right)  $.

\begin{lemma}
The polyadic convolution algebra $\mathrm{C}_{\star}^{\left(  n^{\prime
},n\right)  }$ (\ref{cs}) is associative.
\end{lemma}

\begin{proof}
To prove the claimed associativity of polyadic convolution $\mathbf{\mu}_{\star
}^{\left(  n_{\star}\right)  }$ we express (\ref{a}) in Sweedler notation.
Starting from%
\begin{equation}
\mathfrak{h}=\mathbf{\mu}_{\star}^{\left(  n_{\star}\right)  }\circ\left(
\mathfrak{g}\otimes\mathfrak{f}^{\left(  n_{\star}+1\right)  }\otimes
\mathfrak{f}^{\left(  n_{\star}+2\right)  }\otimes\mathfrak{\ldots}%
\otimes\mathfrak{f}^{\left(  2n_{\star}-1\right)  }\right)  ,\ \ \mathfrak{h}%
\in\operatorname*{Hom}\nolimits_{\Bbbk}\left(  C^{\otimes\left(  n^{\prime
}-1\right)  },A^{\otimes\left(  n-1\right)  }\right)  , \label{h}%
\end{equation}
where $\mathfrak{g}$ is given by (\ref{ms}), it follows that $\mathfrak{h}$
should not depend of place of $\mathfrak{g}$ in (\ref{h}). Applying
$\mathfrak{h}$ to $c\in C$ twice, we obtain for its Sweedler components
$\mathfrak{h}_{\left[  j\right]  },\ \ j\in1,\ldots,n-1,$%
\begin{align}
&  \mathfrak{h}_{\left[  j\right]  }\circ\left(  c^{\left(  1\right)  }%
\otimes\ldots\otimes c^{\left(  n^{\prime}-1\right)  }\right)  =\left(
\mu^{\left(  n\right)  }\right)  ^{\circ\ell}\left[  \left(  \mu^{\left(
n\right)  }\right)  ^{\circ\ell}\left[  \mathrm{f}_{\left[  j\right]
}^{\left(  1\right)  }\left(  \overset{n^{\prime}-1}{\overbrace{c_{\left[
1\right]  }^{\left(  1\right)  },\ldots,c_{\left[  1\right]  }^{\left(
n^{\prime}-1\right)  }}}\right)  ,\mathrm{f}_{\left[  j\right]  }^{\left(
2\right)  }\left(  \overset{n^{\prime}-1}{\overbrace{c_{\left[  2\right]
}^{\left(  1\right)  },\ldots,c_{\left[  2\right]  }^{\left(  n^{\prime
}-1\right)  }}}\right)  ,\right.  \right. \nonumber\\
&  \left.  \ldots,\mathrm{f}_{\left[  j\right]  }^{\left(  n_{\star}\right)
}\left(  \overset{n^{\prime}-1}{\overbrace{c_{\left[  2n_{\star}-1\right]
}^{\left(  1\right)  },\ldots,c_{\left[  2n_{\star}-1\right]  }^{\left(
n^{\prime}-1\right)  }}}\right)  \right]  ,\mathrm{f}_{\left[  j\right]
}^{\left(  n_{\star}+1\right)  }\left(  \overset{n^{\prime}-1}{\overbrace
{c_{\left[  1\right]  }^{\left(  1\right)  },\ldots,c_{\left[  1\right]
}^{\left(  n^{\prime}-1\right)  }}}\right)  ,\mathrm{f}_{\left[  j\right]
}^{\left(  n_{\star}+2\right)  }\left(  \overset{n^{\prime}-1}{\overbrace
{c_{\left[  2\right]  }^{\left(  1\right)  },\ldots,c_{\left[  2\right]
}^{\left(  n^{\prime}-1\right)  }}}\right)  ,\nonumber\\
&  \left.  \ldots,\mathrm{f}_{\left[  j\right]  }^{\left(  2n_{\star
}-1\right)  }\left(  \overset{n^{\prime}-1}{\overbrace{c_{\left[  2n_{\star
}-1\right]  }^{\left(  1\right)  },\ldots,c_{\left[  2n_{\star}-1\right]
}^{\left(  n^{\prime}-1\right)  }}}\right)  \right]  , \label{hj}%
\end{align}
and here coassociativity and (\ref{ns}) gives $\left(  \Delta^{\left(  n^{\prime
}\right)  }\right)  ^{\circ2l^{\prime}}\left(  c\right)  =c_{\left[  1\right]
}\otimes c_{\left[  2\right]  }\otimes\ldots\otimes c_{\left[  2n_{\star
}-1\right]  }$. Since $n$-ary algebra multiplication $\mu^{\left(  n\right)
}$ is associative, the internal $\left(  \mu^{\left(  n\right)  }\right)
^{\circ\ell}$ in (\ref{hj}) can be on any place, and $\mathfrak{g}$ in
(\ref{h}) can be on any place as well. This means that the polyadic
convolution product $\mathbf{\mu}_{\star}^{\left(  n_{\star}\right)  }$ is associative.
\end{proof}

Observe the polyadic version of the identity used in (\ref{mi}): for any
$\mathfrak{f}\in\operatorname*{Hom}\nolimits_{\Bbbk}\left(  C^{\otimes\left(
n^{\prime}-1\right)  },A^{\otimes\left(  n-1\right)  }\right)  $%
\begin{equation}
\operatorname*{id}\nolimits_{A}^{\otimes\left(  n-1\right)  }\circ
\,\mathfrak{f}\circ\operatorname*{id}\nolimits_{C}^{\otimes\left(  n^{\prime
}-1\right)  }=\mathfrak{f}. \label{fi}%
\end{equation}

\begin{proposition}
If the polyadic associative algebra {\upshape$\mathrm{A}^{\left(  n\right)  }%
$} is unital with {\upshape$\mathbf{\eta}^{\left(  r,n\right)  }%
:K^{r}\rightarrow A^{\otimes\left(  n-1\right)  }$}, and the polyadic
coassociative coalgebra {\upshape$\mathrm{C}^{\left(  n^{\prime}\right)  }$}
is counital with $\mathbf{\varepsilon}^{\left(  n^{\prime},r^{\prime}\right)
}:C^{\otimes\left(  n^{\prime}-1\right)  }\rightarrow K^{r^{\prime}}$, both
over the same polyadic field $\Bbbk$, then the polyadic convolution algebra
{\upshape$\mathrm{C}_{\star}^{\left(  n^{\prime},n\right)  }$} (\ref{cs}) is
unital, and its unit is given by {\upshape$\mathbf{e}_{\star}^{\left(
n^{\prime},n\right)  }$} \emph{(\ref{ns})}.
\end{proposition}

\begin{proof}
In analogy with (\ref{mi}) we compose%
\begin{align}
&  \mathfrak{f}=\left(  \left(  \mathbf{\mu}^{\left(  n\right)  }\right)
^{\circ\ell}\right)  ^{\otimes\left(  n-1\right)  }\circ\mathbf{\tau}%
_{medial}^{\left(  n-1,n_{\star}\right)  }\circ\left(  \left(  \mathbf{\eta
}^{\left(  r,n\right)  }\right)  ^{\otimes\left(  n_{\star}-1\right)  }%
\otimes\operatorname*{id}\nolimits_{A}^{\otimes\left(  n-1\right)  }\right)
\circ\left(  \left(  \mathbf{\gamma}^{\left(  r^{\prime},r\right)  }\right)
^{\otimes\left(  n_{\star}-1\right)  }\otimes\operatorname*{id}\nolimits_{A}%
^{\otimes\left(  n-1\right)  }\right) \nonumber\\
&  \circ\left(  \left(  \operatorname*{id}\nolimits_{K}^{\otimes r^{\prime}%
}\right)  ^{\otimes\left(  n_{\star}-1\right)  }\otimes\mathfrak{f}\right)
\circ\left(  \left(  \mathbf{\varepsilon}^{\left(  n^{\prime},r^{\prime
}\right)  }\right)  ^{\otimes\left(  n_{\star}-1\right)  }\otimes
\operatorname*{id}\nolimits_{C}^{\otimes\left(  n^{\prime}-1\right)  }\right)
\circ\mathbf{\tau}_{medial}^{\left(  n_{\star},n^{\prime}-1\right)  }%
\circ\left(  \left(  \mathbf{\Delta}^{\left(  n^{\prime}\right)  }\right)
^{\circ\ell^{\prime}}\right)  ^{\otimes\left(  n^{\prime}-1\right)
}\nonumber\\
&  =\left(  \left(  \mathbf{\mu}^{\left(  n\right)  }\right)  ^{\circ\ell
}\right)  ^{\otimes\left(  n-1\right)  }\circ\mathbf{\tau}_{medial}^{\left(
n-1,n_{\star}\right)  }\circ\left(  \mathbf{\eta}^{\left(  r,n\right)  }%
\circ\mathbf{\gamma}^{\left(  r^{\prime},r\right)  }\circ\mathbf{\varepsilon
}^{\left(  n^{\prime},r^{\prime}\right)  }\right)  ^{\otimes\left(  n_{\star
}-1\right)  }\nonumber\\
&  \circ\left(  \operatorname*{id}\nolimits_{A}^{\otimes\left(  n-1\right)
}\circ\,\mathfrak{f}\circ\operatorname*{id}\nolimits_{C}^{\otimes\left(
n^{\prime}-1\right)  }\right)  \circ\mathbf{\tau}_{medial}^{\left(  n_{\star
},n^{\prime}-1\right)  }\circ\left(  \left(  \mathbf{\Delta}^{\left(
n^{\prime}\right)  }\right)  ^{\circ\ell^{\prime}}\right)  ^{\otimes\left(
n^{\prime}-1\right)  }=\mathbf{\mu}_{\star}^{\left(  n_{\star}\right)  }%
\circ\left(  \left(  \mathbf{e}_{\star}^{\left(  n^{\prime},n\right)
}\right)  ^{\otimes\left(  n_{\star}-1\right)  }\otimes\mathfrak{f}\right)  ,
\end{align}
which coincides with the polyadic unit definition (\ref{e}). We use the
identity (\ref{fi}) and the axioms for a polyadic unit (\ref{mu}) and counit
(\ref{ep}). The same derivation can be made for any place of $\mathbf{\mu
}_{\star}^{\left(  n_{\star}\right)  }$.
\end{proof}

As in the general theory of $n$-ary groups \cite{dor3}, the invertibility of maps
in $\mathrm{C}_{\star}^{\left(  n^{\prime},n\right)  }$ should be defined not
by using the unit, but by using the querelement (\ref{q}).

\begin{definition}
\label{def-coq}For a fixed $\mathfrak{f}\in\mathrm{C}_{\star}^{\left(
n^{\prime},n\right)  }$ its \textit{coquerelement} $\mathbf{q}_{\star}\left(
\mathfrak{f}\right)  \in\mathrm{C}_{\star}^{\left(  n^{\prime},n\right)  }$ is
the querelement in the $n_{\star}$-ary convolution product%
\begin{equation}
\mathbf{\mu}_{\star}^{\left(  n_{\star}\right)  }\circ\left(  \mathfrak{f}%
^{\otimes\left(  n_{\star}-1\right)  }\otimes\mathbf{q}_{\star}\left(
\mathfrak{f}\right)  \right)  =\mathfrak{f}, \label{mfq}%
\end{equation}
where $\mathbf{q}_{\star}\left(  \mathfrak{f}\right)  $ can be on any place
and $n_{\star}\geq3$. The maps in a polyadic convolution algebra which have
a coquerelement are called \textit{coquerable}.
\end{definition}

Define the \textit{positive convolution power} $\ell_{\star}$ of an element
$\mathfrak{f}\in\mathrm{C}_{\star}^{\left(  n^{\prime},n\right)  }$ not
recursively as in \cite{pos}, but through the $\ell_{\star}$-iterated
multiplication (\ref{mx})%
\begin{equation}
\mathfrak{f}^{\left\langle \ell_{\star}\right\rangle }=\left(  \mathbf{\mu
}_{\star}^{\left(  n_{\star}\right)  }\right)  ^{\circ\ell_{\star}}%
\circ\left(  \mathfrak{f}^{\otimes\left(  \ell_{\star}\left(  n_{\star
}-1\right)  +1\right)  }\right)  , \label{fl}%
\end{equation}
and an element in the \textit{negative convolution power} $\mathfrak{f}%
^{\left\langle -\ell_{\star}\right\rangle }$ satisfies the equation%
\begin{equation}
\left(  \mathbf{\mu}_{\star}^{\left(  n_{\star}\right)  }\right)  ^{\circ
\ell_{\star}}\circ\left(  \mathfrak{f}^{\left\langle \ell_{\star
}-1\right\rangle }\otimes\mathfrak{f}^{\otimes\left(  n_{\star}-2\right)
}\otimes\mathfrak{f}^{\left\langle -\ell_{\star}\right\rangle }\right)
=\left(  \mathbf{\mu}_{\star}^{\left(  n_{\star}\right)  }\right)  ^{\circ
\ell_{\star}}\circ\left(  \mathfrak{f}^{\otimes\ell_{\star}\left(  n_{\star
}-1\right)  }\otimes\mathfrak{f}^{\left\langle -\ell_{\star}\right\rangle
}\right)  =\mathfrak{f}. \label{mnl}%
\end{equation}
It follows from (\ref{fl}) that the polyadic analogs of the exponent laws
hold%
\begin{align}
\mathbf{\mu}_{\star}^{\left(  n_{\star}\right)  }\circ\left(  \mathfrak{f}%
^{\left\langle \ell_{\star}^{\left(  1\right)  }\right\rangle }\otimes
\mathfrak{f}^{\left\langle \ell_{\star}^{\left(  2\right)  }\right\rangle
}\otimes\ldots\otimes\mathfrak{f}^{\left\langle \ell_{\star}^{\left(
n_{\star}\right)  }\right\rangle }\right)   &  =\mathfrak{f}^{\left\langle
\ell_{\star}^{\left(  1\right)  }+\ell_{\star}^{\left(  2\right)  }\ldots
+\ell_{\star}^{\left(  n_{\star}\right)  }+1\right\rangle },\\
\left(  \mathfrak{f}^{\left\langle \ell_{\star}^{\left(  1\right)
}\right\rangle }\right)  ^{\left\langle \ell_{\star}^{\left(  2\right)
}\right\rangle }  &  =\mathfrak{f}^{\left\langle \ell_{\star}^{\left(
1\right)  }\ell_{\star}^{\left(  2\right)  }\left(  n_{\star}-1\right)
+\ell_{\star}^{\left(  1\right)  }+\ell_{\star}^{\left(  2\right)
}\right\rangle }.
\end{align}

Comparing (\ref{mfq}) and (\ref{mnl}), we have%
\begin{equation}
\mathbf{q}_{\star}\left(  \mathfrak{f}\right)  =\mathfrak{f}^{\left\langle
-1\right\rangle }.
\end{equation}
An arbitrary polyadic power $\ell_{Q}$ of the coquerelement $\mathbf{q}%
_{\star}^{\circ\ell_{Q}}\left(  \mathfrak{f}\right)  $ is defined by
(\ref{mfq}) recursively and can be expressed through the negative polyadic
power of $\mathfrak{f}$ (see, e.g. \cite{dud07} for $n$-ary groups). In terms
of the Heine numbers \cite{heine} (or $q$-deformed numbers \cite{kac/che})%
\begin{equation}
\left[  \left[  l\right]  \right]  _{q}=\frac{q^{l}-1}{q-1},\ \ \ l\in
\mathbb{N}_{0},\ \ q\in\mathbb{Z},
\end{equation}
we obtain \cite{dup2018a}%
\begin{equation}
\mathbf{q}_{\star}^{\circ\ell_{Q}}\left(  \mathfrak{f}\right)  =\mathfrak{f}%
^{\left\langle -\left[  \left[  \ell_{Q}\right]  \right]  _{2-n_{\star}%
}\right\rangle }.
\end{equation}

\section{\textsc{Polyadic bialgebras}}

The next step is to combine algebras and coalgebras into a common algebraic
structure in some \textquotedblleft natural\textquotedblright\ way.
Informally, a bialgebra is defined as a vector space which is
\textquotedblleft simultaneously\textquotedblright\ an algebra and a coalgebra with
some compatibility conditions (e.g., \cite{swe,abe}).

In search of a polyadic analog of bialgebras, we observe two structural
differences with the binary case: \textbf{1}) since the unit and counit do not
necessarily exist, we obtain \textsf{4 different kinds} of bialgebras (similar
to the unit and zero in \textsc{Table \ref{tab1}}); \textbf{2}) where the most
exotic is the possibility of \textsf{unequal arities} of multiplication and
comultiplication $n\neq n^{\prime}$ (see \textbf{Assertion }\ref{as-dual}).
Initially, we take them as arbitrary and then try to find restrictions arising from
some \textquotedblleft natural\textquotedblright\ relations.

Let $\mathrm{B}_{vect}$ be a polyadic vector space over the polyadic field
$\Bbbk^{\left(  m_{k},n_{k}\right)  }$ as (see (\ref{av}) and (\ref{cvect}))%
\begin{equation}
\mathrm{B}_{vect}=\left\langle B,K\mid\nu^{\left(  m\right)  };\nu
_{k}^{\left(  m_{k}\right)  },\mu_{k}^{\left(  n_{k}\right)  };\rho^{\left(
r\right)  }\right\rangle ,
\end{equation}
where $\nu^{\left(  m\right)  }:B^{\times m}\rightarrow B$ is $m$-ary addition
and $\rho^{\left(  r\right)  }:K^{\times r}\times B\rightarrow B$ is $r$-place
action (see (\ref{r})).

\begin{definition}
\label{def-bialg}A \textit{polyadic bialgebra} $\mathrm{B}^{\left(  n^{\prime
},n\right)  }$ is $\mathrm{B}_{vect}$ equipped with a $\Bbbk$-linear $n$-ary
multiplication map $\mathbf{\mu}^{\left(  n\right)  }:B^{\otimes n}\rightarrow
B$ and a $\Bbbk$-linear $n^{\prime}$-ary comultiplication map $\mathbf{\Delta
}^{\left(  n^{\prime}\right)  }:B\rightarrow B^{\otimes n^{\prime}}$ such that:

\begin{enumerate}
\item
\begin{enumerate}
\item $\mathrm{B}_{A}^{\left(  n\right)  }=\left\langle \mathrm{B}_{vect}%
\mid\mathbf{\mu}^{\left(  n\right)  }\right\rangle $ is a $n$-ary algebra;

\item The map $\mathbf{\mu}^{\left(  n\right)  }$ is a coalgebra
(homo)morphism (\ref{hc}).
\end{enumerate}

\item
\begin{enumerate}
\item $\mathrm{B}_{C}^{\left(  n^{\prime}\right)  }=\left\langle
\mathrm{B}_{vect}\mid\mathbf{\Delta}^{\left(  n^{\prime}\right)
}\right\rangle $ is a $n^{\prime}$-ary coalgebra;

\item The map $\mathbf{\Delta}^{\left(  n^{\prime}\right)  }$ is an algebra
(homo)morphism (\ref{ha}).
\end{enumerate}
\end{enumerate}
\end{definition}

The equivalence of the compatibility conditions \textbf{1}\textsf{b}\textbf{)}
and \textbf{2}\textsf{b}\textbf{)} can be expressed in the form (polyadic
analog of \textquotedblleft$\Delta\circ\mu=\left(  \mu\otimes\mu\right)
\left(  \operatorname*{id}\otimes\tau\otimes\operatorname*{id}\right)  \left(
\Delta\otimes\Delta\right)  $\textquotedblright)%
\begin{equation}
\mathbf{\Delta}^{\left(  n^{\prime}\right)  }\circ\mathbf{\mu}^{\left(
n\right)  }=\left(  \overset{n^{\prime}}{\overbrace{\mathbf{\mu}^{\left(
n\right)  }\otimes\ldots\otimes\mathbf{\mu}^{\left(  n\right)  }}}\right)
\circ\mathbf{\tau}_{medial}^{\left(  n,n^{\prime}\right)  }\circ\left(
\overset{n}{\overbrace{\mathbf{\Delta}^{\left(  n^{\prime}\right)  }%
\otimes\ldots\otimes\mathbf{\Delta}^{\left(  n^{\prime}\right)  }}}\right)  ,
\label{dm}%
\end{equation}
where $\mathbf{\tau}_{medial}^{\left(  n,n^{\prime}\right)  }$ is the medial
map (\ref{an}) acting on $B$, while the diagram%
\begin{equation}
\begin{diagram} B^{\otimes n } & \rTo^{\left(\mathbf{\Delta}^{\left( n^{\prime}\right) }\right)^{\otimes n }}& & \left( B^{\otimes n^{\prime}}\right) ^{\otimes n}&& \rTo^{\mathbf{\tau}_{medial}^{\left( n,n^{\prime}\right) } } & \left( B^{\otimes n}\right) ^{\otimes n^{\prime}} \\ \dTo^{\mathbf{\mu}^{\left( n\right) }} && &&& & \dTo_{\left(\mathbf{\mu}^{\left( n\right) }\right)^{\otimes n ^{\prime}}} \\ B &&& \rTo^{\mathbf{\Delta}^{\left( n^{\prime}\right) }} & & &B^{\otimes n^{\prime}} \\ \end{diagram}
\end{equation}
commutes.

In an elementwise description it is the commutation of $n$-ary multiplication and
$n^{\prime}$-ary comultiplication
\begin{equation}
\Delta^{\left(  n^{\prime}\right)  }\left(  \mu^{\left(  n\right)  }\left[
b_{1},\ldots,b_{n}\right]  \right)  =\mu^{\left(  n\right)  }\left[
\Delta^{\left(  n^{\prime}\right)  }\left(  b_{1}\right)  ,\ldots
,\Delta^{\left(  n^{\prime}\right)  }\left(  b_{n}\right)  \right]  ,
\label{dnm}%
\end{equation}
which in the Sweedler notation becomes%
\begin{align}
&  \mu^{\left(  n\right)  }\left[  b_{1},\ldots,b_{n}\right]  _{\left[
1\right]  }\otimes\mu^{\left(  n\right)  }\left[  b_{1},\ldots,b_{n}\right]
_{\left[  2\right]  }\otimes\ldots\otimes\mu^{\left(  n\right)  }\left[
b_{1},\ldots,b_{n}\right]  _{\left[  n^{\prime}\right]  }\nonumber\\
&  =\mu^{\left(  n\right)  }\left[  b_{\left[  1\right]  }^{\left(  1\right)
},\ldots,b_{\left[  1\right]  }^{\left(  n\right)  }\right]  \otimes
\mu^{\left(  n\right)  }\left[  b_{\left[  2\right]  }^{\left(  1\right)
},\ldots,b_{\left[  2\right]  }^{\left(  n\right)  }\right]  \otimes
\ldots\otimes\mu^{\left(  n\right)  }\left[  b_{\left[  n^{\prime}\right]
}^{\left(  1\right)  },\ldots,b_{\left[  n^{\prime}\right]  }^{\left(
n\right)  }\right]  .
\end{align}

Consider the example of a nonderived bialgebra $\mathrm{B}^{\left(
n,n\right)  }$ which follows from the von Neumann higher $n$-regularity
relations \cite{dup18,dup/mar2001a,dup/mar7,dup/mar2018a}.

\begin{example}
[\textsl{von Neumann }$n$\textsl{-regular bialgebra}]\label{ex-vn}Let
$\mathrm{B}^{\left(  n,n\right)  }=\left\langle B\mid\mathbf{\mu}^{\left(
n\right)  },\mathbf{\Delta}^{\left(  n\right)  }\right\rangle $ be a polyadic
bialgebra generated by the elements $b_{i}\in B,\ \ i=1,\ldots,n-1$ subject to
the \textsf{nonderived} $n$-ary multiplication%
\begin{align}
\mu^{\left(  n\right)  }\left(  b_{1},b_{2},b_{3}\ldots,b_{n-2},b_{n-1}%
,b_{1}\right)   &  =b_{1},\\
\mu^{\left(  n\right)  }\left(  b_{2},b_{3},b_{4}\ldots b_{n-1},b_{1}%
,b_{2}\right)   &  =b_{2},\\
&  \vdots\nonumber\\
\mu^{\left(  n\right)  }\left(  b_{n-1},b_{1},b_{2}\ldots b_{n-3}%
,b_{n-2},b_{n-1}\right)   &  =b_{n-1},
\end{align}
and the \textsf{nonderived} $n$-ary comultiplication (cf. \ref{d3})%
\begin{align}
\Delta^{\left(  n\right)  }\left(  b_{1}\right)   &  =b_{1}\otimes
b_{2}\otimes b_{3}\ldots,b_{n-2}\otimes b_{n-1}\otimes b_{1},\\
\Delta^{\left(  n\right)  }\left(  b_{2}\right)   &  =b_{2}\otimes
b_{3}\otimes b_{4}\ldots,b_{n-1}\otimes b_{1}\otimes b_{2},\\
&  \vdots\nonumber\\
\Delta^{\left(  n\right)  }\left(  b_{n-1}\right)   &  =b_{n-1}\otimes
b_{1}\otimes b_{2}\otimes\ldots\otimes b_{n-3}\otimes b_{n-2}\otimes b_{n-1}.
\end{align}
It is straightforward to check that the compatibility condition (\ref{dnm})
holds.  Many possibilities exist for choosing other operations---algebra
addition, field addition and multiplication, action---so to demonstrate
the compatibility we have confined ourselves to only the algebra multiplication and comultiplication.
\end{example}

If the $n$-ary algebra $\mathrm{B}_{A}^{\left(  n\right)  }$ has unit
\textsf{and/or} $n^{\prime}$-ary coalgebra $\mathrm{B}_{C}^{\left(  n^{\prime
}\right)  }$ has counit $\mathbf{\varepsilon}^{\left(  n^{\prime},r^{\prime
}\right)  }$, we should add the following additional axioms.

\begin{definition}
[\textsl{Unit axiom}]\label{def-bi-un}If $\mathrm{B}_{A}^{\left(  n\right)
}=\left\langle \mathrm{B}_{vect}\mid\mathbf{\mu}^{\left(  n\right)
}\right\rangle $ is unital, then the unit $\mathbf{\eta}^{\left(  r,n\right)
}$ is a (homo)morphism of the coalgebra $\mathrm{B}_{C}^{\left(  n^{\prime
}\right)  }=\left\langle \mathrm{B}_{vect}\mid\mathbf{\Delta}^{\left(
n^{\prime}\right)  }\right\rangle $ (see (\ref{hc}))%
\begin{equation}
\left(  \overset{n-1}{\overbrace{\mathbf{\Delta}^{\left(  n^{\prime}\right)
}\otimes\ldots\otimes\mathbf{\Delta}^{\left(  n^{\prime}\right)  }}}\right)
\circ\mathbf{\eta}^{\left(  r,n\right)  }=\left(  \overset{n^{\prime}%
}{\overbrace{\mathbf{\eta}^{\left(  r,n\right)  }\otimes\ldots\otimes
\mathbf{\eta}^{\left(  r,n\right)  }}}\right)
\end{equation}
such that the diagram%
\begin{equation}
\begin{diagram} K^{\otimes r} & \rTo^{\mathbf{\eta}^{\left( r,n\right) }} & B^{\otimes \left( n-1\right)} \\ \dTo^{\simeq} & & \dTo_{\left(\mathbf{\Delta}^{\left( n^{\prime}\right) }\right)^{\otimes \left( n-1\right) }} \\ K^{\otimes r n^{\prime}}& \rTo^{\;\;\left( \mathbf{\eta}^{\left( r,n\right) }\right)^{n^{\prime}}} & B^{\otimes \left( n-1\right) n^{\prime}}\\ \end{diagram}
\end{equation}
commutes.
\end{definition}

\begin{definition}
[\textsl{Counit axiom}]\label{def-bi-coun}If $\mathrm{B}_{C}^{\left(
n^{\prime}\right)  }=\left\langle \mathrm{B}_{vect}\mid\mathbf{\Delta
}^{\left(  n^{\prime}\right)  }\right\rangle $ is counital, then the counit
$\mathbf{\varepsilon}^{\left(  n^{\prime},r^{\prime}\right)  }$ is a
(homo)morphism of the algebra $\mathrm{B}_{A}^{\left(  n\right)
}=\left\langle \mathrm{B}_{vect}\mid\mathbf{\mu}^{\left(  n\right)
}\right\rangle $ (see (\ref{ha}))%
\begin{equation}
\mathbf{\varepsilon}^{\left(  n^{\prime},r^{\prime}\right)  }\circ\left(
\overset{n^{\prime}-1}{\overbrace{\mathbf{\mu}^{\left(  n\right)  }%
\otimes\ldots\otimes\mathbf{\mu}^{\left(  n\right)  }}}\right)  =\left(
\overset{n}{\overbrace{\mathbf{\varepsilon}^{\left(  n^{\prime},r^{\prime
}\right)  }\otimes\ldots\otimes\mathbf{\varepsilon}^{\left(  n^{\prime
},r^{\prime}\right)  }}}\right)
\end{equation}
such that the diagram%
\begin{equation}
\begin{diagram} K^{\otimes r} & \lTo^{\mathbf{\varepsilon}^{\left( n^{\prime },r^{\prime}\right) } }& B^{\otimes \left( n^{\prime}-1\right)} \\ \uTo^{\simeq} & & \uTo_{\left( \mathbf{\mu}^{\left( n\right) }\right)^{\otimes \left( n^{\prime}-1\right) }} \\ K^{\otimes r n}& \lTo^{\;\;} & B^{\otimes \left( n^{\prime}-1\right) n}\\ \end{diagram}
\end{equation}
commutes.
\end{definition}

If \textsf{both} the polyadic unit and polyadic counit exist, then we include
their compatibility condition%
\begin{equation}
\left(  \mathbf{\varepsilon}^{\left(  n^{\prime},r^{\prime}\right)  }\right)
^{\otimes\left(  n-1\right)  }\circ\left(  \mathbf{\eta}^{\left(  r,n\right)
}\right)  ^{\otimes\left(  n^{\prime}-1\right)  }\simeq\operatorname*{id}%
\nolimits_{K}, \label{en}%
\end{equation}
such that the diagram%
\begin{equation}
\begin{diagram} & B^{\otimes \left( n^{\prime}-1\right) \left(n-1\right)}&\\ \ruTo(1,2)^{\left( \mathbf{\eta}^{\left( r,n\right) }\right)^{\otimes\left(n^{\prime}-1\right)} }&&\rdTo(1,2)^{\left(\mathbf{\varepsilon}^{\left( n^{\prime },r^{\prime}\right) } \right)^{\otimes\left(n-1\right)}}\\ K^{\otimes r \left( n^{\prime}-1\right)} &\rTo^{\simeq} &K^{\otimes r \left( n-1\right)} \end{diagram}
\end{equation}
commutes.

\begin{assertion}
There are \textsf{four kinds} of polyadic bialgebras depending on whether the unit
${\mathbf{\eta}}^{\left(  r,n\right)  }$ and counit ${\mathbf{\varepsilon}%
}^{\left(  n^{\prime},r^{\prime}\right)  }$ exist:

{\upshape\textbf{1) }} nonunital-noncounital; {\upshape\textbf{2) }}
unital-noncounital; {\upshape\textbf{3) }} nonunital-counital;
{\upshape\textbf{4) }} unital-counital.
\end{assertion}

\begin{definition}
A polyadic bialgebra $\mathrm{B}^{\left(  n^{\prime},n\right)  }$ is called
\textit{totally co-commutative}, if%
\begin{align}
\mathbf{\mu}^{\left(  n\right)  }  &  =\mathbf{\mu}^{\left(  n\right)  }%
\circ\mathbf{\tau}_{n},\label{tcm}\\
\mathbf{\Delta}^{\left(  n^{\prime}\right)  }  &  =\mathbf{\tau}_{n^{\prime}%
}\circ\mathbf{\Delta}^{\left(  n^{\prime}\right)  }, \label{tcd}%
\end{align}
where $\mathbf{\tau}_{n}\in\mathrm{S}_{n}$, $\mathbf{\tau}_{n^{\prime}}%
\in\mathrm{S}_{n^{\prime}}$, and $\mathrm{S}_{n},\mathrm{S}_{n^{\prime}}$ are
the symmetry permutation groups on $n$ and $n^{\prime}$ elements respectively.
\end{definition}

\begin{definition}
A polyadic bialgebra $\mathrm{B}^{\left(  n^{\prime},n\right)  }$ is called
\textit{medially co-commutative}, if%
\begin{align}
\mathbf{\mu}_{op}^{\left(  n\right)  }  &  \equiv\mathbf{\mu}^{\left(
n\right)  }\circ\mathbf{\tau}_{op}^{\left(  n\right)  }=\mathbf{\mu}^{\left(
n\right)  },\label{mcm}\\
\mathbf{\Delta}_{cop}^{\left(  n^{\prime}\right)  }  &  \equiv\mathbf{\tau
}_{op}^{\left(  n^{\prime}\right)  }\circ\mathbf{\Delta}^{\left(  n^{\prime
}\right)  }=\mathbf{\Delta}^{\left(  n^{\prime}\right)  }, \label{mcd}%
\end{align}
where $\mathbf{\tau}_{op}^{\left(  n\right)  }$ and $\mathbf{\tau}%
_{op}^{\left(  n^{\prime}\right)  }$ are the medially allowed polyadic twist
maps (\ref{top}).
\end{definition}

\section{\textsc{Polyadic Hopf algebras}}

Here we introduce the most general approach to \textquotedblleft
polyadization\textquotedblright\ of the Hopf algebra concept
\cite{abe,swe,radford}. Informally, the transition from bialgebra to Hopf algebra
is, in some sense, \textquotedblleft dualizing\textquotedblright\ the passage
from semigroup (containing noninvertible elements) to group (in which all
elements are invertible). Schematically, if multiplication $\mu=\left(
\cdot\right)  $ in a semigroup $G$ is binary, the invertibility of all
elements demands \textsf{two} extra and \textsf{necessary} set-ups:
\textbf{1}) \textsl{An additional element} (identity $e\in G$ or the
corresponding map from a one point set to group $\epsilon$); \textbf{2})
\textsl{An additional map} (inverse $\iota:G\rightarrow G$), such that
$g\cdot\iota\left(  g\right)  =e$ in diagrammatic form is $\mu\circ\left(
\operatorname*{id}\nolimits_{G}\times\iota\right)  \circ D_{2}=\epsilon$
($D_{2}:G\rightarrow G\times G$ is the diagonal map). When \textquotedblleft
dualizing\textquotedblright, in a (binary) bialgebra $B$ (with multiplication
$\mu$ and comultiplication $\Delta$) again \textsf{two} set-ups should be
considered in order to get a (binary) Hopf algebra: \textbf{1}) \textsl{An analog of
identity} $e_{\star}=\eta\varepsilon$ (where $\eta:k\rightarrow B$ is unit and
$\varepsilon:B\rightarrow k$ is counit); \textbf{2}) \textsl{An analog of
inverse} $S:B\rightarrow B$ called the \textit{antipode}, such that $\mu
\circ\left(  \operatorname*{id}\nolimits_{B}\otimes S\right)  \circ
\Delta=e_{\star}$ or in terms of the (binary) convolution product
$\operatorname*{id}\nolimits_{B}\star\,S=e_{\star}$. By multiplying both sides
by $S$ from the left and by $\operatorname*{id}\nolimits_{B}$ from the right,
we obtain weaker (von Neumann regularity) conditions $S\star\operatorname*{id}%
\nolimits_{B}\star\,S=S$, $\operatorname*{id}\nolimits_{B}\star\,S\star
\operatorname*{id}\nolimits_{B}=\operatorname*{id}\nolimits_{B}$, which do not
contain an identity $e_{\star}$ and lead to the concept of weak Hopf algebras
\cite{dup/li3,dup/li2,szl}.

The crucial peculiarity of the polyadic generalization is the possible absence
of an identity or \textbf{1}) in both cases. The role and necessity of the
polyadic identity (\ref{e}) is not so important: there polyadic groups without
identity exist (see, e.g. \cite{galmak1}, and the discussion after
(\ref{me})). Invertibility is determined by the querelement (\ref{q}) in
$n$-ary group or the quermap (\ref{mq}) in polyadic algebra. So there are two
ways forward: \textquotedblleft dualize\textquotedblright\ the quermap (\ref{mq})
directly (as in the binary case) or use the most general version of the
polyadic convolution product (\ref{mff}) and apply possible restrictions, if
any. We will choose the second method, because the first one is a particular
case of it. Thus, if the standard (binary) antipode is the convolution inverse
(coinverse) to the identity in a bialgebra, then its polyadic counterpart
should be a coquerelement (\ref{mfq}) of some polyadic \textsf{analog} for the
identity map in the polyadic bialgebra. We consider two possibilities to
define a polyadic analog of identity: \textbf{1}) \textsl{Singular case}. The
comultiplication is binary $n^{\prime}=2$; \textbf{2}) \textsl{Symmetric
case}. The arities of multiplication and comultiplication need not be
binary, but should coincide $n=n^{\prime}$.

In the \textsl{singular case} a polyadic multivalued map in
$\operatorname*{End}\nolimits_{\Bbbk}\left(  B,B^{\otimes\left(  n-1\right)
}\right)  $ is a reminder of how an identity can be defined: its components are
to be functions of one variable. That is, with more than one argument it is not
possible to determine its value when these are unequal.

\begin{definition}
We take for a \textit{singular polyadic identity} $\mathfrak{Id}_{0}$ the
diagonal map $\mathfrak{Id}_{0}=\mathfrak{D}\in\operatorname*{End}%
\nolimits_{\Bbbk}\left(  B,B^{\otimes\left(  n-1\right)  }\right)  $, such
that $b\mapsto b^{\otimes\left(  n-1\right)  }$, for any $b\in B$.
\end{definition}

We call the polyadic convolution product (\ref{mff}) with the binary
comultiplication $n^{\prime}=2$ \textit{reduced} and denote it by
$\mathbf{\bar{\mu}}_{\star}^{\left(  n_{\star}\right)  }$ which in
Sweedler notation can be obtained from (\ref{ms})%
\begin{align}
&  \mathbf{\bar{\mu}}_{\star}^{\left(  n_{\star}\right)  }\circ\left(
\mathfrak{f}^{\left(  1\right)  }\otimes\mathfrak{f}^{\left(  2\right)
}\otimes\mathfrak{\ldots}\otimes\mathfrak{f}^{\left(  n_{\star}\right)
}\right)  =\mathfrak{g},\ \ \ \mathfrak{f}^{\left(  i\right)  },\mathfrak{g}%
\in\operatorname*{End}\nolimits_{\Bbbk}\left(  B,B^{\otimes\left(  n-1\right)
}\right)  ,\nonumber\\
&  \mathfrak{g}_{\left[  j\right]  }\circ\left(  b\right)  =\left(
\mu^{\left(  n\right)  }\right)  ^{\circ\ell}\left[  \mathrm{f}_{\left[
j\right]  }^{\left(  1\right)  }\left(  b_{\left[  1\right]  }\right)
,\mathrm{f}_{\left[  j\right]  }^{\left(  2\right)  }\left(  b_{\left[
2\right]  }\right)  ,\ldots,\mathrm{f}_{\left[  j\right]  }^{\left(  n_{\star
}\right)  }\left(  b_{\left[  n_{\star}\right]  }\right)  \right]
,\nonumber\\
&  \mathrm{f}_{\left[  j\right]  }^{\left(  i\right)  }\in\operatorname*{End}%
\nolimits_{\Bbbk}\left(  B,B\right)  ,\ \ i\in1,\ldots,n_{\star}%
,\ \ j\in1,\ldots,n-1,\ \ \ b\in B,
\end{align}

The consistency condition (\ref{ns}) becomes \textit{reduced}%
\begin{equation}
n_{\star}=\ell\left(  n-1\right)  +1=\ell^{\prime}+1.
\end{equation}

\begin{definition}
The set of the multivalued maps $\mathfrak{f}^{\left(  i\right)  }%
\in\operatorname*{End}\nolimits_{\Bbbk}\left(  B,B^{\otimes\left(  n-1\right)
}\right)  $ (together with the polyadic identity $\mathfrak{Id}_{0}$) endowed
with the reduced convolution product $\mathbf{\bar{\mu}}_{\star}^{\left(
n_{\star}\right)  }$ is called a \textit{reduced }$n_{\star}$\textit{-ary}
\textit{convolution algebra}%
\begin{equation}
\mathrm{C}_{\star}^{\left(  2,n\right)  }=\left\langle \operatorname*{End}%
\nolimits_{\Bbbk}\left(  B,B^{\otimes\left(  n-1\right)  }\right)
\mid\mathbf{\bar{\mu}}_{\star}^{\left(  n_{\star}\right)  }\right\rangle .
\end{equation}

\end{definition}

\begin{remark}
The reduced convolution algebra $\mathrm{C}_{\star}^{\left(  2,n\right)  }$
having $n\neq2$ is not derived (\textbf{Definition} \ref{def-der}).
\end{remark}

Having the distinguished element $\mathfrak{Id}_{0}\in\mathrm{C}_{\star
}^{\left(  2,n\right)  }$ as an analog of $\operatorname*{id}\nolimits_{B}$
and the querelement (\ref{mfq}) for any $\mathfrak{f}\in\mathrm{C}_{\star
}^{\left(  2,n\right)  }$ (the polyadic version of inverse in the convolution
algebra), we are now in a position to \textquotedblleft
polyadize\textquotedblright\ the concept of the (binary) antipode.

\begin{definition}
A multivalued map $\mathbf{Q}_{0}:B\rightarrow B^{\otimes\left(  n-1\right)
}$ in the polyadic bialgebra $\mathrm{B}^{\left(  2,n\right)  }$ is called a
\textit{singular querantipode}, if it is the coquerelement of the polyadic
identity $\mathbf{Q}_{0}=\mathbf{q}_{\star}\left(  \mathfrak{Id}_{0}\right)  $
in the reduced $n_{\star}$-ary convolution algebra%
\begin{equation}
\mathbf{\bar{\mu}}_{\star}^{\left(  n_{\star}\right)  }\circ\left(
\mathfrak{Id}_{0}^{\otimes\left(  n_{\star}-1\right)  }\otimes\mathbf{Q}%
_{0}\right)  =\mathfrak{Id}_{0}, \label{mid}%
\end{equation}
where $\mathbf{Q}_{0}$ can be on any place, such that the diagram%
\begin{equation}
\begin{diagram} B && \rTo^{\mathbf{\Delta}^{\circ\left( n_{\star}-1\right) }} & &B^{\otimes n_{\star}} \\ \dTo^{\mathfrak{Id}_0} && & & \dTo_{\mathfrak{Id}_0^{\otimes\left( n_{\star}-1\right) }\otimes\mathbf{Q}_0} \\ B^{\otimes\left( n-1\right) }&\lTo^{\left( \left( \mathbf{\mu}^{\left( n\right) }\right) ^{\circ\ell}\right) ^{\otimes\left( n-1\right) }}&\left( B^{\otimes n_{\star}}\right) ^{\otimes \left( n-1\right) }& \lTo^{\mathbf{\tau }_{medial}^{\left( n-1,n_{\star}\right) } } & \left( B^{\otimes\left( n-1\right) }\right) ^{\otimes n_{\star}} \\ \end{diagram}
\end{equation}
commutes.
\end{definition}

In Sweedler notation%
\begin{equation}
\left(  \mu^{\left(  n\right)  }\right)  ^{\circ\ell}\left[  b_{\left[
1\right]  },b_{\left[  2\right]  },\ldots,b_{\left[  n_{\star}-1\right]
},\mathrm{Q}_{0\,\left[  j\right]  }\left(  b_{\left[  n_{\star}\right]
}\right)  \right]  =b,\ \ \ \ \ \ j\in1,\ldots,n-1,\ \ \ i\in1,\ldots
,n_{\star},
\end{equation}
where $\mathbf{\Delta}^{\circ\left(  n_{\star}-1\right)  }\left(  b\right)
=b_{\left[  1\right]  }\otimes b_{\left[  2\right]  }\otimes\ldots\otimes
b_{\left[  n_{\star}\right]  }$,\ $\mathbf{Q}_{0}\circ\left(  b\right)
=\mathrm{Q}_{0\,\left[  1\right]  }\left(  b\right)  \otimes\ldots
\otimes\mathrm{Q}_{0\,\left[  n-1\right]  }\left(  b\right)  $%
,\ \ $b,b_{\left[  i\right]  }\in B$.

\begin{definition}
A polyadic bialgebra $\mathrm{B}^{\left(  2,n\right)  }$ equipped with the
reduced $n_{\star}$-ary convolution product $\mathbf{\bar{\mu}}_{\star
}^{\left(  n_{\star}\right)  }$ and the singular querantipode $\mathbf{Q}_{0}$
(\ref{mid}) is called a \textit{singular polyadic Hopf algebra} and is denoted
by $\mathrm{H}_{\text{\textit{sing}}}^{\left(  n\right)  }=\left\langle
\mathrm{B}^{\left(  n,n\right)  }\mid\mathbf{\bar{\mu}}_{\star}^{\left(
n_{\star}\right)  },\mathbf{Q}_{0}\right\rangle $.
\end{definition}

Due to their exotic properties we will not consider singular polyadic Hopf
algebras $\mathrm{H}_{\text{\textit{sing}}}^{\left(  n\right)  }$ in detail.

In the \textsl{symmetric case} a polyadic identity-like map in
$\operatorname*{End}\nolimits_{\Bbbk}\left(  B^{\otimes\left(  n-1\right)
},B^{\otimes\left(  n-1\right)  }\right)  $ can be defined in a more natural way.

\begin{definition}
A \textit{symmetric polyadic identity} $\mathfrak{Id}:$ $B^{\otimes\left(
n-1\right)  }\rightarrow B^{\otimes\left(  n-1\right)  }$ is a polyadic tensor
product of ordinary identities in $\mathrm{B}^{\left(  n,n\right)  }$%
\begin{equation}
\mathfrak{Id}=\overset{n-1}{\overbrace{\operatorname*{id}\nolimits_{B}%
\otimes\ldots\otimes\operatorname*{id}\nolimits_{B}}},\ \ \ \operatorname*{id}%
\nolimits_{B}:B\rightarrow B. \label{id}%
\end{equation}

\end{definition}

Indeed, for any map $\mathfrak{f}\in\operatorname*{End}\nolimits_{\Bbbk
}\left(  B^{\otimes\left(  n-1\right)  },B^{\otimes\left(  n-1\right)
}\right)  $, obviously $\mathfrak{Id}\circ\mathfrak{f}=\mathfrak{f}%
\circ\mathfrak{Id}=\mathfrak{f}$.

The numbers of iterations are now equal $\ell=\ell^{\prime}$, and the
consistency condition (\ref{ns}) becomes%
\begin{equation}
n_{\star}-1=\ell\left(  n-1\right)  . \label{nsy}%
\end{equation}

\begin{definition}
The set of the multiplace multivalued maps $\mathfrak{f}^{\left(  i\right)
}\in\operatorname*{End}\nolimits_{\Bbbk}\left(  B^{\otimes\left(  n-1\right)
},B^{\otimes\left(  n-1\right)  }\right)  $ (together with the polyadic
identity $\mathfrak{Id}$) endowed with the symmetric convolution product
$\mathbf{\hat{\mu}}_{\star}^{\left(  n_{\star}\right)  }=\mathbf{\mu}_{\star
}^{\left(  n_{\star}\right)  }|_{n=n^{\prime}}$ (\ref{mff}) is called a
\textit{symmetric }$n_{\star}$\textit{-ary} \textit{convolution algebra}%
\begin{equation}
\mathrm{C}_{\star}^{\left(  n,n\right)  }=\left\langle \operatorname*{End}%
\nolimits_{\Bbbk}\left(  B^{\otimes\left(  n-1\right)  },B^{\otimes\left(
n-1\right)  }\right)  \mid\mathbf{\hat{\mu}}_{\star}^{\left(  n_{\star
}\right)  }\right\rangle .
\end{equation}

\end{definition}

For a polyadic analog of antipode in the symmetric case we have

\begin{definition}
A multiplace multivalued map $\mathbf{Q}_{\operatorname*{id}}:B^{\otimes
\left(  n-1\right)  }\rightarrow B^{\otimes\left(  n-1\right)  }$ in the
polyadic bialgebra $\mathrm{B}^{\left(  n,n\right)  }$ is called a
\textit{symmetric querantipode}, if it is the coquerelement (see (\ref{mfq}))
of the polyadic identity $\mathbf{Q}_{\operatorname*{id}}=\mathbf{q}_{\star
}\left(  \mathfrak{Id}\right)  $ in the symmetric $n_{\star}$-ary convolution
algebra%
\begin{equation}
\mathbf{\hat{\mu}}_{\star}^{\left(  n_{\star}\right)  }\circ\left(
\mathfrak{Id}^{\otimes\left(  n_{\star}-1\right)  }\otimes\mathbf{Q}%
_{\operatorname*{id}}\right)  =\mathfrak{Id}, \label{mnd}%
\end{equation}
where $\mathbf{Q}_{\operatorname*{id}}$ can be on any place, such that the
diagram%
\begin{equation}
\begin{diagram} B^{\otimes\left( n-1\right) } & \rTo^{ \left( \left( \mathbf{\Delta}^{\left( n\right) }\right) ^{\circ\ell}\right) ^{\otimes\left( n-1\right) }} &\left( B^{\otimes n_{\star}}\right) ^{\otimes \left( n-1\right) }&\rTo^{\mathbf{\tau }_{medial}^{ \left( n_{\star }, n-1\right)} } & \left( B^{\otimes\left( n-1\right) }\right) ^{\otimes n_{\star}} \\ \dTo^{\mathfrak{Id}} && & & \dTo_{\mathfrak{Id}^{\otimes\left( n_{\star}-1\right) }\otimes\mathbf{Q}_{\operatorname*{id}}} \\ B^{\otimes\left( n-1\right) }&\lTo^{\left( \left( \mathbf{\mu}^{\left( n\right) }\right) ^{\circ\ell}\right) ^{\otimes\left( n-1\right) }}&\left( B^{\otimes n_{\star}}\right) ^{\otimes \left( n-1\right) }& \lTo^{\mathbf{\tau }_{medial}^{\left( n-1,n_{\star}\right) } } & \left( B^{\otimes\left( n-1\right) }\right) ^{\otimes n_{\star}} \\ \end{diagram}
\end{equation}
commutes.
\end{definition}

In Sweedler notation we obtain (see (\ref{mff}) and (\ref{ms}))%
\begin{equation}
\left(  \mu^{\left(  n\right)  }\right)  ^{\circ\ell}\left[  b_{\left[
1\right]  }^{\left(  j\right)  },b_{\left[  2\right]  }^{\left(  j\right)
},b_{\left[  n_{\star}-1\right]  }^{\left(  j\right)  },\mathrm{Q}_{\left[
j\right]  }\left(  b_{\left[  n_{\star}\right]  }^{\left(  1\right)
},b_{\left[  n_{\star}\right]  }^{\left(  2\right)  },\ldots,b_{\left[
n_{\star}\right]  }^{\left(  n-1\right)  }\right)  \right]  =b^{\left(
j\right)  },\ j\in1,\ldots,n-1,\ i\in1,\ldots,n_{\star}, \label{mb}%
\end{equation}
where $\left(  \Delta^{\left(  n\right)  }\right)  ^{\circ\ell}\left(
b^{\left(  j\right)  }\right)  =b_{\left[  1\right]  }^{\left(  j\right)
}\otimes b_{\left[  2\right]  }^{\left(  j\right)  }\otimes\ldots\otimes
b_{\left[  n_{\star}\right]  }^{\left(  j\right)  },\ \ b_{\left[  i\right]
}^{\left(  j\right)  }\in B$, $\ell\in\mathbb{N}$, $\mathrm{Q}_{\left[
j\right]  }\in\operatorname*{End}\nolimits_{\Bbbk}\left(  B^{\otimes\left(
n-1\right)  },B\right)  $ are components of $\mathbf{Q}_{\operatorname*{id}}$,
and the convolution product arity is $n_{\star}=\ell\left(  n-1\right)  +1$
(\ref{nsy}).

\begin{definition}
\label{def-hopf}A polyadic bialgebra $\mathrm{B}^{\left(  n,n\right)  }$
equipped with the symmetric $n_{\star}$-ary convolution product $\mathbf{\hat
{\mu}}_{\star}^{\left(  n_{\star}\right)  }$ and the symmetric querantipode
$\mathbf{Q}_{\operatorname*{id}}$ (\ref{mnd}) is called a \textit{symmetric
polyadic Hopf algebra} and is denoted by $\mathrm{H}_{sym}^{\left(  n\right)
}=\left\langle \mathrm{B}^{\left(  n,n\right)  }\mid\mathbf{\hat{\mu}}_{\star
}^{\left(  n_{\star}\right)  },\mathbf{Q}_{\operatorname*{id}}\right\rangle $.
\end{definition}

\begin{example}
In the case where $n=n^{\prime}=3$ and $\ell=1$ we have $\Delta^{\left(
3\right)  }\left(  b^{\left(  j\right)  }\right)  =b_{\left[  1\right]
}^{\left(  j\right)  }\otimes b_{\left[  2\right]  }^{\left(  j\right)
}\otimes b_{\left[  3\right]  }^{\left(  j\right)  }$, $j=1,2,3$,%
\begin{align}
\mu^{\left(  3\right)  }\left[  b_{\left[  1\right]  }^{\left(  1\right)
},b_{\left[  2\right]  }^{\left(  1\right)  },\mathrm{Q}_{\left[  1\right]
}\left(  b_{\left[  3\right]  }^{\left(  1\right)  },b_{\left[  3\right]
}^{\left(  2\right)  }\right)  \right]   &  =b^{\left(  1\right)
},\ \ \ \ \ \mu^{\left(  3\right)  }\left[  b_{\left[  1\right]  }^{\left(
2\right)  },b_{\left[  2\right]  }^{\left(  2\right)  },\mathrm{Q}_{\left[
2\right]  }\left(  b_{\left[  3\right]  }^{\left(  1\right)  },b_{\left[
3\right]  }^{\left(  2\right)  }\right)  \right]  =b^{\left(  2\right)
},\nonumber\\
\mu^{\left(  3\right)  }\left[  b_{\left[  1\right]  }^{\left(  1\right)
},\mathrm{Q}_{\left[  1\right]  }\left(  b_{\left[  2\right]  }^{\left(
1\right)  },b_{\left[  2\right]  }^{\left(  2\right)  }\right)  ,b_{\left[
3\right]  }^{\left(  1\right)  }\right]   &  =b^{\left(  1\right)
},\ \ \ \ \ \mu^{\left(  3\right)  }\left[  b_{\left[  1\right]  }^{\left(
2\right)  },\mathrm{Q}_{\left[  2\right]  }\left(  b_{\left[  2\right]
}^{\left(  1\right)  },b_{\left[  2\right]  }^{\left(  2\right)  }\right)
,b_{\left[  3\right]  }^{\left(  2\right)  }\right]  =b^{\left(  2\right)
},\nonumber\\
\mu^{\left(  3\right)  }\left[  \mathrm{Q}_{\left[  1\right]  }\left(
b_{\left[  1\right]  }^{\left(  1\right)  },b_{\left[  1\right]  }^{\left(
2\right)  }\right)  ,b_{\left[  2\right]  }^{\left(  1\right)  },b_{\left[
3\right]  }^{\left(  1\right)  }\right]   &  =b^{\left(  1\right)
},\ \ \ \ \ \mu^{\left(  3\right)  }\left[  \mathrm{Q}_{\left[  2\right]
}\left(  b_{\left[  1\right]  }^{\left(  1\right)  },b_{\left[  1\right]
}^{\left(  2\right)  }\right)  ,b_{\left[  2\right]  }^{\left(  2\right)
},b_{\left[  3\right]  }^{\left(  2\right)  }\right]  =b^{\left(  2\right)  },
\label{mbb}%
\end{align}
which can be compared with the binary case ($b_{\left[  1\right]  }S\left(
b_{\left[  2\right]  }\right)  =S\left(  b_{\left[  1\right]  }\right)
b_{\left[  2\right]  }=\eta\left(  \varepsilon\left(  b\right)  \right)  $)
and (\ref{q}), (\ref{mfq}).
\end{example}

Recall that the main property of the antipode $S$ of a binary bialgebra $B$ is
its \textquotedblleft anticommutation\textquotedblright\ with the
multiplication $\mu$ and comultiplication $\Delta$ (e.g., \cite{swe})%
\begin{align}
S\circ\mu &  =\mu\circ\tau_{op}\circ\left(  S\otimes S\right)  ,\ \ \ S\circ
\eta=\eta,\label{sm}\\
\Delta\circ S  &  =\tau_{op}\circ\left(  S\otimes S\right)  \circ
\Delta,\ \ \ \varepsilon\circ S=\varepsilon, \label{ds}%
\end{align}
where $\tau_{op}$ is the binary twist (see (\ref{ti})). The first relation
means that $S$ is an algebra anti-endomorphism, because in the elementwise
description $S\left(  a\cdot b\right)  =S\left(  b\right)  \cdot S\left(
a\right)  $, $a,b\in B$, $\left(  \cdot\right)  \equiv\mu$.

We propose the polyadic analogs of (\ref{sm})-(\ref{ds}) without proofs, which
are too cumbersome, but their derivations almost coincide with those for the
binary case.

\begin{proposition}
The querantipode $\mathbf{Q}_{\operatorname*{id}}:B^{\otimes\left(
n-1\right)  }\rightarrow B^{\otimes\left(  n-1\right)  }$ of the polyadic
bialgebra $\mathrm{B}^{\left(  n,n\right)  }=\left\langle B\mid\mathbf{\mu
}^{\left(  n\right)  },\mathbf{\Delta}^{\left(  n\right)  }\right\rangle $
satisfies the polyadic version of \textquotedblleft
antimultiplicativity\textquotedblright\ (\textquotedblleft antialgebra
map\textquotedblright)%
\begin{equation}
\mathbf{Q}_{\operatorname*{id}}\circ\left(  \left(  \mathbf{\mu}^{\left(
n\right)  }\right)  ^{\circ\ell}\right)  ^{\otimes\left(  n-1\right)
}=\left(  \left(  \mathbf{\mu}^{\left(  n\right)  }\right)  ^{\circ\ell
}\right)  ^{\otimes\left(  n-1\right)  }\circ\mathbf{\tau}_{op}^{\left(
\ell_{\tau}\right)  }\circ\mathbf{Q}_{\operatorname*{id}}^{\otimes n_{\star}%
}\circ\mathbf{\tau}_{medial}^{\left(  n_{\star},n-1\right)  }, \label{qm}%
\end{equation}
and \textquotedblleft anticomultiplicativity\textquotedblright%
\ (\textquotedblleft anticoalgebra map\textquotedblright)%
\begin{equation}
\left(  \left(  \mathbf{\Delta}^{\left(  n\right)  }\right)  ^{\circ\ell
}\right)  ^{\otimes\left(  n-1\right)  }\circ\mathbf{Q}_{\operatorname*{id}%
}=\mathbf{\tau}_{op}^{\left(  \ell_{\tau}\right)  }\circ\mathbf{Q}%
_{\operatorname*{id}}^{\otimes n_{\star}}\circ\mathbf{\tau}_{medial}^{\left(
n-1,n_{\star}\right)  }\circ\left(  \left(  \mathbf{\Delta}^{\left(  n\right)
}\right)  ^{\circ\ell}\right)  ^{\otimes\left(  n-1\right)  }, \label{aqm}%
\end{equation}
where $\mathbf{\tau}_{medial}^{\left(  n,m\right)  }$ is the medial map
\emph{(\ref{an})}, $\mathbf{\tau}_{op}^{\left(  \ell_{\tau}\right)  }$ is the
polyadic twist \emph{(\ref{top})} and $\ell_{\tau}=\left(  n-1\right)
n_{\star}$ should be allowed \emph{(}see \textsc{Table} \ref{tab-twist}%
\emph{)}.
\end{proposition}

\begin{proposition}
If the polyadic unit $\mathbf{\eta}^{\left(  r,n\right)  }$ \emph{(\ref{mu})}
and counit $\mathbf{\varepsilon}^{\left(  n,r\right)  }$ \emph{(\ref{ep})} in
$\mathrm{B}^{\left(  n,n\right)  }$ exist, then%
\begin{align}
\mathbf{Q}_{\operatorname*{id}}\circ\mathbf{\eta}^{\left(  r,n\right)  }  &
=\mathbf{\eta}^{\left(  r,n\right)  }\mathbf{,}\\
\mathbf{\varepsilon}^{\left(  n,r\right)  }\circ\mathbf{Q}_{\operatorname*{id}%
}  &  =\mathbf{\varepsilon}^{\left(  n,r\right)  }\mathbf{.}%
\end{align}

\end{proposition}

\begin{example}
If $n=3$, $\ell=1$, $n_{\star}=3$, $\ell_{\tau}=6$ (and (\ref{t6})), then
using Sweedler notation, for (\ref{qm}) we have%
\begin{align}
\mathrm{Q}_{\left[  1\right]  }\left(  \mu^{\left(  3\right)  }\left[
a_{1},a_{2},a_{3}\right]  ,\mu^{\left(  3\right)  }\left[  b_{1},b_{2}%
,b_{3}\right]  \right)   &  =\mu^{\left(  3\right)  }\left[  \mathrm{Q}%
_{\left[  2\right]  }\left(  a_{2},b_{2}\right)  ,\mathrm{Q}_{\left[
1\right]  }\left(  a_{1},b_{1}\right)  ,\mathrm{Q}_{\left[  1\right]  }\left(
a_{3},b_{3}\right)  \right]  ,\\
\mathrm{Q}_{\left[  2\right]  }\left(  \mu^{\left(  3\right)  }\left[
a_{1},a_{2},a_{3}\right]  ,\mu^{\left(  3\right)  }\left[  b_{1},b_{2}%
,b_{3}\right]  \right)   &  =\mu^{\left(  3\right)  }\left[  \mathrm{Q}%
_{\left[  2\right]  }\left(  a_{1},b_{1}\right)  ,\mathrm{Q}_{\left[
2\right]  }\left(  a_{3},b_{3}\right)  ,\mathrm{Q}_{\left[  1\right]  }\left(
a_{2},b_{2}\right)  \right]  ,
\end{align}
where $\mathbf{Q}_{\operatorname*{id}}\circ\left(  a,b\right)  =\mathrm{Q}%
_{\left[  1\right]  }\left(  a,b\right)  \otimes\mathrm{Q}_{\left[  2\right]
}\left(  a,b\right)  \in\operatorname*{End}\nolimits_{\Bbbk}\left(  B\otimes
B,B\otimes B\right)  $, $a,b,a_{i},b_{i}\in B$, (cf. (\ref{sm})).
\end{example}

The key property of the binary antipode $S$ is its involutivity
$S^{\circ2}=\operatorname*{id}_{B}$ for either commutative ($\mu=\mu\circ
\tau_{op}$) or co-commutative ($\Delta=\tau_{op}\circ\Delta$) Hopf algebras,
which follows from (\ref{sm}) or (\ref{ds}) applied to $S\star S^{\circ2}$
giving $\eta\varepsilon$ ($=S\star\operatorname*{id}_{B}$).

\begin{proposition}
If in a symmetric Hopf algebra $\mathrm{H}_{sym}^{\left(  n\right)  }$ either
multiplication or comultiplication is invariant under polyadic twist map
$\mathbf{\tau}_{op}^{\left(  \ell_{\tau}\right)  }$ \emph{(\ref{top})}, then
the querantipode $\mathbf{Q}_{\operatorname*{id}}$ \emph{(\ref{mnd})}
satisfies%
\begin{equation}
\mathbf{\hat{\mu}}_{\star}^{\left(  n_{\star}\right)  }\left[  \overset
{\left(  n_{\star}-1\right)  }{\overbrace{\mathbf{Q}_{\operatorname*{id}%
},\ldots,\mathbf{Q}_{\operatorname*{id}}}},\mathbf{Q}_{\operatorname*{id}%
}\circ\mathbf{Q}_{\operatorname*{id}}\right]  =\mathbf{Q}_{\operatorname*{id}%
}, \label{mqq}%
\end{equation}
where $\mathbf{Q}_{\operatorname*{id}}^{\circ2}$ can be on any place, or the
convolution querelement \emph{(\ref{mfq})} of the querantipode $\mathbf{Q}%
_{\operatorname*{id}}$ is%
\begin{equation}
\mathbf{q}_{\star}\left(  \mathbf{Q}_{\operatorname*{id}}\right)
=\mathbf{Q}_{\operatorname*{id}}^{\circ2}. \label{qq}%
\end{equation}

\end{proposition}

\begin{proof}
The proposition follows from applying either (\ref{qm}) or (\ref{aqm}) to the l.h.s. of
(\ref{mqq}), to use (\ref{mfq}).
\end{proof}

\section{\textsc{Towards polyadic quantum groups}}

Bialgebras with a special relaxation of co-commutativity, almost
co-commutativity, are the ground objects in the construction of quantum groups
identified with the non-commutative and non-co-commutative quasitriangular
Hopf algebras \cite{dri0,dri2}.

\subsection{Quantum Yang-Baxter equation}

Here we recall the binary case (informally) in a notation that will allow us
to provide the \textquotedblleft polyadization\textquotedblright\ in a
clearer way.

Let us consider a (binary) bialgebra $\mathrm{B}^{\left(  2,2\right)
}=\left\langle B\mid\mu,\Delta\right\rangle $, where $\mu=\mu^{\left(  2\right)  }$
is the binary multiplication, $\Delta=\Delta^{\left(  2\right)  }$ (see
\textbf{Definition} \ref{def-bialg}), and the opposite comultiplication is given by
$\Delta_{cop}\equiv\tau_{op}\circ\Delta$, where $\tau_{op}$ is the binary
twist (\ref{ti}). To relax the co-commutativity condition ($\Delta_{cop}=\Delta$), the
following construction inspired by \textsf{conjugation} in groups was
proposed \cite{dri0,dri2}. A bialgebra $\mathrm{B}^{\left(  2,2\right)  }$ is
\textit{almost co-commutative}, if there exists $\mathit{R}\in B\otimes B$
such that (in the elementwise notation)%
\begin{equation}
\mu\left[  \Delta_{cop}\left(  b\right)  ,\mathit{R}\right]  =\mu\left[
\mathit{R},\Delta\left(  b\right)  \right]  ,\ \ \ \ \forall b\in B.
\label{md}%
\end{equation}

A \textsf{fixed} element $\mathit{R}$ of a bialgebra satisfying (\ref{md}) is
called a \textit{universal }$R$\textit{-matrix}. For a co-commutative
bialgebra we have $\mathit{R}=e_{B}\otimes e_{B}$, where $e_{B}\in B$ is the
unit (element) of the algebra $\left\langle B\mid\mu\right\rangle $.

If we demand that $\left\langle B\mid\Delta_{cop}\right\rangle $ is the
opposite coalgebra of $\left\langle B\mid\Delta\right\rangle $, and therefore
$\Delta_{cop}$ be coassociative, then $\mathit{R}$ cannot be
\textsf{arbitrary}, but has to satisfy some additional conditions, which we will call
the \textit{almost co-commutativity equations} for the $R$-matrix. Indeed, using
(\ref{md}) we can write%
\begin{align}
&  \mu\left[  \left(  \Delta_{cop}\otimes\operatorname*{id}\nolimits_{B}%
\right)  \circ\Delta_{cop}\left(  b\right)  ,\mu\left[  \left(  \mathit{R}%
\otimes e_{B}\right)  ,\left(  \Delta\otimes\operatorname*{id}\nolimits_{B}%
\right)  \left(  \mathit{R}\right)  \right]  \right] \nonumber\\
&  =\mu\left[  \mu\left[  \left(  \mathit{R}\otimes e_{B}\right)  ,\left(
\Delta\otimes\operatorname*{id}\nolimits_{B}\right)  \left(  \mathit{R}%
\right)  \right]  ,\left(  \Delta\otimes\operatorname*{id}\nolimits_{B}%
\right)  \circ\Delta\left(  b\right)  \right]  ,\\
&  \mu\left[  \left(  \operatorname*{id}\nolimits_{B}\otimes\Delta
_{cop}\right)  \circ\Delta_{cop}\left(  b\right)  ,\mu\left[  \left(
e_{B}\otimes\mathit{R}\right)  ,\left(  \operatorname*{id}\nolimits_{B}%
\otimes\Delta\right)  \left(  \mathit{R}\right)  \right]  \right] \nonumber\\
&  =\mu\left[  \mu\left[  \left(  e_{B}\otimes\mathit{R}\right)  ,\left(
\operatorname*{id}\nolimits_{B}\otimes\Delta\right)  \left(  \mathit{R}%
\right)  \right]  ,\left(  \operatorname*{id}\nolimits_{B}\otimes
\Delta\right)  \circ\Delta\left(  b\right)  \right]  ,
\end{align}
Therefore, the coassociativity of $\Delta_{cop}$ leads to the \textsf{first}
almost co-commutativity equation%
\begin{equation}
\mu\left[  \left(  \mathit{R}\otimes e_{B}\right)  ,\left(  \Delta
\otimes\operatorname*{id}\nolimits_{B}\right)  \left(  \mathit{R}\right)
\right]  =\mu\left[  \left(  e_{B}\otimes\mathit{R}\right)  ,\left(
\operatorname*{id}\nolimits_{B}\otimes\Delta\right)  \left(  \mathit{R}%
\right)  \right]  . \label{mr}%
\end{equation}
On the other hand, directly from (\ref{md}), we have relations which can be treated
as the next \textsf{two }almost co-commutativity equations
(\textsf{unconnected} to the coassociativity of $\Delta_{cop}$)%
\begin{align}
\mu\left[  \left(  \mathit{R}\otimes e_{B}\right)  ,\left(  \Delta
\otimes\operatorname*{id}\nolimits_{B}\right)  \left(  \mathit{R}\right)
\right]   &  =\mu\left[  \left(  \Delta_{cop}\otimes\operatorname*{id}%
\nolimits_{B}\right)  \left(  \mathit{R}\right)  ,\left(  \mathit{R}\otimes
e_{B}\right)  \right] \nonumber\\
&  =\mu\left[  \left(  \tau_{op}\otimes\operatorname*{id}\nolimits_{B}\right)
\circ\left(  \Delta\otimes\operatorname*{id}\nolimits_{B}\right)  \left(
\mathit{R}\right)  ,\left(  \mathit{R}\otimes e_{B}\right)  \right]
,\label{mr1}\\
\mu\left[  \left(  e_{B}\otimes\mathit{R}\right)  ,\left(  \operatorname*{id}%
\nolimits_{B}\otimes\Delta\right)  \left(  \mathit{R}\right)  \right]   &
=\mu\left[  \left(  \operatorname*{id}\nolimits_{B}\otimes\Delta_{cop}\right)
\left(  \mathit{R}\right)  ,\left(  e_{B}\otimes\mathit{R}\right)  \right]
\nonumber\\
&  =\mu\left[  \left(  \operatorname*{id}\nolimits_{B}\otimes\tau_{op}\right)
\circ\left(  \operatorname*{id}\nolimits_{B}\otimes\Delta\right)  \left(
\mathit{R}\right)  ,\left(  e_{B}\otimes\mathit{R}\right)  \right]  .
\label{mr2}%
\end{align}

The equations (\ref{mr})--(\ref{mr2}) for the components of%
\begin{equation}
\mathit{R}=\sum_{\alpha}\mathsf{r}_{\alpha}^{\left(  1\right)  }%
\otimes\mathsf{r}_{\alpha}^{\left(  2\right)  }\in B\otimes B \label{rb}%
\end{equation}
are on $B\otimes B\otimes B$. In components the almost co-commutativity
(\ref{md}) can be expressed as follows%
\begin{equation}
\sum_{\left[  b\right]  }\sum_{\alpha}\mu\left[  b_{\left[  2\right]
},\mathsf{r}_{\alpha}^{\left(  1\right)  }\right]  \otimes\mu\left[
b_{\left[  1\right]  },\mathsf{r}_{\alpha}^{\left(  2\right)  }\right]
=\sum_{\left[  b\right]  ^{\prime}}\sum_{\alpha^{\prime}}\mu\left[
\mathsf{r}_{\alpha^{\prime}}^{\left(  1\right)  },b_{\left[  1\right]
^{\prime}}\right]  \otimes\mu\left[  \mathsf{r}_{\alpha^{\prime}}^{\left(
2\right)  },b_{\left[  2\right]  ^{\prime}}\right]  . \label{mbr}%
\end{equation}
Now introduce the \textquotedblleft extended\textquotedblright\ form of the
$R$-matrix $\mathcal{R}_{ij}\in B\otimes B\otimes B$, $i,j=1,2,3$, by%
\begin{align}
\mathcal{R}_{12}  &  =\sum_{\alpha}\mathsf{r}_{\alpha}^{\left(  1\right)
}\otimes\mathsf{r}_{\alpha}^{\left(  2\right)  }\otimes e_{B}\equiv
\mathit{R}\otimes e_{B},\label{r12}\\
\mathcal{R}_{13}  &  =\sum_{\alpha}\mathsf{r}_{\alpha}^{\left(  1\right)
}\otimes e_{B}\otimes\mathsf{r}_{\alpha}^{\left(  2\right)  }=\left(
\operatorname*{id}\nolimits_{B}\otimes\tau_{op}\right)  \circ\left(
\mathit{R}\otimes e_{B}\right)  ,\label{r13}\\
\mathcal{R}_{23}  &  =\sum_{\alpha}e_{B}\otimes\mathsf{r}_{\alpha}^{\left(
1\right)  }\otimes\mathsf{r}_{\alpha}^{\left(  2\right)  }\equiv e_{B}%
\otimes\mathit{R}. \label{r23}%
\end{align}

Obviously, one can try to solve (\ref{mr})--(\ref{mr2}) with respect to the
$\mathsf{r}_{\alpha}^{\left(  1\right)  }$, $\mathsf{r}_{\alpha}^{\left(
2\right)  }$ directly, but then we are confronted with a difficulty arising
from the Sweedler components, because now (see (\ref{dn})--(\ref{dv}))%
\begin{align}
\left(  \Delta\otimes\operatorname*{id}\nolimits_{B}\right)  \left(
\mathit{R}\right)   &  =\sum_{\left[  \mathsf{r}_{\alpha}^{\left(  1\right)
}\right]  }\sum_{\alpha}\mathsf{r}_{\alpha,\left[  1\right]  }^{\left(
1\right)  }\otimes\mathsf{r}_{\alpha,\left[  2\right]  }^{\left(  1\right)
}\otimes\mathsf{r}_{\alpha}^{\left(  2\right)  },\label{dr1}\\
\left(  \operatorname*{id}\nolimits_{B}\otimes\Delta\right)  \left(
\mathit{R}\right)   &  =\sum_{\left[  \mathsf{r}_{\alpha}^{\left(  2\right)
}\right]  }\sum_{\alpha}\mathsf{r}_{\alpha}^{\left(  1\right)  }%
\otimes\mathsf{r}_{\alpha,\left[  1\right]  }^{\left(  2\right)  }%
\otimes\mathsf{r}_{\alpha,\left[  2\right]  }^{\left(  2\right)  }.
\label{dr2}%
\end{align}

To avoid computations in the Sweedler components, one can substitute them by
the components of $\mathit{R}$ directly as $\mathsf{r}_{\left[  j\right]
}^{\left(  i\right)  }\longrightarrow\mathsf{r}^{\left(  i\right)  }$
(schematically). This allows us to express (\ref{dr1})--(\ref{dr2}) solely
through elements of the \textquotedblleft extended\textquotedblright%
\ $R$-matrix $\mathcal{R}_{ij}$\textit{ }by%
\begin{align}
\left(  \Delta\otimes\operatorname*{id}\nolimits_{B}\right)  \left(
\mathit{R}\right)   &  =\mu\left[  \mathcal{R}_{13},\mathcal{R}_{23}\right]
\equiv\sum_{\alpha,\beta}\mathsf{r}_{\alpha}^{\left(  1\right)  }%
\otimes\mathsf{r}_{\beta}^{\left(  1\right)  }\otimes\mu\left[  \mathsf{r}%
_{\alpha}^{\left(  2\right)  },\mathsf{r}_{\beta}^{\left(  2\right)  }\right]
,\label{drr1}\\
\left(  \operatorname*{id}\nolimits_{B}\otimes\Delta\right)  \left(
\mathit{R}\right)   &  =\mu\left[  \mathcal{R}_{13},\mathcal{R}_{12}\right]
\equiv\sum_{\alpha,\beta}\mu\left[  \mathsf{r}_{\alpha}^{\left(  1\right)
},\mathsf{r}_{\beta}^{\left(  1\right)  }\right]  \otimes\mathsf{r}_{\beta
}^{\left(  2\right)  }\otimes\mathsf{r}_{\alpha}^{\left(  2\right)  },
\label{drr2}%
\end{align}
which do not contain Sweedler components of $\mathit{R}$ at all. The
equations (\ref{drr1})--(\ref{drr2}) \textsf{define} a \textit{quasitriangular
}$R$-matrix \cite{dri0}. The corresponding almost co-commutative (binary)
bialgebra $\mathrm{B}_{braid}^{\left(  2\right)  }=\left\langle \mathrm{B}%
^{\left(  2,2\right)  },\mathit{R}\right\rangle $ is called a
\textit{quasitriangular almost co-commutative bialgebra} (or \textit{braided
bialgebra} \cite{kassel}). Only for them can the almost co-commutativity equations
(\ref{mr})--(\ref{mr2}) be expressed solely in terms of $R$-matrix
components or through the \textquotedblleft extended\textquotedblright%
\ $R$-matrix $\mathcal{R}_{ij}$, using (\ref{drr1})--(\ref{drr2}).

\begin{theorem}
In the binary case, three almost co-commutativity equations for the $R$-matrix
\textsf{coincide} with%
\begin{equation}
\mu^{\circ2}\left[  \mathcal{R}_{12},\mathcal{R}_{13},\mathcal{R}_{23}\right]
=\mu^{\circ2}\left[  \mathcal{R}_{23},\mathcal{R}_{13},\mathcal{R}%
_{12}\right]  . \label{rrr}%
\end{equation}

\end{theorem}

Conversely, any quasitriangular $R$-matrix is a solution of (\ref{rrr}) by the
above construction. The equation for the \textquotedblleft
extended\textquotedblright\ $R$-matrix $\mathcal{R}_{ij}$ (\ref{rrr}) is
called the \textit{quantum Yang-Baxter equation} \cite{lam/rad,majid} (or the
triangle relation \cite{dri2}). In terms of the $R$-matrix components
(\ref{rb}) the quantum Yang-Baxter equation (\ref{rrr}) takes the form%
\begin{equation}
\sum_{\alpha,\beta,\gamma}\mu\left[  \mathsf{r}_{\alpha}^{\left(  1\right)
},\mathsf{r}_{\beta}^{\left(  1\right)  }\right]  \otimes\mu\left[
\mathsf{r}_{\alpha}^{\left(  2\right)  },\mathsf{r}_{\gamma}^{\left(
1\right)  }\right]  \otimes\mu\left[  \mathsf{r}_{\beta}^{\left(  2\right)
},\mathsf{r}_{\gamma}^{\left(  2\right)  }\right]  =\sum_{\alpha^{\prime
},\beta^{\prime},\gamma^{\prime}}\mu\left[  \mathsf{r}_{\beta^{\prime}%
}^{\left(  1\right)  },\mathsf{r}_{\alpha^{\prime}}^{\left(  1\right)
}\right]  \otimes\mu\left[  \mathsf{r}_{\gamma^{\prime}}^{\left(  1\right)
},\mathsf{r}_{\alpha^{\prime}}^{\left(  2\right)  }\right]  \otimes\mu\left[
\mathsf{r}_{\gamma^{\prime}}^{\left(  2\right)  },\mathsf{r}_{\beta^{\prime}%
}^{\left(  2\right)  }\right]  . \label{mmrr}%
\end{equation}

Let us consider modules over the braided bialgebra $\mathrm{B}_{braid}%
^{\left(  2,2\right)  }$ and recall \cite{dri89} how the universal $R$-matrix
generalizes the standard flip $\mathbf{\tau}_{V_{1}V_{2}}:V_{1}\otimes
V_{2}\rightarrow V_{2}\otimes V_{1}$. Define the isomorphism of modules (which in
our notation correspond to a 1-place action $\rho$ (\ref{r}))
$\mathbf{c}_{V_{1}V_{2}}:V_{1}\otimes V_{2}\rightarrow V_{2}\otimes V_{1}$ by%
\begin{equation}
\mathbf{c}_{V_{1}V_{2}}\circ\left(  v_{1}\otimes v_{2}\right)  =\mathbf{\tau
}_{V_{1}V_{2}}\circ\mathit{R}\circ\left(  v_{1}\otimes v_{2}\right)
=\sum_{\alpha}\rho\left(  \mathsf{r}_{\alpha}^{\left(  2\right)  }\mid
v_{2}\right)  \otimes\rho\left(  \mathsf{r}_{\alpha}^{\left(  1\right)  }\mid
v_{1}\right)  ,v_{i}\in V_{i},\ \mathsf{r}_{\alpha}^{\left(  i\right)  }\in
B,i=1,2. \label{cr}%
\end{equation}
The quasitriangularity (\ref{drr1})--(\ref{drr2}) and (\ref{cr}) on
$V_{1}\otimes V_{2}\otimes V_{3}$ leads to (see, e.g., \cite{kassel})%
\begin{align}
\left(  \mathbf{c}_{V_{1}V_{3}}\otimes\operatorname*{id}\nolimits_{V_{2}%
}\right)  \circ\left(  \operatorname*{id}\nolimits_{V_{1}}\otimes
\mathbf{c}_{V_{2}V_{3}}\right)   &  =\mathbf{c}_{V_{1}\otimes V_{2},V_{3}%
},\label{cv1}\\
\left(  \operatorname*{id}\nolimits_{V_{2}}\otimes\mathbf{c}_{V_{1}V_{3}%
}\right)  \circ\left(  \mathbf{c}_{V_{1}V_{2}}\otimes\operatorname*{id}%
\nolimits_{V_{3}}\right)   &  =\mathbf{c}_{V_{1},V_{2}\otimes V_{3}}.
\label{cv2}%
\end{align}

Similarly, the quantum Yang-Baxter equation (\ref{rrr}) gives the
\textit{braid equation} \cite{dri89} mapping $V_{1}\otimes V_{2}\otimes
V_{3}\rightarrow V_{3}\otimes V_{2}\otimes V_{1}$:%
\begin{equation}
\left(  \mathbf{c}_{V_{2}V_{3}}\otimes\operatorname*{id}\nolimits_{V_{1}%
}\right)  \circ\left(  \operatorname*{id}\nolimits_{V_{2}}\otimes
\mathbf{c}_{V_{1}V_{3}}\right)  \circ\left(  \mathbf{c}_{V_{1}V_{2}}%
\otimes\operatorname*{id}\nolimits_{V_{3}}\right)  =\left(  \operatorname*{id}%
\nolimits_{V_{3}}\otimes\mathbf{c}_{V_{1}V_{2}}\right)  \circ\left(
\mathbf{c}_{V_{1}V_{3}}\otimes\operatorname*{id}\nolimits_{V_{2}}\right)
\circ\left(  \operatorname*{id}\nolimits_{V_{1}}\otimes\mathbf{c}_{V_{2}V_{3}%
}\right)  . \label{cii}%
\end{equation}

Putting $V_{1}=V_{2}=V_{3}=V$ shows that $\mathbf{c}_{VV}$ is a solution of
the braid equation (\ref{cii}) for any module $V$, if the $\mathit{R}$ is a
solution of the Yang-Baxter equation \cite{dri89,kassel}.

\subsection{$n^{\prime}$-ary braid equation}

Let us consider possible higher arity generalizations of the braid equation
(\ref{cii}), informally. Introduce the modules $V_{i}$ over the polyadic
bialgebra $\mathrm{B}^{\left(  n^{\prime},n\right)  }$ (\textbf{Definition
}\ref{def-bialg}) by the $r$-place actions $\rho_{V_{i}}^{\left(  r\right)
}\left(  b_{1},\ldots b_{r}\mid v_{i}\right)  $, $b_{j}\in B$, $v_{i}\in
V_{i}$, $i=1,\ldots s$, $j=1,\ldots,r$ (see (\ref{r1})). Define the following
morphisms of modules%
\begin{equation}
\mathbf{c}_{V_{1}\ldots V_{n^{\prime}}}:V_{1}\otimes\ldots\otimes
V_{n^{\prime}}\rightarrow V_{n^{\prime}}\otimes\ldots\otimes V_{1}.
\label{cvn}%
\end{equation}

We use the shorthand notation $\mathbf{c}_{V^{n^{\prime}}}\equiv
\mathbf{c}_{V_{1}\ldots V_{n^{\prime}}}$, $\operatorname*{id}_{V}%
\equiv\operatorname*{id}_{V_{i}}$ and introduce indices manifestly only when
it will be needed.

\begin{proposition}
The $n^{\prime}$-ary braid equation has the form%
\begin{align}
&  \left(  \mathbf{c}_{V^{n^{\prime}}}\otimes\overset{n^{\prime}-1}%
{\overbrace{\operatorname*{id}\nolimits_{V}\otimes\ldots\otimes
\operatorname*{id}\nolimits_{V}}}\right)  \circ\left(  \operatorname*{id}%
\nolimits_{V}\otimes\mathbf{c}_{V^{n^{\prime}}}\otimes\overset{n^{\prime}%
-2}{\overbrace{\operatorname*{id}\nolimits_{V}\otimes\ldots\otimes
\operatorname*{id}\nolimits_{V}}}\right)  \circ\ldots\nonumber\\
&  \circ\left(  \overset{n^{\prime}-1}{\overbrace{\operatorname*{id}%
\nolimits_{V}\otimes\ldots\otimes\operatorname*{id}\nolimits_{V}}}%
\otimes\mathbf{c}_{V^{n^{\prime}}}\right)  \circ\left(  \mathbf{c}%
_{V^{n^{\prime}}}\otimes\overset{n^{\prime}-1}{\overbrace{\operatorname*{id}%
\nolimits_{V}\otimes\ldots\otimes\operatorname*{id}\nolimits_{V}}}\right)
\nonumber\\
&  =\left(  \overset{n^{\prime}-1}{\overbrace{\operatorname*{id}%
\nolimits_{V}\otimes\ldots\otimes\operatorname*{id}\nolimits_{V}}}%
\otimes\mathbf{c}_{V^{n^{\prime}}}\right)  \circ\left(  \mathbf{c}%
_{V^{n^{\prime}}}\otimes\overset{n^{\prime}-1}{\overbrace{\operatorname*{id}%
\nolimits_{V}\otimes\ldots\otimes\operatorname*{id}\nolimits_{V}}}\right)
\circ\ldots\nonumber\\
&  \circ\left(  \overset{n^{\prime}-2}{\overbrace{\operatorname*{id}%
\nolimits_{V}\otimes\ldots\otimes\operatorname*{id}\nolimits_{V}}}%
\otimes\mathbf{c}_{V^{n^{\prime}}}\otimes\operatorname*{id}\nolimits_{V}%
\right)  \circ\left(  \overset{n^{\prime}-1}{\overbrace{\operatorname*{id}%
\nolimits_{V}\otimes\ldots\otimes\operatorname*{id}\nolimits_{V}}}%
\otimes\mathbf{c}_{V^{n^{\prime}}}\right)  , \label{nc}%
\end{align}
where each side consists of $\left(  n^{\prime}+1\right)  $ brackets with
$\left(  2n^{\prime}-1\right)  $ multipliers.
\end{proposition}

\begin{proof}
Use the associative quiver technique from \cite{dup2018a} (The Post-like
quiver in \textbf{Section 6}).
\end{proof}

\begin{remark}
\label{rem-add}There can be additional equations depending on the concrete
values of $n^{\prime}$ which can contain a different number of brackets
determined by the corresponding diagram commutation.
\end{remark}

\begin{example}
In case $n^{\prime}=3$ we have the \textit{ternary braided equation} for
$\mathbf{c}_{V_{1}V_{2}V_{3}}:V_{1}\otimes V_{2}\otimes V_{3}\rightarrow
V_{3}\otimes V_{2}\otimes V_{1}$ on the tensor product of modules
$V_{1}\otimes V_{2}\otimes V_{3}$ $\otimes V_{4}\otimes V_{5}$, as%
\begin{align}
&  \left(  \mathbf{c}_{V_{3}V_{4}V_{5}}\otimes\operatorname*{id}%
\nolimits_{V_{2}}\otimes\operatorname*{id}\nolimits_{V_{1}}\right)
\circ\left(  \operatorname*{id}\nolimits_{V_{3}}\otimes\mathbf{c}_{V_{2}%
V_{5}V_{4}}\otimes\operatorname*{id}\nolimits_{V_{1}}\right)  \circ\left(
\operatorname*{id}\nolimits_{V_{3}}\otimes\operatorname*{id}\nolimits_{V_{2}%
}\otimes\mathbf{c}_{V_{1}V_{4}V_{5}}\right)  \circ\left(  \mathbf{c}%
_{V_{1}V_{2}V_{3}}\otimes\operatorname*{id}\nolimits_{V_{4}}\otimes
\operatorname*{id}\nolimits_{V_{5}}\right)  =\nonumber\\
&  \left(  \operatorname*{id}\nolimits_{V_{5}}\otimes\operatorname*{id}%
\nolimits_{V_{4}}\otimes\mathbf{c}_{V_{1}V_{2}V_{3}}\right)  \circ\left(
\mathbf{c}_{V_{1}V_{4}V_{5}}\otimes\operatorname*{id}\nolimits_{V_{2}}%
\otimes\operatorname*{id}\nolimits_{V_{3}}\right)  \circ\left(
\operatorname*{id}\nolimits_{V_{1}}\otimes\mathbf{c}_{V_{2}V_{5}V_{4}}%
\otimes\operatorname*{id}\nolimits_{V_{3}}\right)  \circ\left(
\operatorname*{id}\nolimits_{V_{1}}\otimes\operatorname*{id}\nolimits_{V_{2}%
}\otimes\mathbf{c}_{V_{3}V_{4}V_{5}}\right)  . \label{ccc}%
\end{align}
The ternary compatibility conditions for $\mathbf{c}_{V_{1}V_{2}V_{3}}$
(corresponding to (\ref{cv1})--(\ref{cv2})) are%
\begin{align}
\left(  \mathbf{c}_{V_{1}V_{4}V_{5}}\otimes\operatorname*{id}\nolimits_{V_{2}%
}\otimes\operatorname*{id}\nolimits_{V_{3}}\right)  \circ\left(
\operatorname*{id}\nolimits_{V_{1}}\otimes\mathbf{c}_{V_{2}V_{5}V_{4}}%
\otimes\operatorname*{id}\nolimits_{V_{3}}\right)  \circ\left(
\operatorname*{id}\nolimits_{V_{1}}\otimes\operatorname*{id}\nolimits_{V_{2}%
}\otimes\mathbf{c}_{V_{3}V_{4}V_{5}}\right)   &  =\mathbf{c}_{V_{1}\otimes
V_{2}\otimes V_{3},V_{4},V_{5}},\label{cc1}\\
\left(  \operatorname*{id}\nolimits_{V_{3}}\otimes\mathbf{c}_{V_{2}V_{5}V_{4}%
}\otimes\operatorname*{id}\nolimits_{V_{1}}\right)  \circ\left(
\operatorname*{id}\nolimits_{V_{3}}\otimes\operatorname*{id}\nolimits_{V_{2}%
}\otimes\mathbf{c}_{V_{1}V_{4}V_{5}}\right)  \circ\left(  \mathbf{c}%
_{V_{1}V_{2}V_{3}}\otimes\operatorname*{id}\nolimits_{V_{4}}\otimes
\operatorname*{id}\nolimits_{V_{5}}\right)   &  =\mathbf{c}_{V_{1},V_{2}%
,V_{3}\otimes V_{4}\otimes V_{5}}. \label{cc2}%
\end{align}

\end{example}

Now we follow the opposite (to the standard \cite{dri89}), but consistent
way: using the equations (\ref{ccc})--(\ref{cc2}) we find polyadic analogs of the
corresponding equations for the $R$-matrix and the quasitriangularity conditions
(\ref{drr1})--(\ref{drr2}), which will fix the comultiplication structure of a
polyadic bialgebra $\mathrm{B}^{\left(  n^{\prime},n\right)  }$.

\subsection{Polyadic almost co-commutativity}

We will see that the almost co-commutativity equations for the $R$-matrix are more
complicated in the polyadic case, because the main condition (\ref{md}) will
have a different form coming from $n$-ary group theory \cite{galmak1}.
Indeed, let $\mathrm{G}^{\left(  n\right)  }=\left\langle G\mid\mu^{\left(
n\right)  }\right\rangle $ be an $n$-ary group and $\mathrm{H}^{\prime
}=\left\langle H^{\prime}\mid\mu^{\left(  n\right)  }\right\rangle $,
$\mathrm{H}^{\prime\prime}=\left\langle H^{\prime\prime}\mid\mu^{\left(
n\right)  }\right\rangle $ are its $n$-ary subgroups. Recall \cite{galmak1}
that $\mathrm{H}^{\prime}$ and $\mathrm{H}^{\prime\prime}$ are
\textit{semiconjugated} in $\mathrm{G}^{\left(  n\right)  }$, if there exist
$g\in G$, such that $\mu^{\left(  n\right)  }\left[  g,h_{1}^{\prime}%
,\ldots,h_{n-1}^{\prime}\right]  =\mu^{\left(  n\right)  }\left[
h_{1}^{\prime\prime},\ldots,h_{n-1}^{\prime\prime},g\right]  $, $h_{i}%
^{\prime}\in H^{\prime}$, $h_{i}^{\prime\prime}\in H^{\prime\prime}$, and if
$g$ can be on any place, then $\mathrm{H}^{\prime}$ and $\mathrm{H}%
^{\prime\prime}$ are \textit{conjugated} in $\mathrm{G}^{\left(  n\right)  }$.
Based on this notion and on analogy with (\ref{me}), we can \textquotedblleft
polyadize\textquotedblright\ the almost co-commutativity condition (\ref{md}) in the
following way.

Let $\mathrm{B}^{\left(  n^{\prime},n\right)  }=\left\langle B\mid\mu^{\left(
n\right)  },\Delta^{\left(  n^{\prime}\right)  }\right\rangle $, be a polyadic
bialgebra (see \textbf{Definition} \ref{def-bialg}), and the opposite
comultiplication $\Delta_{cop}^{\left(  n^{\prime}\right)  }=\tau
_{op}^{\left(  n^{\prime}\right)  }\circ\Delta^{\left(  n^{\prime}\right)  }$,
where $\tau_{op}^{\left(  n^{\prime}\right)  }$ is the polyadic twist
(\ref{top}).

\begin{definition}
A polyadic bialgebra $\mathrm{B}^{\left(  n^{\prime},n\right)  }$ is called
\textit{polyadic sequenced almost co-commutative}, if there exist fixed
$\left(  n-1\right)  $ elements $\mathit{R}_{i}^{\left(  n^{\prime}\right)
}\in B^{\otimes n^{\prime}}$, $i=1,\ldots,n-1$, called a \textit{polyadic }%
$R$\textit{-matrix sequence}, such that%
\begin{align}
&  \mu^{\left(  n\right)  }\left[  \Delta_{cop}^{\left(  n^{\prime}\right)
}\left(  b\right)  ,\mathit{R}_{1}^{\left(  n^{\prime}\right)  }%
,\mathit{R}_{2}^{\left(  n^{\prime}\right)  },\ldots,\mathit{R}_{n-1}^{\left(
n^{\prime}\right)  }\right] \nonumber\\
&  =\mu^{\left(  n\right)  }\left[  \mathit{R}_{1}^{\left(  n^{\prime}\right)
},\Delta^{\left(  n^{\prime}\right)  }\left(  b\right)  ,\mathit{R}%
_{2}^{\left(  n^{\prime}\right)  },\ldots,\mathit{R}_{n-1}^{\left(  n^{\prime
}\right)  }\right] \nonumber\\
&  \vdots\nonumber\\
&  =\mu^{\left(  n\right)  }\left[  \mathit{R}_{1}^{\left(  n^{\prime}\right)
},\mathit{R}_{2}^{\left(  n^{\prime}\right)  },\ldots,\mathit{R}%
_{n-1}^{\left(  n^{\prime}\right)  },\Delta^{\left(  n^{\prime}\right)
}\left(  b\right)  \right]  ,\ \ \ \ \forall b\in B. \label{mndb}%
\end{align}

\end{definition}

\begin{definition}
A polyadic bialgebra $\mathrm{B}^{\left(  n^{\prime},n\right)  }$ is called
\textit{polyadic sequenced almost semico-commutative}, if only the first and
the last relations in (\ref{mndb}) hold %
\begin{equation}
\mu^{\left(  n\right)  }\left[  \Delta_{cop}^{\left(  n^{\prime}\right)
}\left(  b\right)  ,\mathit{R}_{1}^{\left(  n^{\prime}\right)  }%
,\mathit{R}_{2}^{\left(  n^{\prime}\right)  },\ldots,\mathit{R}_{n-1}^{\left(
n^{\prime}\right)  }\right]  =\mu^{\left(  n\right)  }\left[  \mathit{R}%
_{1}^{\left(  n^{\prime}\right)  },\mathit{R}_{2}^{\left(  n^{\prime}\right)
},\ldots,\mathit{R}_{n-1}^{\left(  n^{\prime}\right)  },\Delta^{\left(
n^{\prime}\right)  }\left(  b\right)  \right]  ,\ \ \ \ \forall b\in B.
\label{mnd1}%
\end{equation}

\end{definition}

\begin{remark}
Using $\left(  n-1\right)  $ polyadic $R$-matrices $\mathit{R}_{i}^{\left(
n^{\prime}\right)  }$ is the only way to build a polyadic analog for
the almost commutativity concept, since now there is no binary multiplication.
\end{remark}

The definition (\ref{mndb}) is too general and needs to consider $\left(
n-1\right)  $ different polyadic analogs of the $R$-matrix which might not be unique.
Therefore, in a similar way to the correspondence of the neutral sequence
(\ref{me}) and the polyadic unit (\ref{e}), we arrive at

\begin{definition}
A polyadic bialgebra $\mathrm{B}^{\left(  n^{\prime},n\right)  }=\left\langle
B\mid\mu^{\left(  n\right)  },\Delta^{\left(  n^{\prime}\right)
}\right\rangle $ is called \textit{polyadic almost (semi)co-commutative}, if
there exists \textsf{one} fixed element $\mathit{R}^{\left(  n^{\prime
}\right)  }\in B^{\otimes n^{\prime}}$ called a $n^{\prime}$-\textit{ary }%
$R$\textit{-matrix}, such that%
\begin{equation}
\mu^{\left(  n\right)  }\left[  \Delta_{cop}^{\left(  n^{\prime}\right)
}\left(  b\right)  ,\overset{n-1}{\overbrace{\mathit{R}^{\left(  n^{\prime
}\right)  },\ldots,\mathit{R}^{\left(  n^{\prime}\right)  }}}\right]
=\mu^{\left(  n\right)  }\left[  \overset{n-1}{\overbrace{\mathit{R}^{\left(
n^{\prime}\right)  },\ldots,\mathit{R}^{\left(  n^{\prime}\right)  }}}%
,\Delta^{\left(  n^{\prime}\right)  }\left(  b\right)  \right]
,\ \ \ \ \forall b\in B. \label{mrr}%
\end{equation}

\end{definition}

In components the $n^{\prime}$-ary $R$-matrix $\mathit{R}^{\left(
n^{\prime}\right)  }$ is%
\begin{equation}
\mathit{R}^{\left(  n^{\prime}\right)  }=\sum_{\alpha}\mathsf{r}_{\alpha
}^{\left(  1\right)  }\otimes\ldots\otimes\mathsf{r}_{\alpha}^{\left(
n^{\prime}\right)  },\ \ \ \ \mathsf{r}_{\alpha}^{\left(  i\right)  }\in B.
\label{rn}%
\end{equation}

\begin{remark}
\label{rem-tau} Polyadic almost co-commutativity (\ref{mrr}) can be
expressed in component form, as in the binary case (\ref{mbr}), only if we
know concretely the polyadic twist $\tau_{op}^{\left(  n^{\prime}\right)  }%
\in\mathrm{S}_{n^{\prime}}$ (where $\mathrm{S}_{n^{\prime}}$ is the symmetry
permutation group ON $n^{\prime}$ elements), which is not unique for
arbitrary $n^{\prime}>2$.
\end{remark}

\begin{example}
For $\mathrm{B}^{\left(  3,3\right)  }$ the ternary almost
(semi)co-commutativity (\ref{mrr}) is given by%
\begin{equation}
\mu^{\left(  3\right)  }\left[  \Delta_{cop}^{\left(  3\right)  }\left(
b\right)  ,\mathit{R}^{\left(  3\right)  },\mathit{R}^{\left(  3\right)
}\right]  =\mu^{\left(  3\right)  }\left[  \mathit{R}^{\left(  3\right)
},\mathit{R}^{\left(  3\right)  },\Delta^{\left(  3\right)  }\left(  b\right)
\right]  ,\ \ \ \ \forall b\in B, \label{md3}%
\end{equation}
which with $\tau_{op}^{\left(  3\right)  }=\binom{123}{321}$ becomes, in components,%
\begin{align}
&  \sum_{\left[  b\right]  }\sum_{\alpha,\beta}\mu^{\left(  3\right)  }\left[
b_{\left[  3\right]  },\mathsf{r}_{\alpha}^{\left(  1\right)  },\mathsf{r}%
_{\beta}^{\left(  1\right)  }\right]  \otimes\mu^{\left(  3\right)  }\left[
b_{\left[  2\right]  },\mathsf{r}_{\alpha}^{\left(  2\right)  },\mathsf{r}%
_{\beta}^{\left(  2\right)  }\right]  \otimes\mu^{\left(  3\right)  }\left[
b_{\left[  1\right]  },\mathsf{r}_{\alpha}^{\left(  3\right)  },\mathsf{r}%
_{\beta}^{\left(  3\right)  }\right] \nonumber\\
&  =\sum_{\left[  b\right]  ^{\prime}}\sum_{\alpha^{\prime}\beta^{\prime}}%
\mu^{\left(  3\right)  }\left[  \mathsf{r}_{\alpha^{\prime}}^{\left(
1\right)  },\mathsf{r}_{\beta^{\prime}}^{\left(  1\right)  },b_{\left[
1\right]  ^{\prime}}\right]  \otimes\mu^{\left(  3\right)  }\left[
\mathsf{r}_{\alpha^{\prime}}^{\left(  2\right)  },\mathsf{r}_{\beta^{\prime}%
}^{\left(  2\right)  },b_{\left[  2\right]  ^{\prime}}\right]  \otimes
\mu^{\left(  3\right)  }\left[  \mathsf{r}_{\alpha^{\prime}}^{\left(
3\right)  },\mathsf{r}_{\beta^{\prime}}^{\left(  3\right)  },b_{\left[
3\right]  ^{\prime}}\right]  . \label{m3}%
\end{align}

\end{example}

\begin{example}
In the exotic mixed case $\mathrm{B}^{\left(  4,3\right)  }$ where the polyadic
twist \textquotedblleft without fixed points\textquotedblright\ (\ref{top}) is
$\tau_{op}^{\left(  4\right)  }=\binom{1234}{3142}$, the polyadic almost
co-commutativity (\ref{mrr}) becomes%
\begin{align}
&  \sum_{\left[  b\right]  }\sum_{\alpha,\beta}\mu^{\left(  3\right)  }\left[
b_{\left[  3\right]  },\mathsf{r}_{\alpha}^{\left(  1\right)  },\mathsf{r}%
_{\beta}^{\left(  1\right)  }\right]  \otimes\mu^{\left(  3\right)  }\left[
b_{\left[  1\right]  },\mathsf{r}_{\alpha}^{\left(  2\right)  },\mathsf{r}%
_{\beta}^{\left(  2\right)  }\right]  \otimes\mu^{\left(  3\right)  }\left[
b_{\left[  4\right]  },\mathsf{r}_{\alpha}^{\left(  3\right)  },\mathsf{r}%
_{\beta}^{\left(  3\right)  }\right]  \otimes\mu^{\left(  3\right)  }\left[
b_{\left[  2\right]  },\mathsf{r}_{\alpha}^{\left(  4\right)  },\mathsf{r}%
_{\beta}^{\left(  4\right)  }\right]  =\nonumber\\
&  \sum_{\left[  b\right]  ^{\prime}}\sum_{\alpha^{\prime}\beta^{\prime}}%
\mu^{\left(  3\right)  }\left[  \mathsf{r}_{\alpha^{\prime}}^{\left(
1\right)  },\mathsf{r}_{\beta^{\prime}}^{\left(  1\right)  },b_{\left[
1\right]  ^{\prime}}\right]  \otimes\mu^{\left(  3\right)  }\left[
\mathsf{r}_{\alpha^{\prime}}^{\left(  2\right)  },\mathsf{r}_{\beta^{\prime}%
}^{\left(  2\right)  },b_{\left[  2\right]  ^{\prime}}\right]  \otimes
\mu^{\left(  3\right)  }\left[  \mathsf{r}_{\alpha^{\prime}}^{\left(
3\right)  },\mathsf{r}_{\beta^{\prime}}^{\left(  3\right)  },b_{\left[
3\right]  ^{\prime}}\right]  \otimes\mu^{\left(  3\right)  }\left[
\mathsf{r}_{\alpha^{\prime}}^{\left(  4\right)  },\mathsf{r}_{\beta^{\prime}%
}^{\left(  4\right)  },b_{\left[  4\right]  ^{\prime}}\right]  . \label{m4}%
\end{align}

\end{example}

\subsection{Equations for the $n^{\prime}$-ary $R$-matrix}

Here we consider the most consistent way (from a categorical viewpoint) to
derive equations for the polyadic $R$-matrix, in other words, through using the
braided equation (\ref{cii}) (and $n^{\prime}$-ary braided equation
(\ref{nc})) with the concrete choice of the braiding $\mathbf{c}%
_{V^{n^{\prime}}}$.

Suppose that the $n^{\prime}$-ary braiding $\mathbf{c}_{V^{n^{\prime}}}$ is
defined still by a $1$-place action $\rho^{\left(  1\right)  }$, as in the
binary case (\ref{cr}). At first glance, we could define the braiding (similar
to (\ref{cr}))%
\begin{align}
&  \mathbf{c}_{V_{1}\ldots V_{n^{\prime}}}\circ\left(  v_{1}\otimes
\ldots\otimes v_{n^{\prime}}\right)  =\tau_{V_{1}\ldots V_{n^{\prime}}}%
\circ\mathit{R}^{\left(  n^{\prime}\right)  }\circ\left(  v_{1}\otimes
\ldots\otimes v_{n^{\prime}}\right) \nonumber\\
&  =\tau_{V_{1}\ldots V_{n^{\prime}}}\circ\left(  \sum_{\alpha}\rho^{\left(
1\right)  }\left(  \mathsf{r}_{\alpha}^{\left(  1\right)  }\mid v_{1}\right)
\otimes\ldots\otimes\rho^{\left(  1\right)  }\left(  \mathsf{r}_{\alpha
}^{\left(  n^{\prime}\right)  }\mid v_{n^{\prime}}\right)  \right)
,\ \ v_{i}\in V_{i},\ \ \mathsf{r}_{\alpha}^{\left(  i\right)  }\in B,
\label{cvr}%
\end{align}
where $\rho^{\left(  1\right)  }:B\otimes V_{i}\rightarrow V_{i}$ is the
$1$-place action (see (\ref{r})). We recall that only the $n$-ary composition
of $1$-place actions ($n$ is the arity of multiplication $\mu^{\left(
n\right)  }$) is defined here (see \cite{dup2017})%
\begin{equation}
\overset{n}{\overbrace{\rho^{\left(  1\right)  }\left(  b_{1}\mid\rho^{\left(
1\right)  }\left(  b_{2}\ldots\rho^{\left(  1\right)  }\left(  b_{n}\mid
v\right)  \right)  \right)  }}=\rho^{\left(  1\right)  }\left(  \mu^{\left(
n\right)  }\left[  b_{1},\ldots b_{n}\right]  \mid v\right)  ,\ \ b_{i}\in
B,\ \ v\in V. \label{rbr}%
\end{equation}

As in the binary case (\ref{r12})--(\ref{r23}), we need the \textquotedblleft
extended\textquotedblright\ polyadic $R$-matrix.

\begin{remark}
\label{rem-unit}The standard definition of the \textquotedblleft
extended\textquotedblright\ $n^{\prime}$-ary $R$-matrix can be possible, if
the algebra $\left\langle B\mid\mu^{\left(  n\right)  }\right\rangle $
contains \textsf{one} polyadic unit (element) $e_{B}$, because in the polyadic
case there are new intriguing possibilities (which did not exist in the binary case)
of having several units, or even where all elements are units (see the discussion after
(\ref{me}) and \cite{dup2018a}).
\end{remark}

\begin{definition}
The \textquotedblleft extended\textquotedblright\ form of the $n^{\prime}$-ary
$R$-matrix is defined by $\mathcal{R}_{i_{1}\ldots i_{n^{\prime}}}^{\left(
2n^{\prime}-1\right)  }\in B^{\otimes\left(  2n^{\prime}-1\right)  }$, such
that%
\begin{equation}
\mathcal{R}_{i_{1}\ldots i_{n^{\prime}}}^{\left(  2n^{\prime}-1\right)  }%
=\sum_{\alpha}e_{B}\otimes\ldots\otimes\mathsf{r}_{\alpha}^{\left(
i_{1}\right)  }\otimes\ldots\otimes\mathsf{r}_{\alpha}^{\left(  i_{n^{\prime}%
}\right)  }\otimes\ldots\otimes e_{B},\ \ \ i_{1},\ldots,i_{n^{\prime}}%
\in\left\{  1,\ldots2n^{\prime}-1\right\}  \label{rext}%
\end{equation}
where $\mathsf{r}_{\alpha}^{\left(  i_{k}\right)  }$ are on the $i_{k}$-place.
\end{definition}

In this way we can express in terms of the \textquotedblleft
extended\textquotedblright\ $n^{\prime}$-ary $R$-matrix (\ref{rext}) the
$n^{\prime}$-ary braided equation (\ref{nc}), in full analogy with the binary
case (\ref{rrr}).

\begin{example}
For the ternary case%
\begin{align}
&  \mathbf{c}_{V_{1}V_{2}V_{3}}\circ\left(  v_{1}\otimes v_{2}\otimes
v_{3}\right)  =\tau_{V_{1}V_{2}V_{3}}\circ\mathit{R}^{\left(  3\right)  }%
\circ\left(  v_{1}\otimes v_{2}\otimes v_{3}\right) \nonumber\\
&  =\sum_{\alpha}\rho^{\left(  1\right)  }\left(  \mathsf{r}_{\alpha}^{\left(
1\right)  }\mid v_{3}\right)  \otimes\rho^{\left(  1\right)  }\left(
\mathsf{r}_{\alpha}^{\left(  2\right)  }\mid v_{2}\right)  \otimes
\rho^{\left(  1\right)  }\left(  \mathsf{r}_{\alpha}^{\left(  3\right)  }\mid
v_{1}\right)  ,\ \ v_{i}\in V_{i},\ \ \mathsf{r}_{\alpha}^{\left(  i\right)
}\in B, \label{cv3}%
\end{align}
and we define $\mathcal{R}_{i_{1}i_{2}i_{3}}^{\left(  5\right)  }$ by
(\ref{rext}), $i_{1},i_{2},i_{3}\in\left\{  1,\ldots,5\right\}  $,
$\tau_{V_{1}V_{2}V_{3}}=\binom{123}{321}$, and consider the ternary braid
equation (\ref{ccc}). Using (\ref{rbr}) we obtain (informally)%
\begin{equation}
\mathcal{R}_{123}^{\left(  5\right)  }\mathcal{R}_{145}^{\left(  5\right)
}\mathcal{R}_{254}^{\left(  5\right)  }\mathcal{R}_{345}^{\left(  5\right)
}=\mathcal{R}_{345}^{\left(  5\right)  }\mathcal{R}_{254}^{\left(  5\right)
}\mathcal{R}_{145}^{\left(  5\right)  }\mathcal{R}_{123}^{\left(  5\right)  }.
\label{r5}%
\end{equation}

\end{example}

\begin{remark}
Unfortunately, a \textquotedblleft linear\textquotedblright\
$\mathit{R}^{\left(  n^{\prime}\right)  }$ $n^{\prime}$-ary braiding
$\mathbf{c}_{V_{1}\ldots V_{n^{\prime}}}$ (as in (\ref{cvr}) and (\ref{cv3}))
is not consistent with the polyadic analog of the quasitriangularity equations
(\ref{cv1})--(\ref{cv2}), because the polyadic almost co-commutativity
(\ref{mrr}) contains $\left(  n-1\right)  $ copies of $n^{\prime}$-ary
$R$-matrix $\mathit{R}^{\left(  n^{\prime}\right)  }$.
\end{remark}

Therefore, in order to agree with (\ref{mrr}), instead of (\ref{cvr}), we have

\begin{definition}
The polyadic braiding $\mathbf{c}_{V_{1}\ldots V_{n^{\prime}}}$ is defined by%
\begin{align}
&  \mathbf{c}_{V_{1}\ldots V_{n^{\prime}}}\circ\left(  v_{1}\otimes
\ldots\otimes v_{n^{\prime}}\right)  =\tau_{V_{1}\ldots V_{n^{\prime}}}%
\circ\rho^{\left(  n-1\right)  }\left(  \overset{n-1}{\overbrace
{\mathit{R}^{\left(  n^{\prime}\right)  },\ldots,\mathit{R}^{\left(
n^{\prime}\right)  }}}\mid\left(  v_{1}\otimes\ldots\otimes v_{n^{\prime}%
}\right)  \right) \nonumber\\
&  =\tau_{V_{1}\ldots V_{n^{\prime}}}\circ\left(  \sum_{\alpha_{1}%
,\ldots\alpha_{n-1}}\rho^{\left(  n-1\right)  }\left(  \mathsf{r}_{\alpha_{1}%
}^{\left(  1\right)  },\ldots\mathsf{r}_{\alpha_{n-1}}^{\left(  1\right)
}\mid v_{1}\right)  \otimes\ldots\otimes\rho^{\left(  n-1\right)  }\left(
\mathsf{r}_{\alpha_{1}}^{\left(  n^{\prime}\right)  },\ldots\mathsf{r}%
_{\alpha_{n-1}}^{\left(  n^{\prime}\right)  }\mid v_{n^{\prime}}\right)
\right)  ,\nonumber\\
&  v_{i}\in V_{i},\ \ \mathsf{r}_{\alpha}^{\left(  i\right)  }\in B,
\label{cvv}%
\end{align}
where $\rho^{\left(  n-1\right)  }:B^{n-1}\otimes V\rightarrow V$ the $\left(
n-1\right)  $-place action (see (\ref{r})).
\end{definition}

\begin{remark}
The twist of the modules $\tau_{V_{1}\ldots V_{n^{\prime}}}$ should be
compatible with the polyadic twist $\tau_{op}^{\left(  n^{\prime}\right)  }$
in (\ref{mndb}). In the binary case they are both the same flip $\binom
{12}{21}$, but in the $n^{\prime}$-ary case they can be different.
\end{remark}

\begin{example}
Consider the ternary braided equation (\ref{ccc}), but now for the braiding
$\mathbf{c}_{V_{1}V_{2}V_{3}}$, instead of (\ref{cv3}), where we have%
\begin{align}
&  \mathbf{c}_{V_{1}V_{2}V_{3}}\circ\left(  v_{1}\otimes v_{2}\otimes
v_{3}\right)  =\tau_{V_{1}V_{2}V_{3}}\circ\rho^{\left(  2\right)  }\left(
\mathit{R}^{\left(  3\right)  },\mathit{R}^{\left(  3\right)  }\mid\left(
v_{1}\otimes v_{2}\otimes v_{3}\right)  \right) \nonumber\\
&  =\sum_{\alpha,\beta}\rho^{\left(  2\right)  }\left(  \mathsf{r}_{\alpha
}^{\left(  3\right)  },\mathsf{r}_{\beta}^{\left(  3\right)  }\mid
v_{3}\right)  \otimes\rho^{\left(  2\right)  }\left(  \mathsf{r}_{\alpha
}^{\left(  2\right)  },\mathsf{r}_{\beta}^{\left(  2\right)  }\mid
v_{2}\right)  \otimes\rho^{\left(  2\right)  }\left(  \mathsf{r}_{\alpha
}^{\left(  1\right)  },\mathsf{r}_{\beta}^{\left(  1\right)  }\mid
v_{1}\right)  ,\ \ v_{i}\in V_{i},\ \ \mathsf{r}_{\alpha,\beta}^{\left(
i\right)  }\in B, \label{cvv2}%
\end{align}
where $\rho^{\left(  2\right)  }:B\otimes B\otimes V\rightarrow V$ is a
$2$-place action (\ref{r}). In this way (\ref{cvv2}) is consistent with
(\ref{md3}). In each place of the $2$-place action $\rho^{\left(
2\right)  }$ we then obtain the relation (\ref{r5}).
\end{example}

\subsection{Polyadic triangularity}

A polyadic analog of triangularity \cite{dri0} can be defined, if we
rewrite (\ref{drr2}) as%
\begin{equation}
\left(  \operatorname*{id}\nolimits_{B}\otimes\Delta\right)  \left(
\mathit{R}\right)  =\mu\left[  \mathcal{R}_{13},\mathcal{R}_{12}\right]
\equiv\sum_{\alpha,\beta}\mu\circ\tau_{op}\left[  \mathsf{r}_{\alpha}^{\left(
1\right)  },\mathsf{r}_{\beta}^{\left(  1\right)  }\right]  \otimes
\mathsf{r}_{\alpha}^{\left(  2\right)  }\otimes\mathsf{r}_{\beta}^{\left(
2\right)  }, \label{idr}%
\end{equation}
where $\tau_{op}$ is the binary twist. Instead of the $R$-matrix formulation
(the left equality in (\ref{idr})), we use the component approach by
\cite{radford}, and propose the following

\begin{definition}
A polyadic almost co-commutative bialgebra $\mathrm{B}^{\left(  n^{\prime
},n\right)  }=\left\langle B\mid\mu^{\left(  n\right)  },\Delta^{\left(
n^{\prime}\right)  }\right\rangle $ with the polyadic $R$-matrix
$\mathit{R}^{\left(  n^{\prime}\right)  }=\sum_{\alpha}\mathsf{r}_{\alpha
}^{\left(  1\right)  }\otimes\ldots\otimes\mathsf{r}_{\alpha}^{\left(
n^{\prime}\right)  },\ \ \ \ \mathsf{r}_{\alpha}^{\left(  i\right)  }\in B$ is
called \textit{quasipolyangular}, if the following $n^{\prime}$ relations hold%
\begin{align}
&  \sum_{\alpha}\Delta^{\left(  n^{\prime}\right)  }\left(  \mathsf{r}%
_{\alpha}^{\left(  1\right)  }\right)  \otimes\mathsf{r}_{\alpha}^{\left(
2\right)  }\otimes\ldots\otimes\mathsf{r}_{\alpha}^{\left(  n^{\prime}\right)
}=\sum_{\alpha_{1},\ldots\alpha_{n^{\prime}}}\mathsf{r}_{\alpha_{1}}^{\left(
1\right)  }\otimes\mathsf{r}_{\alpha_{2}}^{\left(  1\right)  }\otimes
\ldots\otimes\mathsf{r}_{\alpha_{n^{\prime}}}^{\left(  1\right)  }\nonumber\\
&  \otimes\left(  \mu^{\left(  n\right)  }\right)  ^{\circ\ell}\left[
\mathsf{r}_{\alpha_{1}}^{\left(  2\right)  }\otimes\mathsf{r}_{\alpha_{2}%
}^{\left(  2\right)  }\otimes\ldots\otimes\mathsf{r}_{\alpha_{n^{\prime}}%
}^{\left(  2\right)  }\right]  \otimes\ldots\otimes\left(  \mu^{\left(
n\right)  }\right)  ^{\circ\ell}\left[  \mathsf{r}_{\alpha_{1}}^{\left(
n^{\prime}\right)  }\otimes\mathsf{r}_{\alpha_{2}}^{\left(  n^{\prime}\right)
}\otimes\ldots\otimes\mathsf{r}_{\alpha_{n^{\prime}}}^{\left(  n^{\prime
}\right)  }\right]  , \label{qp1}%
\end{align}%
\begin{align}
&  \sum_{\alpha}\mathsf{r}_{\alpha}^{\left(  1\right)  }\otimes\Delta^{\left(
n^{\prime}\right)  }\left(  \mathsf{r}_{\alpha}^{\left(  2\right)  }\right)
\otimes\ldots\otimes\mathsf{r}_{\alpha}^{\left(  n^{\prime}\right)  }%
=\sum_{\alpha_{1},\ldots\alpha_{n^{\prime}}}\left(  \mu^{\left(  n\right)
}\right)  ^{\circ\ell}\circ\tau_{op}^{\left(  n^{\prime}\right)  }\left[
\mathsf{r}_{\alpha_{1}}^{\left(  1\right)  }\otimes\mathsf{r}_{\alpha_{2}%
}^{\left(  1\right)  }\otimes\ldots\otimes\mathsf{r}_{\alpha_{n^{\prime}}%
}^{\left(  1\right)  }\right] \nonumber\\
&  \otimes\mathsf{r}_{\alpha_{1}}^{\left(  2\right)  }\otimes\mathsf{r}%
_{\alpha_{2}}^{\left(  2\right)  }\otimes\ldots\otimes\mathsf{r}%
_{\alpha_{n^{\prime}}}^{\left(  2\right)  }\otimes\ldots\otimes\left(
\mu^{\left(  n\right)  }\right)  ^{\circ\ell}\left[  \mathsf{r}_{\alpha_{1}%
}^{\left(  n^{\prime}\right)  }\otimes\mathsf{r}_{\alpha_{2}}^{\left(
n^{\prime}\right)  }\otimes\ldots\otimes\mathsf{r}_{\alpha_{n^{\prime}}%
}^{\left(  n^{\prime}\right)  }\right]  , \label{qp2}%
\end{align}%
\[
\vdots
\]%
\begin{align}
&  \sum_{\alpha}\mathsf{r}_{\alpha}^{\left(  1\right)  }\otimes\mathsf{r}%
_{\alpha}^{\left(  2\right)  }\otimes\ldots\otimes\Delta^{\left(  n^{\prime
}\right)  }\left(  \mathsf{r}_{\alpha}^{\left(  n^{\prime}\right)  }\right)
=\sum_{\alpha_{1},\ldots\alpha_{n^{\prime}}}\left(  \mu^{\left(  n\right)
}\right)  ^{\circ\ell}\circ\tau_{op}^{\left(  n^{\prime}\right)  }\left[
\mathsf{r}_{\alpha_{1}}^{\left(  1\right)  }\otimes\mathsf{r}_{\alpha_{2}%
}^{\left(  1\right)  }\otimes\ldots\otimes\mathsf{r}_{\alpha_{n^{\prime}}%
}^{\left(  1\right)  }\right] \nonumber\\
&  \otimes\left(  \mu^{\left(  n\right)  }\right)  ^{\circ\ell}\circ\tau
_{op}^{\left(  n^{\prime}\right)  }\left[  \mathsf{r}_{\alpha_{1}}^{\left(
2\right)  }\otimes\mathsf{r}_{\alpha_{2}}^{\left(  2\right)  }\otimes
\ldots\otimes\mathsf{r}_{\alpha_{n^{\prime}}}^{\left(  2\right)  }\right]
\otimes\ldots\otimes\mathsf{r}_{\alpha_{1}}^{\left(  n^{\prime}\right)
}\otimes\mathsf{r}_{\alpha_{2}}^{\left(  n^{\prime}\right)  }\otimes
\ldots\otimes\mathsf{r}_{\alpha_{n^{\prime}}}^{\left(  n^{\prime}\right)  },
\label{qp3}%
\end{align}
where $\tau_{op}^{\left(  n^{\prime}\right)  }$ is the polyadic twist map
(\ref{top}). The arity shape of a quasipolyangular $\mathrm{B}^{\left(
n^{\prime},n\right)  }$ is fixed by
\begin{equation}
n^{\prime}=\ell\left(  n-1\right)  +1,\ \ \ \ell\in\mathbb{N}. \label{nln}%
\end{equation}

\end{definition}

\begin{remark}
\label{rem-qr}As opposed to the binary case (\ref{drr1})--(\ref{drr2}),
the right hand sides here can be expressed in terms of the extended $R$-matrix
in the first equation (\ref{qp1}) and the last one (\ref{qp3}) only, because
in the intermediate equations the sequences of $R$-matrix elements are
permuted. For instance, it is clear that the binary product $\sum
_{\alpha,\beta}\left(  \mathsf{r}_{\beta}^{\left(  1\right)  }\otimes
\mathsf{r}_{\alpha}^{\left(  2\right)  }\otimes e_{B}\right)  \cdot\left(
\mathsf{r}_{\alpha}^{\left(  1\right)  }\otimes\mathsf{r}_{\beta}^{\left(
2\right)  }\otimes e_{B}\right)  =\sum_{\alpha,\beta}\left(  \mathsf{r}%
_{\beta}^{\left(  1\right)  }\cdot\mathsf{r}_{\alpha}^{\left(  1\right)
}\otimes\mathsf{r}_{\alpha}^{\left(  2\right)  }\cdot\mathsf{r}_{\beta
}^{\left(  2\right)  }\otimes e_{B}\right)  $ cannot be expressed in terms of
the extended binary $R$-matrix (\ref{r12}).
\end{remark}

\subsection{Almost co-medial polyadic bialgebras}

The previous considerations showed that co-commutativity and almost
co-commutativity in the polyadic case are not unique and do not describe the
bialgebras to the fullest extent. This happens because mediality is a more
general and consequent property of polyadic algebraic structures, while
commutativity can be treated as a particular case of it (see \textbf{Subsection}
\ref{subsec-med} and (\ref{ti})). Therefore, we propose here to deform
\textsf{co-mediality} (rather than co-commutativity as in \cite{dri0,dri2,dri89}).

Let $\mathrm{B}^{\left(  n^{\prime},n\right)  }=\left\langle B\mid\mu^{\left(
n\right)  },\Delta^{\left(  n^{\prime}\right)  }\right\rangle $, be a polyadic
bialgebra (see \textbf{Definition} \ref{def-bialg}). Now we deform the
co-mediality condition (\ref{com}) in a similar way to the polyadic $R$-matrix
(\ref{mrr}).

\begin{definition}
A polyadic bialgebra $\mathrm{B}^{\left(  n^{\prime},n\right)  }$ is called
\textit{polyadic sequenced almost co-medial}, if there exist $\left(
n^{\prime}-1\right)  $ fixed elements $\mathit{M}_{i}^{\left(  n^{\prime
2}\right)  }\in B^{\otimes n^{\prime2}}$, $i=1,\ldots,n-1$, called a
\textit{polyadic }$M$-\textit{matrix sequence}, such that (see (\ref{com}) and
(\ref{mrr}))%
\begin{align}
&  \mu^{\left(  n\right)  }\left[  \tau_{medial}^{\left(  n^{\prime}%
,n^{\prime}\right)  }\circ\left(  \left(  \Delta^{\left(  n^{\prime}\right)
}\right)  ^{\otimes n^{\prime}}\right)  \circ\Delta^{\left(  n^{\prime
}\right)  }\left(  b\right)  ,\mathit{M}_{1}^{\left(  n^{\prime2}\right)
},\mathit{M}_{2}^{\left(  n^{\prime2}\right)  },\ldots,\mathit{M}%
_{n-1}^{\left(  n^{\prime2}\right)  }\right] \nonumber\\
&  =\mu^{\left(  n\right)  }\left[  \mathit{M}_{1}^{\left(  n^{\prime
2}\right)  },\left(  \left(  \Delta^{\left(  n^{\prime}\right)  }\right)
^{\otimes n^{\prime}}\right)  \circ\Delta^{\left(  n^{\prime}\right)  }\left(
b\right)  ,\mathit{M}_{2}^{\left(  n^{\prime2}\right)  },\ldots,\mathit{M}%
_{n-1}^{\left(  n^{\prime2}\right)  }\right] \nonumber\\
&  =\mu^{\left(  n\right)  }\left[  \mathit{M}_{1}^{\left(  n^{\prime
2}\right)  },\mathit{M}_{2}^{\left(  n^{\prime2}\right)  },\ldots
,\mathit{M}_{n-1}^{\left(  n^{\prime2}\right)  },\left(  \left(
\Delta^{\left(  n^{\prime}\right)  }\right)  ^{\otimes n^{\prime}}\right)
\circ\Delta^{\left(  n^{\prime}\right)  }\left(  b\right)  \right]
,\ \ \ \ \forall b\in B, \label{mnt}%
\end{align}
where $\mathbf{\tau}_{medial}^{\left(  n^{\prime},n^{\prime}\right)  }$ is the
polyadic medial map (\ref{an}).
\end{definition}

\begin{definition}
A polyadic bialgebra $\mathrm{B}^{\left(  n^{\prime},n\right)  }$ is called
\textit{polyadic sequenced almost (semi)co-medial}, if only the first and the
last relations in (\ref{mnt}) hold%
\begin{align}
&  \mu^{\left(  n\right)  }\left[  \tau_{medial}^{\left(  n^{\prime}%
,n^{\prime}\right)  }\circ\left(  \left(  \Delta^{\left(  n^{\prime}\right)
}\right)  ^{\otimes n^{\prime}}\right)  \circ\Delta^{\left(  n^{\prime
}\right)  }\left(  b\right)  ,\mathit{M}_{1}^{\left(  n^{\prime2}\right)
},\ldots,\mathit{M}_{n-1}^{\left(  n^{\prime2}\right)  }\right] \nonumber\\
&  =\mu^{\left(  n\right)  }\left[  \mathit{M}_{1}^{\left(  n^{\prime
2}\right)  },\ldots,\mathit{M}_{n-1}^{\left(  n^{\prime2}\right)  },\left(
\left(  \Delta^{\left(  n^{\prime}\right)  }\right)  ^{\otimes n^{\prime}%
}\right)  \circ\Delta^{\left(  n^{\prime}\right)  }\left(  b\right)  \right]
,\ \ \ \ \forall b\in B,
\end{align}

\end{definition}

If all the elements in the sequence (similar to the neutral sequence for
$n$-ary groups (\ref{me})) are the same $\mathit{M}_{1}^{\left(  n^{\prime
2}\right)  }=\mathit{M}_{2}^{\left(  n^{\prime2}\right)  }=\ldots
=\mathit{M}_{n-1}^{\left(  n^{\prime2}\right)  }\equiv\mathit{M}^{\left(
n^{\prime2}\right)  }$, we have

\begin{definition}
\label{def-medbi}A polyadic bialgebra $\mathrm{B}^{\left(  n^{\prime
},n\right)  }$ is called \textit{polyadic almost (semi)co-medial}, if there
exist \textsf{one} fixed element $\mathit{M}^{\left(  n^{\prime2}\right)  }\in
B^{\otimes n^{\prime2}}$ called a \textit{polyadic }$M$\textit{-matrix}, such
that (see (\ref{com}) and (\ref{mrr}))%
\begin{align}
&  \mu^{\left(  n\right)  }\left[  \tau_{medial}^{\left(  n^{\prime}%
,n^{\prime}\right)  }\circ\left(  \left(  \Delta^{\left(  n^{\prime}\right)
}\right)  ^{\otimes n^{\prime}}\right)  \circ\Delta^{\left(  n^{\prime
}\right)  }\left(  b\right)  ,\overset{n-1}{\overbrace{\mathit{M}^{\left(
n^{\prime2}\right)  },\ldots,\mathit{M}^{\left(  n^{\prime2}\right)  }}%
}\right] \nonumber\\
&  =\mu^{\left(  n\right)  }\left[  \overset{n-1}{\overbrace{\mathit{M}%
^{\left(  n^{\prime2}\right)  },\ldots,\mathit{M}^{\left(  n^{\prime2}\right)
}}},\left(  \left(  \Delta^{\left(  n^{\prime}\right)  }\right)  ^{\otimes
n^{\prime}}\right)  \circ\Delta^{\left(  n^{\prime}\right)  }\left(  b\right)
\right]  ,\ \ \ \ \forall b\in B. \label{mnt1}%
\end{align}

\end{definition}

\begin{remark}
\label{rem-uniq}The main advantage of the polyadic almost co-mediality property over
polyadic almost co-commutativity is the \textsf{uniqueness} of the medial
map $\mathbf{\tau}_{medial}^{\left(  n,n\right)  }$ and \textsf{nonuniqueness}
of the polyadic twist map $\mathbf{\tau}_{op}^{\left(  \ell_{\tau}\right)  }$
(\ref{top}).
\end{remark}

The polyadic $M$-matrix $\mathit{M}^{\left(  n^{\prime}\right)  }$ in
components is given by%
\begin{equation}
\mathit{M}^{\left(  n^{\prime2}\right)  }=\sum_{\alpha}\mathsf{m}_{\alpha
}^{\left(  1\right)  }\otimes\ldots\otimes\mathsf{m}_{\alpha}^{\left(
n^{\prime2}\right)  },\ \ \ \ \mathsf{m}_{\alpha}^{\left(  i\right)  }\in
B,\ \ i=1,\ldots,n^{\prime2}. \label{pmm}%
\end{equation}

\begin{example}
In the binary case for $\mathrm{B}^{\left(  2,2\right)  }=\left\langle
B\mid\mu,\Delta\right\rangle $ we have an almost co-mediality (\ref{mnt1}) as%
\begin{equation}
\mu\left[  \tau_{medial}\circ\left(  \Delta\otimes\Delta\right)  \circ
\Delta\left(  b\right)  ,\mathit{M}^{\left(  4\right)  }\right]  =\mu\left[
\mathit{M}^{\left(  4\right)  },\left(  \Delta\otimes\Delta\right)
\circ\Delta\left(  b\right)  \right]  ,\ \ \ \ \forall b\in B. \label{mnt2}%
\end{equation}
which gives, in components (cf. for $R$-matrix (\ref{mbr}))%
\begin{align}
&  \sum_{\left[  b\right]  _{\left[  b\right]  }}\sum_{\alpha}\mu\left[
b_{\left[  1\right]  _{\left[  1\right]  }},\mathsf{m}_{\alpha}^{\left(
1\right)  }\right]  \otimes\mu\left[  b_{\left[  2\right]  _{\left[  1\right]
}},\mathsf{m}_{\alpha}^{\left(  2\right)  }\right]  \otimes\mu\left[
b_{\left[  1\right]  _{\left[  2\right]  }},\mathsf{m}_{\alpha}^{\left(
3\right)  }\right]  \otimes\mu\left[  b_{\left[  2\right]  _{\left[  2\right]
}},\mathsf{m}_{\alpha}^{\left(  4\right)  }\right] \nonumber\\
&  =\sum_{\left[  b\right]  _{\left[  b\right]  }^{\prime}}\sum_{\alpha
^{\prime}}\mu\left[  \mathsf{m}_{\alpha^{\prime}}^{\left(  1\right)
},b_{\left[  1\right]  _{\left[  1\right]  }^{\prime}}\right]  \otimes
\mu\left[  \mathsf{m}_{\alpha^{\prime}}^{\left(  2\right)  },b_{\left[
1\right]  _{\left[  2\right]  }^{\prime}}\right]  \otimes\mu\left[
\mathsf{m}_{\alpha^{\prime}}^{\left(  3\right)  },b_{\left[  2\right]
_{\left[  1\right]  }^{\prime}}\right]  \otimes\mu\left[  \mathsf{m}%
_{\alpha^{\prime}}^{\left(  4\right)  },b_{\left[  2\right]  _{\left[
2\right]  }^{\prime}}\right]  . \label{mb2}%
\end{align}

\end{example}

Let us clarify the connection between the almost co-commutativity and
almost co-mediality properties.

\begin{theorem}
If $\mathrm{B}^{\left(  n^{\prime},n\right)  }$ is polyadic almost
(semi)co-commutative with the polyadic twist map $\tau_{op}^{\left(
n^{\prime}\right)  }$ \emph{(\ref{top})} and the $n^{\prime}$-ary $R$-matrix
$\mathit{R}^{\left(  n^{\prime}\right)  }$ \emph{(\ref{rn})}, then
\emph{(\ref{mrr})} can be presented in the \textquotedblleft
medial-like\textquotedblright\ form%
\begin{align}
&  \mu^{\left(  n\right)  }\left[  \tau_{R}^{\left(  n^{\prime},n^{\prime
}\right)  }\circ\left(  \left(  \Delta^{\left(  n^{\prime}\right)  }\right)
^{\otimes n^{\prime}}\right)  \circ\Delta^{\left(  n^{\prime}\right)  }\left(
b\right)  ,\overset{n-1}{\overbrace{\mathit{M}_{R}^{\left(  n^{\prime
2}\right)  },\ldots,\mathit{M}_{R}^{\left(  n^{\prime2}\right)  }}}\right]
\nonumber\\
&  =\mu^{\left(  n\right)  }\left[  \overset{n-1}{\overbrace{\mathit{M}%
_{R}^{\left(  n^{\prime2}\right)  },\ldots,\mathit{M}_{R}^{\left(  n^{\prime
2}\right)  }}},\left(  \left(  \Delta^{\left(  n^{\prime}\right)  }\right)
^{\otimes n^{\prime}}\right)  \circ\Delta^{\left(  n^{\prime}\right)  }\left(
b\right)  \right]  ,\ \ \ \ \forall b\in B, \label{mnr}%
\end{align}
where%
\begin{align}
\tau_{R}^{\left(  n^{\prime},n^{\prime}\right)  }  &  =\overset{n^{\prime}%
}{\overbrace{\tau_{op}^{\left(  n^{\prime}\right)  }\otimes\ldots\otimes
\tau_{op}^{\left(  n^{\prime}\right)  }}},\label{trn}\\
\mathit{M}_{R}^{\left(  n^{\prime2}\right)  }  &  =\overset{n^{\prime}%
}{\overbrace{\mathit{R}^{\left(  n^{\prime}\right)  }\otimes\ldots
\otimes\mathit{R}^{\left(  n^{\prime}\right)  }}}. \label{mrn}%
\end{align}

\end{theorem}

\begin{proof}
Applying (\ref{mrr}) to each Sweedler component $b_{\left[  i\right]  }$ of
$\Delta^{\left(  n^{\prime}\right)  }\left(  b\right)  $, $i=1,\ldots
,n^{\prime}$, we obtain $n^{\prime}$ relations for the polyadic almost
(semi)co-commutativity. Then multiplying them tensorially, we obtain%
\begin{align*}
&  \left(  \overset{n^{\prime}}{\overbrace{\tau_{op}^{\left(  n^{\prime
}\right)  }\otimes\ldots\otimes\tau_{op}^{\left(  n^{\prime}\right)  }}%
}\right)  \circ\left(  \Delta^{\left(  n^{\prime}\right)  }\left(  b_{\left[
1\right]  }\right)  \otimes\ldots\otimes\Delta^{\left(  n^{\prime}\right)
}\left(  b_{\left[  n^{\prime}\right]  }\right)  \right)  \circ\left(
\overset{n^{\prime}}{\overbrace{\mathit{R}^{\left(  n^{\prime}\right)
}\otimes\ldots\otimes\mathit{R}^{\left(  n^{\prime}\right)  }}}\right) \\
&  =\left(  \overset{n^{\prime}}{\overbrace{\mathit{R}^{\left(  n^{\prime
}\right)  }\otimes\ldots\otimes\mathit{R}^{\left(  n^{\prime}\right)  }}%
}\right)  \circ\left(  \Delta^{\left(  n^{\prime}\right)  }\left(  b_{\left[
1\right]  }\right)  \otimes\ldots\otimes\Delta^{\left(  n^{\prime}\right)
}\left(  b_{\left[  n^{\prime}\right]  }\right)  \right)  ,
\end{align*}
which immediately gives (\ref{mnr}). The converse statement is obvious.
\end{proof}

\begin{corollary}
Polyadic almost co-commutativity is a \textsf{particular case} of
polyadic co-mediality with the special \textquotedblleft
medial-like\textquotedblright\ twist map $\mathbf{\tau}_{R}^{\left(
n^{\prime},n^{\prime}\right)  }$ (\ref{trn}) and the composite $M$-matrix
(\ref{mrn}) consisting of $n^{\prime}$ copies of the $R$-matrix (\ref{rn}).
\end{corollary}

\begin{example}
In the binary case we compare the medial map (\ref{ti}) with the composed
\textquotedblleft medial-like\textquotedblright\ twist map (\ref{trn}) as%
\begin{align}
\mathbf{\tau}_{medial}  &  =\operatorname*{id}\nolimits_{B}\otimes
\mathbf{\tau}_{op}\otimes\operatorname*{id}\nolimits_{B},\label{tmed}\\
\mathbf{\tau}_{R}  &  =\mathbf{\tau}_{op}\otimes\mathbf{\tau}_{op},
\label{trr}%
\end{align}
or in components%
\begin{align}
&  b_{1}\otimes b_{2}\otimes b_{3}\otimes b_{4}\overset{\tau_{medial}}%
{\mapsto}b_{1}\otimes b_{3}\otimes b_{2}\otimes b_{4},\label{b1}\\
&  b_{1}\otimes b_{2}\otimes b_{3}\otimes b_{4}\overset{\tau_{R}}{\mapsto
}b_{2}\otimes b_{1}\otimes b_{4}\otimes b_{3}. \label{b2}%
\end{align}
This shows manifestly the difference between (polyadic) almost
co-commutativity and (polyadic) almost co-mediality.
\end{example}

\subsection{Equations for the $M$-matrix}

Let us find the equations for the $M$-matrix (\ref{mnt1}) using the medial
analog of the $n^{\prime}$-ary braid equation. Now the morphism of modules
$\mathbf{c}_{V_{1}\ldots V_{n^{\prime2}}}$ becomes (see for the $R$-matrix
(\ref{cvr}))%
\begin{align}
&  \mathbf{c}_{V_{1}\ldots V_{n^{\prime2}}}\circ\left(  v_{1}\otimes
\ldots\otimes v_{n^{\prime2}}\right)  =\tau_{medial,V_{1}\ldots V_{n^{\prime
2}}}^{\left(  n^{\prime},n^{\prime}\right)  }\circ\rho^{\left(  n-1\right)
}\left(  \overset{n-1}{\overbrace{\mathit{M}^{\left(  n^{\prime2}\right)
},\ldots,\mathit{M}^{\left(  n^{\prime2}\right)  }}}\mid\left(  v_{1}%
\otimes\ldots\otimes v_{n^{\prime2}}\right)  \right)  =\nonumber\\
&  \tau_{medial,V_{1}\ldots V_{n^{\prime2}}}^{\left(  n^{\prime},n^{\prime
}\right)  }\circ\left(  \sum_{\alpha_{1},\ldots\alpha_{n-1}}\rho^{\left(
n-1\right)  }\left(  \mathsf{m}_{\alpha_{1}}^{\left(  1\right)  }%
,\ldots\mathsf{m}_{\alpha_{n-1}}^{\left(  1\right)  }\mid v_{1}\right)
\otimes\ldots\otimes\rho^{\left(  n-1\right)  }\left(  \mathsf{m}_{\alpha_{1}%
}^{\left(  n^{\prime2}\right)  },\ldots\mathsf{m}_{\alpha_{n-1}}^{\left(
n^{\prime2}\right)  }\mid v_{n^{\prime2}}\right)  \right)  ,\nonumber\\
&  v_{i}\in V_{i},\ \ \ \ \mathsf{r}_{\alpha}^{\left(  i\right)  }\in
B,\ \ \ \ i=1,\ldots,n^{\prime2},
\end{align}
where $\tau_{medial,V_{1}\ldots V_{n^{\prime2}}}^{\left(  n^{\prime}%
,n^{\prime}\right)  }$ is the medial map (\ref{an}) acting on $n^{\prime2}$
modules $V_{i}$, $\mathit{M}^{\left(  n^{\prime2}\right)  }$ is the polyadic
$M$-matrix (\ref{pmm}), and $\rho^{\left(  n-1\right)  }$ is the $\left(
n-1\right)  $-place action (\ref{r}). Now instead of the $n^{\prime}$-ary
braid equation (\ref{nc}) we can have (see \textit{Remark} \ref{rem-add})

\begin{proposition}
The $n^{\prime}$-ary medial braid equation is%
\begin{align}
&  \left(  \mathbf{c}_{V^{n^{\prime2}}}\otimes\overset{n^{\prime2}%
-1}{\overbrace{\operatorname*{id}\nolimits_{V}\otimes\ldots\otimes
\operatorname*{id}\nolimits_{V}}}\right)  \circ\left(  \operatorname*{id}%
\nolimits_{V}\otimes\mathbf{c}_{V^{n^{\prime2}}}\otimes\overset{n^{\prime2}%
-2}{\overbrace{\operatorname*{id}\nolimits_{V}\otimes\ldots\otimes
\operatorname*{id}\nolimits_{V}}}\right)  \circ\ldots\nonumber\\
&  \circ\left(  \overset{n^{\prime2}-1}{\overbrace{\operatorname*{id}%
\nolimits_{V}\otimes\ldots\otimes\operatorname*{id}\nolimits_{V}}}%
\otimes\mathbf{c}_{V^{n^{\prime2}}}\right)  \circ\left(  \mathbf{c}%
_{V^{n^{\prime2}}}\otimes\overset{n^{\prime2}-1}{\overbrace{\operatorname*{id}%
\nolimits_{V}\otimes\ldots\otimes\operatorname*{id}\nolimits_{V}}}\right)
\nonumber\\
&  =\left(  \overset{n^{\prime2}-1}{\overbrace{\operatorname*{id}%
\nolimits_{V}\otimes\ldots\otimes\operatorname*{id}\nolimits_{V}}}%
\otimes\mathbf{c}_{V^{n^{\prime2}}}\right)  \circ\left(  \mathbf{c}%
_{V^{n^{\prime2}}}\otimes\overset{n^{\prime2}-1}{\overbrace{\operatorname*{id}%
\nolimits_{V}\otimes\ldots\otimes\operatorname*{id}\nolimits_{V}}}\right)
\circ\ldots\nonumber\\
&  \circ\left(  \overset{n^{\prime2}-2}{\overbrace{\operatorname*{id}%
\nolimits_{V}\otimes\ldots\otimes\operatorname*{id}\nolimits_{V}}}%
\otimes\mathbf{c}_{V^{n^{\prime2}}}\otimes\operatorname*{id}\nolimits_{V}%
\right)  \circ\left(  \overset{n^{\prime2}-1}{\overbrace{\operatorname*{id}%
\nolimits_{V}\otimes\ldots\otimes\operatorname*{id}\nolimits_{V}}}%
\otimes\mathbf{c}_{V^{n^{\prime2}}}\right)  , \label{mbe}%
\end{align}
where we use the notation $\mathbf{c}_{V^{n^{\prime2}}}\equiv\mathbf{c}%
_{V_{1}\ldots V_{n^{\prime2}}}$, $\operatorname*{id}_{V}\equiv
\operatorname*{id}_{V_{i}}$, and each side consists of $\left(  n^{\prime
2}+1\right)  $ brackets with $\left(  2n^{\prime2}-1\right)  $ multipliers.
\end{proposition}

\begin{proof}
This follows from the associative quiver technique \cite{dup2018a}.
\end{proof}

We observe that even in the binary case the medial braid equations are
cumbersome and nontrivial.

\begin{example}
In the binary case $n^{\prime}=2$ we have the map $\mathbf{c}_{V^{n^{\prime2}%
}}$ (see (\ref{tmed}), (\ref{b1}))%
\begin{equation}
\mathbf{c}_{V_{1}V_{2}V_{3}V_{4}}:V_{1}\otimes V_{2}\otimes V_{3}\otimes
V_{4}\rightarrow V_{1}\otimes V_{3}\otimes V_{2}\otimes V_{4} \label{c4}%
\end{equation}
which acts on $V_{1}\otimes V_{2}\otimes V_{3}\otimes V_{4}\otimes
V_{5}\otimes V_{6}\otimes V_{7}$. There are two medial braid equations which
correspond to diagrams of different lengths (cf. the standard braid equation
(\ref{cii}))

\textbf{1) }$V_{1}\otimes V_{2}\otimes V_{3}\otimes V_{4}\otimes V_{5}\otimes
V_{6}\otimes V_{7}\longrightarrow V_{1}\otimes V_{4}\otimes V_{5}\otimes
V_{6}\otimes V_{3}\otimes V_{2}\otimes V_{7}$ :%
\begin{align}
&  \left(  \mathbf{c}_{V_{1}V_{5}V_{4}V_{6}}\otimes\operatorname*{id}%
\nolimits_{V_{3}}\otimes\operatorname*{id}\nolimits_{V_{2}}\otimes
\operatorname*{id}\nolimits_{V_{7}}\right)  \circ\left(  \operatorname*{id}%
\nolimits_{V_{1}}\otimes\mathbf{c}_{V_{5}V_{6}V_{4}V_{3}}\otimes
\operatorname*{id}\nolimits_{V_{2}}\otimes\operatorname*{id}\nolimits_{V_{7}%
}\right) \nonumber\\
&  \circ\left(  \operatorname*{id}\nolimits_{V_{1}}\otimes\operatorname*{id}%
\nolimits_{V_{5}}\otimes\mathbf{c}_{V_{6}V_{3}V_{4}V_{2}}\otimes
\operatorname*{id}\nolimits_{V_{7}}\right)  \circ\left(  \operatorname*{id}%
\nolimits_{V_{1}}\otimes\operatorname*{id}\nolimits_{V_{5}}\otimes
\operatorname*{id}\nolimits_{V_{6}}\otimes\mathbf{c}_{V_{3}V_{2}V_{4}V_{7}%
}\right) \nonumber\\
&  \circ\left(  \mathbf{c}_{V_{1}V_{6}V_{5}V_{3}}\otimes\operatorname*{id}%
\nolimits_{V_{2}}\otimes\operatorname*{id}\nolimits_{V_{4}}\otimes
\operatorname*{id}\nolimits_{V_{7}}\right)  \circ\left(  \operatorname*{id}%
\nolimits_{V_{1}}\otimes\mathbf{c}_{V_{6}V_{3}V_{5}V_{2}}\otimes
\operatorname*{id}\nolimits_{V_{4}}\otimes\operatorname*{id}\nolimits_{V_{7}%
}\right) \nonumber\\
&  \circ\left(  \operatorname*{id}\nolimits_{V_{1}}\otimes\operatorname*{id}%
\nolimits_{V_{6}}\otimes\mathbf{c}_{V_{3}V_{2}V_{5}V_{4}}\otimes
\operatorname*{id}\nolimits_{V_{7}}\right)  \circ\left(  \operatorname*{id}%
\nolimits_{V_{1}}\otimes\operatorname*{id}\nolimits_{V_{6}}\otimes
\operatorname*{id}\nolimits_{V_{6}}\otimes\mathbf{c}_{V_{2}V_{4}V_{5}V_{7}%
}\right) \nonumber\\
&  \circ\left(  \mathbf{c}_{V_{1}V_{3}V_{6}V_{2}}\otimes\operatorname*{id}%
\nolimits_{V_{4}}\otimes\operatorname*{id}\nolimits_{V_{5}}\otimes
\operatorname*{id}\nolimits_{V_{7}}\right)  \circ\left(  \operatorname*{id}%
\nolimits_{V_{1}}\otimes\mathbf{c}_{V_{3}V_{2}V_{6}V_{4}}\otimes
\operatorname*{id}\nolimits_{V_{5}}\otimes\operatorname*{id}\nolimits_{V_{7}%
}\right) \nonumber\\
&  \circ\left(  \operatorname*{id}\nolimits_{V_{1}}\otimes\operatorname*{id}%
\nolimits_{V_{3}}\otimes\mathbf{c}_{V_{2}V_{4}V_{6}V_{5}}\otimes
\operatorname*{id}\nolimits_{V_{7}}\right)  \circ\left(  \operatorname*{id}%
\nolimits_{V_{1}}\otimes\operatorname*{id}\nolimits_{V_{3}}\otimes
\operatorname*{id}\nolimits_{V_{2}}\otimes\mathbf{c}_{V_{4}V_{5}V_{6}V_{7}%
}\right) \nonumber\\
&  \circ\left(  \mathbf{c}_{V_{1}V_{2}V_{3}V_{4}}\otimes\operatorname*{id}%
\nolimits_{V_{5}}\otimes\operatorname*{id}\nolimits_{V_{6}}\otimes
\operatorname*{id}\nolimits_{V_{7}}\right)  =\left(  \operatorname*{id}%
\nolimits_{V_{1}}\otimes\operatorname*{id}\nolimits_{V_{4}}\otimes
\operatorname*{id}\nolimits_{V_{5}}\otimes\mathbf{c}_{V_{6}V_{2}V_{3}V_{7}%
}\right) \nonumber\\
&  \circ\left(  \mathbf{c}_{V_{1}V_{5}V_{4}V_{6}}\otimes\operatorname*{id}%
\nolimits_{V_{2}}\otimes\operatorname*{id}\nolimits_{V_{3}}\otimes
\operatorname*{id}\nolimits_{V_{7}}\right)  \circ\left(  \operatorname*{id}%
\nolimits_{V_{1}}\otimes\mathbf{c}_{V_{5}V_{6}V_{4}V_{2}}\otimes
\operatorname*{id}\nolimits_{V_{3}}\otimes\operatorname*{id}\nolimits_{V_{7}%
}\right) \nonumber\\
&  \circ\left(  \operatorname*{id}\nolimits_{V_{1}}\otimes\operatorname*{id}%
\nolimits_{V_{5}}\otimes\mathbf{c}_{V_{6}V_{2}V_{4}V_{3}}\otimes
\operatorname*{id}\nolimits_{V_{7}}\right)  \circ\left(  \operatorname*{id}%
\nolimits_{V_{1}}\otimes\operatorname*{id}\nolimits_{V_{5}}\otimes
\operatorname*{id}\nolimits_{V_{6}}\otimes\mathbf{c}_{V_{2}V_{3}V_{4}V_{7}%
}\right) \nonumber\\
&  \circ\left(  \mathbf{c}_{V_{1}V_{6}V_{5}V_{2}}\otimes\operatorname*{id}%
\nolimits_{V_{3}}\otimes\operatorname*{id}\nolimits_{V_{4}}\otimes
\operatorname*{id}\nolimits_{V_{7}}\right)  \circ\left(  \operatorname*{id}%
\nolimits_{V_{1}}\otimes\mathbf{c}_{V_{6}V_{2}V_{5}V_{3}}\otimes
\operatorname*{id}\nolimits_{V_{4}}\otimes\operatorname*{id}\nolimits_{V_{7}%
}\right) \nonumber\\
&  \circ\left(  \operatorname*{id}\nolimits_{V_{1}}\otimes\operatorname*{id}%
\nolimits_{V_{6}}\otimes\mathbf{c}_{V_{2}V_{3}V_{5}V_{4}}\otimes
\operatorname*{id}\nolimits_{V_{7}}\right)  \circ\left(  \operatorname*{id}%
\nolimits_{V_{1}}\otimes\operatorname*{id}\nolimits_{V_{6}}\otimes
\operatorname*{id}\nolimits_{V_{2}}\otimes\mathbf{c}_{V_{3}V_{4}V_{5}V_{7}%
}\right) \nonumber\\
&  \circ\left(  \mathbf{c}_{V_{1}V_{2}V_{6}V_{3}}\otimes\operatorname*{id}%
\nolimits_{V_{4}}\otimes\operatorname*{id}\nolimits_{V_{5}}\otimes
\operatorname*{id}\nolimits_{V_{7}}\right)  \circ\left(  \operatorname*{id}%
\nolimits_{V_{1}}\otimes\mathbf{c}_{V_{6}V_{3}V_{5}V_{2}}\otimes
\operatorname*{id}\nolimits_{V_{4}}\otimes\operatorname*{id}\nolimits_{V_{7}%
}\right) \nonumber\\
&  \circ\left(  \operatorname*{id}\nolimits_{V_{1}}\otimes\operatorname*{id}%
\nolimits_{V_{2}}\otimes\mathbf{c}_{V_{3}V_{4}V_{6}V_{5}}\otimes
\operatorname*{id}\nolimits_{V_{7}}\right)  \circ\left(  \operatorname*{id}%
\nolimits_{V_{1}}\otimes\operatorname*{id}\nolimits_{V_{2}}\otimes
\operatorname*{id}\nolimits_{V_{3}}\otimes\mathbf{c}_{V_{4}V_{5}V_{6}V_{7}%
}\right)  . \label{ccb}%
\end{align}

\textbf{2) }$V_{1}\otimes V_{2}\otimes V_{3}\otimes V_{4}\otimes V_{5}\otimes
V_{6}\otimes V_{7}\longrightarrow V_{1}\otimes V_{6}\otimes V_{5}\otimes
V_{4}\otimes V_{2}\otimes V_{3}\otimes V_{7}\ $:%
\begin{align}
&  \left(  \operatorname*{id}\nolimits_{V_{1}}\otimes\operatorname*{id}%
\nolimits_{V_{6}}\otimes\mathbf{c}_{V_{5}V_{2}V_{4}V_{3}}\otimes
\operatorname*{id}\nolimits_{V_{7}}\right)  \circ\left(  \operatorname*{id}%
\nolimits_{V_{1}}\otimes\operatorname*{id}\nolimits_{V_{6}}\otimes
\operatorname*{id}\nolimits_{V_{5}}\otimes\mathbf{c}_{V_{2}V_{3}V_{4}V_{7}%
}\right) \nonumber\\
&  \circ\left(  \mathbf{c}_{V_{1}V_{5}V_{6}V_{2}}\otimes\operatorname*{id}%
\nolimits_{V_{3}}\otimes\operatorname*{id}\nolimits_{V_{4}}\otimes
\operatorname*{id}\nolimits_{V_{7}}\right)  \circ\left(  \operatorname*{id}%
\nolimits_{V_{1}}\otimes\mathbf{c}_{V_{5}V_{2}V_{6}V_{3}}\otimes
\operatorname*{id}\nolimits_{V_{4}}\otimes\operatorname*{id}\nolimits_{V_{7}%
}\right) \nonumber\\
&  \circ\left(  \operatorname*{id}\nolimits_{V_{1}}\otimes\operatorname*{id}%
\nolimits_{V_{5}}\otimes\mathbf{c}_{V_{2}V_{3}V_{6}V_{4}}\otimes
\operatorname*{id}\nolimits_{V_{7}}\right)  \circ\left(  \operatorname*{id}%
\nolimits_{V_{1}}\otimes\operatorname*{id}\nolimits_{V_{5}}\otimes
\operatorname*{id}\nolimits_{V_{2}}\otimes\mathbf{c}_{V_{3}V_{4}V_{6}V_{7}%
}\right) \nonumber\\
&  \circ\left(  \mathbf{c}_{V_{1}V_{2}V_{5}V_{3}}\otimes\operatorname*{id}%
\nolimits_{V_{4}}\otimes\operatorname*{id}\nolimits_{V_{6}}\otimes
\operatorname*{id}\nolimits_{V_{7}}\right)  \circ\left(  \operatorname*{id}%
\nolimits_{V_{1}}\otimes\mathbf{c}_{V_{2}V_{3}V_{5}V_{4}}\otimes
\operatorname*{id}\nolimits_{V_{6}}\otimes\operatorname*{id}\nolimits_{V_{7}%
}\right) \nonumber\\
&  \circ\left(  \operatorname*{id}\nolimits_{V_{1}}\otimes\operatorname*{id}%
\nolimits_{V_{2}}\otimes\mathbf{c}_{V_{3}V_{4}V_{5}V_{6}}\otimes
\operatorname*{id}\nolimits_{V_{7}}\right)  =\left(  \operatorname*{id}%
\nolimits_{V_{1}}\otimes\mathbf{c}_{V_{6}V_{4}V_{5}V_{2}}\otimes
\operatorname*{id}\nolimits_{V_{3}}\otimes\operatorname*{id}\nolimits_{V_{7}%
}\right) \nonumber\\
&  \circ\left(  \operatorname*{id}\nolimits_{V_{1}}\otimes\operatorname*{id}%
\nolimits_{V_{6}}\otimes\mathbf{c}_{V_{4}V_{2}V_{5}V_{3}}\otimes
\operatorname*{id}\nolimits_{V_{7}}\right)  \circ\left(  \operatorname*{id}%
\nolimits_{V_{1}}\otimes\operatorname*{id}\nolimits_{V_{6}}\otimes
\operatorname*{id}\nolimits_{V_{4}}\otimes\mathbf{c}_{V_{2}V_{3}V_{5}V_{7}%
}\right) \nonumber\\
&  \circ\left(  \mathbf{c}_{V_{1}V_{4}V_{6}V_{2}}\otimes\operatorname*{id}%
\nolimits_{V_{3}}\otimes\operatorname*{id}\nolimits_{V_{5}}\otimes
\operatorname*{id}\nolimits_{V_{7}}\right)  \circ\left(  \operatorname*{id}%
\nolimits_{V_{1}}\otimes\mathbf{c}_{V_{4}V_{2}V_{6}V_{3}}\otimes
\operatorname*{id}\nolimits_{V_{5}}\otimes\operatorname*{id}\nolimits_{V_{7}%
}\right) \nonumber\\
&  \circ\left(  \operatorname*{id}\nolimits_{V_{1}}\otimes\operatorname*{id}%
\nolimits_{V_{4}}\otimes\mathbf{c}_{V_{2}V_{3}V_{6}V_{5}}\otimes
\operatorname*{id}\nolimits_{V_{7}}\right)  \circ\left(  \operatorname*{id}%
\nolimits_{V_{1}}\otimes\operatorname*{id}\nolimits_{V_{4}}\otimes
\operatorname*{id}\nolimits_{V_{2}}\otimes\mathbf{c}_{V_{3}V_{5}V_{6}V_{7}%
}\right) \nonumber\\
&  \circ\left(  \mathbf{c}_{V_{1}V_{2}V_{4}V_{3}}\otimes\operatorname*{id}%
\nolimits_{V_{5}}\otimes\operatorname*{id}\nolimits_{V_{6}}\otimes
\operatorname*{id}\nolimits_{V_{7}}\right)  \circ\left(  \operatorname*{id}%
\nolimits_{V_{1}}\otimes\mathbf{c}_{V_{2}V_{3}V_{4}V_{5}}\otimes
\operatorname*{id}\nolimits_{V_{6}}\otimes\operatorname*{id}\nolimits_{V_{7}%
}\right)  . \label{ccbb}%
\end{align}

\end{example}

The equations for the $M$-matrix can be obtained by introducing the
\textquotedblleft extended\textquotedblright\ $M$-matrix, as in the case of
the $R$-matrix (see \textit{Remark} \ref{rem-unit}), and this can also be
possible if the $n$-ary algebra $\left\langle B\mid\mu^{\left(  n\right)
}\right\rangle $ has the unit (element) $e_{B}\in B$.

\begin{definition}
The \textquotedblleft extended\textquotedblright\ $M$-matrix is defined by
$\mathcal{M}_{i_{1}\ldots i_{n^{\prime2}}}^{\left(  2n^{\prime2}-1\right)
}\in B^{\otimes\left(  2n^{\prime2}-1\right)  }$, such that%
\begin{equation}
\mathcal{M}_{i_{1}\ldots i_{n^{\prime2}}}^{\left(  2n^{\prime2}-1\right)
}=\sum_{\alpha}e_{B}\otimes\ldots\otimes\mathsf{m}_{\alpha}^{\left(
i_{1}\right)  }\otimes\ldots\otimes\mathsf{m}_{\alpha}^{\left(  i_{n^{\prime
2}}\right)  }\otimes\ldots\otimes e_{B},\ \ \ i_{1},\ldots,i_{n^{\prime2}}%
\in\left\{  1,\ldots,2n^{\prime2}-1\right\}  \label{mni}%
\end{equation}
where $\mathsf{m}_{\alpha}^{\left(  i_{k}\right)  }$ are on the $i_{k}$-place.
\end{definition}

It is difficult to write the general compatibility equations for the
\textquotedblleft extended\textquotedblright\ $M$-matrix (\ref{mni}).

\begin{example}
In the binary case $n=n^{\prime}=2$ we have for the polyadic $M$-matrix
$\mathit{M}^{\left(  4\right)  }$ in components%
\begin{equation}
\mathit{M}^{\left(  4\right)  }=\sum_{\alpha}\mathsf{m}_{\alpha}^{\left(
1\right)  }\otimes\mathsf{m}_{\alpha}^{\left(  2\right)  }\otimes
\mathsf{m}_{\alpha}^{\left(  3\right)  }\otimes\mathsf{m}_{\alpha}^{\left(
4\right)  },\ \ \ \ \mathsf{m}_{\alpha}^{\left(  i\right)  }\in B,
\end{equation}
and $\mathcal{M}_{i_{1}\ldots i_{4}}^{\left(  7\right)  }\in B^{\otimes7}$
with%
\begin{equation}
\mathcal{M}_{i_{1}\ldots i_{4}}^{\left(  7\right)  }=\sum_{\alpha}e_{B}%
\otimes\ldots\otimes\mathsf{m}_{\alpha}^{\left(  i_{1}\right)  }\otimes
\ldots\otimes\mathsf{m}_{\alpha}^{\left(  i_{4}\right)  }\otimes\ldots\otimes
e_{B},\ \ \ i_{1},\ldots,i_{4}\in\left\{  1,\ldots,7\right\}  . \label{m7}%
\end{equation}
The map of modules $\mathbf{c}_{V_{1}\ldots V_{4}}$ (\ref{c4}) in the manifest
form is%
\begin{align}
&  \mathbf{c}_{V_{1}V_{2}V_{3}V_{4}}\circ\left(  v_{1}\otimes v_{2}\otimes
v_{3}\otimes\otimes v_{4}\right)  =\tau_{medial}\circ\rho\left(
\mathit{M}^{\left(  4\right)  }\mid\left(  v_{1}\otimes v_{2}\otimes
v_{3}\otimes v_{4}\right)  \right) \nonumber\\
&  =\tau_{medial}\circ\left(  \sum_{\alpha}\rho\left(  \mathsf{m}_{\alpha
}^{\left(  1\right)  }\mid v_{1}\right)  \otimes\rho\left(  \mathsf{m}%
_{\alpha}^{\left(  2\right)  }\mid v_{2}\right)  \otimes\rho\left(
\mathsf{m}_{\alpha}^{\left(  3\right)  }\mid v_{3}\right)  \otimes\rho\left(
\mathsf{m}_{\alpha}^{\left(  4\right)  }\mid v_{4}\right)  \right)
,\nonumber\\
&  v_{i}\in V_{i},\ \ \ \ \mathsf{m}_{\alpha}^{\left(  i\right)  }\in
B,\ \ \ \ i=1,2,3,4, \label{cv12}%
\end{align}
where $\tau_{medial}$ is the medial map (\ref{tmed}), and $\rho:B\times
V\rightarrow V$ is the ordinary $1$-place action (\ref{r}).

After inserting (\ref{cv12}) into (\ref{ccb}) and (\ref{ccbb}), using
(\ref{m7}) we obtain the equations for $M$-matrix%
\begin{align}
&  \mathcal{M}_{1546}^{\left(  7\right)  }\mathcal{M}_{5643}^{\left(
7\right)  }\mathcal{M}_{6342}^{\left(  7\right)  }\mathcal{M}_{3247}^{\left(
7\right)  }\mathcal{M}_{1653}^{\left(  7\right)  }\mathcal{M}_{6352}^{\left(
7\right)  }\mathcal{M}_{3254}^{\left(  7\right)  }\mathcal{M}_{2457}^{\left(
7\right)  }\mathcal{M}_{1362}^{\left(  7\right)  }\mathcal{M}_{3264}^{\left(
7\right)  }\mathcal{M}_{2465}^{\left(  7\right)  }\mathcal{M}_{4567}^{\left(
7\right)  }\mathcal{M}_{1234}^{\left(  7\right)  }\nonumber\\
&  =\mathcal{M}_{6237}^{\left(  7\right)  }\mathcal{M}_{1546}^{\left(
7\right)  }\mathcal{M}_{5642}^{\left(  7\right)  }\mathcal{M}_{6243}^{\left(
7\right)  }\mathcal{M}_{2347}^{\left(  7\right)  }\mathcal{M}_{1652}^{\left(
7\right)  }\mathcal{M}_{6253}^{\left(  7\right)  }\mathcal{M}_{2354}^{\left(
7\right)  }\mathcal{M}_{3457}^{\left(  7\right)  }\mathcal{M}_{1263}^{\left(
7\right)  }\mathcal{M}_{2364}^{\left(  7\right)  }\mathcal{M}_{3465}^{\left(
7\right)  }\mathcal{M}_{4567}^{\left(  7\right)  } \label{m1}%
\end{align}
and%
\begin{align}
&  \mathcal{M}_{5243}^{\left(  7\right)  }\mathcal{M}_{2347}^{\left(
7\right)  }\mathcal{M}_{1562}^{\left(  7\right)  }\mathcal{M}_{5263}^{\left(
7\right)  }\mathcal{M}_{2364}^{\left(  7\right)  }\mathcal{M}_{3467}^{\left(
7\right)  }\mathcal{M}_{1253}^{\left(  7\right)  }\mathcal{M}_{2354}^{\left(
7\right)  }\mathcal{M}_{3456}^{\left(  7\right)  }\nonumber\\
&  =\mathcal{M}_{6452}^{\left(  7\right)  }\mathcal{M}_{4253}^{\left(
7\right)  }\mathcal{M}_{2357}^{\left(  7\right)  }\mathcal{M}_{1462}^{\left(
7\right)  }\mathcal{M}_{4263}^{\left(  7\right)  }\mathcal{M}_{2365}^{\left(
7\right)  }\mathcal{M}_{3567}^{\left(  7\right)  }\mathcal{M}_{1243}^{\left(
7\right)  }\mathcal{M}_{2345}^{\left(  7\right)  }, \label{m2}%
\end{align}
which respect the braid equations (\ref{ccb}) and (\ref{ccbb}).
\end{example}

\begin{remark}
The unequal number of terms in (\ref{m1}) and (\ref{m2}) is governed by
different commutative diagrams of modules (\ref{ccb}) and (\ref{ccbb}),
respectively (cf. (\ref{rrr}) and (\ref{cv1})--(\ref{cii})).
\end{remark}

\subsection{Medial analog of triangularity}

Now we consider the possible analogs of the quasitriangularity conditions
(similar to (\ref{drr1})--(\ref{drr2}) and quasipolyangularity (\ref{qp1}%
--(\ref{qp3})) for a polyadic almost co-medial bialgebra $\mathrm{B}^{\left(
n^{\prime},n\right)  }$ (see \textbf{Definition} \ref{def-medbi}).

\begin{definition}
A polyadic almost co-medial bialgebra $\mathrm{B}^{\left(  n^{\prime
},n\right)  }=\left\langle B\mid\mu^{\left(  n\right)  },\Delta^{\left(
n^{\prime}\right)  }\right\rangle $ with the polyadic $M$-matrix
$\mathit{M}^{\left(  n^{\prime2}\right)  }=\sum_{\alpha}\mathsf{m}_{\alpha
}^{\left(  1\right)  }\otimes\ldots\otimes\mathsf{m}_{\alpha}^{\left(
n^{\prime2}\right)  },\ \ \ \ \mathsf{m}_{\alpha}^{\left(  i\right)  }\in B$
is called \textit{medial quasipolyangular}, if the following $n^{\prime2}$
relations hold%
\begin{align}
&  \sum_{\alpha}\left(  \Delta^{\left(  n^{\prime}\right)  }\right)  ^{\otimes
n^{\prime}}\circ\Delta^{\left(  n^{\prime}\right)  }\left(  \mathsf{m}%
_{\alpha}^{\left(  1\right)  }\right)  \otimes\mathsf{m}_{\alpha}^{\left(
2\right)  }\otimes\ldots\otimes\mathsf{m}_{\alpha}^{\left(  n^{\prime
2}\right)  }=\sum_{\alpha_{1},\ldots\alpha_{n^{\prime2}}}\mathsf{m}%
_{\alpha_{1}}^{\left(  1\right)  }\otimes\mathsf{m}_{\alpha_{2}}^{\left(
1\right)  }\otimes\ldots\otimes\mathsf{m}_{\alpha_{n^{\prime2}}}^{\left(
1\right)  }\nonumber\\
&  \otimes\left(  \mu^{\left(  n\right)  }\right)  ^{\circ\ell}\left[
\mathsf{m}_{\alpha_{1}}^{\left(  2\right)  }\otimes\mathsf{m}_{\alpha_{2}%
}^{\left(  2\right)  }\otimes\ldots\otimes\mathsf{m}_{\alpha_{n^{\prime2}}%
}^{\left(  2\right)  }\right]  \otimes\ldots\otimes\left(  \mu^{\left(
n\right)  }\right)  ^{\circ\ell}\left[  \mathsf{m}_{\alpha_{1}}^{\left(
n^{\prime2}\right)  }\otimes\mathsf{m}_{\alpha_{2}}^{\left(  n^{\prime
2}\right)  }\otimes\ldots\otimes\mathsf{m}_{\alpha_{n^{\prime2}}}^{\left(
n^{\prime2}\right)  }\right]  , \label{mqp1}%
\end{align}%
\begin{align}
&  \sum_{\alpha}\mathsf{m}_{\alpha}^{\left(  1\right)  }\otimes\left(
\Delta^{\left(  n^{\prime}\right)  }\right)  ^{\otimes n^{\prime}}\circ
\Delta^{\left(  n^{\prime}\right)  }\left(  \mathsf{m}_{\alpha}^{\left(
2\right)  }\right)  \otimes\ldots\otimes\mathsf{m}_{\alpha}^{\left(
n^{\prime2}\right)  }\nonumber\\
&  =\sum_{\alpha_{1},\ldots\alpha_{n^{\prime2}}}\left(  \mu^{\left(  n\right)
}\right)  ^{\circ\ell}\circ\tau_{medial}^{\left(  n^{\prime2},n^{\prime
2}\right)  }\left[  \mathsf{m}_{\alpha_{1}}^{\left(  1\right)  }%
\otimes\mathsf{m}_{\alpha_{2}}^{\left(  1\right)  }\otimes\ldots
\otimes\mathsf{m}_{\alpha_{n^{\prime2}}}^{\left(  1\right)  }\right]
\otimes\mathsf{m}_{\alpha_{1}}^{\left(  2\right)  }\otimes\mathsf{m}%
_{\alpha_{2}}^{\left(  2\right)  }\otimes\ldots\otimes\mathsf{m}%
_{\alpha_{n^{\prime2}}}^{\left(  2\right)  }\otimes\ldots\nonumber\\
&  \otimes\left(  \mu^{\left(  n\right)  }\right)  ^{\circ\ell}\left[
\mathsf{m}_{\alpha_{1}}^{\left(  n^{\prime2}\right)  }\otimes\mathsf{m}%
_{\alpha_{2}}^{\left(  n^{\prime2}\right)  }\otimes\ldots\otimes
\mathsf{m}_{\alpha_{n^{\prime2}}}^{\left(  n^{\prime2}\right)  }\right]  ,
\label{mqp2}%
\end{align}%
\[
\vdots
\]%
\begin{align}
&  \sum_{\alpha}\mathsf{m}_{\alpha}^{\left(  1\right)  }\otimes\ldots
\otimes\mathsf{m}_{\alpha}^{\left(  n^{\prime2}-1\right)  }\otimes\left(
\Delta^{\left(  n^{\prime}\right)  }\left(  \mathsf{m}_{\alpha}^{\left(
n^{\prime2}\right)  }\right)  \right)  ^{\otimes n^{\prime}}\circ
\Delta^{\left(  n^{\prime}\right)  }\left(  \mathsf{m}_{\alpha}^{\left(
n^{\prime2}\right)  }\right) \nonumber\\
&  =\sum_{\alpha_{1},\ldots\alpha_{n^{\prime2}}}\left(  \mu^{\left(  n\right)
}\right)  ^{\circ\ell}\circ\tau_{medial}^{\left(  n^{\prime2},n^{\prime
2}\right)  }\left[  \mathsf{m}_{\alpha_{1}}^{\left(  1\right)  }%
\otimes\mathsf{m}_{\alpha_{2}}^{\left(  1\right)  }\otimes\ldots
\otimes\mathsf{m}_{\alpha_{n^{\prime2}}}^{\left(  1\right)  }\right]
\otimes\ldots\nonumber\\
&  \otimes\left(  \mu^{\left(  n\right)  }\right)  ^{\circ\ell}\circ
\tau_{medial}^{\left(  n^{\prime2},n^{\prime2}\right)  }\left[  \mathsf{m}%
_{\alpha_{1}}^{\left(  n^{\prime2}-1\right)  }\otimes\mathsf{m}_{\alpha_{2}%
}^{\left(  n^{\prime2}-1\right)  }\otimes\ldots\otimes\mathsf{m}%
_{\alpha_{n^{\prime2}}}^{\left(  n^{\prime2}-1\right)  }\right]
\otimes\mathsf{m}_{\alpha_{1}}^{\left(  n^{\prime2}\right)  }\otimes
\mathsf{m}_{\alpha_{2}}^{\left(  n^{\prime2}\right)  }\otimes\ldots
\otimes\mathsf{m}_{\alpha_{n^{\prime2}}}^{\left(  n^{\prime2}\right)  },
\label{mqp3}%
\end{align}
where $\tau_{medial}^{\left(  n^{\prime2},n^{\prime2}\right)  }$ is the
\textsf{unique} medial twist map (\ref{an}). The arity shape of a medial
quasipolyangular bialgebra $\mathrm{B}^{\left(  n^{\prime},n\right)  }$ is
given by (cf. (\ref{nln}))%
\begin{equation}
n^{\prime2}=\ell\left(  n-1\right)  +1,\ \ \ \ell\in\mathbb{N}. \label{n2l}%
\end{equation}

\end{definition}

\begin{remark}
\label{rem-qm} Similar to \textit{Remark} \ref{rem-qr}, the medial
quasipolyangularity equations (\ref{mqp1})--(\ref{mqp3}) can be expressed in
terms of the extended $M$-matrix for the first equation (\ref{mqp1}) and the
last one (\ref{mqp3}) only, because in the intermediate equations the
sequences of $M$-matrix elements are permuted.
\end{remark}

\begin{example}
In the case where $n^{\prime}=n=2$, $\ell=3$, for the bialgebra $\mathrm{B}%
^{\left(  2,2\right)  }=\left\langle B\mid\mu=\left(  \cdot\right)
,\Delta\right\rangle $ with the polyadic $M$-matrix
\begin{equation}
\mathit{M}^{\left(  4\right)  }=\sum_{\alpha}\mathsf{m}_{\alpha}^{\left(
1\right)  }\otimes\mathsf{m}_{\alpha}^{\left(  2\right)  }\otimes
\mathsf{m}_{\alpha}^{\left(  3\right)  }\otimes\mathsf{m}_{\alpha}^{\left(
4\right)  },\ \ \ \mathsf{m}_{\alpha}^{\left(  i\right)  }\in B \label{m4m}%
\end{equation}
we have the binary medial quasipolyangularity equations%
\begin{align}
&  \sum_{\alpha}\left(  \Delta\otimes\Delta\right)  \circ\Delta\left(
\mathsf{m}_{\alpha}^{\left(  1\right)  }\right)  \otimes\mathsf{m}_{\alpha
}^{\left(  2\right)  }\otimes\mathsf{m}_{\alpha}^{\left(  3\right)  }%
\otimes\mathsf{m}_{\alpha}^{\left(  4\right)  }=\sum_{\alpha_{1},\alpha
_{2},\alpha_{3},\alpha_{4}}\mathsf{m}_{\alpha_{1}}^{\left(  1\right)  }%
\otimes\mathsf{m}_{\alpha_{2}}^{\left(  1\right)  }\otimes\mathsf{m}%
_{\alpha_{3}}^{\left(  1\right)  }\otimes\mathsf{m}_{\alpha_{4}}^{\left(
1\right)  }\nonumber\\
&  \otimes\mathsf{m}_{\alpha_{1}}^{\left(  2\right)  }\cdot\mathsf{m}%
_{\alpha_{2}}^{\left(  2\right)  }\cdot\mathsf{m}_{\alpha_{3}}^{\left(
2\right)  }\cdot\mathsf{m}_{\alpha_{4}}^{\left(  2\right)  }\otimes
\mathsf{m}_{\alpha_{1}}^{\left(  3\right)  }\cdot\mathsf{m}_{\alpha_{2}%
}^{\left(  3\right)  }\cdot\mathsf{m}_{\alpha_{3}}^{\left(  3\right)  }%
\cdot\mathsf{m}_{\alpha_{4}}^{\left(  3\right)  }\otimes\mathsf{m}_{\alpha
_{1}}^{\left(  4\right)  }\cdot\mathsf{m}_{\alpha_{2}}^{\left(  4\right)
}\cdot\mathsf{m}_{\alpha_{3}}^{\left(  4\right)  }\cdot\mathsf{m}_{\alpha_{4}%
}^{\left(  4\right)  }, \label{dd1}%
\end{align}%
\begin{align}
&  \sum_{\alpha}\mathsf{m}_{\alpha}^{\left(  1\right)  }\otimes\left(
\Delta\otimes\Delta\right)  \circ\Delta\left(  \mathsf{m}_{\alpha}^{\left(
2\right)  }\right)  \otimes\mathsf{m}_{\alpha}^{\left(  3\right)  }%
\otimes\mathsf{m}_{\alpha}^{\left(  4\right)  }=\sum_{\alpha_{1},\alpha
_{2},\alpha_{3},\alpha_{4}}\mathsf{m}_{\alpha_{1}}^{\left(  1\right)  }%
\cdot\mathsf{m}_{\alpha_{3}}^{\left(  1\right)  }\cdot\mathsf{m}_{\alpha_{2}%
}^{\left(  1\right)  }\cdot\mathsf{m}_{\alpha_{4}}^{\left(  1\right)
}\nonumber\\
&  \otimes\mathsf{m}_{\alpha_{1}}^{\left(  2\right)  }\otimes\mathsf{m}%
_{\alpha_{2}}^{\left(  2\right)  }\otimes\mathsf{m}_{\alpha_{3}}^{\left(
2\right)  }\otimes\mathsf{m}_{\alpha_{4}}^{\left(  2\right)  }\otimes
\mathsf{m}_{\alpha_{1}}^{\left(  3\right)  }\cdot\mathsf{m}_{\alpha_{2}%
}^{\left(  3\right)  }\cdot\mathsf{m}_{\alpha_{3}}^{\left(  3\right)  }%
\cdot\mathsf{m}_{\alpha_{4}}^{\left(  3\right)  }\otimes\mathsf{m}_{\alpha
_{1}}^{\left(  4\right)  }\cdot\mathsf{m}_{\alpha_{2}}^{\left(  4\right)
}\cdot\mathsf{m}_{\alpha_{3}}^{\left(  4\right)  }\cdot\mathsf{m}_{\alpha_{4}%
}^{\left(  4\right)  }, \label{dd2}%
\end{align}%
\begin{align}
&  \sum_{\alpha}\mathsf{m}_{\alpha}^{\left(  1\right)  }\otimes\mathsf{m}%
_{\alpha}^{\left(  2\right)  }\otimes\left(  \Delta\otimes\Delta\right)
\circ\Delta\left(  \mathsf{m}_{\alpha}^{\left(  3\right)  }\right)
\otimes\mathsf{m}_{\alpha}^{\left(  4\right)  }=\sum_{\alpha_{1},\alpha
_{2},\alpha_{3},\alpha_{4}}\mathsf{m}_{\alpha_{1}}^{\left(  1\right)  }%
\cdot\mathsf{m}_{\alpha_{3}}^{\left(  1\right)  }\cdot\mathsf{m}_{\alpha_{2}%
}^{\left(  1\right)  }\cdot\mathsf{m}_{\alpha_{4}}^{\left(  1\right)
}\nonumber\\
&  \otimes\mathsf{m}_{\alpha_{1}}^{\left(  2\right)  }\cdot\mathsf{m}%
_{\alpha_{3}}^{\left(  2\right)  }\cdot\mathsf{m}_{\alpha_{2}}^{\left(
2\right)  }\cdot\mathsf{m}_{\alpha_{4}}^{\left(  2\right)  }\otimes
\mathsf{m}_{\alpha_{1}}^{\left(  3\right)  }\otimes\mathsf{m}_{\alpha_{2}%
}^{\left(  3\right)  }\otimes\mathsf{m}_{\alpha_{3}}^{\left(  3\right)
}\otimes\mathsf{m}_{\alpha_{4}}^{\left(  3\right)  }\otimes\mathsf{m}%
_{\alpha_{1}}^{\left(  4\right)  }\cdot\mathsf{m}_{\alpha_{2}}^{\left(
4\right)  }\cdot\mathsf{m}_{\alpha_{3}}^{\left(  4\right)  }\cdot
\mathsf{m}_{\alpha_{4}}^{\left(  4\right)  }, \label{dd3}%
\end{align}%
\begin{align}
&  \sum_{\alpha}\mathsf{m}_{\alpha}^{\left(  1\right)  }\otimes\mathsf{m}%
_{\alpha}^{\left(  2\right)  }\otimes\mathsf{m}_{\alpha}^{\left(  3\right)
}\otimes\left(  \Delta\otimes\Delta\right)  \circ\Delta\left(  \mathsf{m}%
_{\alpha}^{\left(  4\right)  }\right)  =\sum_{\alpha_{1},\alpha_{2},\alpha
_{3},\alpha_{4}}\mathsf{m}_{\alpha_{1}}^{\left(  1\right)  }\cdot
\mathsf{m}_{\alpha_{3}}^{\left(  1\right)  }\cdot\mathsf{m}_{\alpha_{2}%
}^{\left(  1\right)  }\cdot\mathsf{m}_{\alpha_{4}}^{\left(  1\right)
}\nonumber\\
&  \otimes\mathsf{m}_{\alpha_{1}}^{\left(  2\right)  }\cdot\mathsf{m}%
_{\alpha_{3}}^{\left(  2\right)  }\cdot\mathsf{m}_{\alpha_{2}}^{\left(
2\right)  }\cdot\mathsf{m}_{\alpha_{4}}^{\left(  2\right)  }\otimes
\mathsf{m}_{\alpha_{1}}^{\left(  3\right)  }\cdot\mathsf{m}_{\alpha_{3}%
}^{\left(  3\right)  }\cdot\mathsf{m}_{\alpha_{2}}^{\left(  3\right)  }%
\cdot\mathsf{m}_{\alpha_{4}}^{\left(  3\right)  }\otimes\mathsf{m}_{\alpha
_{1}}^{\left(  4\right)  }\otimes\mathsf{m}_{\alpha_{2}}^{\left(  4\right)
}\otimes\mathsf{m}_{\alpha_{3}}^{\left(  4\right)  }\otimes\mathsf{m}%
_{\alpha_{4}}^{\left(  4\right)  }. \label{dd4}%
\end{align}

According to \textit{Remark} \ref{rem-qm}, we can express through the extended
$M$-matrix (\ref{m7}) the first medial quasipolyangularity equation
(\ref{dd1}) and the last one (\ref{dd4}) only, as follows%
\begin{align}
\left(  \left(  \Delta\otimes\Delta\right)  \circ\Delta\otimes
\operatorname*{id}\nolimits_{B}\otimes\operatorname*{id}\nolimits_{B}%
\otimes\operatorname*{id}\nolimits_{B}\right)  \left(  \mathit{M}^{\left(
4\right)  }\right)   &  =\mathcal{M}_{1567}^{\left(  7\right)  }%
\cdot\mathcal{M}_{2567}^{\left(  7\right)  }\cdot\mathcal{M}_{3567}^{\left(
7\right)  }\cdot\mathcal{M}_{4567}^{\left(  7\right)  },\label{ddd1}\\
\left(  \operatorname*{id}\nolimits_{B}\otimes\operatorname*{id}%
\nolimits_{B}\otimes\operatorname*{id}\nolimits_{B}\otimes\left(
\Delta\otimes\Delta\right)  \circ\Delta\right)  \left(  \mathit{M}^{\left(
4\right)  }\right)   &  =\mathcal{M}_{1234}^{\left(  7\right)  }%
\cdot\mathcal{M}_{1236}^{\left(  7\right)  }\cdot\mathcal{M}_{1235}^{\left(
7\right)  }\cdot\mathcal{M}_{1237}^{\left(  7\right)  }. \label{ddd2}%
\end{align}

\end{example}

The compatibility of (\ref{ddd1})--(\ref{ddd2}) with the (binary) almost
co-mediality (\ref{mnt2}) leads to

\begin{proposition}
An extended binary $M$-matrix \emph{(\ref{m7})} of the binary almost co-medial
bialgebra $\mathrm{B}^{\left(  2,2\right)  }=\left\langle B\mid\mu=\left(
\cdot\right)  ,\Delta\right\rangle $ satisfies the compatibility equations
\emph{(}cf. \emph{(\ref{rrr}))}%
\begin{align}
\mathcal{M}_{1234}^{\left(  7\right)  }\cdot\mathcal{M}_{1567}^{\left(
7\right)  }\cdot\mathcal{M}_{2567}^{\left(  7\right)  }\cdot\mathcal{M}%
_{3567}^{\left(  7\right)  }\cdot\mathcal{M}_{4567}^{\left(  7\right)  }  &
=\mathcal{M}_{1567}^{\left(  7\right)  }\cdot\mathcal{M}_{3567}^{\left(
7\right)  }\cdot\mathcal{M}_{2567}^{\left(  7\right)  }\cdot\mathcal{M}%
_{4567}^{\left(  7\right)  }\cdot\mathcal{M}_{1234}^{\left(  7\right)
},\label{mm1}\\
\mathcal{M}_{4567}^{\left(  7\right)  }\cdot\mathcal{M}_{1234}^{\left(
7\right)  }\cdot\mathcal{M}_{1236}^{\left(  7\right)  }\cdot\mathcal{M}%
_{1235}^{\left(  7\right)  }\cdot\mathcal{M}_{1237}^{\left(  7\right)  }  &
=\mathcal{M}_{1234}^{\left(  7\right)  }\cdot\mathcal{M}_{1235}^{\left(
7\right)  }\cdot\mathcal{M}_{1236}^{\left(  7\right)  }\cdot\mathcal{M}%
_{1237}^{\left(  7\right)  }\cdot\mathcal{M}_{4567}^{\left(  7\right)  }.
\label{mm2}%
\end{align}

\end{proposition}

\begin{proof}
The identities for the $M$-matrix%
\begin{align}
&  \left(  \mathit{M}^{\left(  4\right)  }\otimes\operatorname*{id}%
\nolimits_{B}\otimes\operatorname*{id}\nolimits_{B}\otimes\operatorname*{id}%
\nolimits_{B}\right)  \circ\left(  \left(  \Delta\otimes\Delta\right)
\circ\Delta\otimes\operatorname*{id}\nolimits_{B}\otimes\operatorname*{id}%
\nolimits_{B}\otimes\operatorname*{id}\nolimits_{B}\right)  \left(
\mathit{M}^{\left(  4\right)  }\right) \nonumber\\
&  =\left(  \tau_{medial}\circ\left(  \Delta\otimes\Delta\right)  \circ
\Delta\otimes\operatorname*{id}\nolimits_{B}\otimes\operatorname*{id}%
\nolimits_{B}\otimes\operatorname*{id}\nolimits_{B}\right)  \left(
\mathit{M}^{\left(  4\right)  }\right)  \circ\left(  \mathit{M}^{\left(
4\right)  }\otimes\operatorname*{id}\nolimits_{B}\otimes\operatorname*{id}%
\nolimits_{B}\otimes\operatorname*{id}\nolimits_{B}\right)  ,
\end{align}%
\begin{align}
&  \left(  \operatorname*{id}\nolimits_{B}\otimes\operatorname*{id}%
\nolimits_{B}\otimes\operatorname*{id}\nolimits_{B}\otimes\mathit{M}^{\left(
4\right)  }\right)  \circ\left(  \operatorname*{id}\nolimits_{B}%
\otimes\operatorname*{id}\nolimits_{B}\otimes\operatorname*{id}\nolimits_{B}%
\otimes\left(  \Delta\otimes\Delta\right)  \circ\Delta\right)  \left(
\mathit{M}^{\left(  4\right)  }\right) \nonumber\\
&  =\left(  \operatorname*{id}\nolimits_{B}\otimes\operatorname*{id}%
\nolimits_{B}\otimes\operatorname*{id}\nolimits_{B}\otimes\tau_{medial}%
\circ\left(  \Delta\otimes\Delta\right)  \circ\Delta\right)  \left(
\mathit{M}^{\left(  4\right)  }\right)  \circ\left(  \operatorname*{id}%
\nolimits_{B}\otimes\operatorname*{id}\nolimits_{B}\otimes\operatorname*{id}%
\nolimits_{B}\otimes\mathit{M}^{\left(  4\right)  }\right)  ,
\end{align}
follow from the almost co-mediality condition(\ref{mnt2}), and then we apply
quasipolyangularity (\ref{ddd1})--(\ref{ddd2}).
\end{proof}

\begin{remark}
Two other compatibility equations corresponding to the intermediate
quasipolyangularity equations (\ref{dd2})--(\ref{dd3}) can be written in
component form only (see \textit{Remark} \ref{rem-qm}).
\end{remark}

The solutions to (\ref{mm1})--(\ref{mm2}) can be found in matrix form by
choosing an appropriate basis and using the standard methods (see, e.g.,
\cite{kassel,lam/rad}).

\section{\textsc{Conclusions}}

We have presented the \textquotedblleft polyadization\textquotedblright\ procedure
of the following algebra-like structures: algebras, coalgebras, bialgebras and Hopf
algebras (see \cite{dup2017a,dup2017} for ring-like structures). In our
concrete constructions the initial arities of operations are taken as
arbitrary, and we then try to restrict them only by means of natural relations which
bring to mind the binary case. This leads to many exotic properties and
unexpected connections between arities and a fixing of their values called
\textquotedblleft quantization\textquotedblright. For instance, the unit and
counit (which do not always exist) can be multivalued many place maps,
polyadic algebras can be zeroless, the qeurelements should be considered
instead of inverse elements under addition and multiplication, a polyadic
bialgebra can consist of an algebra and coalgebra of different arities, and a
polyadic analog of Hopf algebras contains (instead of the ordinary antipode)
the querantipode, which has different properties.

The formulas and constructions introduced for concrete algebra-like structures
can have many applications, e.g., in combinatorics, quantum logic,or
representation theory. As an example, we have introduced possible polyadic analogs
of braidings, almost co-commutativity and a version of the $R$-matrix. A new concept of
deformation (using the medial map) is proposed: this is unique and therefore can be
more consequential and suitable in the polyadic case.


\mbox{}
\bigskip

\newpage
\clearpage
\thispagestyle{plain}
\mbox{}

\pagestyle{fancyref}
\listoftables
\end{document}